\documentclass[a4paper,reqno]{amsart}
\usepackage[english]{babel}
\usepackage{amsmath,amsthm,amssymb,amsfonts}
\usepackage{subcaption}
\usepackage{tikz}
\usetikzlibrary{decorations.pathmorphing,shapes}

\usepackage{enumitem}
\setlist[enumerate]{leftmargin=*}
\setlist[itemize]{leftmargin=*}

\usepackage[unicode]{hyperref}
\newcommand\ringring[1]{%
  {
   \mathop{\kern0pt #1}\limits^{
     \vbox to-1.85ex{
       \kern-2ex 
       \hbox to 0pt{\hss\normalfont\kern.1em \r{}\kern-.45em \r{}\hss}%
       \vss 
     }
   }
  }
}

\DeclareFontFamily{U}{matha}{}
\DeclareFontShape{U}{matha}{m}{n}{
  <-5.5>    matha5
  <5.5-6.5> matha6 
  <6.5-7.5> matha7
  <7.5-8.5> matha8
  <8.5-9.5> matha9
  <9.5-11>  matha10
  <11->     matha12
}{}
\DeclareSymbolFont{matha}{U}{matha}{m}{n}
\DeclareFontSubstitution{U}{matha}{m}{n}
\DeclareFontFamily{U}{mathx}{\hyphenchar\font45}
\DeclareFontShape{U}{mathx}{m}{n}{<-> mathx10}{}
\DeclareSymbolFont{mathx}{U}{mathx}{m}{n}
\DeclareFontSubstitution{U}{mathx}{m}{n}

\DeclareMathDelimiter{\ldbrack}{4}{matha}{"76}{mathx}{"30}
\DeclareMathDelimiter{\rdbrack}{5}{matha}{"77}{mathx}{"38}

\newcommand{\smallo}{\mathfrak{o}}

\newcommand{\dist}{\mathop{\mathrm{dist}}}

\newcommand{\Mthr}{\varsigma_{\mathrm{mult}}}
\newcommand{\Wthr}{\varsigma_{\mathrm{wave}}}

\renewcommand{\mho}{\mathring{\Lambda}}
\newcommand{\mhoI}{\ringring{\Lambda}}

\newcommand{\inj}{\mathrm{inj}}

\newcommand{\lie}{\mathfrak}

\newcommand{\RR}{\mathbb{R}}
\newcommand{\CC}{\mathbb{C}}
\newcommand{\ZZ}{\mathbb{Z}}
\newcommand{\NN}{\mathbb{N}}

\newcommand{\dil}{\mathfrak{D}}
\newcommand{\ltr}{\mathfrak{l}}
\newcommand{\dltr}{\bar{\mathfrak{l}}}

\newcommand{\opL}{\mathcal{L}}

\newcommand{\Sz}{\mathcal{S}}

\newcommand{\Four}{\mathcal{F}}

\newcommand{\Rnon}{\RR^+_0}
\newcommand{\Rpos}{\RR^+}
\newcommand{\Npos}{\NN^+}

\newcommand{\chr}{\mathbf{1}}

\newcommand{\Den}{\mathfrak{d}}

\DeclareMathOperator{\supp}{supp}

\DeclareMathOperator{\Div}{div}

\DeclareMathOperator{\tr}{tr}

\DeclareMathOperator{\sgn}{sgn}

\DeclareMathOperator{\Ker}{ker}

\DeclareMathOperator{\rk}{rank}

\newcommand{\tc}{\,:\,}

\newcommand{\defeq}{\mathrel{:=}}

\newcommand{\Ham}{\mathcal{H}}
\newcommand{\HamA}{\mathcal{A}}
\newcommand{\DHam}{\mathcal{D}_\Ham}

\newcommand{\cQ}{\mathcal{Q}}

\newcommand{\cA}{\mathcal{A}}
\newcommand{\cB}{\mathcal{B}}

\newcommand{\bx}{\boldsymbol{x}}
\newcommand{\bxi}{\boldsymbol{\xi}}

\newcommand{\by}{\boldsymbol{y}}

\renewcommand{\Re}{\mathop{\mathfrak{Re}}}
\renewcommand{\Im}{\mathop{\mathfrak{Im}}}

\newtheorem{thm}{Theorem}[section]
\newtheorem{prp}[thm]{Proposition}
\newtheorem{lem}[thm]{Lemma}
\newtheorem{cor}[thm]{Corollary}
\theoremstyle{definition}
\newtheorem{rem}[thm]{Remark}

\numberwithin{equation}{section}

\title[Wave equation on M\'etivier groups]{An FIO-based approach to $L^p$-bounds \\ for the wave equation on $2$-step Carnot groups:\\ the case of M\'etivier groups}
\author{Alessio Martini}
\address[A. Martini]{Dipartimento di Scienze Matematiche \\ Politecnico di Torino \\ Corso Duca degli Abruzzi 24 \\ 10129 Torino \\ Italy}
\email{alessio.martini@polito.it}
\author{Detlef M\"uller}
\address[D. M\"uller]{Mathematisches Seminar \\ C.-A.-Universit\"at zu Kiel \\ Heinrich-Hecht-Platz 6 \\ 24118 Kiel \\ Germany}
\email{mueller@math.uni-kiel.de}

\thanks{The first-named author gratefully acknowledges the financial support of Compagnia di San Paolo; he is a member of Gruppo Nazionale per l'Analisi Matematica, la Probabilit\`a e le loro Applicazioni (GNAMPA) of Istituto Nazionale di Alta Matematica (INdAM). The second-named author equally acknowledges the support by the Deutsche Forschungsgemeinschaft under DFG-Grant MU 761/12-1}


\begin{document}
\begin{abstract}
Let $\opL$ be a homogeneous left-invariant sub-Laplacian on a $2$-step Carnot group. We devise a new geometric approach to sharp fixed-time $L^p$-bounds with loss of derivatives for the wave equation driven by $\opL$, based on microlocal analysis and highlighting the role of the underlying sub-Riemannian geodesic flow.
A major challenge here stems from the fact that, differently from the Riemannian setting, the conjugate locus of a point on a sub-Riemannian manifold may cluster at the point itself, thus making it indispensable to deal with caustics even when studying small-time wave propagation.

Our analysis of the wave propagator on a $2$-step Carnot group allows us to essentially reduce microlocally to two conic regions in frequency space: an \emph{anti-FIO region}, which seems not amenable to FIO techniques, and an \emph{FIO region}. For the latter, we construct a parametrix by means of FIOs with complex phase, by suitably adapting a construction from the elliptic setting due to Laptev, Safarov and Vassiliev, which remains valid beyond caustics. A substantial problem arising here is that, after a natural decomposition and subsequent scalings, one effectively needs to deal with the long-time behaviour and according control of $L^1$-norms of the corresponding contributions to the wave propagator, a new phenomenon that is specific to sub-elliptic settings.

For the class of M\'etivier groups, we show how our approach, in combination with a variation of the key method of Seeger, Sogge and Stein for proving Lebesgue-space estimates for FIOs, yields $L^p$-bounds for the wave equation, which are sharp up to the endpoint regularity. In particular, we extend previously known results for distinguished sub-Laplacians on groups of Heisenberg type, by means of a more general and robust approach. The study of the wave equation on wider classes of $2$-step Carnot groups via this approach will pose further challenges that we plan to address in subsequent works.
\end{abstract}

\maketitle

\tableofcontents

\section{Introduction}

\subsection{Statement of the results}

Let $G$ be a stratified Lie group. In other words, $G$ is a connected, simply connected nilpotent Lie group, whose Lie algebra $\lie{g}$ decomposes as a direct sum $\lie{g}_1 \oplus \lie{g}_2 \oplus \dots \oplus \lie{g}_k$ of nonzero subspaces, called layers, such that $[\lie{g}_1,\lie{g}_j] = \lie{g}_{j+1}$ for $j=1,\dots,k-1$ and $[\lie{g}_1,\lie{g}_k] = \{0\}$; the positive integer $k$ is known as the step of $G$. We denote by $d_j$ the dimension of the $j$th layer $\lie{g}_j$, and let $d = d_1+d_2+\dots+d_k$ and $Q = d_1 +2d_2+\dots+kd_k$ be the topological and homogeneous dimensions of $G$.

Let $X_1,\dots,X_{d_1}$ be a system of left-invariant vector fields on $G$ that form a basis of the first layer $\lie{g}_1$, and let
\begin{equation}\label{eq:opL}
\opL = -\sum_{j=1}^{d_1} X_j^2
\end{equation}
be the corresponding homogeneous sub-Laplacian on $G$. Notice that $\opL$ is essentially selfadjoint and positive on $L^2(G)$, where Lebesgue spaces on $G$ are defined with respect to a (left and right) Haar measure.

When $k=1$, the group $G$ is isomorphic to $\RR^d$ and $\opL$ is the usual (positive definite) Laplace operator. For $k>1$, instead, the group $G$ is not abelian and the sub-Laplacian $\opL$ is not an elliptic operator, as the $X_j$ do not span the whole of $\lie{g}$; nevertheless, by H\"ormander's classical result \cite{Ho_hypo}, $\opL$ is hypoelliptic and satisfies subelliptic estimates.

The choice of a sub-Laplacian $\opL$ induces a left-invariant sub-Riemannian structure on $G$ (see, e.g., \cite{ABB,Mo}), where the horizontal distribution is spanned at every point by the vector fields $X_1,\dots,X_{d_1}$ and equipped with the metric that makes those vector fields orthonormal. With this sub-Riemannian structure, $G$ is also known as a Carnot group, and the homogeneous dimension $Q$ is the Hausdorff dimension of the underlying metric space.

Much as in \cite{Fol}, let us introduce, for $p \in [1,\infty]$ and $s \in \RR$, a scale of $L^p$-Sobolev spaces $L^p_s(G)$ on $G$ adapted to the sub-Laplacian $\opL$, where
\[
\|f\|_{L^p_s(G)} \defeq \|(1+\opL)^{s/2} f\|_{L^p(G)}.
\]
In this work, we discuss the validity of fixed-time $L^p$-estimates with loss of derivatives of the form
\begin{equation}\label{eq:fixedtimewave}
\|u(t,\cdot)\|_{L^p(G)} \lesssim_{p,s,t} \| u(0,\cdot) \|_{L^p_s(G)} + \| \partial_t u(0,\cdot) \|_{L^p_{s-1}(G)}
\end{equation}
for solutions $u : \RR \times G \to \CC$ of the wave equation
\[
\partial_t^2 u(t,\bx) = -\opL u(t,\bx)
\]
driven by $\opL$. Here $A \lesssim B$ denotes that the inequality $A \leq C B$ holds for some implicit constant $C \in (0,\infty)$, and we shall also write $A \simeq B$ for the conjunction of $A \lesssim B$ and $B \lesssim A$; subscripted variants such as $A \lesssim_{p,s,t} B$ indicate that the implicit constant may depend on the parameters $p,s,t$. 

\smallskip

Estimates of the form \eqref{eq:fixedtimewave} are analogous to the classical Miyachi--Peral estimates \cite{Mi,Pe} for the wave equation driven by the Laplacian in Euclidean space. As in the Euclidean case, a key question here is the sharp range of exponents $p \in [1,\infty]$ and $s \geq 0$ for which such estimates hold true. Notice that, due to energy conservation, when $p=2$ one can take any $s \geq 0$, i.e., there is no loss of derivatives then.

At the other endpoint $p=1$ (equivalently, $p=\infty$), such estimates can essentially be reduced (up to $\varepsilon$-losses in $s$) to spectrally localised $L^1$-estimates for the half-wave propagator of the form
\begin{equation}\label{eq:sploc_wave_est_t}
\|\exp(it\sqrt{\opL}) \, \chi(|t|\sqrt{\opL}/\lambda)\|_{1 \to 1} \lesssim_{s,T} \lambda^s \quad \text{ for } \lambda \gg 1, \ |t| \leq T \in \Rpos,
\end{equation}
where $\chi \in C^\infty_c(\Rpos)$ is a smooth cutoff, and $\Rpos \defeq (0,\infty)$. Notice that, by a scaling argument, for a homogeneous sub-Laplacian $\opL$ the left-hand side of \eqref{eq:sploc_wave_est_t} is independent of $t \in \RR$, so in this case the uniformity in $t$ is automatic.

In the Euclidean case, the half-wave propagator for the Laplacian on $\RR^d$ can be explicitly written as an oscillatory integral operator via the Fourier transform, and the analogue of \eqref{eq:sploc_wave_est_t} holds true whenever $s \geq (d-1)/2$ \cite{Be,Mi,Pe,Sj}. The same result holds true for any elliptic Laplacian on a compact $d$-manifold; indeed, in that case it is possible to construct a parametrix for the half-wave equation in terms of certain Fourier integral operators (FIOs), for which sharp $L^p$-bounds are proved in \cite{SSS}. However, there are a number of obstructions in extending these approaches to the case of the wave equation driven by a nonelliptic sub-Laplacian.

\smallskip

It is worth mentioning that homogeneous left-invariant sub-Laplacians \eqref{eq:opL} on Carnot groups serve as ``local models'' for more general sub-Laplacians on sub-Riemannian manifolds, much in the same way as the Euclidean Laplacian is a model for second-order elliptic operators (see, e.g., \cite{Bel,NRS,RS}). Thus, Carnot groups are a natural setting where to study estimates for sub-Riemannian wave equations.

Interestingly enough, while for elliptic Laplacians there is just one local model in each dimension $d$ (namely, the Laplacian on $\RR^d$), for the class of nonelliptic sub-Laplacians one may have infinitely many non-isomorphic models for a given dimension $d$, even when the step $k$ and the dimensions $d_j$ of the layers are fixed. Take for instance the case of Heisenberg groups, i.e., $2$-step stratified groups $G$ with $d_2=1$ such that the skew-symmetric bilinear map $[\cdot,\cdot] : \lie{g}_1 \times \lie{g}_1 \to \lie{g}_2$ given by the Lie bracket is nondegenerate. Any Heisenberg group of dimension $d \geq 5$, beside a distinguished isotropic Carnot structure (where the vector fields $X_j$ form a symplectic basis for $[\cdot,\cdot]$ on $\lie{g}_1$), admits uncountably many other pairwise non-isometric structures, often referred to as ``nonisotropic Heisenberg groups'' in the literature (see, e.g., \cite[Section 13.2]{ABB}).

As mentioned, in the study of wave equation estimates such as \eqref{eq:sploc_wave_est_t} on Carnot groups, the difficulty of the problem lies in determining the \emph{sharp} range of parameters for which these estimates hold true.
Indeed, the fact that \eqref{eq:sploc_wave_est_t} holds true whenever $s$ is sufficiently large can be readily seen as a direct consequence of spectral multiplier estimates of the form
\begin{equation}\label{eq:spmult_cs}
\sup_{r>0} \|F(r\opL)\|_{1 \to 1} \lesssim_{s,K} \|F\|_{L^\infty_s(\RR)} \qquad\text{whenever } \supp F \subseteq K \Subset \Rpos,
\end{equation}
where $L^q_s(\RR)$ denotes the $L^q$-Sobolev space on $\RR$ of order $s \geq 0$, and $A \Subset B$ indicates that the closure of $A$ is compact and contained in $B$.
For a homogeneous sub-Laplacian $\opL$ on a Carnot group, from the classical Christ--Mauceri--Meda Theorem \cite{C,MaMe} we know that \eqref{eq:spmult_cs}, thus also \eqref{eq:sploc_wave_est_t}, holds true for any $s > Q/2$. In particular, if we define the optimal thresholds
\[
\Mthr(\opL) \defeq \inf \{ s \geq 0 \tc \eqref{eq:spmult_cs} \text{ holds} \}, 
\quad
\Wthr(\opL) \defeq \inf \{ s \geq 0 \tc \eqref{eq:sploc_wave_est_t} \text{ holds} \}
\]
for the multiplier and wave equation estimates, then
\[
\Wthr(\opL) \leq \Mthr(\opL) \leq Q/2.
\]

Notice that $\Mthr(\opL) = d/2 = Q/2$ for the Euclidean Laplacian. For nonabelian Carnot groups, instead, $Q>d$, and the homogeneous dimension $Q$ is arguably the natural dimensional parameter associated to the underlying geometric structure; nevertheless, somewhat surprisingly, $Q/2$ need not be the optimal multiplier threshold $\Mthr(\opL)$ for a nonelliptic sub-Laplacian $\opL$. This discovery was first made in \cite{Heb,MSt_mult} in the case of Heisenberg and related groups, where $\Mthr(\opL) = d/2$.
We now know that $d/2 \leq \Mthr(\opL)<Q/2$ on any $2$-step Carnot group \cite{MM}, and the lower bound $\Mthr(\opL) \geq d/2$ actually holds in arbitrary step \cite{MMNG}.
The equality $\Mthr(\opL)=d/2$ has also been proved for several classes of $2$-step Carnot groups and sub-Laplacians beyond the Heisenberg case \cite{M_HR,MM_newclasses}, as well as in particular examples beyond the Carnot group setting, including in step higher than $2$ (see, e.g., \cite{CCM,CKS,DM,MP} and references therein).
Nevertheless, the determination of the multiplier threshold $\Mthr(\opL)$ remains an open problem for an arbitrary nonelliptic sub-Laplacian $\opL$ on a Carnot group (or a sub-Riemannian manifold), even in the $2$-step case.

For the wave equation threshold $\Wthr(\opL)$, even less is known. For an arbitrary sub-Laplacian $\opL$, we know that $\Wthr(\opL) \geq (d-1)/2$ \cite{MMNG}; moreover, a subordination argument \cite{Mu_ICM} shows that $\Mthr(\opL) \leq \Wthr(\opL) + 1/2$, at least on Carnot groups. In light of this, and by comparison with the Euclidean case, one might conjecture that $\Wthr(\opL) = (d-1)/2$, at least in the cases where $\Mthr(\opL) = d/2$.

To the best of our knowledge, this conjecture has so far been confirmed only in the case where $G$ is a $2$-step group $G$ of Heisenberg type, also known as an H-type group \cite{KR}, and $\opL$ is the distinguished sub-Laplacian thereon.
In this case, from \cite{MSe,MSt} it follows that \eqref{eq:sploc_wave_est_t} indeed holds true for any $s \geq (d-1)/2$.
The class of groups of Heisenberg type includes the Heisenberg groups, as well as variants with $d_2>1$, equipped with a distinguished isotropic Carnot structure (see Section \ref{ss:metivier} below for details); in particular, the nonisotropic Heisenberg groups are not covered by the results of \cite{MSe,MSt}. 

The proof in \cite{MSt}, which only discusses Heisenberg groups and the non-endpoint range $s>(d-1)/2$, is fundamentally based on the representation theory of those groups. The endpoint result is proved in \cite{MSe} for arbitrary H-type groups, by exploiting the Mehler-type formulas for the Schr\"odinger propagators available in that context. In any case, both of the approaches in \cite{MSt} and \cite{MSe} make strong use of symmetry properties of the given sub-Laplacian, and the same is indeed true of most of the sharp spectral multiplier theorems for nonelliptic sub-Laplacians known to date.
The lack of robustness of the existing approaches stands in striking contrast to what is available in the elliptic case, where microlocal analysis techniques have been successfully used to prove sharp multiplier and wave equation estimates in great generality, at least on compact manifolds \cite{SeS,SSS}.

\smallskip

The main objective of this work is to present a substantially different approach to proving (essentially) sharp wave propagator estimates of the form \eqref{eq:sploc_wave_est_t} for sub-Laplacians, which is less reliant on underlying symmetries.
In the same spirit as the works \cite{SeS,SSS} for elliptic Laplacians on manifolds, here we shall construct an approximate FIO representation for the wave equation driven by a sub-Laplacian, which is accurate enough to obtain sharp $L^1$-estimates.

More precisely, a comparison between spectral and frequency localisations allows us to essentially reduce  microlocally to two conic regions in frequency space, an \emph{anti-FIO region} (which seems not amenable to FIO techniques), and an \emph{FIO region}. 
At least for the sub-Laplacians considered in our main theorem, the contribution by the anti-FIO region to the wave propagator can be controlled in a rather easy way, by effectively taking advantage of the limited spatial propagation of the corresponding geodesics. Instead, it is the contribution by the FIO region that will be studied by means of FIO techniques.
Our construction of an FIO parametrix for the latter region is based on an adaptation of the method of \cite{LSV,SV}, where a \emph{global} FIO representation for wave propagators associated to elliptic operators is obtained through the use of a complex phase function, naturally defined in terms of the geodesic flow.

The fact that we are working in a more complex, sub-Riemannian setting imposes quite a number of novel challenges, which do not come up in the elliptic case, and it has indeed been unclear for a long time if an FIO approach might be feasible in sub-Riemannian settings at all. For instance, by means of suitable microlocalisations and subsequent scalings, one is eventually led to proving uniform $L^1$-estimates for certain classes of FIOs with time parameters that may become arbitrarily large, which presents one of these new challenges.

\smallskip

We shall exemplify our approach for the case of a homogeneous left-invariant sub-Laplacian $\opL$ on a \emph{M\'etivier group} $G$, i.e., a $2$-step stratified Lie group for which the skew-symmetric form $\mu [\cdot,\cdot] : \lie{g}_1 \times \lie{g}_1 \to \RR$ is nondegenerate for all $\mu \in \lie{g}_2^* \setminus \{0\}$ \cite{Met}. This class of groups and sub-Laplacians properly extends that considered in \cite{MSe,MSt}, and includes, among others, the nonisotropic Heisenberg groups; we refer to \cite{MSe_maximal,Nie} for further interesting examples of M\'etivier groups with $d_2>1$, demonstrating the greater complexity of this class compared to that of H-type groups.

Due to the robust nature of the employed FIO techniques, we are convinced that the basic approach developed here has the potential for extensions to wider classes of sub-Riemannian manifolds, such as sub-Laplacians on more general classes of Carnot groups, or even sub-Laplacians which do not have any particular symmetry properties at all (such as small, local perturbations of the sub-Laplacian on the Heisenberg group, or sub-Laplacians on arbitrary compact Heisenberg manifolds).
 
\smallskip
 
For $r \in \RR$, we denote by $S^r(\RR^n)$ the standard \emph{symbol class of order $r$}, i.e., the Fr\'echet space of the smooth functions $f : \RR^n \to \CC$ such that
\[
\sup_{\zeta \in \RR^n} (1+|\zeta|)^{|\alpha|-r} |\partial_\zeta^\alpha f(\zeta)| < \infty  \qquad\text{for all } \alpha \in \NN^n.
\]
What we shall prove here is the following result, which extends that for groups of Heisenberg type in \cite{MSe}, except for the endpoint regularity $s = (d-1)/2$.

\begin{thm}\label{thm:main}
Let $G$ be a M\'etivier group of topological dimension $d$, and $\opL$ a homogeneous sub-Laplacian thereon. Let $p \in [1,\infty]$ and $s > (d-1)|1/p-1/2|$.
\begin{enumerate}[label=(\roman*)]
\item\label{en:main_splocest} For all $\chi \in C^\infty_c(\Rpos)$ and all $\lambda \geq 1$,
\[
\sup_{t \in \RR} \| \exp(it\sqrt{\opL}) \chi(|t|\sqrt{\opL}/\lambda) \|_{p \to p} \lesssim_{\chi,p,s} \lambda^s.
\]
In particular, $\Wthr(\opL) = (d-1)/2$.
\item\label{en:main_symbol} Let $\cA$ and $\cB$ be bounded subsets of $S^{-s}(\RR)$ and $S^{1-s}(\RR)$.
Then the operators
\[
a(t\sqrt{\opL}) \exp(it\sqrt{\opL}), \qquad  a(t\sqrt{\opL}) \cos(t\sqrt{\opL}), \qquad b(t\sqrt{\opL}) \frac{\sin(t\sqrt{\opL})}{t\sqrt{\opL}}
\]
are $L^p$-bounded uniformly in $a \in \cA$, $b \in \cB$ and $t \in \RR$.
\item\label{en:main_cauchy} The solution $u = u(t,\bx)$ of the initial value problem
\[
\begin{cases}
\partial_t^2 u = -\opL u \\
u|_{t=0} = f \\
\partial_t u|_{t=0} = g
\end{cases}
\]
for the wave equation associated with the sub-Laplacian $\opL$ on $G$ satisfies
\[
\|u(t,\cdot)\|_{p} \lesssim_{s,p} \|(1+t^2\opL)^{s/2} f\|_{p} + \|t(1+t^2\opL)^{(s-1)/2} g\|_{p}
\]
for all $t \in \RR$. In particular, if the initial data satisfy
\[
f \in L^p_s(G), \qquad g \in L^p_{s-1}(G),
\]
then the solution $u(t,\cdot)$ is in $L^p(G)$ for all $t \in \RR$.
\end{enumerate}
\end{thm}

By \cite[Theorem 1.1(ii)]{MMNG}, the threshold $(d-1)|1/p-1/2|$ in Theorem \ref{thm:main} is sharp, in the sense that it cannot be replaced by any smaller quantity. An interesting open problem is whether, at least when $p \in (1,\infty)$, the estimates of Theorem \ref{thm:main} hold at the endpoint $s = (d-1)|1/p-1/2|$. In the particular case where $G$ is an H-type group with its distinguished sub-Laplacian, such an endpoint result is proved in \cite{MSe}. Nevertheless, none of the approaches in \cite{MSe,MSt} would directly extend from H-type groups to arbitrary M\'etivier groups (compare Remark \ref{rem:notMS} below).

\smallskip

Much as in \cite[Section 1.4]{MSe}, the wave equation estimates of Theorem \ref{thm:main} imply, by subordination,
a spectral multiplier theorem for sub-Laplacians on M\'etivier groups, which extends the results of \cite{Heb,MSt_mult}, in that it covers a larger class of multipliers. In what follows, we denote by $\widehat{F}$ the Fourier transform of a function $F : \RR \to \CC$, given by $\widehat F(\tau) \defeq \int_\RR F(t) \,e^{-i \tau t} \,dt$ whenever $F$ is integrable.

\begin{cor}\label{cor:MH}
Let $G$ be a M\'etivier group of topological dimension $d$, and $\opL$ a homogeneous sub-Laplacian thereon. Let $\chi \in C^\infty_c(\Rpos)$ be nonzero. Let $F : \RR \to \CC$ be a Borel function. Let $s>(d-1)/2$.
\begin{enumerate}[label=(\roman*)]
\item If $\supp F \subseteq K$ for some $K \Subset \Rpos$, then
\begin{equation}\label{en:MH_cptsupp}
\sup_{t>0} \|F(t\sqrt{\opL})\|_{1\to 1} \lesssim_{s,K} \int_\RR |\widehat{F}(\tau)| \, (1+|\tau|)^s \,d\tau.
\end{equation}
\item\label{en:MH_mult} If
\begin{equation}\label{eq:smoothness_mult}
 \sup_{t>0} \int_{\RR} |\widehat{F(t\cdot) \chi}(\tau)| \, (1+|\tau|)^{s} \,d\tau  < \infty,
\end{equation}
then $F(\sqrt{\opL})$ is of weak type $(1,1)$ and bounded on $L^p(G)$ for all $p \in (1,\infty)$.
\end{enumerate}
\end{cor}

\begin{rem}
By the Cauchy--Schwarz inequality, the right-hand side of \eqref{en:MH_cptsupp} is controlled by the $L^2$-Sobolev norm $\|F\|_{L^2_\sigma(\RR)}$ of order $\sigma > s+1/2$; thus, the condition \eqref{eq:smoothness_mult} follows from the scale-invariant smoothness condition of Mihlin--H\"ormander type
\begin{equation}\label{eq:mh}
\sup_{t > 0} \| F(t \cdot) \chi \|_{L^2_\sigma(\RR)} < \infty,
\end{equation}
for any $\sigma > s+1/2$. In particular, Corollary \ref{cor:MH}\ref{en:MH_mult} applies whenever \eqref{eq:mh} is satisfied for some $\sigma > d/2$, so we recover the multiplier theorem of \cite{Heb,MSt_mult} in the case where $G$ is a M\'etivier group. On the other hand, for any $s \geq 0$, there exist compactly supported functions $F$ such that $\int_{\RR} |\widehat{F}(\tau)|\, (1+|\tau|)^s \,d\tau < \infty$ but $\|F\|_{L^2_\sigma(\RR)} = \infty$ for all $\sigma > s$, thus Corollary \ref{cor:MH} indeed covers a larger class of multipliers.
\end{rem}

As mentioned, it would be of great interest to investigate whether the results above hold true for larger classes of sub-Laplacians on sub-Riemannian manifolds. As a matter of fact, some parts of the proof of Theorem \ref{thm:main} work in greater generality than that of M\'etivier groups. In order not to affect the clarity of the presentation, in this paper we shall limit the discussion to the setting of $2$-step Carnot groups $G$ and homogeneous left-invariant sub-Laplacians thereon. Nevertheless, we shall not restrict ourselves to the M\'etivier group case immediately, but we shall introduce the appropriate restrictions on $G$ at each stage of the proof.

\subsection{Some heuristics}\label{ss:intro_heuristics}
Let us first discuss some heuristics about the validity of a spectrally localised estimate of the form \eqref{eq:sploc_wave_est_t}.

In the case where $\opL$ is the Laplace operator on $\RR^d$, due to homogeneity and translation-invariance, one can reduce to $t=1$ and work with convolution kernels; further, by means of the Fourier transform one can explicitly write
\begin{equation}\label{eq:euclidean_FIO}
\exp(i\sqrt{\opL}) \chi(\sqrt{\opL}/\lambda) \delta_0(x) = (2\pi)^{-d} \int_{\RR^d} e^{i(x\cdot\xi + |\xi|)} \chi(|\xi|/\lambda) \,d\xi.
\end{equation}
The key argument of \cite{SSS} can then be used to obtain a sharp $L^1$-estimate, in terms of powers of $\lambda$, for the kernel $\exp(i\sqrt{\opL}) \chi(\sqrt{\opL}/\lambda) \delta_0$: namely, the integral in the right-hand side of \eqref{eq:euclidean_FIO}, where we already have the frequency localisation $|\xi|\simeq \lambda$, can be decomposed into regions of angular aperture $\lambda^{-1/2}$ according to the direction of $\xi$ (see Figure \ref{fig:SSSfreq}); in each region the phase function $x \cdot \xi + |\xi|$ can effectively be linearised (indicating that this portion of the wave propagator acts as a translation in the direction of $\xi$ on the space side), thus showing that the corresponding integral gives a unit contribution to the $L^1$-norm; the total $L^1$-norm is therefore proportional to the number of pieces, i.e., to $\lambda^{(d-1)/2}$.

\begin{figure}
\centering
\begin{subcaptionblock}{.4\textwidth}
\centering\begin{tikzpicture}
		\useasboundingbox (-2.3,-2.3) rectangle (2.3,2.3);

		\fill [green,even odd rule] (0,0) circle (2cm) circle (1cm);

		\fill [red] (40:1) arc (40:50:1) -- (50:2) arc (50:40:2) -- cycle;
		\fill [blue] (100:1) arc (100:110:1) -- (110:2) arc (110:100:2) -- cycle;

\foreach \t in {0,10,...,350}
   { \draw (\t:1) -- (\t:2); }

		\draw [gray,->] (-2.3,0) -- (2.3,0) node [black,below] {$\xi_1$};
		\draw [gray,->] (0,-2.3) -- (0,2.3) node [black,left] {$\xi_2$};

\end{tikzpicture}
\caption{frequency picture} \label{fig:SSSfreq}
\end{subcaptionblock}%
\begin{subcaptionblock}{.4\textwidth}
\centering\begin{tikzpicture}
		\useasboundingbox (-2.3,-2.3) rectangle (2.3,2.3);
		
		\fill [green!30!white,even odd rule] (0,0) circle (1.6cm) circle (1.4cm);

		\draw[->,thick,red,decoration={snake, amplitude=.8},decorate] (45:.35) -- (45:1.3);
    \fill [red!20!white,rotate around={45:(0,0)}] (-.1,-.25) -- (.1,-.25) -- (.1,.25) -- (-.1,.25) -- cycle;
    \fill [red,rotate around={45:(0,0)}] (1.4,-.25) -- (1.6,-.25) -- (1.6,.25) -- (1.4,.25) -- cycle;

		\draw[->,thick,blue,decoration={snake, amplitude=.8},decorate] (105:.35) -- (105:1.3);
    \fill [blue!20!white,rotate around={105:(0,0)}] (-.1,-.25) -- (.1,-.25) -- (.1,.25) -- (-.1,.25) -- cycle;
    \fill [blue,rotate around={105:(0,0)}] (1.4,-.25) -- (1.6,-.25) -- (1.6,.25) -- (1.4,.25) -- cycle;

		\draw (0,0) circle (1.5cm);

		\draw [gray,->] (-2.3,0) -- (2.3,0) node [black,below] {$x_1$}; 
		\draw [gray,->] (0,-2.3) -- (0,2.3) node [black,left] {$x_2$};

\end{tikzpicture}
\caption{space picture} \label{fig:SSSspace}
\end{subcaptionblock}

\caption{Frequency decomposition in $\lambda \times \lambda^{1/2}$ boxes of the Euclidean wave propagator at time $t=1$ according to the key method of \cite{SSS}; after the application of the wave propagator, the dual boxes in the space picture form a $1/\lambda$-neighbourhood of the unit sphere.}
\label{fig:SSS}

\end{figure}

We point out that the phase function $x\cdot \xi + |\xi|$ in \eqref{eq:euclidean_FIO} carries relevant geometric information. In particular, the $\xi$-gradient of the phase function vanishes where $x = - \xi/|\xi|$, corresponding to the fact that the singular support of the wave propagator kernel $\exp(i \sqrt{\opL}) \delta_0$ at time $1$ is the unit sphere $S = \{ x \tc |x| = 1\}$; this agrees with the idea that waves propagate along geodesics (in this case, straight lines emanating from the origin). Moreover, the aforementioned decomposition argument from \cite{SSS} applied to \eqref{eq:euclidean_FIO} can be used to show that the function $\exp(i\sqrt{\opL}) \chi(\sqrt{\opL}/\lambda) \delta_0$ is essentially concentrated in a $1/\lambda$-neighbourhood of $S$ (see Figure \ref{fig:SSSspace}); this is also intuitively expected from the fact that $S$ is the singular support of the wave propagator kernel $\exp(i \sqrt{\opL}) \delta_0$, while the spectral localisation kernel $\chi(\sqrt{\opL}/\lambda) \delta_0$ effectively acts as a convolution by a bump function at scale $1/\lambda$.
As a matter of fact, if one assumed that $\exp(i\sqrt{\opL}) \chi(\sqrt{\opL}/\lambda) \delta_0$ were actually supported in such $1/\lambda$-neighbourhood, then by the Cauchy--Schwarz inequality and the unitarity of the half-wave propagator one would immediately obtain that
\begin{multline}\label{eq:sharp_euclidean_euristics}
\| \exp(i \sqrt{\opL}) \chi(\sqrt{\opL}/\lambda) \delta_0 \|_{1} 
\lesssim \lambda^{-1/2} \| \exp(i \sqrt{\opL}) \chi(\sqrt{\opL}/\lambda) \delta_0 \|_{2} \\
= \lambda^{-1/2} \| \chi(\sqrt{\opL}/\lambda) \delta_0 \|_{2} \simeq \lambda^{(d-1)/2},
\end{multline}
i.e., this heuristic argument yields the expected power growth in $\lambda$.

\smallskip

When $\opL$ is a more general elliptic Laplacian on a $d$-manifold $M$, a similar approach is possible by means of the theory of Fourier integral operators. Indeed, one can represent, locally, for small times, and up to smoothing terms, the half-wave propagator $\exp(-it\sqrt{\opL})$ as a Fourier integral operator (FIO) of the form
\begin{equation}\label{eq:elliptic_FIO}
\Psi_t f(x) = \iint e^{i\phi(t,x,y,\xi)} \, q(t,x,y,\xi) \, f(y) \,dy \,d\xi
\end{equation}
for a suitable phase function $\phi$ and amplitude $q$; the method of \cite{SSS} can then be used to obtain $L^1$-estimates for such an operator $\Psi_t$. At least when $M$ is compact, these local estimates can be combined to obtain sharp local-in-time $L^1$-estimates with loss of derivatives for the wave propagator.

Again, the phase function $\phi$ in the FIO representation \eqref{eq:elliptic_FIO} carries relevant geometric information. Namely, let $(y,\xi) \mapsto (x^t(y,\xi),\xi^t(y,\xi))$ be the homogeneous canonical transformation of $T^* M \setminus \{0\}$ at time $t \in \RR$ generated by the square root $\Ham = \Ham(x,\xi)$ of the principal symbol of $\opL$, i.e., satisfying the following Hamilton equations:
\[
\dot x^t = \nabla_{\xi} \Ham(x^t,\xi^t), \qquad \dot \xi^t = - \nabla_{x} \Ham(x^t,\xi^t), \qquad x^0 = y, \qquad \xi^0 = \xi;
\]
notice that $t \mapsto x^t(y,\xi)$ is the speed-one geodesic starting at $y$ with initial vector $\nabla_\xi \Ham(y,\xi)$ for the Riemannian geometric structure on $M$ associated with $\opL$.
Then, the phase function $\phi$ parametrizes the Lagrangian manifold associated with this canonical transformation, in the sense that
\begin{equation}\label{eq:phase_param1}
\partial_{\xi} \phi(t,x,y,\xi) =0 \iff x = x^t(y,\xi)
\end{equation}
and moreover
\begin{equation}\label{eq:phase_param2}
\partial_{\xi} \phi(t,x,y,\xi) =0 \ \Longrightarrow\  \begin{cases}
\nabla_x \phi(t,x,y,\xi) = \xi^t(y,\xi), \\
\det \partial_\xi \nabla_x \phi(t,x,y,\xi) \neq 0.
\end{cases}
\end{equation}

In most constructions of the FIO representation above (see, e.g., \cite{Ho_spec,Sh,S,Tr}), the phase function $\phi$ is actually chosen so that it solves the \emph{eikonal equation}
\begin{equation}\label{eq:eikonal}
\partial_t\phi(t,x,y,\xi) + \Ham(x,\nabla_x \phi(t,x,y,\xi)) = 0.
\end{equation}
On the other hand, as pointed out in \cite{LSV,SV}, the requirement that $\phi$ satisfies the eikonal equation, or more generally the requirement that $\phi$ is \emph{real-valued}, forces the construction of the FIO representation \eqref{eq:elliptic_FIO} to be local; in particular, the presence of \emph{caustics} (i.e., conjugate points along geodesics, where $\rk \partial_\xi x^t < d-1$) turns out to be an obstruction for a global FIO representation of the wave propagator.

As mentioned, for an elliptic Laplacian $\opL$ on a compact manifold, this locality constraint is not an obstruction to obtaining sharp local-in-time estimates for the wave propagator. Crucially, under these ellipticity and compactness assumptions we have a \emph{positive injectivity radius} $t_{\inj}(\opL)$ for the Riemannian metric associated with $\opL$, and conjugate points do not appear for $|t| < t_{\mathrm{inj}}(\opL)$. Indeed, for $|t| < t_{\inj}(\opL)$ and any $y \in M$, the ``geodesic sphere'' $\{ x^t(y,\xi) \tc \xi \in T^*_y M \setminus \{0\}\}$ is diffeomorphic to a standard sphere. Thus, at a geometric level the (local) situation is fully analogous to that in Euclidean space, and in this respect it is not entirely surprising that the approach used for the Euclidean wave equation can be adapted to that driven by an elliptic Laplacian on a compact manifold.

\smallskip

Things change considerably when we weaken the ellipticity assumption, and consider the wave equation driven by a sub-Laplacian $\opL$ on a sub-Riemannian manifold $M$. In this case, the principal symbol of $\opL$ may be a degenerate (positive semidefinite) quadratic form on $T^* M$, so the set $\DHam$ where the symbol does not vanish, i.e., where its square root $\Ham$ is smooth, is a proper conical subset of $T^* M \setminus \{0\}$. 
Crucially, 
even if we restrict to the Hamiltonian flow generated by $\Ham$ on $\DHam$, there is no longer a ``positive injectivity radius'', and there may exist conjugate points (where $\rk \partial_\xi x^t(y,\xi) < d-1$) to a given point $y \in M$ for arbitrarily small $t$ (cf.\ \cite[Theorem 12.16]{ABB}). Moreover, the geodesic spheres $\{ x^t(y,\xi) \tc \xi \in T^*_y M \cap \DHam \}$ centred at $y$ need no longer be diffeomorphic to standard spheres, even for small $t$, and may be adjacent to their centre (see Figure \ref{fig:wavefronts}).

These observations reveal, on the one hand, some of the many obstructions in extending the classical approach summarised above to the case of sub-Riemannian wave equations; on the other hand, they indicate that the simple ``geometric heuristics'' discussed above for the small-time wave propagator are not directly applicable in the sub-Riemannian case, where a more intricate analysis is to be expected. To further appreciate this, notice that the heuristic computation \eqref{eq:sharp_euclidean_euristics} for the Euclidean Laplacian crucially depends on the fact that $\| \chi(\sqrt{\opL}/\lambda) \delta_0 \|_{2} \simeq \lambda^{d/2}$. For a sub-Laplacian on a Carnot group, however, scaling considerations would give $\| \chi(\sqrt{\opL}/\lambda) \delta_0 \|_{2} \simeq \lambda^{Q/2}$. As the homogeneous dimension $Q$ may be much larger than the topological dimension $d$ for nonelliptic $\opL$, a direct analogue of \eqref{eq:sharp_euclidean_euristics} would not yield the desired growth $\lambda^{(d-1)/2}$ in any case.

\begin{figure}

\includegraphics{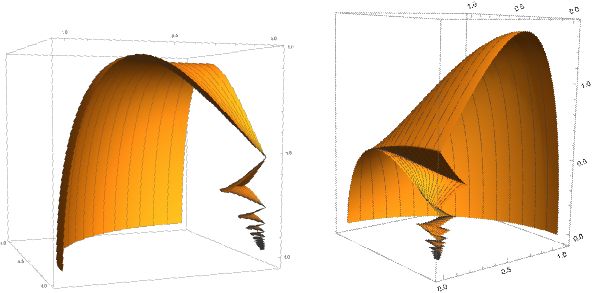}

\caption{Sections of geodesic spheres on isotropic (left) and nonisotropic (right) Heisenberg groups; in either case, the ``zig-zags'' around the vertical axis converge to the centre of the sphere.}
\label{fig:wavefronts}

\end{figure}

\smallskip

As just mentioned, a serious obstruction in running the classical arguments in the sub-Riemannian case is the existence of conjugate points for arbitrarily small times. From this point of view, the ``global FIO approach'' proposed in \cite{LSV,SV} becomes of great interest: while the construction in \cite{LSV,SV} is presented with elliptic operators $\opL$ in mind, one of its main features is the possibility of obtaining a global (in time and space) FIO representation for the wave propagator, in terms of oscillatory integrals of the form
\[
\Xi_t f(x) = \iint e^{i\phi(t,x,y,\xi)} \, q(t,y,\xi) \, f(y) \, \Den_\phi(t,x,y,\xi) \,dy \,d\xi.
\]
Here the phase function $\phi$ may be complex-valued (with $\Im \phi \geq 0$) and need not satisfy the eikonal equation \eqref{eq:eikonal} off the critical set $\{ \partial_\xi \phi = 0 \}$, but still parametrizes the Lagrangian manifold associated with the canonical transformation $(y,\xi) \mapsto (x^t(y,\xi),\xi^t(y,\xi))$ in the sense of \eqref{eq:phase_param1}-\eqref{eq:phase_param2}; moreover, the density $\Den_\phi$ is chosen as a suitable branch of the square root of the determinant of the mixed hessian matrix $\partial_\xi \nabla_x \phi$ of $\phi$. The presence of a nonzero imaginary part $\Im \phi$ is a crucial feature of this construction, in that it allows the representation to make sense globally, irrespective of the existence of caustics.

One of the main contributions of the present work is to show that the complex-phase FIO approach of \cite{LSV,SV} can be adapted to the sub-Riemannian setting, at least in the case where $\opL$ is a sub-Laplacian on a $2$-step group $G$; in addition, we manage to show that the corresponding FIO representation for the wave propagator can be effectively used to prove sharp fixed-time $L^1$-estimates for the wave propagator, at least when $G$ is M\'etivier. A number of features of the nonelliptic and noncompact setting at hand must be taken into account, thus making the analysis much more involved than that used in the compact elliptic case. As we shall see, among other things, a substantially more complicated ``anisotropic'' decomposition of the oscillatory integral is needed here; nevertheless, interestingly enough, each of the resulting pieces, after a suitable scaling, becomes amenable to (a variation of) the approach of \cite{SSS}, and the estimates thus obtained can be summed to obtained the desired sharp $L^1$-estimate for the wave propagator.

We point out that microlocal analysis and Fourier integral operator techniques have already been used in the study of the wave equation driven by a sub-Laplacian, especially for the analysis of propagation of singularities (see, e.g., \cite{CHT,GHK,Las,Le_prop,Le,Mel,Nach} and references therein); this also includes the use of a complex-phase oscillatory integral construction of a particular ``gaussian beam'' solution \cite{Le}. Our study, however, does not focus on particular solutions, and information on propagation of singularities in itself is not enough for our purposes, as we are looking for global, sharp $L^1$-estimates for the whole (spectrally localised) wave propagator.

\subsection{Sketch of the proof}\label{ss:sketch}
For the reader's convenience, we now summarize the main steps of the proof, including pointers to the corresponding sections of the paper. This allows us to introduce some of the notation that will be used throughout.

Recall that a $2$-step group $G$ can be identified to its Lie algebra $\lie{g} = \lie{g}_1 \oplus \lie{g}_2$ via the exponential map; correspondingly, we shall use the notation $\bx = (x,u) \in \lie{g}_1 \oplus \lie{g}_2$ for a point of $G$. In these exponential coordinates, Lebesgue measure is a Haar measure, and the group law is given by
\begin{equation}\label{eq:group_law}
(x,u) \cdot_G (x',u') = (x+x',u+u'+[x,x']/2).
\end{equation}
The sub-Laplacian $\opL$ on $G$ is invariant under left translations and is also homogeneous with respect to the automorphic dilations
\begin{equation}\label{eq:dilations}
\dil_r(x,u) = (rx,r^2 u).
\end{equation}
Similar properties are enjoyed by the sub-Riemannian structure on $G$, and in particular the sub-Riemannian distance of $\bx =(x,u) \in G$ from the origin satisfies
\[
\dist(0,\bx) \simeq \max\{ |x|,|u|^{1/2}\}.
\]

Our aim is to prove, when $G$ is a M\'etivier group, the estimate
\begin{equation}\label{eq:sketch_target}
\| \exp(\pm i\sqrt{\opL}) \chi(\sqrt{\opL}/\lambda) \|_{1 \to 1} \lessapprox \lambda^{(d-1)/2}, \qquad \lambda \gg 1,
\end{equation}
for some $\chi \in C^\infty_c(\Rpos)$. Here $A \lessapprox B$ stands for $A \lesssim_\varepsilon \lambda^\varepsilon B$ for all $\varepsilon > 0$. Thanks to left-invariance, this is the same as proving an $L^1$-estimate for the convolution kernel,
\[
\| \exp(\pm i\sqrt{\opL}) \chi(\sqrt{\opL}/\lambda) \delta_0 \|_{1} \lessapprox \lambda^{(d-1)/2}, \qquad \lambda \gg 1.
\]
As $\chi(\sqrt{\opL}/\lambda) \delta_0$ is essentially supported (up to negligible Schwartz tails) in a sub-Riemannian ball centred at the origin of radius $\simeq 1/\lambda$, we can use the finite propagation speed property for the wave equation to deduce that, after applying the half-wave propagator at time $\pm 1$, we obtain a function which has fast decay outside a sub-Riemannian ball of radius $\simeq 1+1/\lambda \simeq 1$; thus, as we show in Section \ref{s:spatialloc}, we are reduced to proving that
\[
\| \chr_{\overline{B}(0,2)} \exp(\pm i\sqrt{\opL}) \chi(\sqrt{\opL}/\lambda) \delta_0 \|_{1} \lessapprox \lambda^{(d-1)/2}, \qquad \lambda \gg 1,
\]
where $\overline{B}(\bx,r)$ denotes a closed sub-Riemannian ball of centre $\bx$ and radius $r$, and $\chr_S$ is the characteristic function of a set $S$.

As we aim at using microlocal analysis techniques to prove the above estimate, it becomes important to relate the spectral localisation given by $\chi(\sqrt{\opL}/\lambda)$ to a frequency localisation. As seen in \eqref{eq:euclidean_FIO}, when $\opL$ is the Laplacian on $\RR^d$, this spectral localisation corresponds to the frequency localisation $|\xi| \simeq \lambda$. For a sub-Laplacian $\opL$ on a $2$-step group, however, things are more complicated: indeed, if we write $\bxi = (\xi,\mu)$ for the frequency variable dual to $\bx = (x,u)$, then the spectral localisation $\sqrt{\opL} \simeq \lambda$ cannot simply correspond to $|\bxi| \simeq \lambda$, as the two conditions do not scale the same way under the automorphic dilations \eqref{eq:dilations}.
A refinement of these scaling considerations, combined with vanishing moment properties, allows us in Section \ref{s:spectrumvsfrequency} to show that, up to a negligible remainder, under the spectral localisation $\sqrt{\opL} \simeq \lambda$ we can restrict to the frequency localisation
\[
\max\{ |\xi|, |\mu|^{1/2} \} \approx \lambda
\]
(see Figure \ref{fig:freq_loc}), where $A \approx B$ stands for $A \lessapprox B$ and $B \lessapprox A$.

\smallskip

Let $\langle \cdot,\cdot\rangle$ be the inner product on $\lie{g}_1$ determined by the sub-Laplacian $\opL$, i.e., such that the vector fields $X_j$ in \eqref{eq:opL} form an orthonormal frame for the horizontal distribution. If $J_\mu$ is the skew-symmetric endomorphism of $\lie{g}_1$ defined by
\begin{equation}\label{eq:Jmu}
\langle J_\mu x, x' \rangle = \mu[x,x'] \qquad \text{for all } x,x' \in \lie{g}_1
\end{equation}
for any $\mu \in \lie{g}_2^*$, then the principal symbol of the sub-Laplacian $\opL$ can be written as $\Ham(\bx,\bxi)^2$, where
\[
\Ham(\bx,\bxi) = |\xi + J_\mu x/2|
\]
(recall that $\bxi = (\xi,\mu)$ denotes the frequency variable dual to $\bx = (x,u)$).
Notice that $J_\mu$ corresponds to the matrix of Poisson brackets of the symbols of the vector fields $X_j$ of \eqref{eq:opL}, which also appears in \cite[Corollary 1.4]{Mu} in reference to local solvability problems. Moreover, $G$ is M\'etivier if and only if $J_\mu$ is nondegenerate for all $\mu \neq 0$.

By looking at the Hamiltionian flow generated by $\Ham$ (see Section \ref{s:flow}), one notices that unit-length sub-Riemannian geodesics on a M\'etivier group starting at the origin, with initial covector $\bxi$ satisfying 
\begin{equation}\label{eq:freq_antiFIO}
|\mu| \approx \lambda^2, \qquad |\xi| \lessapprox \lambda,
\end{equation}
remain in a Euclidean ball of radius $\lessapprox 1/\lambda$ centred at the origin (see Figure \ref{fig:geo_A}). This suggests that, for the frequency region \eqref{eq:freq_antiFIO}, one could take advantage of a reduced propagation, thus obtaining the desired $L^1$-estimate via a simple Cauchy--Schwarz argument, much as in \eqref{eq:sharp_euclidean_euristics}. As discussed in Section \ref{s:antiFIO}, this idea can be made precise by means of a variant of the ``first-layer weighted Plancherel-type estimates'' used in \cite{Heb,MSt_mult} and many other works on sharp multiplier theorems on $2$-step groups. Thus, for the frequency region \eqref{eq:freq_antiFIO}, which we shall call the \emph{anti-FIO region}, since it does not seem amenable to FIO techniques, there is also no need to use FIO techniques to achieve the desired estimate.

\smallskip

We thus remain with the frequency region
\begin{equation}\label{eq:freq_FIO}
|\xi| \approx \lambda, \qquad |\mu| \lll \lambda^2,
\end{equation}
where $A \lll B$ means $A \lesssim \lambda^{-\varepsilon} B$ for some $\varepsilon > 0$; we shall call \eqref{eq:freq_FIO} the \emph{FIO region}, as it will be tackled via FIO techniques. We shall actually split this region further.

\smallskip

In the part where $|\mu| \ll |\xi|$, the sub-Laplacian $\opL$ microlocally behaves like an elliptic operator in a neighbourhood of the origin. Thus, this frequency region, called the \emph{elliptic region}, can be tackled by adapting standard techniques for elliptic operators.
A similar observation was already used in \cite{MMNG} to prove necessary conditions for spectral multiplier and wave equation bounds for sub-Laplacians; in Section \ref{s:elliptic} we shall make use of the ``microlocal parametrix'' from \cite{MMNG} and the key method of \cite{SSS} to prove the desired estimate for the wave propagator in the elliptic region.

\smallskip

We are finally reduced to the \emph{nonelliptic region}
\[
|\xi| \approx \lambda, \qquad |\xi| \lesssim |\mu| \lll \lambda^2.
\]
It is convenient here to split this region dyadically according to the value of $|\mu|/|\xi|$, i.e., set $|\mu|/|\xi| \simeq 2^\ell$ for some $\ell \in \ZZ$ with $1 \lesssim 2^\ell \lll \lambda$. We now scale automorphically by $2^\ell$ each of the pieces; this allows us to achieve, after scaling, a frequency localisation where $|\xi| \simeq |\mu|$, at the cost of moving from time $\pm 1$ to time $t = \pm 2^\ell$ in the wave propagator. Namely, after scaling, we are reduced to the region
\begin{equation}\label{eq:freq_nonelliptic_scaled}
|\xi| \simeq |\mu| \approx \lambda/|t|, \qquad 1 \lesssim |t| \lll \lambda.
\end{equation}
We point out that, as the number of dyadic pieces grows logarithmically in $\lambda$, for our purposes it will be enough to obtain uniform estimates for each single piece, because we allow for $\lambda^\varepsilon$-losses in the final estimate. The advantage of the new frequency localisation \eqref{eq:freq_nonelliptic_scaled} after scaling is that we effectively need to work with a single ``isotropic'' frequency scale $\lambda/|t| \ggg 1$, much as in the Euclidean and elliptic cases. On the other hand, here we must work with the wave propagator at arbitrarily large time $|t|$, and correspondingly, with a spatial localisation to a sub-Riemannian ball of radius $\simeq |t|$, which instead brings considerable new challenges.

This is where the approach of \cite{LSV,SV} to obtain a global FIO representation for the wave propagator comes into play. Indeed, by an adaptation of that approach, in Section \ref{s:LSVparametrix} we obtain an approximate FIO representation for the wave propagators $\cos(t\sqrt{\opL})$ and $\sqrt{\opL} \sin(t\sqrt{\opL})$, by using the phase function
\begin{equation}\label{eq:phase_intro}
\phi_*(t,\bx,\by,\bxi) = (\bx - \bx_*^t) \cdot \bxi_*^t + \frac{i}{4} \langle |J_\mu| (x-x_*^t),x-x_*^t \rangle.
\end{equation}
Here $t \mapsto (\bx_*^t(\by,\bxi),\bxi_*^t(\by,\bxi))$ denotes the integral curve starting at $(\by,\bxi)$ relative to the Hamiltonian flow generated by the square root $\Ham$ of the symbol of $\opL$, and as usual we write $\bx_*^t = (x_*^t,u_*^t)$ and $\bxi_*^t = (\xi_*^t,\mu_*^t)$. Moreover $|J_\mu| \defeq (-J_\mu^2)^{1/2}$ for all $\mu \in \lie{g}_2^*$, where $J_\mu$ is the skew-symmetric endomorphism of $\lie{g}_1$ defined in \eqref{eq:Jmu}.

\smallskip

The particular form of the imaginary part $\Im \phi_*$ of the phase function in \eqref{eq:phase_intro} is somewhat different from what is suggested in \cite{CLV,LSV,SV} for the wave equation driven by elliptic operators. As a matter of fact, those works construct a global FIO parametrix for the wave propagator up to smoothing terms; however, this information per se is not enough to achieve uniform $L^1$-estimates in our case, given that we are working with an arbitrarily large time parameter. Indeed, the construction of the amplitude of the FIO parametrix goes through an iterative solution of certain ``transport equations'', where each iteration involves an integration in the time variable from $0$ to $t$, thus leading to a potential blow-up in $|t|$. However, our particular choice of the imaginary part in \eqref{eq:phase_intro} leads to substantial simplifications in the construction of the amplitude: namely, in Section \ref{s:opLambda_2step} we manage to obtain explicit and relatively simple expressions for the coefficients appearing in the iterative solution of the transport equations, which are valid on any $2$-step group. In turn, these expressions show that there is no blow-up in $|t|$, at least in the case of M\'etivier groups considered here.

Moreover, the fact that the imaginary part $\Im \phi_*$ in \eqref{eq:phase_intro} at time $t=0$ depends only on the first-layer spatial variables $x,y$ and the second-layer frequency variable $\mu$ significantly helps to keep control on $L^1$-norms when we introduce the complex phase in Section \ref{ss:introcomplex}.

\smallskip

Thanks to the FIO parametrix construction, we are finally reduced to proving $L^1$-estimates for oscillatory integrals of the form
\[
I_m(t,\bx) = \int e^{i\phi(t,\bx,\bxi)} \, \chi(|\xi|/2^m) \, \chi(|\mu|/2^m) \, \Den_\phi(t,\bx,\bxi) \,d\bxi,
\]
where $\phi \defeq \phi_*|_{\by = 0}$
and $\Den_\phi \defeq \sqrt{\det \partial_{\bxi} \nabla_{\bx} \phi}$, under the frequency localisation
\[
2^m \approx \lambda/|t|, \qquad 1 \lesssim |t| \lll \lambda.
\]
Notice that we can take $\by = 0$ as we are working with convolution kernels, thanks to translation-invariance. One can actually show that $|\Den_\phi| \simeq |t|^{1/2}$ here, and in particular the mixed hessian of $\phi$ is never degenerate. Now, an adaptation of the key method of \cite{SSS}, taking also the imaginary part of $\phi$ into account, shows that
\[
\|\chr_K \, I_m(t,\cdot)\|_{1} \lesssim_{K,T} 2^{m(d-1)/2} \qquad \text{for all } K \Subset G, \ |t|\leq T;
\]
this is enough for our purposes as long as $|t|$ remains bounded. However, we also need to consider large values of $|t|$, and the $t$-dependence of expressions such as $\bx^t \defeq \bx_*^t|_{\by = 0}$ and $\bxi^t \defeq \bxi_*^t|_{\by = 0}$ appearing in $\phi$ makes it possible for the above estimate again to blow up for large $|t|$; specifically, the method of \cite{SSS} is based on repeated integration by parts in $\bxi$, and $\bxi$-differentiation of $\bx^t$ and $\bxi^t$ produces powers of $|t|$. On the other hand, as we show in Section \ref{s:sss}, a further decomposition and suitable non-linear changes of variables in $\mu$ (see Figure \ref{fig:freq_mushear}) allow us to take advantage of the ``almost-periodicity'' in $t$ of the aforementioned expressions (cf.\ Remark \ref{rem:notMS} below), and eventually reduce to ``normalised'' oscillatory integrals, to which again an adaptation of the method of \cite{SSS} can be applied. The final estimate actually entails some growth in $|t|$, namely, we obtain
\[
\| \chr_{\overline{B}(0,4|t|)} I_m(t,\cdot)\|_1 \lessapprox |t|^{d_2-1/2} 2^{m(d-1)/2},
\]
but the inequality $d_2 \leq (d-1)/2$, which is true on M\'etivier groups, and the frequency localisation $2^m \approx \lambda/|t|$ shows that this bound is still good enough to obtain the desired estimate \eqref{eq:sketch_target} for the wave propagator.

\subsection{Notation}\label{ss:notation}
We write $\chr_{S}$ for the characteristic function of a set $S$. We denote by $\Rnon$ and $\Rpos$ the sets $[0,\infty)$ and $(0,\infty)$ of nonnegative and positive real numbers respectively. We write $\NN$ for the set of natural numbers (including $0$) and $\Npos$ for the set $\NN \setminus \{0\}$ of positive integers. Furthermore, we write $\dot \RR^n$ for $\RR^n \setminus \{0\}$.

If $A$ and $B$ are two nonnegative numbers, we write $A \lesssim B$ to indicate that there exists a constant $C \in (0,\infty)$ such that $A \leq C B$. We also write $A \simeq B$ to denote the conjunction of $A \lesssim B$ and $B \lesssim A$. Subscripted variants such as $\lesssim_p$ and $\simeq_p$ are used to indicate that the implicit constants may depend on a parameter $p$.

Unless otherwise specified, elements of $\RR^n$ or $\CC^n$ are thought of as column vectors.
For a scalar-valued function $f = f(x)$ with variable $x \in \RR^n$, we write $\nabla_x f$ for the column vector of partial derivatives $(\partial_{x_j} f)_j$, and $\partial_x f = (\nabla_x f)^T$ for the corresponding row vector. Correspondingly, if $v = v(x)$ is a (column) vector-valued function, then $\partial_x v$ is a matrix, whose columns are the partial derivatives $\partial_{x_j} v$.
The hessian matrix of a scalar-valued function $f = f(x)$ is therefore denoted by $\partial_x \nabla_x f$.
We shall also write $\langle v,w\rangle$ for $\sum_{j} v_j \overline{w_j}$ and $v \cdot w$ for $\sum_j v_j w_j$; clearly $\langle v,w \rangle = v \cdot w$ whenever $w$ is real. Due to our choice of coordinates on $G$ in Section \ref{ss:fcalculus} below, the notation $\langle \cdot,\cdot \rangle$ is consistent with that for the inner product associated with the sub-Laplacian introduced in Section \ref{ss:sketch}.

\section{Preliminaries}

\subsection{A dyadic partition of unity}\label{ss:dyadicpartition}
Throughout this work, we shall frequently make use of dyadic decompositions implemented through a partition of unity.
To this purpose, we fix an even nonnegative function $\chi_1 \in C^\infty_c(\dot\RR)$ such that $\supp \chi_1|_{\Rpos} \subseteq [1/2,2]$ and
\[
\sum_{k \in \ZZ} \chi_1(2^k \cdot) = 1 \quad\text{on } \Rpos.
\]
We also set
\[
\chi_0 \defeq \sum_{k > 0} \chi_1(2^k \cdot);
\]
notice that $\chi_0 \in C^\infty_c(\RR)$ and $\supp \chi_0 \subseteq [-1,1]$, and moreover
\begin{equation}\label{eq:dyadicpartition}
\chi_0 + \sum_{k \in \NN} \chi_1(2^{-k} \cdot) = 1 \quad\text{on } \RR.
\end{equation}
Further, we choose $\tilde\chi_1 \in C^\infty_c(0,\infty)$ such that $\tilde\chi_1 \chi_1 = \chi_1$.

\subsection{Functional calculus for the sub-Laplacian}\label{ss:fcalculus}
We assume from now on that $G$ is a $2$-step stratified group.
As in Section \ref{ss:sketch}, via the exponential map we identify $G$ with its Lie algebra $\lie{g} = \lie{g}_1 \oplus \lie{g}_2$.
Let $\opL$ be a homogeneous sub-Laplacian on $G$, as in \eqref{eq:opL}, and let $\langle \cdot,\cdot \rangle$ the corresponding inner product on $\lie{g}_1$, i.e., the one for which the vector fields in \eqref{eq:opL} form an orthonormal frame for the horizontal distribution.

\smallskip

We choose linear coordinates on $\lie{g}_1$ and $\lie{g}_2$, so that these spaces are identified with $\RR^{d_1}$ and $\RR^{d_2}$ respectively; the coordinates on $\lie{g}_1$ are chosen to be orthonormal with respect to the inner product. As a consequence, $G$ is identified with $\RR^{d_1}_x \times \RR^{d_2}_u$, with group law \eqref{eq:group_law},
and the Lebesgue measure $dx \, du$ is the Haar measure on $G$. We also define the automorphic dilations $\dil_r$ on $G$ as in \eqref{eq:dilations}, and we denote by $d = d_1+d_2$ and $Q = d_1 +2d_2$ the topological and homogeneous dimensions.

We denote by $\Sz(G)$ the Fr\'echet space of Schwartz functions on $G$ and by $\Sz'(G)$ the dual space of tempered distributions.
For a left-invariant linear operator $T$ acting on functions on $G$, we denote by $k_{T}$ its convolution kernel, i.e.,
\[
T \phi = \phi * k_T.
\]
The Schwartz Kernel Theorem ensures the existence of $k_T$ as a tempered distribution whenever $T : \Sz(G) \to \Sz'(G)$ is bounded; in this case, $T$ also extends to compactly supported distributions and we can write $k_T = T \delta_0$. This applies, e.g., when $T = F(\sqrt{\opL})$ for some bounded Borel function $F : \Rnon \to \CC$.

\smallskip

Let $\Sz(\Rnon)$ and $\Sz_e(\Rnon)$ be the Fr\'echet spaces of the restrictions to $\Rnon$ of Schwartz and even Schwartz functions on $\RR$ respectively. The following result, known as Hulanicki's Theorem \cite{H}, shows that $k_{F(\sqrt{\opL})}$ is an actual function on $G$ under suitable assumptions on $F$.

\begin{prp}\label{prp:hulanicki}
The linear mapping $F \mapsto k_{F(\sqrt{\opL})}$ is continuous
\begin{enumerate}[label=(\roman*)]
\item\label{en:hulanicki_even} from $\Sz_e(\Rnon)$ to $\Sz(G)$, and
\item\label{en:hulanicki_l1} from $\Sz(\Rnon)$ to $L^1(G)$.
\end{enumerate}
\end{prp}
\begin{proof}
The result in \cite{H} states the boundedness of the linear mapping $H \mapsto k_{H(\opL)}$ from $\Sz(\RR)$ to $\Sz(G)$. The change of variables $F(\zeta) = H(\zeta^2)$ immediately gives the boundedness result in part \ref{en:hulanicki_even}.

As for part \ref{en:hulanicki_l1}, let us write $F \in \Sz(\Rnon)$ as $F(\zeta) = F_e(\zeta) + \zeta F_o(\zeta)$, where $F_e,F_o \in \Sz_e(\RR)$. Consequently, by \eqref{eq:dyadicpartition}, we can decompose
\[
F(\zeta) 
= F_{-1}(\zeta) + \sum_{k=0}^\infty 2^{-k} F_k(2^k \zeta), \qquad \zeta > 0,
\]
where $F_{-1}(\zeta) = F_e(\zeta) + \zeta F_o(\zeta) \chi_0(1/\zeta)$ and $F_k(\zeta) = \zeta F_o(2^{-k} \zeta) \chi_1(1/\zeta)$ for $k \in \NN$. It is easily seen that $\{F_k\}_{k \geq -1}$ is a bounded subset of $\Sz_e(\Rnon)$, with bounds only depending on the Schwartz bounds of $F$.
As a consequence, by part \ref{en:hulanicki_even} there exists a continuous seminorm $\varrho$ on $\Sz(\Rnon)$ such that
\[
\sup_{k \geq -1} \|k_{F_k(\sqrt{\opL})}\|_{1} \lesssim \varrho(F),
\]
and therefore
\[
\|k_{F_{\sqrt{\opL}}}\|_1 \leq \|k_{F_{-1}(\sqrt{\opL})}\|_1 + \sum_{k=0}^\infty 2^{-k} \|k_{F_{k}(2^k \sqrt{\opL})}\|_1 \lesssim \varrho(F),
\]
where we used the fact that, by automorphic scaling, $\|k_{F_{k}(2^k \sqrt{\opL})}\|_1 = \|k_{F_{k}(\sqrt{\opL})}\|_1$.
\end{proof}

\begin{rem}
Notice that the boundedness in part \ref{en:hulanicki_even} does not hold if $\Sz_e(\Rnon)$ is replaced by $\Sz(\Rnon)$, as shown for example by the Poisson kernels, corresponding to $F(\zeta) = e^{-\zeta}$. It is well known that Hulanicki's result holds more generally for left-invariant sub-Laplacians on Lie groups with polynomial growth (see, e.g., \cite{MRT}), and that actually similar results hold in even greater generality (with integral kernels replacing convolution kernels) for operators $\opL$ satisfying Gaussian-type heat kernel bounds on doubling manifolds (see, e.g., \cite[Theorem 6.1(iii)]{M_Kohn} or \cite{KP}).
\end{rem}

Another basic property that we shall frequently use is \emph{finite propagation speed} for the wave equation driven by $\opL$ (see, e.g., \cite{CM,Mel,Si}, which apply in greater generality than Carnot groups). Let $\overline{B}(0,r)$ denotes the closed sub-Riemannian ball on $G$ centred at the identity and of radius $r$.

\begin{prp}\label{prp:fps}
For all $t \in \RR$,
\begin{equation}\label{eq:fps}
\supp k_{\cos(t \sqrt{\opL})} \subseteq \overline{B}(0,|t|), \qquad t \in \RR,
\end{equation}
More generally, if $F : \RR \to \CC$ is an even function whose Fourier transform $\widehat{F}$ is supported in $[-r,r]$, then
\begin{equation}\label{eq:fps2}
\supp k_{F(\sqrt{\opL})} \subseteq \overline{B}(0,r).
\end{equation}
\end{prp}

We shall also make use of the following \emph{vanishing moment property}.

\begin{prp}
The following hold.
\begin{enumerate}[label=(\roman*)]
\item\label{en:vm_poly} For every polynomial $p$ on $G$, there exists $N \in \NN$ such that $\opL^N p = 0$.
\item\label{en:vm_moment} For all polynomials $p$ and all $\chi \in C^\infty_c(\Rpos)$,
\begin{equation}\label{eq:vanishingmoments}
\int_G p \, k_{\chi(\sqrt{\opL})} = 0 .
\end{equation}
\end{enumerate}
\end{prp}
\begin{proof}
Part \ref{en:vm_poly} is easily seen by observing that $\opL$ is a differential operator with polynomial coefficients, thus $\opL$ maps polynomials into polynomials; moreover, by homogeneity, an application of $\opL$ decreases the $\dil_r$-homogeneity degree by $2$. 

As for part \ref{en:vm_moment}, if $N \in \NN$ is taken as in part \ref{en:vm_poly}, then we can write $k_{\chi(\sqrt{\opL})} = \opL^N k_{\chi_N(\sqrt{\opL})}$ for a suitable $\chi_N \in C^\infty_c(\Rpos)$ and use repeated integration by parts.
\end{proof}

\subsection{H-type and M\'etivier groups}\label{ss:metivier}
For all $\mu \in \lie{g}^*_2$, let $J_\mu : \lie{g}_1 \to \lie{g}_1$ be the skew-symmetric linear map defined by \eqref{eq:Jmu}, where $\langle \cdot, \cdot \rangle$ is the inner product associated with the sub-Laplacian $\opL$.

Recall that the group $G$ is called a \emph{M\'etivier group} if, for any $\mu \neq 0$, the skew-symmetric bilinear form $\mu[\cdot,\cdot]$ on $\lie{g}_1$ is non-degenerate; this is equivalent to the fact that $J_{\mu}$ is non-degenerate for any $\mu \neq 0$, i.e., to the fact that
\begin{equation}\label{eq:metivier}
|J_\mu x| \simeq |\mu||x| 
\end{equation}
for all $\mu \in \lie{g}_2^*$ and $x \in \lie{g}_1$. If the stronger condition
\[
|J_\mu x| = |\mu| |x|
\]
holds, then $G$ is called a \emph{Heisenberg-type group}; this is equivalent to saying that
\begin{equation}\label{eq:Htype}
-J_\mu^2 = |\mu|^2 I
\end{equation}
for all $\mu \in \lie{g}_2^*$.

\begin{prp}\label{prp:modJmu_analytic}
Let 
\[
\Omega \defeq \{ \mu \in \lie{g}_2^* \tc \rk J_\mu \text{ is maximal} \}.
\]
Let moreover $P_0^\mu$ denote the orthogonal projection onto $\Ker J_\mu$.
Then:
\begin{enumerate}[label=(\roman*)]
\item $\Omega$ is a homogeneous Zariski-open subset of $\lie{g}_2^* \setminus \{0\}$.
\item $\mu \mapsto P_0^\mu$ is a $0$-homogeneous rational function on $\lie{g}_2^*$, real-analytic on $\Omega$.
\item $\mu \mapsto |J_\mu| \defeq (-J_\mu^2)^{1/2}$ is a $1$-homogeneous Lipschitz-continuous function on $\lie{g}_2^*$, which is real-analytic on $\Omega$.
\end{enumerate}
\end{prp}

\begin{rem}
If $G$ is a M\'etivier group, then $\Omega=\lie{g}_2^*\setminus \{0\}$. This is a crucial difference with the homogeneous Zariski-open set $\lie{g}_{2,r}^*$ considered, e.g., in \cite[Lemma 4]{MM_newclasses}, related to the analyticity of all eigenvalues and eigenprojections of $J_\mu$, which in general is smaller than $\Omega$. The fact that we do not need to consider separate eigenvalues, but we can work with $J_\mu$ (and $|J_\mu|$) as a whole, is one of the reasons why our approach extends to non-H-type M\'etivier groups (see Remark \ref{rem:notMS} below).
\end{rem}

\begin{proof}
As $J_\mu$ is a skew-symmetric endomorphism of $\lie{g}_1$ depending linearly on $\mu$, $B_\mu \defeq -J_\mu^2$ is a nonnegative symmetric endomorphism of $\lie{g}_1$ and a degree-$2$ homogeneous polynomial in $\mu$. We can therefore write the characteristic polynomial of $B_\mu$ as
\[
\det(B_\mu - \lambda I) = \sum_{j=0}^{d_1} (-1)^j p_j(\mu) \lambda^j,
\]
where each $p_j$ is a $2(d_1-j)$-homogeneous polynomial; moreover, if
\[
b_1^\mu \geq \dots \geq b_{d_1}^\mu \geq 0
\]
are the eigenvalues of $B_\mu$ (repeated according to their multiplicities), then
\[
\det(B_\mu - \lambda I)=\prod_{k=1}^{d_1}( b_k^\mu-\lambda),
\]
and so 
\[
p_j(\mu) = \sum_{\substack{A \subseteq \{1,\dots,d_1\} \\ |A| = d_1-j}} \prod_{k \in A} b_k^\mu.
\]
Let $r \defeq \max \{ \rk J_\mu \tc \mu \in \lie{g}_2^* \}$. Then clearly $p_j \equiv 0$ for $j < d_1-r$, and $p_{d_1-r}(\mu)=\prod_{k=1}^{r} b_k^\mu$, so 
\[
\Omega = \{ \mu \in \lie{g}_2^* \tc \rk J_\mu = r \} = \{ \mu \in \lie{g}_2^* \tc p_{d_1-r}(\mu) \neq 0 \},
\]
whence it follows that $\Omega$ is a homogeneous Zariski-open subset of $\lie{g}_2^* \setminus \{0\}$.

Moreover,
\[
q^\mu(\lambda) \defeq \sum_{j=0}^{r} (-1)^j p_{j+d_1-r}(\mu) \lambda^j  = (-\lambda)^{-(d_1-r)} \det(B_\mu-\lambda I) 
\]
is a degree-$r$ polynomial in $\lambda$, whose coefficients are polynomials in $\mu$, with the property that
\[
q^\mu(b_k^\mu) = 0 \ \text{for } k=1,\dots,r, \qquad q^\mu(0) = p_{d_1-r}(\mu),
\]
whence, for all $\mu \in \Omega$,
\[
q^\mu(B_\mu) = p_{d_1-r}(\mu) P_0^\mu,
\]
as $P_0^\mu$ is the eigenprojection of $B_\mu$ relative to the eigenvalue $0$; this shows that $P_0^\mu = q^\mu(B_\mu)/p_{d_1-r}(\mu)$ is a $0$-homogeneous rational function of $\mu$ on $\Omega$.

Clearly $|J_\mu| = \sqrt{B_\mu}$ is a $1$-homogeneous function of $\mu$, and its Lipschitz continuity is proved, e.g., in \cite[Corollary 5.3]{Bhatia_abs}.
Now, it is well known (see, e.g., \cite[Lemma 4]{MM_newclasses}) that the eigenvalues $b_k^\mu$ are continuous functions of $\mu \in \lie{g}_2^*$. In particular, for any fixed $\mu_0 \in \Omega$, there exists an open neighbourhood $U \subseteq \Omega$ of $\mu_0$ such that $M \defeq \sup_{\mu \in U} b_1^\mu < \infty$ and $m \defeq \inf_{\mu \in U} b_r^\mu > 0$. In particular, if $\Gamma_{m,M} = \{ z \in \CC \tc m/2 \leq \Re z \leq 2M, |\Im z| \leq 1 \}$, then all the nonzero eigenvalues of $B_\mu$ are contained in $\Gamma_{m,M}$ for all $\mu \in U$, and we can write $|J_\mu|$ via the contour integral
\[
|J_\mu| = \sqrt{B_\mu} = \frac{1}{2\pi i} \int_{\partial \Gamma_{m,M}} z^{1/2} (z-B_\mu)^{-1} \,dz,
\]
where $z \mapsto z^{1/2}$ is the standard branch of the complex square root on $\{ z \in \CC \tc \Im z > 0\}$ (cf.\ \cite[Section II.1.4]{Kato}); this representation shows that $\mu \mapsto |J_\mu|$ is real-analytic on $U$, and therefore (by taking an arbitrary $\mu_0$) on the whole of $\Omega$.
\end{proof}

\section{Spatial localisation via finite propagation speed}\label{s:spatialloc}

Our main objective in this and the following sections will be the proof, through a series of reductions, of the key estimate \eqref{eq:sketch_target} with $\chi=\chi_1$, that is,
\begin{equation}\label{eq:target_first}
\| \exp(\pm i\sqrt{\opL}) \chi_1(\sqrt{\opL}/\lambda) \delta_0 \|_{1} \lessapprox \lambda^{(d-1)/2}
\end{equation}
for all $\lambda \geq 1$. We shall see in Section \ref{s:mainproofs} how Theorem \ref{thm:main} follows from \eqref{eq:target_first}.

We start by proving the following localisation estimate. The proof, which exploits the finite propagation speed property stated in Proposition \ref{prp:fps}, is analogous to that of \cite[Proposition 8.8]{MSe}; we include the details as similar ideas will be used again in later sections.

\begin{prp}\label{prp:spatial_loc}
For all $\lambda \geq 1$,
\[
\exp(\pm i\sqrt{\opL}) k_{\chi_1(\sqrt{\opL}/\lambda)} = \chr_{\overline{B}(0,2)} \exp(\pm i\sqrt{\opL})  k_{\chi_1(\sqrt{\opL}/\lambda)} + R_\lambda,
\]
where
\[
\|R_\lambda\|_1 \lesssim_N \lambda^{-N}
\]
for all $N \in \NN$.
\end{prp}
\begin{proof}
Let us write
\begin{equation}\label{eq:exp_cos_sin}
\exp(\pm i \sqrt{\opL}) \chi_1(\sqrt{\opL}/\lambda) = \cos(\sqrt{\opL}) \chi_1(\sqrt{\opL}/\lambda) \pm i \lambda^{-1} \sqrt{\opL} \sin(\sqrt{\opL}) \chi_2(\sqrt{\opL}/\lambda),
\end{equation}
where $\chi_2(s) \defeq |s|^{-1} \chi_1(s)$. Notice that $\chi_2$ is also an even nonnegative function in $C^\infty_c(\dot\RR)$ with $\supp \chi_2|_{\Rpos} \subseteq [1/2,2]$.

Now, for $r=1,2$ and any $\delta > 0$,
\begin{equation}\label{eq:fourier_support_dec}
\chi_r(\sqrt{\opL}/\lambda) = \frac{1}{2\pi} \int_{\RR} \widehat{\chi_r}(\tau) \,\cos(\tau\sqrt{\opL}/\lambda) \,d\tau
= \chi_r^{0,\delta,\lambda}(\sqrt{\opL}/\lambda) + \chi_r^{\infty,\delta,\lambda}(\sqrt{\opL}/\lambda),
\end{equation}
where $\chi_r^{0,\delta,\lambda} \in \Sz_e(\RR^+)$ is given by
\[
\chi_r^{0,\delta,\lambda}(s) = \frac{1}{2\pi} \int_{\RR} \chi_0(2\tau/\lambda^{\delta}) \, \widehat{\chi_r}(\tau) \,\cos(\tau s) \,d\tau
\]
and satisfies $\supp\widehat{\chi_r^{0,\delta,\lambda}} \subseteq [-\lambda^{\delta},\lambda^{\delta}]$. By finite propagation speed \eqref{eq:fps2}, we deduce that 
\begin{equation}\label{eq:support_1de}
\supp k_{\chi_r^{0,\delta,\lambda}(\sqrt{\opL}/\lambda)} \subseteq \overline{B}(0,\lambda^{\delta-1})
\end{equation}
and
\begin{multline*}
\supp (\cos(\sqrt{\opL}) k_{\chi_1^{0,\delta,\lambda}(\sqrt{\opL}/\lambda)}) \cup \supp (\sqrt{\opL} \sin(\sqrt{\opL}) k_{\chi_2^{0,\delta,\lambda}(\sqrt{\opL}/\lambda)}) \\
\subseteq  \overline{B}(0,1+\lambda^{\delta-1}) \subseteq  \overline{B}(0,2),
\end{multline*}
if $\lambda \geq 1$ and $\delta \leq 1$. So
\[\begin{split}
R_\lambda 
&= \chr_{G \setminus \overline{B}(0,2)} k_{\exp(\pm i \sqrt{\opL}) \chi_1(\sqrt{\opL}/\lambda)} \\
&= \chr_{G \setminus \overline{B}(0,2)} k_{\cos(\sqrt{\opL}) \chi_1^{\infty,\delta,\lambda}(\sqrt{\opL}/\lambda)} \pm i \lambda^{-1} \chr_{G \setminus \overline{B}(0,2)} k_{\sqrt{\opL} \sin(\sqrt{\opL}) \chi_2^{\infty,\delta,\lambda}(\sqrt{\opL}/\lambda)}.
\end{split}\]

On the other hand, as
\[
\widehat{\chi_{r}^{\infty,\delta,\lambda}}(\tau) = (1-\chi_0(2\tau/\lambda^{\delta})) \, \widehat{\chi_r}(\tau),
\]
for $r=1,2$, and $1-\chi_0$ vanishes identically on $[-1/4,1/4]$, it is easily checked that
\begin{equation}\label{eq:schwartz_bound_error}
\varrho(\chi_{r}^{\infty,\delta,\lambda}) \lesssim_{\varrho,N,\delta} \lambda^{-N}
\end{equation}
for any continuous seminorm $\varrho$ on $\Sz_e(\Rnon)$ and all $\lambda \geq 1$ and $N \in \NN$, thus also
\[
\varrho(\cos(\cdot) \chi_{1}^{\infty,\delta,\lambda}(\cdot/\lambda)), \varrho(\cdot \sin(\cdot) \chi_{2}^{\infty,\delta,\lambda}(\cdot/\lambda)) \lesssim_{\varrho,N,\delta} \lambda^{-N},
\]
and therefore, by Proposition \ref{prp:hulanicki},
\[
\|R_\lambda\|_1 
\leq \|k_{\cos(\sqrt{\opL}) \chi_{1}^{\infty,\delta,\lambda}(\sqrt{\opL}/\lambda) } \|_1 + \|k_{\sqrt{\opL} \sin(\sqrt{\opL}) \chi_{2}^{\infty,\delta,\lambda}(\sqrt{\opL}/\lambda) } \|_1 
\lesssim_{N,\delta} \lambda^{-N},
\]
as desired.
\end{proof}

Thanks to Proposition \ref{prp:spatial_loc}, the proof of \eqref{eq:target_first} reduces to that of
\begin{equation}\label{sploc_scloc_wave_est}
\|\chr_{\overline{B}(0,2)} \exp(\pm i \sqrt{\opL}) k_{\chi_1(\sqrt{\opL}/\lambda)} \|_{1} \lessapprox \lambda^{(d-1)/2}
\end{equation}
for all $\lambda \geq 1$.

\section{Spectral versus frequency localisation}\label{s:spectrumvsfrequency}

Recall that $G$ is identified with $\RR^{d_1}_x \times \RR^{d_2}_u$. We write $D_x$ and $D_u$ for the vectors $(-i\partial_{x_j})_{j=1}^{d_1}$ and $(-i\partial_{u_j})_{j=1}^{d_2}$ of partial derivatives in $x$ and $u$. Thus, by \eqref{eq:dyadicpartition}, we can write
\begin{equation}\label{eq:freq_dec}
k_{\chi_1(\sqrt{\opL}/\lambda)}
= \sum_{j \in \ZZ} \sum_{k \in \ZZ} \chi_1(\sqrt{\opL}/\lambda) \, \chi_1(2^{-j}|D_x|) \, \chi_1(2^{-k}|D_u|) \delta_0. 
\end{equation}
As in Section \ref{ss:sketch}, we shall use the symbols $\xi$ and $\mu$ to denote the frequency variables dual to $x$ and $u$; with this notation, the operators $|D_x|$ and $|D_u|$ correspond, via the Euclidean Fourier transform, to the multiplication operators by $|\xi|$ and $|\mu|$. Furthermore, we shall use the notation $\bx = (x,u)$ and $\bxi = (\xi,\mu)$.

We now want to show that a number of summands in the above sum \eqref{eq:freq_dec} are ``negligible'' to the purpose of proving the bound \eqref{sploc_scloc_wave_est}.

\begin{prp}\label{prp:spectrum_frequency}
Let $\epsilon_1,\epsilon_2,\epsilon_3,\epsilon_4>0$, and set $\vec\epsilon = (\epsilon_1,\epsilon_2,\epsilon_3,\epsilon_4)$. Then, for all $\lambda \geq 1$,
\begin{equation}\label{eq:spectrum_frequency_dec}
k_{\chi_1(\sqrt{\opL}/\lambda)} = A_{\vec\epsilon,\lambda} + B_{\vec\epsilon,\lambda} + R_{\vec\epsilon,\lambda},
\end{equation}
where
\begin{equation}\label{eq:regionsAB}
\begin{aligned}
A_{\vec\epsilon,\lambda} &\defeq \sum_{j \tc 2^j < \lambda^{1+\epsilon_2}} \sum_{k \tc \lambda^{2-\epsilon_3-\epsilon_4} \leq 2^k < \lambda^{2+\epsilon_1}} \chi_1(\sqrt{\opL}/\lambda) \, \chi_1(2^{-j}|D_x|) \, \chi_1(2^{-k}|D_u|) \delta_0, \\
B_{\vec\epsilon,\lambda} &\defeq \sum_{j \tc \lambda^{1-\epsilon_3} \leq 2^j < \lambda^{1+\epsilon_2}} \sum_{k \tc 2^k < \lambda^{2-\epsilon_3-\epsilon_4}} \chi_1(\sqrt{\opL}/\lambda) \, \chi_1(2^{-j}|D_x|) \, \chi_1(2^{-k}|D_u|) \delta_0,
\end{aligned}
\end{equation}
and
\begin{equation}\label{eq:frequency_remainder}
\|\exp(\pm i \sqrt{\opL}) R_{\vec\epsilon,\lambda}\|_{1} \lesssim_{\vec\epsilon,N} \lambda^{-N}
\end{equation}
for all $N \in \NN$.
\end{prp}

\begin{figure}

\begin{tikzpicture}[scale=1.2]

	\fill[fill=yellow!20] (0,2.8) -- (0,3.6) -- (3.6,3.6) -- (3.6,2.8); 
	\fill[fill=orange!20] (2.3,0) -- (2.3,2.8) -- (3.6,2.8) -- (3.6,0); 
	\fill[fill=green!20] (0,0) -- (0,2.3) -- (1.7,2.3) -- (1.7,0); 
	\fill[fill=blue!20] (1.7,0) node[below left] {$\lambda^{1-\epsilon_3}$} -- (2.3,0) node[below right] {$\lambda^{1+\epsilon_2}$} -- (2.3,2.2) -- (1.7,2.2); 
	\fill[fill=red!20] (0,2.2) node[below left] {$\lambda^{2-\epsilon_3-\epsilon_4}$} -- (0,2.8) node[above left] {$\lambda^{2+\epsilon_1}$} -- (2.3,2.8) -- (2.3,2.2);

  \draw[thick,->] (-0.2,0) -- (3.7,0) node[right] {$|\xi|$};
  \draw[thick,->] (0,-0.2) -- (0,3.7) node[above] {$|\mu|$};

  \draw (1.7,.1) -- ++(0,-.2) ;
  \draw (2,.1) -- ++(0,-.2) node[below] {$\lambda$} ;
	\draw (2.3,.1) -- ++(0,-.2) ;
  \draw (0.1,2.5) -- ++(-.2,0) node[left] {$\lambda^2$} ;
  \draw (0.1,2.8) -- ++(-.2,0) ;
  \draw (0.1,2.2) -- ++(-.2,0) ;

  \node at (1.1,2.5) {$A$};
  \node at (2,1.1) {$B$};
	\node at (1.9,3.2) {$R^{(1)}$};
	\node at (3,1.6) {$R^{(2)}$};
	\node at (.9,1.1) {$R^{(3)}_{}$};

\end{tikzpicture}

\caption{Frequency localisation according to Proposition \ref{prp:spectrum_frequency}.}
\label{fig:freq_loc}

\end{figure}
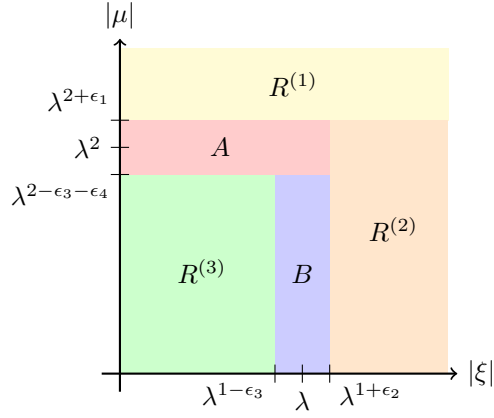

\begin{proof}
From \eqref{eq:spectrum_frequency_dec} and \eqref{eq:freq_dec} it is clear that
\[
R_{\vec\epsilon,\lambda} = R_{\vec\epsilon,\lambda}^{(1)} + R_{\vec\epsilon,\lambda}^{(2)} + R_{\vec\epsilon,\lambda}^{(3)},
\]
where
\begin{equation}\label{eq:Rremainders}
\begin{aligned}
R_{\vec\epsilon,\lambda}^{(1)} &\defeq \sum_{j \in \ZZ} \sum_{k \tc 2^k \geq \lambda^{2+\epsilon_1}} \chi_1(\sqrt{\opL}/\lambda) \, \chi_1(2^{-j}|D_x|) \, \chi_1(2^{-k}|D_u|) \delta_0, \\
R_{\vec\epsilon,\lambda}^{(2)} &\defeq \sum_{j \tc 2^j \geq \lambda^{1+\epsilon_2}} \sum_{k \tc 2^k < \lambda^{2+\epsilon_1}} \chi_1(\sqrt{\opL}/\lambda) \, \chi_1(2^{-j}|D_x|) \, \chi_1(2^{-k}|D_u|) \delta_0, \\
R_{\vec\epsilon,\lambda}^{(3)} &\defeq \sum_{j \tc 2^j < \lambda^{1-\epsilon_3}} \sum_{k \tc 2^k < \lambda^{2-\epsilon_3-\epsilon_4}} \chi_1(\sqrt{\opL}/\lambda) \, \chi_1(2^{-j}|D_x|) \, \chi_1(2^{-k}|D_u|) \delta_0 .
\end{aligned}
\end{equation}
In addition, as $\chi_1 \tilde\chi_1 = \chi_1$,
\[
\tilde\chi_1(\sqrt{\opL}/\lambda) R_{\vec\epsilon,\lambda} = R_{\vec\epsilon,\lambda}.
\]
Thus, in order to prove the estimate \eqref{eq:frequency_remainder}, it is enough to prove that
\begin{equation}\label{eq:frequency_remainder_l1}
\| R_{\vec\epsilon,\lambda}^{(\ell)} \|_1 \lesssim_{\vec\epsilon,N} \lambda^{-N}
\end{equation}
for all $N \in \NN$, $\lambda \geq 1$, and $\ell = 1,2,3$.
Indeed, by combining the estimates \eqref{eq:frequency_remainder_l1}, we would obtain
\[
\| R_{\vec\epsilon,\lambda} \|_1 \lesssim_{\vec\epsilon,N} \lambda^{-N};
\]
as moreover
\[
\|\exp(\pm i \sqrt{\opL}) \, \tilde\chi_1(\sqrt{\opL}/\lambda)\|_{1 \to 1} \lesssim \lambda^{N_0}
\]
for some $N_0 \in \NN$ by Proposition \eqref{prp:hulanicki}, we could then conclude that
\[
\| \exp(\pm i \sqrt{\opL}) R_{\vec\epsilon,\lambda} \|_1 \leq \|\exp(\pm i \sqrt{\opL}) \tilde\chi_1(\sqrt{\opL}/\lambda) \|_{1\to 1} \| R_{\vec\epsilon,\lambda} \|_{1} \lesssim_{\vec\epsilon,N} \lambda^{-N},
\]
as desired.

Notice that
\[
\chi_1(\sqrt{\opL}/\lambda) \, \chi_1(2^{-j}|D_x|) \, \chi_1(2^{-k}|D_u|) \delta_0 
= \left[ (F^1_j \otimes F^2_k) * S_\lambda \right],
\]
where $*$ is the convolution on $G$, while, by homogeneity,
\[
S_\lambda(x,u) = \lambda^Q S(\lambda x, \lambda^2 u), \qquad F^1_j(x) = 2^{jd_1} F^1(2^j x), \qquad F^2_k(u) = 2^{kd_2} F^2(2^k u),
\]
with $S \defeq k_{\chi_1(\sqrt{\opL})} \in \Sz(G)$ (by Proposition \eqref{prp:hulanicki}), $F^1 \defeq \Four_{\RR^{d_1}}^{-1} (\chi_1(|\cdot|)) \in \Sz(\RR^{d_1})$ and $F^2 \defeq \Four_{\RR^{d_2}}^{-1} (\chi_1(|\cdot|)) \in \Sz(\RR^{d_2})$. Moreover, in light of \eqref{eq:vanishingmoments}, all moments of the function $S$ vanish, and the same is true of the functions $F^1$ and $F^2$, since their Euclidean Fourier transforms are supported away from the origin.

We start by proving \eqref{eq:frequency_remainder_l1} for $\ell=1$.
Notice that, for all $k \in \ZZ$,
\begin{equation}\label{eq:recomp}
\sum_{j \in \ZZ} \left[ (F^1_j \otimes F^2_k) * S_\lambda \right] = (\delta_0 \otimes F^2_k) * S_\lambda
\end{equation}
and
\[\begin{split}
(\delta_0 \otimes F^2_k) * S_\lambda(x,u) 
&= \int_{\RR^{d_2}} F^2_k(u') S_\lambda(x,u-u') \,du' \\
&= \int_{\RR^{d_2}} F^2(u') S_\lambda(x,u-2^{-k} u') \,du',
\end{split}\]
which has the same $L^1(G)$-norm as
\begin{equation}\label{eq:1st_canc_resc}
\lambda^{-Q} (\delta_0 \otimes F^2_k) * S_\lambda(\lambda^{-1}x,\lambda^{-2}u)  
= \int_{\RR^{d_2}} F^2(u') S(x,u-2^{-k} \lambda^2 u') \,du'.
\end{equation}

Notice now that, if we set
\[
f(s) \defeq S(x,u-s2^{-k} \lambda^2 u'),
\]
then, for all $h \in \NN$,
\[
f^{(h)}(s) = (-2^{-k} \lambda^2)^h (u' \cdot \nabla_u)^h S(x,u-s2^{-k} \lambda^2 u'),
\]
thus, by Taylor's formula with integral remainder,
\[\begin{split}
&S(x,u-2^{-k} \lambda^2 u') \\
&= f(1) 
= \sum_{h=0}^N \frac{f^{(h)}(0)}{h!} + \frac{1}{N!} \int_0^1 f^{(N+1)}(s) \,(1-s)^{N} \,ds \\
&= \sum_{h=0}^N \frac{(-2^{-k} \lambda^2)^h}{h!}  (u' \cdot \nabla_u)^h S(x,u) \\
&\quad+ \frac{(-2^{-k} \lambda^2)^{N+1}}{N!} \int_0^1 (u' \cdot \nabla_u)^{N+1} S(x,u-s2^{-k} \lambda^2 u') \, (1-s)^{N} \,ds
\end{split}\]
for any $N \in \NN$.

By plugging this expression into \eqref{eq:1st_canc_resc} and exploiting the vanishing moments of $F^2$, we obtain that
\[
\begin{split}
&\lambda^{-Q} (\delta_0 \otimes F^2_k) * S_\lambda(\lambda^{-1}x,\lambda^{-2}u)  \\
&= \int_{\RR^{d_2}} F^2(u')  \frac{(-2^{-k} \lambda^2)^{N+1}}{N!} \int_0^1 (u' \cdot \nabla_u)^{N+1} S(x,u-s2^{-k} \lambda^2 u') \, (1-s)^{N} \,ds \,du' \\
&= \frac{(-2^{-k} \lambda^2)^{N+1}}{N!}  \int_0^1 \int_{\RR^{d_2}} F^2(u') (u' \cdot \nabla_u)^{N+1} S(x,u-s2^{-k} \lambda^2 u')   \,du' \,(1-s)^{N} \,ds.
\end{split}
\]
Thus
\[
\begin{split}
\|(\delta_0 \otimes F^2_k) * S_\lambda \|_1 
&\lesssim_N (2^{-k} \lambda^2)^{N+1} \int_G \int_0^1 \int_{\RR^{d_2}} |F^2(u')| |u'|^{N+1} \\
&\qquad\times |\nabla_u^{N+1} S(x,u-s2^{-k} \lambda^2 u')|   \,du' \,(1-s)^{N} \,ds \,dx\,du \\
&= (2^{-k} \lambda^2)^{N+1} \int_0^1 (1-s)^N \,ds \\
&\qquad\times \int_{\RR^{d_2}} |F^2(u')| |u'|^{N+1}  \,du'  \int_G |\nabla_u^{N+1} S(x,u)| \,dx\,du \\
&\lesssim_N (2^{-k} \lambda^2)^{N+1},
\end{split}
\]
as $F^2$ and $S$ are in the Schwartz class. As a consequence, for all $\lambda \geq 1$,
\begin{equation}\label{eq:1st_negligible}
\sum_{k \tc 2^k \geq \lambda^{2+\epsilon_1}} \|(\delta_0 \otimes F^2_k) * S_\lambda \|_1 
\lesssim_{N} \sum_{k \tc 2^k \geq \lambda^{2+\epsilon_1}} (2^{-k} \lambda^2)^{N+1} \lesssim \lambda^{-\epsilon_1 (N+1)}.
\end{equation}
In conclusion, by \eqref{eq:recomp} and \eqref{eq:1st_negligible},
\[
\begin{split}
\|R^{(1)}_{\vec\epsilon,\lambda}\|_1 
&= \left\| \sum_{j \in \ZZ} \sum_{k \tc 2^k \geq \lambda^{2+\epsilon_1}} (F^1_j \otimes F^2_k) * S_\lambda \right\|_1 \\
&\leq \sum_{k \tc 2^k \geq \lambda^{2+\epsilon_1}} \|(\delta_0 \otimes F^2_k) * S_\lambda \|_1 \lesssim_{\epsilon_1,N} \lambda^{-N}
\end{split}
\]
for all $N \in \NN$ and $\lambda \geq 1$; this proves \eqref{eq:frequency_remainder_l1} for $\ell=1$.

We now prove the analogous bound for $\ell=2$. Observe that, for all $j \in \ZZ$,
\begin{equation}\label{eq:2nd_recomp}
\sum_{k \tc 2^k < \lambda^{2+\epsilon_1}} (F^1_j \otimes F^2_k) = F^1_j \otimes \tilde F^2_{k_0(\epsilon_1;\lambda)},
\end{equation}
where
\[
k_0(\epsilon_1;\lambda) \defeq \max \{ k \tc 2^k < \lambda^{2+\epsilon_1} \}+1
\]
and $\tilde F^2_{k} \defeq \sum_{k' < k} F^2_{k'} = 2^{k d_2} \tilde F^2(2^k \cdot)$, with $\tilde F^2 \defeq \Four_{\RR^{d_2}}^{-1}(\chi_0(|\cdot|)) \in \Sz(\RR^{d_2})$.
Notice that, for all $j,k \in \ZZ$,
\[\begin{split}
&(F^1_j \otimes \tilde F^2_k) * S_\lambda(x,u) \\
&= \int_G F^1_j(x') \, \tilde F^2_k(u') \, S_\lambda(x-x',u-u'+[x,x']/2) \,dx' \,du' \\
&= \int_G F^1(x') \, \tilde F^2_k(u') \, S_\lambda(x-2^{-j}x',u-u'+2^{-j}[x,x']/2) \,dx' \,du' ,
\end{split}\]
which has the same $L^1(G)$-norm as
\begin{equation}\label{eq:2nd_canc_resc}
\begin{split}
&\lambda^{-Q}(F^1_j \otimes \tilde F^2_k) * S_\lambda(\lambda^{-1}x,\lambda^{-2}u) \\
&= \int_G F^1(x') \, \tilde F^2_k(u') \, S(x-2^{-j}\lambda x',u-\lambda^2 u'+2^{-j}\lambda [x,x']/2) \,dx' \,du'.
\end{split}
\end{equation}

Arguing much as above, if we set
\[
f(s) \defeq S(x-s2^{-j}\lambda x',u-\lambda^2 u'+s2^{-j}\lambda [x,x']/2),
\]
then, for all $h \in \NN$,
\[\begin{split}
f^{(h)}(s) 
&= (2^{-j} \lambda)^h \sum_{\ell=0}^h \binom{h}{\ell} (-1)^\ell 2^{\ell-h} \\
&\quad\times ([x,x'] \cdot \nabla_u)^{h-\ell} (x' \cdot \nabla_x)^\ell S(x-s2^{-j}\lambda x',u-\lambda^2 u'+s2^{-j}\lambda [x,x']/2),
\end{split}\]
thus, by Taylor's formula, for all $N \in \NN$,
\[\begin{split}
&S(x-2^{-j}\lambda x',u-\lambda^2 u'+2^{-j}\lambda [x,x']/2) \\
&=f(1) 
= \sum_{h=0}^N \frac{f^{(h)}(0)}{h!} + \frac{1}{N!} \int_0^1 f^{(N+1)}(s) \,(1-s)^{N} \,ds \\
&= \sum_{h=0}^N \frac{(2^{-j}\lambda)^h}{h!} \sum_{\ell=0}^h \binom{h}{\ell} (-1)^\ell 2^{\ell-h} ([x,x'] \cdot \nabla_u)^{h-\ell} (x' \cdot \nabla_x)^\ell S(x,u-\lambda^2 u')\\
&\quad+ \frac{(2^{-j}\lambda)^{N+1}}{N!} \sum_{\ell=0}^{N+1} \binom{N+1}{\ell} (-1)^\ell 2^{\ell-N-1} \int_0^1 (1-s)^N \\
&\quad\times ([x,x'] \cdot \nabla_u)^{N+1-\ell} (x' \cdot \nabla_x)^\ell S(x-s2^{-j}\lambda x',u-\lambda^2 u'+s2^{-j}\lambda [x,x']/2) \,ds.
\end{split}\]
By plugging this expression into \eqref{eq:2nd_canc_resc} and exploiting the vanishing moments of $F^1$, we deduce that
\[\begin{split}
&\lambda^{-Q}(F^1_j \otimes \tilde F^2_k) * S_\lambda(\lambda^{-1}x,\lambda^{-2}u) \\
&=\frac{(2^{-j}\lambda)^{N+1}}{N!} \sum_{\ell=0}^{N+1} \binom{N+1}{\ell} (-1)^\ell 2^{\ell-N-1} \int_0^1 (1-s)^N \int_G F^1(x') \, \tilde F^2_k(u') \\
&\times ([x,x'] \cdot \nabla_u)^{N+1-\ell} (x' \cdot \nabla_x)^\ell S(x-s2^{-j}\lambda x',u-\lambda^2 u'+s2^{-j}\lambda [x,x']/2) \,dx' \,du' \,ds,
\end{split}\]
and thus
\[\begin{split}
&\|(F^1_j \otimes \tilde F^2_k) * S_\lambda\|_1\\
&\lesssim_N (2^{-j} \lambda)^{N+1} \sum_{\ell=0}^{N+1} \int_G \int_0^1 (1-s)^N \int_G |F^1(x')| \, |\tilde F^2_k(u')| \, |[x,x']|^{N+1-\ell} \, |x'|^{\ell}  \\
&\quad\times |\nabla_u^{N+1-\ell} \nabla_x^\ell S(x-s2^{-j}\lambda x',u-\lambda^2 u'+s2^{-j}\lambda [x,x']/2)| \,dx' \,du' \,ds \,dx \,du \\
&= (2^{-j} \lambda)^{N+1} \sum_{\ell=0}^{N+1} \int_0^1 (1-s)^N \int_G \int_G |F^1(x')| \, |\tilde F^2_k(u')| \, |[x,x']|^{N+1-\ell} \, |x'|^{\ell}  \\
&\qquad\times |\nabla_u^{N+1-\ell} \nabla_x^\ell S(x,u)| \,dx \,du \,dx' \,du' \,ds \\
&\lesssim_N (2^{-j} \lambda)^{N+1} \sum_{\ell=0}^{N+1} \int_0^1 (1-s)^N \,ds \int_{\RR^{d_1}} |F^1(x')| \, |x'|^{N+1}  \,dx' \int_{\RR^{d_2}} |\tilde F^2_k(u')| \,du' \\ 
&\qquad\times \int_G  |\nabla_u^{N+1-\ell} \nabla_x^\ell S(x,u)| |x|^{N+1-\ell} \,dx \,du \\
&\lesssim_N (2^{-j} \lambda)^{N+1},
\end{split}\]
where we used the fact that $F^1$, $\tilde F^2$ and $S$ are in the Schwartz class, that the $L^1$-norm of $\tilde F^2_k$ does not depend on $k$, and (in the intermediate change of variables) that $[x,x'] = [x+s 2^{-j}\lambda x',x']$. Thus
\begin{equation}\label{eq:2nd_negligible}
\sum_{j \tc 2^j \geq \lambda^{1+\epsilon_2}} \| (F^1_j \otimes \tilde F^2_{k}) * S_\lambda \|_1 
\lesssim_N \sum_{j \tc 2^j \geq \lambda^{1+\epsilon_2}} (2^{-j} \lambda)^{N+1}
\lesssim \lambda^{-\epsilon_2 (N+1)} 
\end{equation}
for any $N \in \NN$, $k \in \ZZ$, $\lambda \geq 1$, $\epsilon_2 > 0$.

As a consequence, by \eqref{eq:2nd_recomp} and \eqref{eq:2nd_negligible},
\[
\begin{aligned}
\|R_{\vec\epsilon,\lambda}^{(2)}\|_1
&= \left\| \sum_{j \tc 2^j \geq \lambda^{1+\epsilon_2}} \sum_{k \tc 2^k < \lambda^{2+\epsilon_1}} (F^1_j \otimes F^2_k) * S_\lambda \right\|_1 \\
&\leq \sum_{j \tc 2^j \geq \lambda^{1+\epsilon_2}} \| (F^1_j \otimes \tilde F^2_{k_0(\epsilon_1;\lambda)}) * S_\lambda \|_1
\lesssim_{\epsilon_2,N} \lambda^{-N} 
\end{aligned}
\]
for all $\lambda \geq 1$, $N \in \NN$, $\epsilon_1,\epsilon_2 >0$. This proves \eqref{eq:frequency_remainder_l1} for $\ell=2$.

Finally, we consider the analogous bound for $\ell=3$.
We can write
\begin{equation}\label{eq:3rd_recomp}
\sum_{j \tc 2^j < \lambda^{1-\epsilon_3}} \sum_{k \tc 2^k < \lambda^{2-\epsilon_3-\epsilon_4}} F^1_j \otimes F^2_k = \tilde F^1_{j_*(\epsilon_3;\lambda)} \otimes \tilde F^2_{k_*(\epsilon_3,\epsilon_4;\lambda)},
\end{equation}
where
\begin{equation}\label{eq:3rd_defjk}
\begin{aligned}
j_*(\epsilon_3;\lambda) &\defeq \max \{ j \tc 2^j < \lambda^{1-\epsilon_3} \}+1, \\
k_*(\epsilon_3,\epsilon_4;\lambda) &\defeq \max \{ k \tc 2^k < \lambda^{2-\epsilon_3-\epsilon_4} \}+1,
\end{aligned}
\end{equation}
and $\tilde F^1_{j} \defeq \sum_{j' < j} F^1_{j'} = 2^{j d_1} \tilde F^1(2^j \cdot)$, with $\tilde F^1 \defeq \Four_{\RR^{d_1}}^{-1}(\chi_0(|\cdot|)) \in \Sz(\RR^{d_1})$.

Now, for all $j,k \in \ZZ$,
\[\begin{split}
&(\tilde F^1_j \otimes \tilde F^2_k) * S_\lambda(x,u) \\
&= \int_G \tilde F^1_j(x-x') \, \tilde F^2_k(u-u'-[x,x']/2) \,S_\lambda(x',u') \,dx' \,du' \\
&= \int_G \tilde F^1_j(x-\lambda^{-1} x') \, \tilde F^2_k(u-\lambda^{-2} u'-\lambda^{-1}[x,x']/2) \,S(x',u') \,dx' \,du' ,
\end{split}\]
which has the same $L^1(G)$-norm as
\begin{equation}\label{eq:3rd_canc_resc}
\begin{split}
&2^{-d_1 j -d_2 k}(\tilde F^1_j \otimes \tilde F^2_k) * S_\lambda(2^{-j}x,2^{-k}u) \\
&= \int_G \tilde F^1(x-\lambda^{-1}2^j x') \, \tilde F^2(u-\lambda^{-2}2^k u'-\lambda^{-1} 2^{k-j} [x,x']/2) \,S(x',u') \,dx' \,du' .
\end{split}
\end{equation}

We now define
\[
f(s) \defeq \tilde F^1(x-s\lambda^{-1}2^j x') \, \tilde F^2(u-s(\lambda^{-2}2^k u'+\lambda^{-1} 2^{k-j} [x,x']/2))
\]
and observe that, for all $h \in \NN$,
\[\begin{split}
f^{(h)}(s) &= \sum_{a+b+c=h} \frac{h!}{a! b! c!} (-\lambda^{-1}2^j)^a (-\lambda^{-2} 2^k)^b (-\lambda^{-1} 2^{k-j}/2)^c \\
&\qquad\times (x' \cdot \nabla_x)^a \tilde F^1(x-s\lambda^{-1}2^j x') \\
&\qquad\times (u'\cdot\nabla_u)^b ([x,x']\cdot \nabla_u)^c \tilde F^2(u-s(\lambda^{-2}2^k u'+\lambda^{-1} 2^{k-j} [x,x']/2)),
\end{split}\]
whence, by Taylor's formula, for all $N \in \NN$,
\[\begin{split}
&\tilde F^1(x-\lambda^{-1}2^j x') \, \tilde F^2(u-\lambda^{-2}2^k u'-\lambda^{-1} 2^{k-j} [x,x']/2) \\
&=f(1) 
= \sum_{h=0}^N \frac{f^{(h)}(0)}{h!} + \frac{1}{N!} \int_0^1 f^{(N+1)}(s) \,(1-s)^{N} \,ds \\
&= \sum_{h=0}^N \sum_{a+b+c=h} \frac{(-1)^h}{a! b! c!} (\lambda^{-1}2^j)^a (\lambda^{-2} 2^k)^b (\lambda^{-1} 2^{k-j}/2)^c \\
&\qquad\times (x' \cdot \nabla_x)^a \tilde F^1(x) \,(u'\cdot\nabla_u)^b ([x,x']\cdot \nabla_u)^c \tilde F^2(u) \\
&\quad+ \sum_{a+b+c=N+1} \frac{(-1)^{N+1} (N+1)}{a! b! c!} (\lambda^{-1}2^j)^a (\lambda^{-2} 2^k)^b (\lambda^{-1} 2^{k-j}/2)^c \\
&\qquad\times \int_0^1 (x' \cdot \nabla_x)^a \tilde F^1(x-s\lambda^{-1}2^j x') \\
&\qquad\times (u'\cdot\nabla_u)^b ([x,x']\cdot \nabla_u)^c \tilde F^2(u-s(\lambda^{-2}2^k u'+\lambda^{-1} 2^{k-j} [x,x']/2)) \,(1-s)^N \,ds.
\end{split}\]
By plugging this into \eqref{eq:3rd_canc_resc} and exploiting the vanishing moments of $S$, 
\[\begin{split}
&2^{-d_1 j -d_2 k}(\tilde F^1_j \otimes \tilde F^2_k) * S_\lambda(2^{-j}x,2^{-k}u) \\
&= \sum_{a+b+c=N+1} \frac{(-1)^{N+1} (N+1)}{a! b! c!} (\lambda^{-1}2^j)^a (\lambda^{-2} 2^k)^b (\lambda^{-1} 2^{k-j}/2)^c \\
&\qquad\times \int_0^1 \int_G (x' \cdot \nabla_x)^a \tilde F^1(x-s\lambda^{-1}2^j x') \\
&\qquad\times (u'\cdot\nabla_u)^b ([x,x']\cdot \nabla_u)^c \tilde F^2(u-s(\lambda^{-2}2^k u'+\lambda^{-1} 2^{k-j} [x,x']/2)) \\
&\qquad\times S(x',u') \,dx' \,du' \,(1-s)^N \,ds,
\end{split}\]
thus
\[\begin{split}
&\|(\tilde F^1_j \otimes \tilde F^2_k) * S_\lambda\|_1 \\
&\lesssim_N \sum_{a+b+c=N+1} (\lambda^{-1} 2^j)^{a+b} (\lambda^{-1} 2^{k-j})^{b+c}
 \int_G \int_0^1 \int_G |x'|^a |\nabla_x^a \tilde F^1(x-s\lambda^{-1}2^j x')| \\
&\qquad\times |u'|^b |[x,x']|^c |\nabla_u^{b+c} \tilde F^2(u-s(\lambda^{-2}2^k u'+\lambda^{-1} 2^{k-j} [x,x']/2))| \\
&\qquad\times |S(x',u')| \,dx' \,du' \,(1-s)^N \,ds \,dx \,du \\
&= \sum_{a+b+c=N+1} (\lambda^{-1} 2^j)^{a+b} (\lambda^{-1} 2^{k-j})^{b+c}
 \int_0^1 \int_G \int_G |x'|^a |\nabla_x^a \tilde F^1(x)|\\
&\qquad\times  |u'|^b \, |[x,x']|^c \, |\nabla_u^{b+c} \tilde F^2(u)| 
 |S(x',u')| \,dx \,du \,dx' \,du' \,(1-s)^N \,ds \\
&\lesssim_N \sum_{a+b+c=N+1} (\lambda^{-1} 2^j)^{a+b} (\lambda^{-1} 2^{k-j})^{b+c} \int_0^1 (1-s)^N \,ds  
\int_G \int_{\RR^{d_1}} |x|^c \, |\nabla_x^a \tilde F^1(x)| \,dx \\
&\qquad\times \int_{\RR^{d_2}} |\nabla_u^{b+c} \tilde F^2(u)| \,du \int_G |x'|^{a+c} \, |u'|^b \, |S(x',u')| \,dx' \,du' \\
&\lesssim_N \sum_{a+b+c=N+1} (\lambda^{-1} 2^j)^{a+b} (\lambda^{-1} 2^{k-j})^{b+c},
\end{split}\]
as $\tilde F^1$, $\tilde F^2$ and $S$ are in the Schwartz class. In particular, if $\lambda^{1-\epsilon_3} \leq 2^j < 2\lambda^{1-\epsilon_3}$ and $2^k < 2\lambda^{2-\epsilon_3-\epsilon_4}$, then $2^{k-j} < 2\lambda^{1-\epsilon_4}$, so in this case the previous inequality gives
\begin{equation}\label{eq:3rd_negligible}
\|(\tilde F^1_j \otimes \tilde F^2_k) * S_\lambda\|_1 \lesssim_N \sum_{a+b+c=N+1} \lambda^{-(a+b)\epsilon_3-(b+c)\epsilon_4} \lesssim_N \lambda^{-(1+N)\min\{\epsilon_3,\epsilon_4\}}
\end{equation}
for all $\lambda \geq 1$. It is clear from \eqref{eq:3rd_defjk} that these conditions on $j$ and $k$ are satisfied when $j=j_*(\epsilon_3;\lambda)$, $k=k_*(\epsilon_3,\epsilon_4;\lambda)$. In conclusion, by \eqref{eq:3rd_recomp} and \eqref{eq:3rd_negligible},
\[
\begin{aligned}
\|R^{(3)}_{\vec\epsilon,\lambda}\|_1 
&= \left\| \sum_{j \tc 2^j < \lambda^{1-\epsilon_3}} \sum_{k \tc 2^k < \lambda^{2-\epsilon_3-\epsilon_4}} (F^1_j \otimes F^2_k) * S_\lambda \right\|_{1} \\
&= \|(\tilde F^1_{j_*(\epsilon_3;\lambda)} \otimes \tilde F^2_{k_*(\epsilon_3,\epsilon_4;\lambda)}) * S_\lambda\|_1 
\lesssim_{\epsilon_3,\epsilon_4,N} \lambda^{-N}
\end{aligned}
\]
for all $\lambda \geq 1$ and $N \in \NN$. This proves \eqref{eq:frequency_remainder_l1} for $\ell=3$.
\end{proof}

\begin{rem}
The term $R^{(1)}_{\vec\epsilon,\lambda}$ in \eqref{eq:Rremainders} actually vanishes for $\lambda \gg 1$, due to the fact that $|D_u| \lesssim \opL$ spectrally on any $2$-step group; see also \cite[eq.\ (56)]{MSe} for a similar use of this joint spectral localisation. In the above proof we present a different argument showing the negligibility of $R^{(1)}_{\vec\epsilon,\lambda}$, as it appears to be more robust and potentially amenable to generalisations to less symmetric settings.
\end{rem}

Before we turn to estimating the contributions by the main terms $A_{\vec\epsilon,\lambda}$ and $B_{\vec\epsilon,\lambda}$ in \eqref{eq:spectrum_frequency_dec}, it will be useful to gain a good understanding of the underlying sub-Riemannian geometry on $G$.

\section{The geodesic flow on \texorpdfstring{$2$}{2}-step groups}\label{s:flow}

It is well known that quite explicit formulas for the sub-Riemannian geodesic flow on a $2$-step group can be obtained, and various instances thereof can be found in multiple places in the literature (see, e.g., \cite{ABB,Gav,GHK,Li,MSe} and references therein). For the reader's convenience, we briefly present here the derivation of the formulas that we need later.

With ``geodesic flow'' here we refer to the Hamiltonian flow on the cotangent bundle $T^* G$ generated by the square root $\Ham$ of the principal symbol of the sub-Laplacian $\opL$. Projecting the flow curves onto $G$ yields the so-called ``normal geodesics'', which are locally length-minimising. On arbitrary sub-Riemannian manifolds, there may exist length-minimising curves that are not obtained this way, also known as ``strictly abnormal length minimisers''; such curves, however, do not occur in our $2$-step setting (see, e.g., \cite[Corollary 12.14]{ABB}).

Recall that the choice of a homogeneous sub-Laplacian $\opL$ on the $2$-step group $G$ determines an inner product $\langle\cdot,\cdot\rangle$ on $\lie{g}_1$.
Recall moreover from \eqref{eq:Jmu} the definition of the skew-symmetric linear map $J_\mu : \lie{g}_1 \to \lie{g}_1$ for all $\mu \in \lie{g}_2^*$.
The symbol of the sub-Laplacian $\opL$ can then be written as $\Ham(\bx,\bxi)^2$, where
\begin{equation}\label{eq:HamH}
\Ham(\bx,\bxi) \defeq |\xi + J_\mu x/2|;
\end{equation}
notice that, in these coordinates, $\Ham(\bx,\bxi)^2$ is homogeneous of degree $2$ in $\bxi$, so the full symbol and the principal symbol of $\opL$ coincide.

\subsection{The flow starting at the origin}
We start with studying the flow starting at the origin of $G$. 
Thanks to the left-invariance of $\opL$, this actually determines the flow starting at any point of $G$, as we shall see later.

\begin{prp}
Let $\bxi_0 = (\xi_0,\mu_0) \in \RR^d$.
\begin{enumerate}[label=(\roman*)]
\item\label{en:geodesic_quadratic_A} The solution curve $(\bx(t),\bxi(t))$ of the Hamilton equations
\[
\dot \bx = \nabla_{\bxi} \HamA, 
\quad \dot\bxi = -\nabla_{\bx} \HamA, 
\]
with Hamiltonian $\HamA(\bx,\bxi) \defeq \Ham(\bx,\bxi)^2/2$ and initial datum
\[
\bx(0) = 0, 
\quad \bxi(0) = \bxi_0, 
\]
is
\begin{equation}\label{eq:HamA_x}
\begin{aligned}
x(t) &= \frac{\exp(t J_{\mu_0}) - I}{J_{\mu_0}} \xi_0, \\
\xi(t) &= \frac{1}{2}(I+ \exp(t J_{\mu_0})) \xi_0, \\
u(t) &= \frac{1}{2} \int_0^t \left[\frac{\exp(\tau J_{\mu_0}) - I}{J_{\mu_0}} \xi_0, \exp(\tau J_{\mu_0}) \xi_0\right] \,d\tau, \\
\mu(t) &= \mu_0.
\end{aligned}
\end{equation}

\item\label{en:geodesic_quadratic_H}
The solution curve $(\bx(t),\bxi(t))$ of the Hamilton equations
\begin{equation}\label{eq:Ham}
\dot \bx = \nabla_{\bxi} \Ham, 
\quad \dot\bxi = -\nabla_{\bx} \Ham, 
\end{equation}
with Hamiltonian $\Ham(\bx,\bxi)$ as in \eqref{eq:HamH} and initial datum
\[
\bx(0) = 0, 
\quad \bxi(0) = \bxi_0, 
\]
with $\xi_0 \neq 0$, is
\begin{equation}\label{eq:Ham_x}
\begin{aligned}
x(t) &= \frac{\exp(t J_{\mu_0/|\xi_0|}) - I}{J_{\mu_0}} \xi_0, \\
\xi(t) &= \frac{1}{2}(I+ \exp(t J_{\mu_0/|\xi_0|})) \xi_0, \\
u(t) &= \frac{1}{2} \int_0^t \left[\frac{\exp(\tau J_{\mu_0/|\xi_0|}) - I}{J_{\mu_0}} \xi_0, \exp(\tau J_{\mu_0/|\xi_0|}) \frac{\xi_0}{|\xi_0|}\right] \,d\tau, \\
\mu(t) &= \mu_0.
\end{aligned}
\end{equation}
\end{enumerate}
\end{prp}
\begin{proof}
\ref{en:geodesic_quadratic_A}. First, we compute
\begin{align*}
\nabla_\xi \HamA &= \xi + J_\mu x/2, & \nabla_x\HamA &= - \frac{1}{2} J_\mu(\xi + J_\mu x/2), \\
 \nabla_\mu\HamA &= \frac{1}{2} [x,\xi+J_\mu x/2], & \nabla_u\HamA &= 0.
\end{align*}
The last equation tells us that $\mu$ is constant along any solution curve. If we now set
\[
\zeta \defeq \xi+J_\mu x/2,
\]
then we deduce that
\[
\dot\zeta = \dot\xi+J_\mu \dot x/2 = J_\mu \zeta,
\]
whence
\[
\zeta(t) = \exp(t J_\mu) \zeta(0) = \exp(t J_\mu) \xi(0),
\]
if we assume that $x(0) = 0$.

From this we deduce
\[
\dot x(t) = \zeta(t) = \exp(t J_\mu) \xi(0), \qquad \dot \xi(t) = \frac{1}{2} J_\mu \zeta(t) = \frac{1}{2} J_\mu \exp(t J_\mu) \xi(0),
\]
whence also
\[
x(t) = \frac{\exp(t J_\mu) - I}{J_\mu} \xi(0), \qquad \xi(t) = \frac{1}{2}(I+ \exp(t J_\mu)) \xi(0).
\]

Finally,
\[
\dot u(t) = \frac{1}{2} [x(t),\zeta(t)] = \frac{1}{2} \left[\frac{\exp(t J_\mu) - I}{J_\mu} \xi(0), \exp(t J_\mu) \xi(0)\right]
\]
and, if we assume that $u(0)=0$, then
\[
u(t) = \frac{1}{2} \int_0^t \left[\frac{\exp(\tau J_\mu) - I}{J_\mu} \xi(0), \exp(\tau J_\mu) \xi(0)\right] \,d\tau,
\]
as desired.

\smallskip

\ref{en:geodesic_quadratic_H}.
From part \ref{en:geodesic_quadratic_A} we know that \eqref{eq:HamA_x} give a solution of
\[
\dot x = \nabla_\xi(\Ham^2/2), \quad \dot u = \nabla_\mu(\Ham^2/2), \quad \dot\xi = -\nabla_x(\Ham^2/2), \quad \dot\mu = -\nabla_u(\Ham^2/2),
\]
that is,
\[
\dot x = \Ham \nabla_\xi \Ham, \quad \dot u = \Ham \nabla_\mu \partial\Ham, \quad \dot\xi = -\Ham \nabla_x \Ham, \quad \dot\mu = -\Ham \nabla_u \Ham.
\]
In addition, the quantity $\Ham$ is constant along the flow, so $\Ham(t) = \Ham(0) = |\xi_0|$, which means that \eqref{eq:HamA_x} actually solve
\[
\dot x = |\xi_0| \nabla_\xi \Ham, \quad \dot u = |\xi_0| \nabla_\mu\Ham, \quad \dot\xi = -|\xi_0| \nabla_x\Ham, \quad \dot\mu = -|\xi_0| \nabla_u\Ham.
\]
Thus, the change of variables $t \mapsto t/|\xi_0|$ in \eqref{eq:HamA_x} gives the desired solution of \eqref{eq:Ham}.
\end{proof}

From now on, let
\begin{equation}\label{eq:flow_def}
t\mapsto \bx^t(\bxi) = (x^t(\bxi),u^t(\bxi)), \qquad t\mapsto \bxi^t(\bxi) = (\xi^t(\bxi),\mu^t(\bxi)) = (\xi^t(\bxi),\mu)
\end{equation}
denote the solution \eqref{eq:Ham_x} of the system of Hamilton equations with initial condition 
$(\xi_0,\mu_0) = \bxi \in \dot\RR^{d_1} \times \RR^{d_2}$. 
We record here some alternative formulas for the functions $\bx^t$ and $\bxi^t$ that will become useful later.

\begin{cor}\label{cor:flow_alt}
Assume that $\bxi \in \dot \RR^{d_1} \times \dot \RR^{d_2}$.
If we set
\begin{equation}\label{eq:theta_mu_xi}
\theta \defeq \frac{t|\mu|}{2|\xi|}, \qquad \bar\mu \defeq \frac{\mu}{|\mu|}, \qquad \bar\xi \defeq \frac{\xi}{|\xi|},
\end{equation}
then
\begin{equation}\label{eq:flow_2step}
\begin{aligned}
x^t &= \frac{1}{|\mu|} \frac{\exp(2\theta J_{\bar\mu}) - I}{J_{\bar\mu}} \xi = t \exp(\theta J_{\bar\mu}) \frac{\sinh(\theta J_{\bar\mu})}{\theta J_{\bar\mu}} \bar\xi, \\
\xi^t &= \frac{1}{2}(\exp(2\theta J_{\bar\mu})+I) \xi = \exp(\theta J_{\bar\mu}) \cosh(\theta J_{\bar\mu}) \xi,\\
u^t &= \frac{t^2}{4} \int_0^1 \left[\frac{\exp(2 \tau \theta J_{\bar\mu}) - I}{\theta J_{\bar\mu}} \bar\xi, \exp(2\tau \theta J_{\bar\mu}) \bar\xi \right] \,d\tau,\\
\mu \cdot u^t &= \frac{t|\xi|}{2} \left[ 1 - \left\langle \frac{\sinh(\theta J_{\bar\mu})}{\theta J_{\bar\mu}} \cosh(\theta J_{\bar\mu}) \bar\xi, \bar\xi \right\rangle \right] .
\end{aligned}
\end{equation}
Note that here $\theta J_{\bar\mu}=J_{t\mu/(2|\xi|)}$.
Moreover, if $G$ is of Heisenberg type, then
\begin{equation}\label{eq:flow_Htype}
\begin{aligned}
x^t &= t \frac{\sin \theta}{\theta} ((\cos\theta) I + (\sin\theta) J_{\bar\mu}) \bar\xi, \\
\xi^t &= \cos \theta ((\cos\theta) I + (\sin\theta) J_{\bar\mu}) \xi,\\
u^t &= \frac{t^2}{4\theta} \left[ 1 - \frac{\sin \theta }{\theta} \cos \theta \right] \bar\mu.
\end{aligned}
\end{equation}
\end{cor}
\begin{proof}
The formulas for an arbitrary $2$-step group are obtained through easy manipulation of the formulas \eqref{eq:Ham_x}. In particular, by \eqref{eq:Ham_x} and \eqref{eq:Jmu},
\[\begin{split}
\mu \cdot u(t) 
&= \frac{1}{2} \int_0^t \left\langle (\exp(\tau J_{\mu/|\xi|}) - I) \xi, \exp(\tau J_{\mu/|\xi|}) \bar\xi \right\rangle \,d\tau \\
&= \frac{|\xi|}{2} \int_0^t \left\langle  \bar\xi,(I- \exp(\tau J_{\mu/|\xi|})) \bar\xi \right\rangle \,d\tau,
\end{split}\]
and integration of this expression yields the claimed formula for $\mu \cdot u^t$.

\smallskip

In the case that $G$ is of Heisenberg type (see \eqref{eq:Htype}), by using the identity $J_{\bar\mu}^2 = -I$ the formulas proved for general $G$ reduce to the given ones for $x^t$ and $\xi^t$, as well as
\[
\mu \cdot u^t = \frac{t|\xi|}{2} \left[ 1 - \frac{\sin \theta }{\theta} \cos \theta \right] .
\]
In order to conclude the proof of \eqref{eq:flow_Htype}, it is thus sufficient to show that $u^t$ is parallel to $\mu$, i.e., that $\nu \cdot u^t = 0$ for all $\nu \in \mu^\perp$.

On the other hand, by polarization of \eqref{eq:Htype}, one gets that
\[
J_\mu J_\nu + J_\nu J_\mu = -\langle \mu, \nu \rangle I
\]
for all $\mu,\nu \in \RR^{d_2}$. In particular, if $\mu \perp \nu$, then $J_\mu$ and $J_\nu$ anticommute, and therefore
\[
(J_\nu J_\mu^k)^* = (-1)^{k+1} J_\mu^k J_\nu = - J_\nu J_\mu^k,
\]
i.e., $J_\nu J_\mu^k$ is skew-symmetric for all $k \in \NN$; as a consequence, if $\mu \perp \nu$, then $J_\nu F(J_\mu)$ is skew-symmetric for any real-valued polynomial $F$, hence for any real-valued function $F$.

Now, if $\nu \in \mu^\perp$, then, again by \eqref{eq:Ham_x} and \eqref{eq:Jmu},
\[\begin{split}
\nu \cdot u^t 
&= \frac{1}{2} \int_0^t \left\langle J_\nu \frac{\exp(\tau J_{\mu/|\xi|}) - I}{J_{\mu/|\xi|}} \bar\xi, \exp(\tau J_{\mu/|\xi|}) \bar\xi \right\rangle \,d\tau \\
&= \frac{1}{2} \int_0^t \left\langle J_\nu \exp(\tau J_{\mu/|\xi|}) \frac{\exp(\tau J_{\mu/|\xi|}) - I}{J_{\mu/|\xi|}} \bar\xi, \bar\xi \right\rangle \,d\tau = 0
\end{split}\]
by skew-symmetricity; so indeed $u^t$ is parallel to $\mu$, as claimed.
\end{proof}

\subsection{The flow starting at an arbitrary point}\label{ss:flowy}
For given $\by=(y,v)\in G$, denote by $\ltr_{\by} : G \to G$ the left-translation map by $\by$, i.e., $\ltr_{\by}(\bx) \defeq \by \cdot_G \bx$, $\bx\in G$. 
Since $\ltr_{\by} : G \to G$ is a diffeomorphism, it induces a vector bundle automorphism $\ltr_{\by}^* : T^* G \to T^* G$, given by $\ltr_{\by}^* (\bx,\bxi) = (\ltr_{\by}(\bx), (D\ltr_{\by}(\bx)^T)^{-1}\bxi)$. As
\[
\ltr_{\by}(\bx)=(y+x,v+u+[y,x]/2)
\]
is an affine map in these coordinates, we see that its differential
\begin{equation}\label{eq:diff_ltr}
\ltr_{\by}' \defeq D \ltr_{\by}(\bx) =
\left(\begin{array}{c|c}
I_{d_1} & 0 \\\hline
[y,\cdot]/2 & I_{d_2} 
\end{array}\right)
\end{equation}
is independent of the base point $\bx \in G$.
Thus $\dltr_{\by} \defeq (D \ltr_{\by}(\bx)^T)^{-1}$ is independent of the base point $\bx$ too, and a cotangent vector $\bxi = (\xi,\mu) \in T^*_{\bx} G$ is transformed under $\ltr_{\by}$ into
\begin{equation}\label{eq:dltr}
\dltr_{\by} \bxi = 
\left(\begin{array}{c|c}
I_{d_1} & -[y,\cdot]/2 \\\hline
0 & I_{d_2} 
\end{array}\right)
 \left(\begin{array}{c} \xi \\\hline \mu \end{array}\right) = 
\left(\begin{array}{c}
  \xi -J_\mu y/2  \\\hline
\mu    
\end{array}\right)
.
\end{equation}

The left-invariance of the sub-Laplacian $\opL$ on $G$ corresponds to the left-invariance of its symbol, hence also of the Hamiltonian $\Ham$, i.e.,
\[
\Ham \circ \ltr_{\by}^* = \Ham \qquad \forall \by \in G.
\]
Notice that $\Ham$ is a smooth function on the open subset
\[
\DHam \defeq \{ (\bx,\bxi) \tc \Ham(\bx,\bxi) \neq 0 \} = \{ (\bx,\bxi) \tc \xi+J_\mu x/2 \neq 0 \}
\]
of $T^* G$. Correspondingly, the Hamiltonian flow generated by $\Ham$ is $\ltr_{\by}^*$-covariant. In particular, if 
$t\mapsto (\bx_*^t(\by,\bxi),\bxi_*^t(\by,\bxi))$ denotes the Hamiltonian flow curve of $\Ham$ starting at $(\by,\bxi) \in \DHam$, then
\[
(\bx_*^t(\by,\bxi),\bxi_*^t(\by,\bxi)) =\ltr_{\by}^*(\bx_*^t(0,\dltr_{\by}^{-1}\bxi),\bxi_*^t(0,\dltr_{\by}^{-1}\bxi)),
\]
i.e.,
\begin{equation}\label{eq:flowy}
\bx_*^t(\by,\bxi) =\by \cdot_G \bx^t(\dltr_{\by}^{-1}\bxi), \qquad \bxi_*^t(\by,\bxi) = \dltr_{\by} \bxi^t(\dltr_{\by}^{-1}\bxi) .
\end{equation}
So we have reduced to the flow starting at $\by=0$ as in \eqref{eq:flow_def}.

\subsection{Properties of the flow}

We record here some identities that the functions $\bx^t_*$ and $\bxi^t_*$ from \eqref{eq:flowy} satisfy due to their being solutions of a system of Hamilton equations with a $1$-homogeneous Hamiltonian $\Ham$. A similar discussion is found in \cite[Sections 1.1 and 1.2]{LSV}.

\begin{lem}\label{lem:symplectic}
For all $t \in \RR$ and $(\by,\bxi) \in \DHam$,
\begin{equation}\label{eq:pres_symplectic}
\begin{aligned}
(\partial_{\by} \bx_*^t)^T (\partial_{\by} \bxi_*^t) - (\partial_{\by} \bxi_*^t)^T (\partial_{\by} \bx_*^t) &= 0,\\
(\partial_{\bxi} \bx_*^t)^T (\partial_{\bxi} \bxi_*^t) - (\partial_{\bxi} \bxi_*^t)^T (\partial_{\bxi} \bx_*^t) &= 0,\\
(\partial_{\by} \bx_*^t)^T (\partial_{\bxi} \bxi_*^t) - (\partial_{\by} \bxi_*^t)^T (\partial_{\bxi} \bx_*^t) &= I,
\end{aligned}
\end{equation}
and moreover
\begin{align}
\label{eq:symplecticy_xi}
(\partial_{\by} \bx_*^t)^T \bxi_*^t &= \bxi,\\
\label{eq:symplecticy_ort}
(\partial_{\bxi} \bx_*^t)^T \bxi_*^t &= 0,\\
\label{eq:symplecticy_ker}
\Ker \partial_{\bxi} \bx_*^t \cap \Ker \partial_{\bxi} \bxi_*^t &= \{0\} .
\end{align}
In particular, for the flow \eqref{eq:flow_def} starting at $\by = 0$, we have
\begin{align}
\label{eq:symplectic_com}
(\partial_{\bxi} \bx^t)^T (\partial_{\bxi} \bxi^t) &= (\partial_{\bxi} \bxi^t)^T (\partial_{\bxi} \bx^t),\\
\label{eq:symplectic_ort}
(\partial_{\bxi} \bx^t)^T \bxi^t &= 0,\\
\label{eq:symplectic_ker}
\Ker \partial_{\bxi} \bx^t \cap \Ker \partial_{\bxi} \bxi^t &= \{0\} .
\end{align}
\end{lem}
\begin{proof}
For any $t \in \RR$, the map $(\by,\bxi) \mapsto (\bx_*^t(\by,\bxi),\bxi^t_*(\by,\bxi))$ is a canonical transformation, i.e., it preserves the symplectic form $d\by \wedge d\bxi$ on $T^* G$. Writing in coordinates the preservation of the symplectic form readily gives the identities \eqref{eq:pres_symplectic}.

The third identity in \eqref{eq:pres_symplectic} clearly implies \eqref{eq:symplecticy_ker}.
Moreover, \eqref{eq:symplecticy_xi} and \eqref{eq:symplecticy_ort} are obtained by multiplying the second and third identities in \eqref{eq:pres_symplectic} on the right by $\bxi$ and observing that 
\[
(\partial_{\bxi} \bxi_*^t) \bxi = \bxi_*^t, \qquad (\partial_{\bxi} \bx_*^t) \bxi = 0
\]
by Euler's identity, as $\bx_*^t$ and $\bxi_*^t$ are $0$-homogeneous and $1$-homogeneous in $\bxi$ respectively.

Finally, the identities \eqref{eq:symplectic_com}, \eqref{eq:symplectic_ort} and \eqref{eq:symplectic_ker} are obtained by evaluation of the previous ones at $\by = 0$.
\end{proof}

\section{The anti-FIO region}\label{s:antiFIO}

Recall that, in light of Proposition \ref{prp:spectrum_frequency}, the proof of \eqref{sploc_scloc_wave_est} is reduced to that of the estimates
\begin{equation}\label{eq:AB_region}
\begin{aligned}
\|\chr_{\overline{B}(0,2)} \exp(\pm i \sqrt{\opL}) A_{\vec\epsilon,\lambda} \|_1 &\lesssim_{\vec\epsilon} \lambda^{(d-1)/2 + \smallo(\vec\epsilon)} \\
\|\chr_{\overline{B}(0,2)} \exp(\pm i \sqrt{\opL}) B_{\vec\epsilon,\lambda} \|_1 &\lesssim_{\vec\epsilon} \lambda^{(d-1)/2 + \smallo(\vec\epsilon)}
\end{aligned}
\end{equation}
for all $\lambda \geq 1$. Here and in the sequel, $\smallo(\vec\epsilon)$ serves as shorthand for a positive quantity that can be made arbitrarily small by taking $\vec\epsilon = (\epsilon_1,\epsilon_2,\epsilon_3,\epsilon_4)$ sufficiently small.

\smallskip

Let us begin with some heuristic considerations as to how to look at the terms $A_{\vec\epsilon,\lambda}$ and $B_{\vec\epsilon,\lambda}$ when $\lambda \gg 1$.
By \eqref{eq:regionsAB}, the kernels $A_{\vec\epsilon,\lambda}$ and $B_{\vec\epsilon,\lambda}$ are obtained by convolving the spectral localisation kernel $\chi_1(\sqrt{\opL}/\lambda) \delta_0$ with frequency localisation kernels of the form $\chi_0(2^{-j} |D_x|) \chi_1(2^{-k} |D_u|) \delta_0$ or $\chi_1(2^{-j} |D_x|) \chi_0(2^{-k} |D_u|) \delta_0$, where $2^j \approx \lambda$ and $2^k \approx \lambda^2$ (see Figure \ref{fig:freq_loc}). 
If we ignore small powers of $\lambda$, all these kernels can be thought of as $L^1$-normalised $\dil_{1/\lambda}$-scaled versions of fixed Schwartz-class kernels; in particular, they
are essentially supported (up to negligible tails) in a small sub-Riemannian ball of radius $\simeq 1/\lambda$ centred at the origin, i.e., in a set of the form
\begin{equation}\label{eq:Wlambda}
W_\lambda = \{ \bx \tc |x|\lesssim 1/\lambda, \ |u| \lesssim 1/\lambda^2 \},
\end{equation}
and moreover we roughly have $\|A_{\vec\epsilon,\lambda}\|_p,\|B_{\vec\epsilon,\lambda}\|_p \simeq_p \lambda^{Q/p'}$ for $p \in [1,\infty]$.

\smallskip

In view of H\"ormander's theorem on the propagation of singularities (see \cite[Theorem 26.1.1]{Ho4} or \cite[Theorem 2.1, Ch.\ IV.2]{T}), 
we expect the singular support of the convolution kernel $\exp(it\sqrt {\opL}) \delta_0$ to be contained in the closure of the geodesic sphere $S(t) \defeq \{ \bx^t(\bxi) \tc \bxi \in \dot\RR^{d_1} \times \RR^{d_2} \}$. In accordance with the heurstics discussed in Section \ref{ss:intro_heuristics}, we may therefore expect both $\exp(\pm i\sqrt {\opL}) A_{\vec\epsilon,\lambda}$ and $\exp(\pm i\sqrt {\opL}) B_{\vec\epsilon,\lambda}$ to be essentially supported in the neighbourhood $W_\lambda \cdot_G S$ of the unit geodesic sphere $S \defeq S(1)$. A more precise picture can be obtained if we further assume that, when studying the effect of the wave propagator $\exp(\pm i \sqrt{\opL})$ on the kernels $A_{\vec\epsilon,\lambda}$ and $B_{\vec\epsilon,\lambda}$, we only consider the geodesics of initial covectors $\bxi$ in the frequency regions $A$ and $B$ of Figure \ref{fig:freq_loc} respectively, and thus split the geodesic sphere $S$ into $S_{A,\lambda}$ and $S_{B,\lambda}$.

Indeed, as seen in Corollary \ref{cor:flow_alt}, the sub-Riemannian geodesics $t \mapsto \bx^t=(x^t,u^t)$ are strongly steered by the parameter $\theta = t|\mu|/(2|\xi|)$. On the other hand, the size of $\theta$, which is the same as that of the ratio $|\mu|/|\xi|$ when $|t|=1$, behaves quite differently for the contributions by $A_{\vec\epsilon,\lambda}$ and by $B_{\vec\epsilon,\lambda}$:
\begin{itemize}
\item under the frequency localisation given by $B_{\vec\epsilon,\lambda}$, we roughly have that $|\xi|\simeq\lambda$ and $|\mu|\lesssim \lambda^2$, so that here $|\theta|\lesssim \lambda$;
\item for the term $A_{\vec\epsilon,\lambda}$, instead, we roughly have that $|\xi|\lesssim \lambda$ and $|\mu|\simeq \lambda^2$, so that here $|\theta|\gtrsim \lambda$.
\end{itemize}
This corresponds to a different behaviour of the corresponding geodesics. Indeed, by the formulas \eqref{eq:flow_2step} for the geodesic flow,
\[
|x^t| = \left| \frac{\sinh(\theta J_{\bar\mu})}{\theta J_{\bar\mu}} \bar\xi\right| \qquad\text{for } |t|=1.
\]
If we assume that $G$ is M\'etivier (see \eqref{eq:metivier}), then $\|J_{\bar\mu}^{-1}\| \simeq 1$, and consequently $|x^t|\lesssim (1+|\theta|)^{-1}$.
In a similar way, from \eqref{eq:flow_2step} one sees that $|u^t| \lesssim (1+|\theta|)^{-1}$.

\begin{figure}

\includegraphics{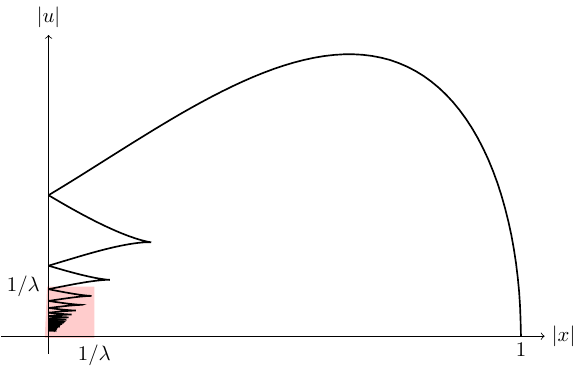}

\caption{Profile of the unit geodesic sphere on an H-type group, highlighting the part related to frequencies where $|\mu|/|\xi| \gtrsim \lambda \gg 1$.}
\label{fig:geo_A}

\end{figure}

Therefore, under the frequency localisation $|\theta| \gtrsim \lambda$ corresponding to $A_{\vec\epsilon,\lambda}$, the set $S_{A,\lambda}$ of endpoints of unit-length geodesics starting at the origin is contained in a Euclidean ball of radius $\simeq 1/\lambda$ (see Figure \ref{fig:geo_A}) and, a fortiori, in the larger cylindrical set
\begin{equation}\label{eq:Zlambda}
Z_\lambda = \{ \bx \tc |x|\lesssim 1/\lambda, \ |u| \lesssim 1 \}.
\end{equation}
Thus, we may expect the kernel $\exp(\pm i\sqrt{\opL}) A_{\vec\epsilon,\lambda}$ to be essentially supported in the product set $W_\lambda \cdot_G Z_\lambda$, which in turn, by \eqref{eq:group_law}, is contained in another cylindrical set of the form $Z_\lambda$ (with larger implicit constants).
In particular, as the sets \eqref{eq:Wlambda} and \eqref{eq:Zlambda} have comparable $x$-sizes, it is as if under this frequency localisation we did not have propagation in $x$.
Assuming this were true, we could simply estimate the $L^1$-norm of $\exp(\pm i\sqrt{\opL}) A_{\vec\epsilon,\lambda}$ by means of the Cauchy--Schwarz inequality:
\[\begin{split}
\| \exp(\pm i\sqrt{\opL}) A_{\vec\epsilon,\lambda} \|_1
&\lesssim |Z_\lambda|^{1/2} \| \exp(\pm i \sqrt{\opL}) A_{\vec\epsilon,\lambda} \|_2
\simeq  \lambda^{-d_1/2} \| A_{\vec\epsilon,\lambda} \|_2 \\
&\simeq \lambda^{-d_1/2} (\lambda^{d_1+2d_2})^{1/2}=\lambda^{d_2}.
\end{split}\]
As $d_2 \leq (d-1)/2$ on M\'etivier groups, this bound would be enough to obtain the first of the estimates \eqref{eq:AB_region}.

\smallskip

We shall next show that these heuristic considerations can indeed be made rigorous by means of weighted $L^2$-estimates with suitable weights on the first layer.

\begin{prp}\label{prp:fbi}
Assume that $G$ is a M\'etivier group.
Let $\vec\epsilon \in (\Rpos)^4$, $\lambda \geq 1$, and $A_{\vec\epsilon,\lambda}$ be defined as in Proposition \ref{prp:spectrum_frequency}. Then
\[
\|\chr_{\overline{B}(0,2)} \exp(\pm i \sqrt{\opL}) A_{\vec\epsilon,\lambda} \|_1 \lesssim_{\vec\epsilon,s} \lambda^s
\]
for all $s > d_2+d_1(\epsilon_3+\epsilon_4)$. Additionally,
\begin{equation}\label{eq:metivier_dim_ineq}
d_2 \leq (d-1)/2.
\end{equation}
\end{prp}
\begin{proof}
As $G$ is M\'etivier, the linear map $\lie{g}_2^* \ni \mu \mapsto J_\mu x \in x^\perp$ is injective for any fixed $x \in \lie{g}_1 \setminus \{0\}$, thus
$d_2 \leq d_1-1$ and \eqref{eq:metivier_dim_ineq} follows.

Moreover, we notice that
\[\begin{split}
A_{\vec\epsilon,\lambda} 
&= \sum_{j \tc 2^j < \lambda^{1+\epsilon_2}} \sum_{k \tc \lambda^{2-\epsilon_3-\epsilon_4} \leq 2^k < \lambda^{2+\epsilon_1}} \chi_1(\sqrt{\opL}/\lambda) \, \chi_1(2^{-j}|D_x|) \, \chi_1(2^{-k}|D_u|) \delta_0 \\
&= \sum_{k \tc \lambda^{2-\epsilon_3-\epsilon_4} \leq 2^k < \lambda^{2+\epsilon_1}} \chi_1(\sqrt{\opL}/\lambda) \, \chi_0(2^{-j_0(\epsilon_2;\lambda)}|D_x|) \, \chi_1(2^{-k}|D_u|) \delta_0
\end{split}\]
where
\[
j_0(\epsilon_2;\lambda) \defeq \max \{ j \tc 2^j < \lambda^{1+\epsilon_2} \}+1.
\]
As the number of summands in the last sum is $\lesssim_{\vec\epsilon} \log\lambda$, it is enough to prove the desired estimate separately for each summand, namely,
\[
\|\chr_{\overline{B}(0,2)} \exp(\pm i \sqrt{\opL}) \chi_1(\sqrt{\opL}/\lambda) \, \chi_0(2^{-j}|D_x|) \, \chi_1(2^{-k}|D_u|) \delta_0 \|_1 \lesssim_{\vec\epsilon,s} \lambda^s
\]
for all $j,k \in \ZZ$ such that $\lambda^{1+\epsilon_2} \leq 2^j < 2\lambda^{1+\epsilon_2}$ and $\lambda^{2-\epsilon_3-\epsilon_4} \leq 2^k < \lambda^{2+\epsilon_1}$.

By automorphic scaling by $\dil_\lambda$, the latter estimate is equivalent to
\[
\|\chr_{\overline{B}(0,2\lambda)} \exp(\pm i \lambda\sqrt{\opL}) \chi_1(\sqrt{\opL}) \, \chi_0(2^{-j}\lambda|D_x|) \, \chi_1(2^{-k}\lambda^2|D_u|) \delta_0 \|_1 \lesssim_{\vec\epsilon,s} \lambda^s.
\]

Let us now write $\varphi_{\lambda,j,k} \defeq \chi_0(2^{-j}\lambda|D_x|) \, \chi_1(2^{-k}\lambda^2|D_u|) \delta_0$, so that 
\begin{multline*}
\exp(\pm i \lambda\sqrt{\opL}) \chi_1(\sqrt{\opL}) \, \chi_0(2^{-j}\lambda|D_x|) \, \chi_1(2^{-k}\lambda^2|D_u|) \delta_0 \\
= \varphi_{\lambda,j,k} * k_{\exp(\pm i \lambda\sqrt{\opL}) \chi_1(\sqrt{\opL}) \tilde\chi_1(2^{-k}\lambda^2|D_u|)}.
\end{multline*}
Now, by the Cauchy--Schwarz inequality, for any $\sigma > d_1/2$,
\[\begin{split}
&\|\chr_{\overline{B}(0,2\lambda)} \exp(\pm i \lambda\sqrt{\opL}) \chi_1(\sqrt{\opL}) \, \chi_0(2^{-j}\lambda|D_x|) \, \chi_1(2^{-k}\lambda^2|D_u|) \delta_0 \|_1 \\
&\leq \left(\int_{\overline{B}(0,2\lambda)} (1+|x|)^{-2\sigma}\right)^{1/2} \\
&\qquad\times \|\varphi_{\lambda,j,k} * k_{\exp(\pm i \lambda\sqrt{\opL}) \chi_1(\sqrt{\opL}) \tilde\chi_1(2^{-k}\lambda^2|D_u|)}\|_{L^2((1+|x|)^{2\sigma} \,d\bx)} \\
&\lesssim_{\sigma} \lambda^{d_2} \|\varphi_{\lambda,j,k}\|_{L^1((1+|x|)^\sigma \,d\bx)} \| k_{\exp(\pm i \lambda\sqrt{\opL}) \chi_1(\sqrt{\opL}) \tilde\chi_1(2^{-k}\lambda^2|D_u|)}\|_{L^2((1+|x|)^{2\sigma} \,d\bx)},
\end{split}\]
as $\sigma>d_1/2$ and $1+|x|$ is a submultiplicative weight on $G$.

Now,
\begin{multline*}
\|\varphi_{\lambda,j,k}\|_{L^1((1+|x|)^\sigma \,dx\,du)} \\
= \| \chi_0(2^{-j}\lambda|D_x|) \delta_0 \|_{L^1((1+|x|)^\sigma \,dx)} \| \chi_1(2^{-k}\lambda^2|D_u|) \delta_0 \|_{L^1(du)} 
\lesssim_{\sigma} 1,
\end{multline*}
where we used that $2^{-j}\lambda \leq \lambda^{-\epsilon_2} \leq 1$.

On the other hand, as $G$ is M\'etivier, the Plancherel-type identity \cite[eq.\ (2.6)]{MP} in combination with \cite[eq.\ (2.14)]{MP} implies that 
\[
\| k_{H(\sqrt{\opL}, |D_u|)}\|_{L^2((1+|x|)^{2\sigma} \,d\bx)}\lesssim_\sigma \| k_{H(\sqrt{\opL}, |D_u|)}\|_2+
\| \opL^{\sigma/2}|D_u|^{-\sigma}  k_{H(\sqrt{\opL}, |D_u|)}\|_2
\]
for any joint spectral multiplier $H$. Thus, we may estimate
\[\begin{split}
&\| k_{\exp(\pm i \lambda\sqrt{\opL}) \chi_1(\sqrt{\opL}) \tilde\chi_1(2^{-k}\lambda^2|D_u|)}\|_{L^2((1+|x|)^{2\sigma} \,d\bx)} \\
&\lesssim_\sigma \| k_{\exp(\pm i \lambda\sqrt{\opL}) \chi_1(\sqrt{\opL}) \tilde\chi_1(2^{-k}\lambda^2|D_u|)}\|_2\\
&\qquad+ \| \opL^{\sigma/2}|D_u|^{-\sigma} k_{\exp(\pm i \lambda\sqrt{\opL}) \chi_1(\sqrt{\opL}) \tilde\chi_1(2^{-k}\lambda^2|D_u|)}\|_2 \\
&\lesssim_{\vec\epsilon,\sigma} \lambda^{\sigma(\epsilon_3+\epsilon_4)} \|k_{\chi_1(\sqrt{\opL})}\|_2 \\
&\lesssim_{\vec\epsilon,\sigma} \lambda^{\sigma(\epsilon_3+\epsilon_4)},
\end{split}\]
where we used that $\opL$ and $|D_u|$ commute, $\tilde\chi_1 \in C^\infty_c(\Rpos)$, $\lambda^{-2} 2^k \geq \lambda^{-(\epsilon_3+\epsilon_4)}$, and  
\[
\|\exp(\pm i \lambda\sqrt{\opL}) \tilde\chi_1(2^{-k}\lambda^2|D_u|)\|_{2 \to 2} \leq \|\tilde\chi_1\|_\infty \lesssim 1.
\]
The desired estimate for $A_{\vec\epsilon,\lambda}$ follows by combining the above.
\end{proof}

\section{The FIO region: preliminary considerations}

\subsection{Splitting into the elliptic and non-elliptic regions}

It remains to prove an estimate of the form
\[
\|\chr_{\overline{B}(0,2)} \exp(\pm i \sqrt{\opL}) B_{\vec\epsilon,\lambda} \|_1 \lesssim_{\vec\epsilon} \lambda^{(d-1)/2+\smallo(\vec\epsilon)}
\]
for all $\lambda \geq 1$, where $B_{\vec\epsilon,\lambda}$ is as in Proposition \ref{prp:spectrum_frequency}. This estimate will be established by means of FIO techniques. We begin by decomposing
\begin{equation}\label{eq:ell_nonell}
B_{\vec\epsilon,\lambda} 
= B^{(0)}_{\vec\epsilon,\lambda} + B^{(1)}_{\vec\epsilon,\lambda},
\end{equation}
where
\begin{align}
\label{eq:elliptic}
B^{(0)}_{\vec\epsilon,\lambda} &=  \chi_1(\sqrt{\opL}/\lambda) \sum_{j \tc \lambda^{1-\epsilon_3} \leq 2^j < \lambda^{1+\epsilon_2}} \chi_1(2^{-j}|D_x|) \, \chi_0(2^{\gamma-j}|D_u|) \delta_0 ,\\
\label{eq:non-elliptic}
B^{(1)}_{\vec\epsilon,\lambda}&= \chi_1(\sqrt{\opL}/\lambda) \sum_{j \tc \lambda^{1-\epsilon_3} \leq 2^j < \lambda^{1+\epsilon_2}} \sum_{k \tc 2^{j-\gamma} \leq 2^k < \lambda^{2-\epsilon_3-\epsilon_4}} \chi_1(2^{-j}|D_x|) \, \chi_1(2^{-k}|D_u|) \delta_0 ,
\end{align}
while $\gamma \in \NN$ is a large parameter to be fixed later.

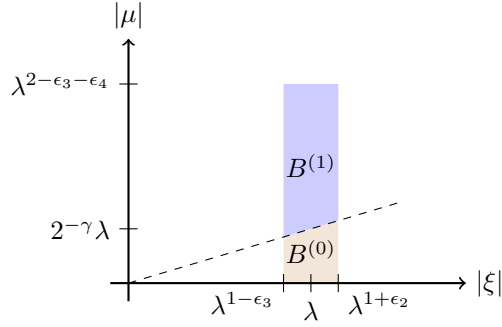
\begin{figure}

\begin{tikzpicture}[scale=1.2]

	\fill[fill=brown!20] (1.7,0) node[below left] {$\lambda^{1-\epsilon_3}$} -- (2.3,0) node[below right] {$\lambda^{1+\epsilon_2}$} -- (2.3,.69) -- (1.7,.51); 
	\fill[fill=blue!20] (1.7,.51) -- (2.3,.69) -- (2.3,2.2) -- (1.7,2.2); 

  \node[left] at (-.1,2.2) {$\lambda^{2-\epsilon_3-\epsilon_4}$};

  \draw[thick,->] (-0.2,0) -- (3.7,0) node[right] {$|\xi|$};
  \draw[thick,->] (0,-0.2) -- (0,2.7) node[above] {$|\mu|$};

  \draw[dashed] (0,0) -- (3,.9) ;

  \draw (1.7,.1) -- ++(0,-.2) ;
  \draw (2,.1) -- ++(0,-.2) node[below] {$\lambda$} ;
	\draw (2.3,.1) -- ++(0,-.2) ;
  \draw (0.1,2.2) -- ++(-.2,0) ;
	\draw (0.1,.6) -- ++(-.2,0) node[left] {$2^{-\gamma} \lambda$} ;
  \node at (2,1.3) {$B^{(1)}$};
  \node at (2,0.35) {$B^{(0)}$};

\end{tikzpicture}

\caption{Splitting into elliptic and non-elliptic regions.}
\label{fig:freq_loc_ell_nonell}

\end{figure}

Looking at Figures \ref{fig:freq_loc} and \ref{fig:freq_loc_ell_nonell}, the splitting \eqref{eq:ell_nonell} essentially corresponds to splitting the region $B$ into the two subregions where $|\mu| \ll |\xi|$ and where $|\mu| \gtrsim |\xi|$. 
Under the spatial localisation $\bx\in \overline{B}(0,2)$, from \eqref{eq:HamH} we see that in the cone $\{ |\mu| \ll |\xi|\}$ the symbol of $\opL$ behaves like an elliptic symbol, as
\[
\Ham(\bx,\bxi)^2 = |\xi + J_\mu x/2|^2\simeq|\xi|^2
\]
there, while the non-ellipticity of $\opL$ manifests itself in the complementary cone. Therefore we can think of $B^{(0)}_{\vec\epsilon,\lambda}$ and $B^{(1)}_{\vec\epsilon,\lambda}$ as the  contribution to $B_{\vec\epsilon,\lambda}$ by the \emph{elliptic region}, and by the \emph{non-elliptic region} respectively.

\smallskip

For the elliptic region, as the number of summands in \eqref{eq:elliptic} is $\lesssim_{\vec\epsilon} \log \lambda$, it is enough to consider each summand separately, i.e., to prove that
\begin{equation}\label{eq:elliptic_region}
\|\chr_{\overline{B}(0,2)} \exp(\pm i \sqrt{\opL}) \chi_1(\sqrt{\opL}/\lambda) \tilde H_j \|_1 \lesssim_{\vec\epsilon} \lambda^{(d-1)/2+\smallo(\vec\epsilon)}
\end{equation}
for all $j \in \NN$ such that
\begin{equation}\label{eq:cond_ell_j}
\lambda^{1-\epsilon_3} \leq 2^j < \lambda^{1+\epsilon_2},
\end{equation}
where
\begin{equation}\label{eq:initial_amplitude_elliptic}
\tilde H_j \defeq \chi_1(2^{-j}|D_x|) \, \chi_0(2^{\gamma-j}|D_u|) \delta_0,
\end{equation}

Similarly, for the non-elliptic region, as the number of pairs $(j,k)$ appearing in the sum in \eqref{eq:non-elliptic} is $\lesssim_{\vec\epsilon} (\log \lambda)^2$, it is enough to consider each summand separately and prove that
\begin{equation}\label{eq:non-elliptic_region}
\|\chr_{\overline{B}(0,2)} \exp(\pm i \sqrt{\opL}) \chi_1(\sqrt{\opL}/\lambda) \tilde H_{j,k} \|_1 \lesssim_{\vec\epsilon,s} \lambda^{(d-1)/2+\smallo(\vec\epsilon)}
\end{equation}
for all $j,k \in \ZZ$ such that
\begin{equation}\label{eq:cond_jk_nonell}
\lambda^{1-\epsilon_3} \leq 2^j < \lambda^{1+\epsilon_2} \qquad\text{and}\qquad  2^{j-\gamma} \leq 2^k < \lambda^{2-\epsilon_3-\epsilon_4},
\end{equation}
where
\begin{equation}\label{eq:initial_amplitude_non-elliptic}
\tilde H_{j,k} \defeq \chi_1(2^{-j}|D_x|) \, \chi_1(2^{-k}|D_u|) \delta_0.
\end{equation}

\subsection{Dropping the spectral localisation}

In order to simplify the subsequent discussion, we now show that in \eqref{eq:elliptic_region} and \eqref{eq:non-elliptic_region} we may drop the spectral localisation $\chi_1(\sqrt{\opL}/\lambda)$, at the cost of slightly relaxing the spatial localisation $\chr_{\overline{B}(0,2)}$.

As a matter of fact,  according to \eqref{eq:cond_ell_j} and \eqref{eq:cond_jk_nonell}, here we have $\lambda \approx 2^j$. Thus, the spectral localisation is essentially captured by the frequency localisation given by $\chi_1(2^{-j}|D_x|)$, i.e., $\lambda\approx|\xi|$ here (see Figure \ref{fig:freq_loc_ell_nonell}).

\begin{prp}\label{prp:sp_loc_removal}
Let $\vec\epsilon \in (\Rpos)^4$. For all $N \in \NN$ and $\lambda \geq 1$,
\begin{multline}\label{eq:sp_loc_removal_non-elliptic}
\|\chr_{\overline{B}(0,2)} \exp(\pm i \sqrt{\opL}) \chi_1(\sqrt{\opL}/\lambda) \tilde H_{j,k} \|_1 \\
\lesssim_{\vec\epsilon,N} \|\chr_{\overline{B}(0,4)} \cos(\sqrt{\opL}) \tilde H_{j,k} \|_1 
+ \lambda^{-1} \|\chr_{\overline{B}(0,4)} \sqrt{\opL} \sin(\sqrt{\opL}) \tilde H_{j,k} \|_1
 + \lambda^{-N}
\end{multline}
for all $j,k \in \ZZ$ satisfying \eqref{eq:cond_jk_nonell}, and
\begin{multline}\label{eq:sp_loc_removal_elliptic}
\|\chr_{\overline{B}(0,2)} \exp(\pm i \sqrt{\opL}) \chi_1(\sqrt{\opL}/\lambda) \tilde H_j \|_1 \\
 \lesssim_{\vec\epsilon,N} \|\chr_{\overline{B}(0,4)} \cos(\sqrt{\opL}) \tilde H_j \|_1 
+ \lambda^{-1} \|\chr_{\overline{B}(0,4)} \sqrt{\opL} \sin(\sqrt{\opL}) \tilde H_j \|_1 
+ \lambda^{-N}
\end{multline}
for all $j \in \NN$ satisfying \eqref{eq:cond_ell_j}.
\end{prp}
\begin{proof}
We only prove the estimate \eqref{eq:sp_loc_removal_non-elliptic}, as \eqref{eq:sp_loc_removal_elliptic} can be proved analogously. For this, we write, as in \eqref{eq:exp_cos_sin},
\[
\exp(\pm i \sqrt{\opL}) \chi_1(\sqrt{\opL}/\lambda) = \cos(\sqrt{\opL}) \chi_1(\sqrt{\opL}/\lambda) \pm i \lambda^{-1} \sqrt{\opL} \sin(\sqrt{\opL}) \chi_2(\sqrt{\opL}/\lambda),
\]
and we decompose, as in \eqref{eq:fourier_support_dec},
\[
\chi_r(\sqrt{\opL}/\lambda) = 
\chi_r^{0,\delta,\lambda}(\sqrt{\opL}/\lambda)
+ \chi_r^{\infty,\delta,\lambda}(\sqrt{\opL}/\lambda)
\]
for $r=1,2$ and some $\delta \in (0,1]$. Then, by \eqref{eq:support_1de},
\[
\supp k_{\chi_r^{0,\delta,\lambda}(\sqrt{\opL}/\lambda)} \subseteq \overline{B}(0,\lambda^{\delta-1}) \subseteq \overline{B}(0,1).
\]
Therefore, much as in the proof of Proposition \ref{prp:spatial_loc},
\[\begin{split}
\chr_{\overline{B}(0,2)} \exp(\pm i\sqrt{\opL}) \chi_1(\sqrt{\opL}/\lambda) 
&= 
\chr_{\overline{B}(0,2)} \chi_1(\sqrt{\opL}/\lambda) \chr_{\overline{B}(0,4)} \cos(\sqrt{\opL})  \\
&\quad+
\chr_{\overline{B}(0,2)} \chi_1^{\infty,\delta,\lambda}(\sqrt{\opL}/\lambda) \chr_{G \setminus \overline{B}(0,4)} \cos(\sqrt{\opL})  \\
&\quad\pm i\lambda^{-1}
\chr_{\overline{B}(0,2)} \chi_2(\sqrt{\opL}/\lambda) \chr_{\overline{B}(0,4)} \sqrt{\opL} \sin(\sqrt{\opL})  \\
&\quad\pm i\lambda^{-1}
\chr_{\overline{B}(0,2)} \chi_2^{\infty,\delta,\lambda}(\sqrt{\opL}/\lambda) \chr_{G \setminus \overline{B}(0,4)} \sqrt{\opL} \sin(\sqrt{\opL})  
\end{split}\]
(here we think of characteristic functions as multiplication operators).
Moreover,
\begin{equation}\label{eq:main_Hmh}
\begin{aligned}
\|\chr_{\overline{B}(0,2)} \chi_1(\sqrt{\opL}/\lambda) \chr_{\overline{B}(0,4)} \cos(\sqrt{\opL}) \tilde H_{j,k} \|_{1} 
&\lesssim \| \chr_{\overline{B}(0,4)} \cos(\sqrt{\opL}) \tilde H_{j,k} \|_{1},\\
\|\chr_{\overline{B}(0,2)} \chi_1(\sqrt{\opL}/\lambda) \chr_{\overline{B}(0,4)} \sqrt{\opL} \sin(\sqrt{\opL}) \tilde H_{j,k} \|_{1} 
&\lesssim \| \chr_{\overline{B}(0,4)} \sqrt{\opL} \sin(\sqrt{\opL}) \tilde H_{j,k} \|_{1}
\end{aligned}
\end{equation}
by the uniform $L^1$-boundedness of the $\chi_1(\sqrt{\opL}/\lambda)$.
On the other hand, 
\begin{equation}\label{eq:l2norm_error}
\|\chi_r^{\infty,\delta,\lambda}(\sqrt{\opL}/\lambda)\|_{2 \to 2} \leq \|\chi_r^{\infty,\delta,\lambda}\|_\infty \lesssim_{N,\delta} \lambda^{-N}
\end{equation}
for $r=1,2$ and all $N \in \NN$, by \eqref{eq:schwartz_bound_error}. Furthermore, by \eqref{eq:initial_amplitude_non-elliptic},
\begin{equation}\label{eq:Hmh_L2}
\|\tilde H_{j,k}\|_2 \simeq 2^{jd_1/2} 2^{kd_2/2}.
\end{equation}
Thus
\begin{equation}\label{eq:error_Hmh}
\begin{split}
&\|\chr_{\overline{B}(0,2)} \chi_1^{\infty,\delta,\lambda}(\sqrt{\opL}/\lambda) \chr_{G \setminus \overline{B}(0,4)} \cos(\sqrt{\opL}) \tilde H_{j,k} \|_1  \\
&\qquad+\lambda^{-1} \|\chr_{\overline{B}(0,2)} \chi_2^{\infty,\delta,\lambda}(\sqrt{\opL}/\lambda) \chr_{G \setminus \overline{B}(0,4)} \sqrt{\opL} \sin(\sqrt{\opL}) \tilde H_{j,k} \|_1  \\
&\lesssim \|\chi_1^{\infty,\delta,\lambda}(\sqrt{\opL}/\lambda) \chr_{G \setminus \overline{B}(0,4)} \cos(\sqrt{\opL}) \tilde H_{j,k} \|_2 \\
&\qquad+\lambda^{-1} \|\chi_1^{\infty,\delta,\lambda}(\sqrt{\opL}/\lambda) \chr_{G \setminus \overline{B}(0,4)} \sqrt{\opL} \sin(\sqrt{\opL}) \tilde H_{j,k} \|_2 \\
&\lesssim_{N,\epsilon_4} \lambda^{-N} 2^{jd_1/2} 2^{kd_2/2} \\
&\leq \lambda^{(1+\epsilon_2) d_1/2 + 2 d_2/2 - N} 
\end{split}
\end{equation}
for all $N \in \NN$,
where we used the Cauchy--Schwarz inequality, the estimates \eqref{eq:l2norm_error} and \eqref{eq:Hmh_L2}, and the constraints \eqref{eq:cond_jk_nonell}. The desired estimate \eqref{eq:sp_loc_removal_non-elliptic} follows by combining \eqref{eq:main_Hmh} and \eqref{eq:error_Hmh}.
\end{proof}

In light of Proposition \ref{prp:sp_loc_removal}, the estimate \eqref{eq:elliptic_region} for the \emph{elliptic region} is reduced to proving that
\begin{equation}\label{eq:elliptic_region2}
\begin{aligned}
\|\chr_{\overline{B}(0,4)} \cos(\sqrt{\opL}) \tilde H_j \|_1 &\lesssim_{\vec\epsilon} \lambda^{(d-1)/2+\smallo(\vec\epsilon)}, \\
\|\chr_{\overline{B}(0,4)} \sqrt{\opL} \sin(\sqrt{\opL}) \tilde H_j \|_1 &\lesssim_{\vec\epsilon} \lambda^{1+(d-1)/2+\smallo(\vec\epsilon)},
\end{aligned}
\end{equation}
where $\tilde H_j$ is given by \eqref{eq:initial_amplitude_elliptic}, for all $\lambda \geq 1$ and $j \in \ZZ$ satisfying \eqref{eq:cond_ell_j}. 
As we shall see, we can treat this region with ``standard'' FIO techniques, without the use of a complex phase: indeed, due to the frequency localisation $|\mu| \lesssim 2^{-\gamma}|\xi|$, as well as the small time involved, this part can be dealt with the techniques of \cite{MMNG,SSS}, provided the parameter $\gamma$ in \eqref{eq:elliptic} is chosen to be sufficiently large. The corresponding discussion is found in Section \ref{s:elliptic} below.

Similarly, by Proposition \ref{prp:sp_loc_removal}, the estimate \eqref{eq:non-elliptic_region} for the \emph{non-elliptic region} is reduced to
\begin{equation}\label{eq:non-elliptic_region_nosploc}
\begin{aligned}
\|\chr_{\overline{B}(0,4)} \cos(\sqrt{\opL}) \tilde H_{j,k} \|_1 &\lesssim_{\vec\epsilon,s} \lambda^{(d-1)/2+\smallo(\vec\epsilon)}, \\
\|\chr_{\overline{B}(0,4)} \sqrt{\opL} \sin(\sqrt{\opL}) \tilde H_{j,k} \|_1 &\lesssim_{\vec\epsilon,s} \lambda^{1+(d-1)/2+\smallo(\vec\epsilon)} 
\end{aligned}
\end{equation}
for all $\lambda \geq 1$ and $j,k \in \ZZ$ satisfying \eqref{eq:cond_jk_nonell}. This estimate, instead, will be tackled by means of a complex-phase FIO parametrix, whose construction, based on ideas in \cite{LSV,SV}, is presented in Section \ref{s:LSVparametrix}. In preparation for this construction, a few additional reductions of the estimate \eqref{eq:non-elliptic_region_nosploc} will be discussed in the rest of this section.

\subsection{Automorphic scaling}\label{ss:scaling}

It will be convenient to re-scale each kernel $\tilde H_{j,k}$ in \eqref{eq:non-elliptic_region_nosploc} via automorphic dilations in such a way that the frequency variable $\bxi = (\xi,\mu)$ satisfies $|\xi| \simeq|\mu|$ after the scaling.
To this end, let us put 
\[
\ell \defeq k-j, \qquad m \defeq 2j-k,
\]
so that $2^j/2^\ell = 2^k/2^{2\ell} = 2^m$, and, by \eqref{eq:cond_jk_nonell},
\begin{equation}\label{eq:cond_lm}
2^{-\gamma} \leq 2^\ell < \lambda^{1-\epsilon_4}, \qquad \lambda^{\epsilon_4-\epsilon_3} < 2^m, \qquad 2^{\ell+m} < \lambda^{1+\epsilon_2}.
\end{equation}
We shall later actually choose $\epsilon_4>\epsilon_3$, so that $m>0$. 
Then, if we scale in $(x,u)$ by $\dil_{2^{-\ell}}$, i.e., in $(\xi,\mu)$ by $\dil_{2^\ell}$, the proof of \eqref{eq:non-elliptic_region_nosploc} is reduced to proving that
\begin{equation}\label{eq:non-elliptic_region2}
\begin{aligned}
\|\chr_{\overline{B}(0,2^{2+\ell})} \cos(2^\ell \sqrt{\opL}) H_m \|_1 &\lesssim_{\vec\epsilon} \lambda^{(d-1)/2+\smallo(\vec\epsilon)} , \\
\|\chr_{\overline{B}(0,2^{2+\ell})} 2^{\ell} \sqrt{\opL} \sin(2^\ell \sqrt{\opL}) H_m \|_1 &\lesssim_{\vec\epsilon} \lambda^{1+(d-1)/2+\smallo(\vec\epsilon)} 
\end{aligned}
\end{equation}
for all $\lambda \geq 1$ and $\ell,m \in \ZZ$ satisfying \eqref{eq:cond_lm}, where
\begin{equation}\label{eq:Hm_def}
H_m \defeq \chi_1(2^{-m}|D_x|) \, \chi_1(2^{-m}|D_u|) \delta_0.
\end{equation}
The advantage of this scaling is that now we have to deal with a single frequency localisation parameter, i.e., $|\xi| \simeq|\mu| \simeq2^m$ on the Fourier support of $H_m$ (see Figure \ref{fig:freq_scaling}); the price to pay is that the estimates \eqref{eq:non-elliptic_region2} involve the wave propagator at time $t = 2^\ell$, which according to \eqref{eq:cond_lm} may become arbitrarily large with $\lambda$.

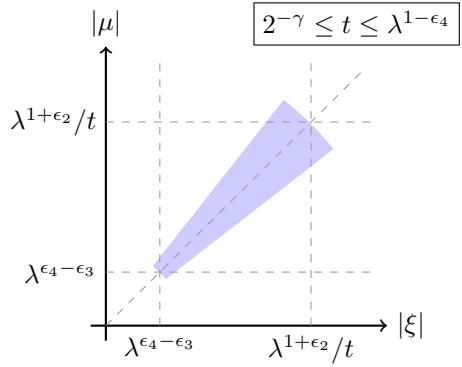
\begin{figure}

\begin{tikzpicture}[scale=1]

  \node at (3.3,4) {\fbox{$2^{-\gamma} \leq t \leq \lambda^{1-\epsilon_4}$}};

	\fill[fill=blue!20] (38:1) arc (38:52:1) -- (52:3.8) arc (52:38:3.8) -- cycle;

  \draw[thick,->] (-0.2,0) -- (3.7,0) node[right] {$|\xi|$};
  \draw[thick,->] (0,-0.2) -- (0,3.7) node[above] {$|\mu|$};

  \draw[dashed,draw=black!40] (.71,0) node[below] {$\lambda^{\epsilon_4-\epsilon_3}$} -- (.71,3.5);
  \draw[dashed,draw=black!40] (0,.71) node[left] {$\lambda^{\epsilon_4-\epsilon_3}$} -- (3.5,.71);
  \draw[dashed,draw=black!40] (2.7,0) node[below] {$\lambda^{1+\epsilon_2}/t$} -- (2.7,3.5);
  \draw[dashed,draw=black!40] (0,2.7) node[left] {$\lambda^{1+\epsilon_2}/t$} -- (3.5,2.7);
  \draw[dashed,draw=black!40] (0,0)  -- (3.4,3.4);

\end{tikzpicture}

\caption{The non-elliptic region after the automorphic scalings.}
\label{fig:freq_scaling}

\end{figure}

\subsection{Introduction of the complex phase}\label{ss:introcomplex}

To deal with the estimates \eqref{eq:non-elliptic_region2}, due to the large time parameter involved, we aim to use the ``complex phase FIO'' approach discussed in Section \ref{s:LSVparametrix} below. To this purpose, it is convenient to turn the ``Euclidean'' frequency localisation expressed by the kernel $H_m$ into an analogous expression involving the complex phase \eqref{eq:phase} at time $t=0$.

\begin{prp}\label{prp:introd_complexphase}
Assume that $\rk J_\mu$ is constant for $\mu \neq 0$.
For all $m,M \in \NN$, we can write the function $H_m$ of \eqref{eq:Hm_def} as
\begin{equation}\label{eq:real_complex_phase}
H_m
= \sum_{h=0}^M \frac{2^{-mh}}{4^h h!} H_{m,h} + R_{m,M},
\end{equation}
where
\begin{equation}\label{eq:initial_amplitude_complex}
H_{m,h}(\bx) = (2\pi)^{-d} \int_{\RR^d} e^{i\bx \cdot \bxi-\langle |J_\mu| x,x\rangle/4} \, \eta_{h} (2^{-m}\bxi)  \,d\bxi
\end{equation}
and $\eta_{h} \in C^\infty_c(\RR^d)$ satisfies
\[
\supp \eta_{h} \subseteq \{ \bxi \in \RR^d \tc 1/2 \leq |\xi| \leq 2, \ 1/2 \leq |\mu| \leq 2 \}.
\]
Moreover, if $\vec\epsilon \in (\Rpos)^4$ is such that $\epsilon_4>\epsilon_3$, then, for all $N \in \NN$, there exists $M = M(\vec\epsilon,N) \in \NN$ such that
\begin{equation}\label{eq:remainder_real_complex}
\|\chr_{\overline{B}(0,2^{2+\ell})} \cos(2^\ell \sqrt{\opL}) R_{m,M} \|_1  + \|\chr_{\overline{B}(0,2^{2+\ell})} 2^\ell \sqrt{\opL} \sin(2^\ell \sqrt{\opL}) R_{m,M} \|_1 \lesssim_{\vec\epsilon,N} \lambda^{-N},
\end{equation}
for all $\lambda \geq 1$ and all $\ell,m \in \NN$ satisfying \eqref{eq:cond_lm}.
\end{prp}
\begin{proof}
Under the above assumption on $J_\mu$, the expression $\langle |J_\mu| x, x \rangle$ is a positive semidefinite quadratic form in $x \in \RR^{d_1}$, while it is smooth and $1$-homogeneous in $\mu \in \dot\RR^{d_2}$ (see Proposition \ref{prp:modJmu_analytic}), and in particular it satisfies the bound
\begin{equation}\label{eq:metivier_nondeg}
0 \leq \langle |J_\mu| x, x \rangle \lesssim |\mu| |x|^2.
\end{equation}
Correspondingly, the operator $\langle |J_\mu| D_\xi, D_\xi \rangle$ is a second-order homogeneous differential operator in the variable $\xi$, with coefficients depending smoothly and $1$-homogeneously on the parameter $\mu \in \dot\RR^{d_2}$.

By Taylor's expansion, for any $M \in \NN$ and $t \in \RR$,
\begin{equation}\label{eq:taylor_exp}
e^{t/4} = \sum_{h=0}^M \frac{t^h}{4^h h!} + t^{M+1} r_M(t),
\end{equation}
where $r_M \in C^\infty(\RR)$. Now, by \eqref{eq:Hm_def},
\[\begin{split}
H_m(\bx) 
&= (2\pi)^{-d} \int e^{i\bx \cdot \bxi} \, \chi_1(2^{-m}|\xi|) \, \chi_1(2^{-m}|\mu|) \,d\bxi \\
&= (2\pi)^{-d} \int e^{i\bx \cdot \bxi - \langle |J_\mu| x,x \rangle/4} \, e^{\langle |J_\mu| x,x \rangle/4} \, \chi_1(2^{-m}|\xi|) \, \chi_1(2^{-m}|\mu|) \,d\bxi ,
\end{split}\]
so, by expanding the term $e^{\langle |J_\mu| x,x \rangle/4}$ via \eqref{eq:taylor_exp}, we obtain the decomposition \eqref{eq:real_complex_phase},
where
\begin{multline}\label{eq:Hmh}
H_{m,h}(\bx) \defeq 2^{mh} (2\pi)^{-d} \int e^{i\bx \cdot \bxi - \langle |J_\mu| x,x \rangle/4} \\
\times \langle |J_\mu| x,x \rangle^h \, \chi_1(2^{-m}|\xi|) \, \chi_1(2^{-m}|\mu|) \,d\bxi, 
\end{multline}
and  
\begin{multline}\label{eq:Rmh}
R_{m,M}(\bx) \defeq (2\pi)^{-d} \int e^{i\bx \cdot \bxi - \langle |J_\mu| x,x \rangle/4} \\
\times \langle |J_\mu| x,x \rangle^{M+1} r_M(\langle |J_\mu| x,x \rangle) \, \chi_1(2^{-m}|\xi|) \, \chi_1(2^{-m}|\mu|) \,d\bxi .
\end{multline}
Note now that, by repeated integration by parts in $\xi$, the expression \eqref{eq:Hmh} turns into \eqref{eq:initial_amplitude_complex},
where
\[
\eta_{h}(\bxi) \defeq \langle |J_\mu| D_\xi,D_\xi \rangle^h \chi_{1}(|\xi|) \chi_{1}(|\mu|);
\]
in particular, $\eta_h$ is smooth and satisfies the required support condition.

\smallskip

We now prove the estimate \eqref{eq:remainder_real_complex}. For this, we decompose
\[
R_{m,M} = \chi_0(\lambda^{-\epsilon_0} 2^m|\cdot|) R_{m,M} + (1-\chi_0(\lambda^{-\epsilon_0} 2^m|\cdot|)) R_{m,M},
\]
where $|\cdot|$ denotes the Euclidean norm on $\RR^d$, while $\epsilon_0 > 0$ is such that
\begin{equation}\label{eq:cond_delta}
2\epsilon_0 < \epsilon_4 - \epsilon_3;
\end{equation}
note that we can choose such an $\epsilon_0$ as we assumed that $\epsilon_4 > \epsilon_3$.

In order to estimate $\chi_0(\lambda^{-\epsilon_0} 2^m|\cdot|) R_{m,M}$, we shall take advantage of the power of $\langle |J_\mu| x, x \rangle$ in \eqref{eq:Rmh}, which can be made arbitrarily small by taking $M$ and $\lambda$ sufficiently large. Indeed, on the support of $\chi_0(\lambda^{-\epsilon_0} 2^m|\bx|) \chi_1(2^{-m}|\xi|) \chi_1(2^{-m}|\mu|)$, by \eqref{eq:metivier_nondeg}, \eqref{eq:cond_lm} and \eqref{eq:cond_delta}, we have that
\begin{equation}\label{eq:RmMbounds}
\begin{gathered}
|\xi| \simeq |\mu| \simeq 2^{m} \leq \lambda^{1+\epsilon_2}, \qquad |\bx| \lesssim 2^{-m} \lambda^{\epsilon_0} \leq \lambda^{\epsilon_0+\epsilon_3-\epsilon_4} \leq 1, \\
|\langle |J_\mu| x,x \rangle| \lesssim |\mu| |x|^2 \lesssim 2^{-m} \lambda^{2\epsilon_0} \leq \lambda^{2\epsilon_0+\epsilon_3-\epsilon_4} \leq 1,
\end{gathered}
\end{equation}
so indeed under the condition \eqref{eq:cond_delta} any factor $\langle |J_\mu| x,x \rangle$ produces a negative power of $\lambda$. 
Notice further that, when taking $\bx$-derivatives of $R_{m,M}$, each derivative may reduce by at most $1$ the exponent of $\langle |J_\mu| x,x \rangle$, and at the same time may produce (by the Chain Rule) one additional factor, which is a component of $i\bxi$ or $|J_\mu|x$ (according to which term is hit by the derivative), but any of these additional factors (and any of their $\bx$-derivatives) is bounded by $2^m$ under the conditions \eqref{eq:RmMbounds}. Similarly, one sees that $\|\partial_{\bx}^\alpha \chi_0(\lambda^{-\epsilon_0} 2^m|\cdot|)\|_\infty \lesssim_{\alpha} 2^{m|\alpha|}$ for all $\alpha \in \NN^d$. Hence, by \eqref{eq:Rmh} and \eqref{eq:RmMbounds}, for any $\alpha \in \NN^{d}$,
\[\begin{split}
&| \partial_{\bx}^\alpha (\chi_0(\lambda^{-\epsilon_0} 2^m|\bx|) R_{m,M}(\bx)) | \\
&\lesssim_{M,\alpha} 2^{m|\alpha|} \lambda^{-(M+1-|\alpha|)(\epsilon_4-\epsilon_3-2\epsilon_0)} \tilde\chi_0(\lambda^{-\epsilon_0} 2^m|\bx|) \int |\chi_1(2^{-m}|\xi|)| \, |\chi_1(2^{-m}|\mu|)| \,d\bxi \\
&\simeq 2^{m(d+|\alpha|)} \lambda^{-(M+1-|\alpha|)(\epsilon_4-\epsilon_3-2\epsilon_0)} \tilde \chi_0(\lambda^{-\epsilon_0} 2^m|\bx|),
\end{split}\]
where $\tilde\chi_0 = \chr_{\supp \chi_0}$;
thus, in combination with \eqref{eq:cond_lm},
\[\begin{split}
\|\partial_{\bx}^\alpha (\chi_0(\lambda^{-\epsilon_0} 2^m|\cdot|) R_{m,M})\|_2 
&\lesssim_{M,\alpha} 2^{m(d/2+|\alpha|)} \lambda^{\epsilon_0 d/2-(M+1-|\alpha|)(\epsilon_4-\epsilon_3-2\epsilon_0)} \\
&\leq \lambda^{(1+\epsilon_2+\epsilon_0) d/2+(1+\epsilon_2)|\alpha|-(M+1-|\alpha|)(\epsilon_4-\epsilon_3-2\epsilon_0)}.
\end{split}\]
As $\opL$ is a second-order differential operator and $|\bx| \lesssim 1$ on $\supp \chi_0(\lambda^{-\epsilon_0} 2^m|\cdot|)$, we conclude that, for any $N \in \NN$, there exists $M_0$ sufficiently large (depending on $\vec\epsilon,\epsilon_0,N$) so that, for all $M \geq M_0$,
\[
\|\chi_0(\lambda^{-\epsilon_0} 2^m|\cdot|)  R_{m,M}\|_2 + \|\opL (\chi_0(\lambda^{-\epsilon_0} 2^m|\cdot|)  R_{m,M})\|_2 \lesssim_{\vec\epsilon,N,M} \lambda^{-N}.
\]
Hence, by the Cauchy--Schwarz inequality and \eqref{eq:cond_lm}, for any $M \geq M_0$,
\[\begin{split}
&\| \chr_{\overline{B}(0,2^{2+\ell})} \cos(2^\ell \sqrt{\opL})  (\chi_0(\lambda^{-\epsilon_0} 2^m|\cdot|) R_{m,M}) \|_1 \\
&+ \| \chr_{\overline{B}(0,2^{2+\ell})} 2^\ell \sqrt{\opL} \sin(2^\ell \sqrt{\opL})  (\chi_0(\lambda^{-\epsilon_0} 2^m|\cdot|) R_{m,M}) \|_1 \\
&\lesssim 2^{\ell Q/2} \|\chi_0(\lambda^{-\epsilon_0} 2^m|\cdot|) R_{m,M}\|_2 +  2^{\ell (2+Q/2)} \|\opL( \chi_0(\lambda^{-\epsilon_0} 2^m|\cdot|) R_{m,M})\|_2 \\
&\lesssim_{\vec\epsilon,N,M} \lambda^{2+Q/2-N},
\end{split}\]
where the uniform $L^2$-boundedness of $\cos(2^\ell \sqrt{\opL})$ and $\sin(2^\ell\sqrt{\opL})/(2^\ell \sqrt{\opL})$ was used.
This shows that, if \eqref{eq:cond_delta} is satisfied, then, for any $N \in \NN$, by taking $M$ sufficiently large (depending on $\vec\epsilon,\epsilon_0,N$), we can ensure that
\[
\begin{aligned}
\|\chr_{\overline{B}(0,2^{2+\ell})} \cos(2^\ell \sqrt{\opL})  (\chi_0(\lambda^{-\epsilon_0} 2^m|\cdot|) R_{m,M}) \|_1 &\lesssim_{\vec\epsilon,N} \lambda^{-N},\\
\| \chr_{\overline{B}(0,2^{2+\ell})} 2^\ell \sqrt{\opL} \sin(2^\ell \sqrt{\opL})  (\chi_0(\lambda^{-\epsilon_0} 2^m|\cdot|) R_{m,M}) \|_1 &\lesssim_{\vec\epsilon,N} \lambda^{-N},
\end{aligned}
\]
for all $\lambda \geq 1$.

In order to prove the analogous estimate for $(1-\chi_0(\lambda^{-\epsilon_0} 2^m|\cdot|) R_{m,M}$, in light of the decomposition \eqref{eq:real_complex_phase}, it will be enough to show that
\begin{equation}\label{eq:large_Hmh}
\begin{aligned}
\|(1-\chi_0(\lambda^{-\epsilon_0} 2^m|\cdot|) H_{m} \|_2 + \|\opL ((1-\chi_0(\lambda^{-\epsilon_0} 2^m|\cdot|) H_{m})\|_2 &\lesssim_{\epsilon_0,N} \lambda^{-N}, \\
\|(1-\chi_0(\lambda^{-\epsilon_0} 2^m|\cdot|) H_{m,h} \|_2 + \|\opL ((1-\chi_0(\lambda^{-\epsilon_0} 2^m|\cdot|) H_{m,h})\|_2  &\lesssim_{\epsilon_0,N,h} \lambda^{-N} 
\end{aligned}
\end{equation}
for all $N,h \in \NN$, all $\lambda \geq 1$, and all $m \in \NN$ satisfying \eqref{eq:cond_lm}.

Starting from the expression \eqref{eq:initial_amplitude_complex}, we see that, for all $a \in \NN^{d_1}$ and $b \in \NN^{d_2}$,
\begin{multline*}
\partial_u^b \partial_x^a H_{m,h}(\bx) \\
= \sum_{a_1+a_2=a} \frac{a!}{a_1!a_2!} 2^{m(|a_2|+|b|)} \int e^{i\bx \cdot\bxi -\langle |J_\mu| x,x \rangle/4} \, Q_{a_1}(\mu,x) \, \eta_{h,a_2,b}(2^{-m}\bxi) \,d\bxi ,
\end{multline*}
where
\begin{align*}
\eta_{h,a_2,b}(\bxi) &\defeq (i\xi)^{a_2} (i\mu)^b \eta_h(\bxi), \\
Q_{a_1}(\mu,x) &\defeq e^{\langle|J_\mu| x,x\rangle/4} \partial_x^{a_1} e^{-\langle|J_\mu| x,x\rangle/4} = \sum_{r=\lceil |a_1|/2 \rceil}^{|a_1|} Q_{a_1,r}(\mu,x)
\end{align*}
and $Q_{a_1,r}(\mu,x)$ is smooth, $(2r-|a_1|)$-homogeneous in $x \in \RR^{d_1}$ and $r$-homogeneous in $\mu \in \dot\RR^{d_2}$. Now, by repeated integration by parts in $\xi$, for all $\alpha \in \NN^{d_1}$,
\begin{multline*}
(-ix)^\alpha \partial_u^b \partial_x^a H_{m,h}(\bx) 
= \sum_{a_1+a_2=a} \frac{a!}{a_1!a_2!} 2^{m(|a_2|+|b|-|\alpha|)} \\
\times \int e^{i\bx \cdot\bxi -\langle |J_\mu| x,x \rangle/4} \, Q_{a_1}(\mu,x) \, (\partial_\xi^\alpha \eta_{h,a_2,b})(2^{-m}\bxi) \,d\bxi ;
\end{multline*}
next, by repeated integration by parts in $\mu$, for all $\beta \in \NN^{d_2}$,
\[\begin{split}
&(-iu)^\beta (-ix)^\alpha \partial_u^b \partial_x^a H_{m,h}(\bx) \\ 
&= \sum_{a_1+a_2=a} \sum_{\beta_1+\beta_2+\beta_3=\beta} \frac{a!}{a_1!a_2!} \frac{\beta!}{\beta_1! \beta_2! \beta_3!} 
 2^{m(|a_2|+|b|-|\alpha|-|\beta_3|)} \\
&\qquad\times \int e^{i\bx\cdot\bxi-\langle|J_\mu| x,x\rangle/4} \, P_{\beta_1}(\mu,x) \, \partial_\mu^{\beta_2} Q_{a_1}(\mu,x) \,
(\partial_\mu^{\beta_3} \partial_\xi^\alpha \eta_{h,a_2,b}) (2^{-m}\bxi) \,d\bxi, 
\end{split}\]
where
\[
P_{\beta_1}(\mu,x) \defeq e^{\langle|J_\mu| x,x\rangle/4} \partial_\mu^{\beta_1} e^{-\langle|J_\mu| x,x\rangle/4} = \sum_{j=0}^{|\beta_1|} P_{\beta_1,j}(\mu,x)
\]
and $P_{\beta_1,j}(\mu,x)$ is smooth and homogeneous of degree $2j$ in $x \in \RR^{d_1}$ and of degree $j-|\beta_1|$ in $\mu \in \dot\RR^{d_2}$.

Notice now that, by homogeneity,
\[\begin{split}
|P_{\beta_1}(\mu,x) \, \partial_\mu^{\beta_2} Q_{a_1}(\mu,x)| 
&\lesssim_{\beta_1,\beta_2,a_1} \sum_{j=0}^{|\beta_1|} \sum_{r=\lceil |a_1|/2\rceil}^{|a_1|} |\mu|^{j-|\beta_1|+r-|\beta_2|} |x|^{2j+2r-|a_1|} \\
&\lesssim_{\beta_1,\beta_2,a_1} |\mu|^{|a_1|/2-|\beta_1|-|\beta_2|} (1+|\mu||x|^2)^{|\beta_1|+|a_1|/2}.
\end{split}\]
As a consequence, also using the fact that $|\mu|^{|a_1|/2} \lesssim_a |\mu|^{|a_1|}$ when $|\mu| \gtrsim 1$, we see that
\[\begin{split}
|u^\beta x^\alpha \partial_u^b \partial_x^a H_{m,h}(\bx)|
&\lesssim_{\alpha,\beta,a,b,h} 2^{m(|a|+|b|-|\alpha|-|\beta|)} 
\int e^{-\langle|J_\mu|x,x\rangle/4} (1+|\mu||x|^2)^{|\beta|+|a|/2} \\
&\qquad\times \chr_{[1/2,2]}(2^{-m}|\xi|) \, \chr_{[1/2,2]}(2^{-m}|\mu|) \,d\bxi \\
&\lesssim_{a,\beta} 2^{m(d+|a|+|b|-|\alpha|-|\beta|)}.
\end{split}\]
Thus, for all $N \in \NN$,
\[
\sup_{\bx \in \RR^d} (1+2^m|\bx|)^N \max_{|\alpha|\leq N} |2^{-m|\alpha|} \partial_{\bx}^\alpha H_{m,h}(\bx)| \lesssim_{h,N} 2^{md}.
\]
Clearly analogous estimates hold with $H_m$ in place of $H_{m,h}$, as $H_m = 2^{md} H_0(2^m \cdot)$ and $H_0 \in \Sz(G)$. In other words, both $H_m$ and $H_{m,h}$ can be thought of as ``normalised'' Schwartz class functions at scale $2^{-m}$; in particular, as $2^m |\bx| \gtrsim \lambda^{\epsilon_0}$ on the support of $1-\chi_0(2^{-m}\lambda^{\epsilon_0} \bx)$, and moreover $2^m \lesssim \lambda^{1+\epsilon_2}$ by \eqref{eq:cond_lm}, one obtains that, for all $k,N \in \NN$ and $\alpha \in \NN^d$,
\[
(1+|\bx|)^k |\partial_{\bx}^\alpha [(1-\chi_0(2^{-m}\lambda^{\epsilon_0} \bx) H_{m,h}(\bx)]| \lesssim_{h,k,\alpha,N} \lambda^{-N},
\]
and analogous estimates for $H_m$. As $\opL$ is a second-order differential operator with polynomial coefficients, these uniform estimates imply the $L^2$-estimates \eqref{eq:large_Hmh}.
\end{proof}

In light of Proposition \ref{prp:introd_complexphase}, the proof of \eqref{eq:non-elliptic_region2} is reduced to proving that
\begin{equation}\label{eq:target_est}
\begin{aligned}
\|\chr_{\overline{B}(0,2^{2+\ell})} \cos(2^\ell \sqrt{\opL}) H_{m,h} \|_1 &\lesssim_{\vec\epsilon,h} \lambda^{(d-1)/2+\smallo(\vec\epsilon)},\\
\|\chr_{\overline{B}(0,2^{2+\ell})} 2^{\ell} \sqrt{\opL} \sin(2^\ell \sqrt{\opL}) H_{m,h} \|_1 &\lesssim_{\vec\epsilon,h} \lambda^{1+(d-1)/2+\smallo(\vec\epsilon)}
\end{aligned}
\end{equation}
for all $\lambda \geq 1$, all $\ell,m \in \NN$ satsifying \eqref{eq:cond_lm}, and all $h \in \NN$, where $H_{m,h}$ is given by \eqref{eq:initial_amplitude_complex}.
The proof of this estimate is given in Proposition \ref{prp:non-elliptic_est} below, after developing a suitable theory of Fourier integral operators with complex phase in Sections \ref{s:LSVparametrix} and \ref{s:sss}.

\section{The contribution by the elliptic region}\label{s:elliptic}

Here we briefly discuss how the estimate \eqref{eq:elliptic_region2} can be proved, by employing the ``microlocal parametrix'' for the sub-Riemannian wave equation discussed in \cite{MMNG}, together with the method of \cite{SSS} to derive $L^1$-estimates for FIOs. In the course of the proof, we will make the choice of the large parameter $\gamma$ appearing in \eqref{eq:elliptic}, which determines the splitting between elliptic and non-elliptic regions.

Since $\lambda \approx 2^j$ by \eqref{eq:cond_ell_j}, the estimate \eqref{eq:elliptic_region2} reduces to proving the following bound.

\begin{prp}\label{prp:elliptic_region}
For all $j \in \NN$,
\begin{equation}\label{eq:elliptic_red1}
\begin{aligned}
\|\chr_{\overline{B}(0,4)} \cos(\sqrt{\opL}) \tilde H_j \|_1 &\lesssim 2^{j(d-1)/2}, \\
\|\chr_{\overline{B}(0,4)} \sqrt{\opL} \sin(\sqrt{\opL}) \tilde H_j \|_1 &\lesssim 2^{j(1+(d-1)/2)},
\end{aligned}
\end{equation}
where $\tilde H_j$ is as in \eqref{eq:initial_amplitude_elliptic}.
\end{prp}
\begin{proof}
By using the Cauchy--Schwarz inequality, the $L^2$-boundedness of $\cos(\sqrt{\opL})$ and $\sin(\sqrt{\opL})/\sqrt{\opL}$, and the fact that $\tilde H_j, \opL \tilde H_j \in \Sz(\RR^d)$, it is clear that the norms in the left-hand sides of \eqref{eq:elliptic_red1} are finite for any $j \in \NN$. So the question here is about large values of $j$, and we may certainly assume that $j \in \Npos$.

\smallskip

Recall that the symbol of $\opL$ at the origin is $|\xi|^2$, so it does not vanish on the cone $\Gamma_\alpha = \{ \bxi \neq 0 \tc |\mu| \leq 4\alpha|\xi| \}$, for any given $\alpha>0$.
Let now $P$ be the (Euclidean) Fourier multiplier operator with smooth symbol
\[
\chi_0(|\mu|/(4\alpha|\xi|)) (1-\chi_0(|\xi|)),
\]
which is supported in $\Gamma_\alpha$. By applying \cite[Theorem 4.13]{MMNG} to the sub-Laplacian $\opL$ and the cone $\Gamma_\alpha$, we can then find a bounded open neighbourhood $U_0$ of the origin and a time $T_0 > 0$ (depending on $\alpha$), and a smooth phase function
\[
w = w(t,\bx,\bxi) : (-T_0,T_0) \times U_0 \times \dot\RR^d \to \RR,
\]
which is $1$-homogeneous in $\bxi$ with
\begin{equation}\label{eq:phase_ell_zerot}
w(0,\bx,\bxi) = \bx \cdot \bxi,
\end{equation}
and such that the following holds: for any open neighbourhood $U$ of the origin with closure contained in $U_0$, we can find a time $T \in (0,T_0)$ such that, if $\chi_U$ is a smooth cutoff supported in $U$ (thought of as a multiplication operator), then we can write, for $|t| \leq T$,
\begin{equation}\label{eq:elliptic_FIO_rep}
\cos(t\sqrt{\opL}) \chi_U P \chi_U = \frac{Q_t+Q_{-t}}{2} + R_t,
\end{equation}
where $R_t$ is a smoothing operator supported in $U_0 \times U_0$, while $Q_t$ is an oscillatory integral operator with integral kernel
\[
Q_t(\bx,\by) = \int_{\RR^d} e^{i(w(t,\bx,\bxi)-\by\cdot\bxi)} \, q(t,\bx,\by,\bxi) \,d\bxi
\]
for a suitable smooth amplitude $q$, which is a classical symbol of order $0$ with respect to the frequency variable $\bxi$ and has compact $(\bx,\by)$-support contained in $U_0 \times U_0$; up to changing the smoothing term $R_t$, we may assume that $q$ vanishes when $\bxi$ is in a neighbourhood of $0$. By \eqref{eq:phase_ell_zerot} the mixed hessian $\partial_{\bxi} \nabla_{\bx} w$ is nondegenerate at $t=0$; up to shrinking $T_0$ and $U_0$, we may assume that
\begin{equation}\label{eq:nondeg_hessian_elliptic_phase}
\left|\det \partial_{\bxi} \nabla_{\bx} w\right| \simeq 1
\end{equation}
for $|t| < T_0$ and $\bx \in U_0$. By differentiating \eqref{eq:elliptic_FIO_rep} with respect to $t$, we also obtain that
\begin{equation}\label{eq:elliptic_FIO_rep_sin}
\sqrt{\opL} \sin(t\sqrt{\opL}) \chi_U P \chi_U = \frac{\mathring{Q}_t-\mathring{Q}_{-t}}{2} + \mathring{R}_t,
\end{equation}
where again $\mathring{R}_t$ is a smoothing operator supported in $U_0 \times U_0$, while $\mathring{Q}_t$ is an oscillatory integral operator with integral kernel
\[
\mathring{Q}_t(\bx,\by) = \int_{\RR^d} e^{i(w(t,\bx,\bxi)-\by\cdot\bxi)} \, \mathring{q}(t,\bx,\by,\bxi) \,d\bxi,
\]
where the amplitude
\[
\mathring{q}(t,\bx,\by,\bxi) \defeq (\partial_t q)(t,\bx,\by,\bxi) + i (\partial_t w)(t,\bx,\bxi) q(t,\bx,\by,\bxi).
\]
is smooth, has order $1$ in $\bxi$ and compact $(\bx,\by)$-support contained in $U_0 \times U_0$.

\smallskip

In the above construction, we now choose a cutoff $\chi_U$ which is identically one on a smaller neighbourhood $V$ of the origin. Moreover, by scaling the whole construction using automorphic dilations, we may actually assume that $T = 1$; however, dilations will also change the neighbourhoods $U$ and $V$, as well as the aperture $\alpha$ of the cone. Nevertheless, the parameter $\gamma \in \NN$ determining the splitting between elliptic and non-elliptic regions (see \eqref{eq:elliptic}) can be chosen sufficiently large that $2^{1-\gamma} \leq \alpha$.

\smallskip

As the symbol of $P$ is identically one on the region
\[
\{ \bxi \tc |\mu| \leq \alpha|\xi|, \ |\xi| \geq 1 \},
\]
while the Fourier support of $\tilde H_j$ is contained in 
\[
\{ \bxi \tc |\mu| \leq 2^{1-\gamma}|\xi|, \  2^{j-1} \leq |\xi| \leq 2^{j+1} \},
\]
(see \eqref{eq:initial_amplitude_elliptic}), we conclude that, with this choice of $\gamma$,
\begin{equation}\label{eq:PHjHj}
P \tilde H_j = \tilde H_j \qquad \forall j \in \Npos.
\end{equation}

Notice now that $\tilde H_j = 2^{dj} \tilde H_0 (2^j \cdot)$, and that $\tilde H_0$ is a Schwartz function. As $1-\chi_U$ vanishes on the neighbourhood $V$ of the origin, by using the Schwartz bounds of $\tilde H_0$ we easily deduce that
\begin{equation}\label{eq:rem_bds_Hj}
\varrho((1-\chi_U) \tilde H_j) \lesssim_{\varrho,N} 2^{-jN} \qquad \forall N \in \NN
\end{equation}
for any continuous seminorm $\varrho$ on $\Sz(\RR^d)$. So, thanks to \eqref{eq:PHjHj}, we can write
\[
\tilde H_j = \chi_U P \chi_U \tilde H_j + \tilde R_j, \qquad \tilde R_j = (1-\chi_U) \tilde H_j + \chi_U P (1-\chi_U) \tilde H_j;
\]
moreover, thanks to \eqref{eq:rem_bds_Hj} and the $L^2$-boundedness and translation-invariance of the Fourier multiplier operator $P$,
\[
\varrho(\tilde R_j) \lesssim_{\varrho,N} 2^{-jN} \qquad \forall N \in \NN
\]
for any continuous seminorm $\varrho$ on $\Sz(\RR^d)$, and in particular
\[
\|\tilde R_j\|_2 + \|\opL \tilde R_j\|_2 \lesssim_{N} 2^{-jN} \qquad \forall N \in \NN,
\]
as $\opL$ has polynomial coefficients. As $\cos(\sqrt{\opL})$ and $\sin(\sqrt{\opL})/\sqrt{\opL}$ are $L^2$-bounded, by the Cauchy--Schwarz inequality we deduce that
\[
\|\chr_{\overline{B}(0,4)} \cos(\sqrt{\opL}) \tilde R_j\|_1 + \|\chr_{\overline{B}(0,4)} \sqrt{\opL} \sin(\sqrt{\opL}) \tilde R_j\|_1 \lesssim_N 2^{-jN} \qquad \forall N \in \NN,
\]
so this contribution is negligible, and \eqref{eq:elliptic_red1} reduces to proving
\[
\begin{aligned}
\|\cos(\sqrt{\opL}) \chi_U P \chi_U \tilde H_j \|_1 &\lesssim 2^{j(d-1)/2},\\
\|\sqrt{\opL} \sin(\sqrt{\opL}) \chi_U P \chi_U \tilde H_j \|_1 &\lesssim 2^{j(1+(d-1)/2)}.\\
\end{aligned}
\]

Now we can apply the representations \eqref{eq:elliptic_FIO_rep} and \eqref{eq:elliptic_FIO_rep_sin} of $\cos(\sqrt{\opL}) \chi_U P \chi_U$ and $\sqrt{\opL} \sin(\sqrt{\opL}) \chi_U P \chi_U$. In fact, due to their compact supports, the contributions of the smoothing terms are easily estimated, as
\[
\| R_1 \tilde H_j\|_1  + \| \mathring{R}_1 \tilde H_j\|_1\lesssim \|\tilde H_j\|_1 \simeq 1,
\]
so it only remains to check that
\[
\| Q_{\pm 1} \tilde H_j\|_1 \lesssim 2^{j(d-1)/2}, \qquad \| \mathring{Q}_{\pm 1} \tilde H_j\|_1 \lesssim 2^{1+j(d-1)/2}.
\]

As $\tilde H_j = 2^{dj} \tilde H_0 (2^j \cdot)$, we can now apply \cite[Lemma 5.1]{MMNG} with $u = \tilde H_0$ and $\lambda=2^j$, and then scale by $2^{-j}$ the integration variable $\bxi$, to obtain the development
\[\begin{split}
Q_{\pm 1} \tilde H_j(\bx) 
&= \sum_{|\alpha| \leq N} c_\alpha 2^{-j|\alpha|} \int_{\RR^d} e^{iw(\pm 1,\bx,\bxi)} \partial_{\by}^\alpha q(\pm 1,\bx,0,\bxi) \partial_{\bxi}^\alpha \eta(2^{-j} \bxi) \,d\bxi \\
&\qquad+ 2^{j(d-(N+1))} R_{N,j}(\bx),
\end{split}\]
where
\[
\eta(\bxi) \defeq \chi_1(|\xi|) \chi_0(2^{\gamma}|\mu|)
\]
is the Fourier transform of $\tilde H_0$ (see \eqref{eq:initial_amplitude_elliptic}), while $\supp R_{N,j} \subseteq U_0$ and $\|R_{N,j}\|_\infty \lesssim_N 1$ uniformly in $j$, so also $\| R_{N,j} \|_1 \lesssim_N 1$, and therefore the contribution of the remainder term $2^{j(d-(N+1))} R_{N,j}$ can be neglected by taking $N$ sufficiently large. 
A straightforward variation of \cite[Lemma 5.1]{MMNG}, allowing for amplitudes of positive $\bxi$-order, also yields the development
\[\begin{split}
\mathring{Q}_{\pm 1} \tilde H_j(\bx) 
&= \sum_{|\alpha| \leq N} c_\alpha 2^{-j|\alpha|} \int_{\RR^d} e^{iw(\pm 1,\bx,\bxi)} \partial_{\by}^\alpha \mathring{q}(\pm 1,\bx,0,\bxi) \partial_{\bxi}^\alpha \eta(2^{-j} \bxi) \,d\bxi \\
&\qquad+ 2^{j(1+d-(N+1))} \mathring{R}_{N,j}(\bx),
\end{split}\]
where again $\| \mathring{R}_{N,j} \|_1 \lesssim_N 1$, and a similar argument allows us to neglect the remainder term.
In conclusion, it remains to show that
\begin{align*}
\int_{U_0} \left|\int_{\RR^d} e^{iw(\pm 1,\bx,\bxi)} \partial_{\by}^\alpha q(\pm 1,\bx,0,\bxi) \partial_{\bxi}^\alpha \eta(2^{-j} \bxi) \,d\bxi\right| \,d\bx &\lesssim_\alpha 2^{j(d-1)/2},\\
\int_{U_0} \left|\int_{\RR^d} e^{iw(\pm 1,\bx,\bxi)} \partial_{\by}^\alpha \mathring{q}(\pm 1,\bx,0,\bxi) \partial_{\bxi}^\alpha \eta(2^{-j} \bxi) \,d\bxi\right| \,d\bx &\lesssim_\alpha 2^{j(1+(d-1)/2)}
\end{align*}
for all $\alpha \in \NN^d$ and $j \in \Npos$. Thanks to the nondegeneracy \eqref{eq:nondeg_hessian_elliptic_phase} of the mixed hessian of the phase function $w$, the fact that the amplitudes $q$ and $\mathring{q}$ are classical symbols of orders $0$ and $1$ with compact $\bx$-support, together with the fact that $|\bxi| \simeq 1$ for $\bxi \in \supp \eta$, these Fourier integral operators are amenable to the classical arguments of \cite{SSS} (see, e.g., \cite[\S IX.4]{St} or \cite[Proposition 6.25]{S}), thus yielding the desired estimates.
\end{proof}

\section{A parametrix for the non-elliptic region}\label{s:LSVparametrix}

This and the next sections will be devoted to the proof of the estimate \eqref{eq:target_est} for the non-elliptic region. Recall that therein, due to the automorphic scaling discussed in Section \ref{ss:scaling}, the time parameter of interest is no longer $1$, but $2^\ell$, which may become arbitrarily large (see \eqref{eq:cond_lm}). On the other hand, the scaling allows us to work in the ``frequency region'' $|\xi| \simeq |\mu| \simeq 2^m$, see also Figure \ref{fig:freq_scaling}.

We shall now look for approximate representations of
\[
\cos(t\sqrt{\opL}) H(\bx), \qquad \sqrt{\opL} \sin(t\sqrt{\opL}) H(\bx),
\]
where $t \in \RR$ and $H$ is one of the functions $H_{m,h}$ of \eqref{eq:initial_amplitude_complex}, in terms of oscillatory integrals (which we shall refer to as \emph{FIO kernels}) of the form
\begin{equation}\label{eq:fio_Iq}
I[q](t,\bx) \defeq (2\pi)^{-d} \int_{\RR^d} e^{i\phi(t,\bx,\bxi)} \, q(t,\bxi) \, \Den_\phi(t,\bx,\bxi) \,d\bxi,
\end{equation}
where the terms $q$ are appropriate $\bx$-independent ``amplitudes'' depending on $H$.

In the oscillatory integrals \eqref{eq:fio_Iq}, partly following \cite{LSV,SV}, we shall use the complex-valued phase function
\begin{equation}\label{eq:phase}
\begin{split}
\phi(t,\bx,\bxi) 
&\defeq (\bx-\bx^t(\bxi))\cdot \bxi^t(\bxi) + \frac{i}{4} \langle |J_\mu| (x-x^t(\bxi)), x-x^t(\bxi) \rangle \\
&= (x-x^t(\bxi))\cdot \xi^t(\bxi) + (u-u^t(\bxi)) \cdot \mu + \frac{i}{4} \langle |J_\mu| (x-x^t(\bxi)), x-x^t(\bxi) \rangle
\end{split}
\end{equation}
defined in terms of the Hamiltonian flow \eqref{eq:flow_def} associated with the sub-Laplacian $\opL$.
The dot notation and the bracket notation for the inner products in \eqref{eq:phase} are actually interchangeable, as they are applied to real vectors (see Section \ref{ss:notation}); for later computations, it is anyway convenient to use two different symbols here.
By Proposition \ref{prp:modJmu_analytic}, we know that $|J_\mu|$ is a $1$-homogeneous Lipschitz-continuous function of $\mu \in \RR^{d_2}$, which is real-analytic on a homogeneous Zariski-open set $\Omega \subseteq \dot\RR^{d_2}$; as a consequence, the phase function $\phi$ is real-analytic on $\RR \times \dot \RR^{d_1} \times \Omega$.

Moreover, the density $\Den_\phi$ in \eqref{eq:fio_Iq} is given by	
\begin{equation}\label{eq:den}
\Den_\phi(t,\bx,\bxi) \defeq \sqrt{\det \Phi(t,\bx,\bxi)}, \qquad \Phi(t,\bx,\bxi) \defeq \partial_{\bxi} \nabla_{\bx} \phi(t,\bx,\bxi),
\end{equation}
where a continuous determination of the square root of $\det \Phi(t,\bx,\bxi)$ is chosen which equals $1$ when $t=0$; as we shall see, $\det \Phi(t,\bx,\bxi)$ never vanishes when $\bxi \in \dot\RR^{d_1} \times \Omega$, is independent of $\bx$, and equals $1$ at $t=0$, so $\Den_\phi$ is well defined.

\begin{rem}
The first step in many of the FIO parametrix constructions for the wave equation in the literature, including \cite{LSV}, is a reduction to the half-wave equation, which is made possible by the fact that the square root of an elliptic Laplacian on a compact manifold is a classical first-order pseudodifferential operator \cite{See}. However, for our nonelliptic sub-Laplacian $\opL$, the usual construction of fractional powers via pseudodifferential calculus breaks down, as one would be led to considering symbols in ``exotic'' symbol classes of type $(1/2,1/2)$, for which a full asymptotic calculus fails. To address the latter problem, a considerable amount of literature has been devoted to the development of pseudodifferential calculi adapted to nonelliptic sub-Laplacians in various settings (see, e.g., \cite{BG,CGGP,Po,SY,T_noncomm,vErp} and references therein). However, here we shall circumvent this issue, by directly constructing a parametrix for the second-order wave operator $\partial_t^2+\opL$, much as in \cite{SV}. As a consequence, our FIO representation directly applies to $\cos(t\sqrt{\opL})$ and $\sqrt{\opL} \sin(t\sqrt{\opL})$, rather than $\exp(it\sqrt{\opL})$.
\end{rem}

\subsection{Basic properties of the phase function}

For brevity, we shall sometimes omit variables and write, e.g., $\phi$ in place of $\phi(t,\bx,\bxi)$. Moreover, for any function $p = p(t,\bx,\bxi)$, we shall write
\begin{equation}\label{eq:def_underlining}
\underline{p} \defeq p|_{\bx=\bx^t} = p(t,\bx^t(\bxi),\bxi)
\end{equation}
for the $\bx$-evaluation of the function along the geodesic flow.

\begin{lem}\label{lem:phder}
The following hold for $\bxi \in \dot\RR^{d_1} \times \Omega$.
\begin{enumerate}[label=(\roman*)]
\item\label{en:phder_grad} We have
\[
\partial_{\bxi} \phi = (\bx-\bx^t)^T \underline{\Phi} + \frac{i}{4} \sum_{j=1}^{d_2} \langle (\partial_{\mu_j} |J_\mu|) (x-x^t), x-x^t \rangle \left( \begin{array}{c} 0 \\\hline e_j \end{array} \right)^T,
\]
where $e_1,\dots,e_{d_2}$ is the standard basis of $\RR^{d_2}$,
and
\[
\underline{\partial_{\bxi} \phi} = 0, \qquad \underline{\phi} = 0, \qquad \underline{\nabla_{\bx} \phi} = \bxi^t .
\]
\item\label{en:phder_hess} We have
\[
\Phi = \left(\begin{array}{c|c}
\Phi_0 & * \\\hline
0 & I
\end{array}\right),
\]
where
\[
\Phi_0 = \Phi_0(t,\bxi) = \partial_{\xi} \xi^t - \frac{i}{2} |J_\mu| \partial_\xi x^t
\]
does not depend on $\bx$, and equals $I$ at $t=0$. 
As a consequence,
\[
\Den_\phi = \Den_\phi(t,\bxi) = \sqrt{\det \Phi_0} = \sqrt{\det \underline{\Phi}}.
\]
\item\label{en:phder_hess_nondeg} $\Phi$ and $\Phi_0$ are nondegenerate for all $\bxi \in \dot\RR^{d_1} \times \Omega$. More precisely, with the notation of Corollary \ref{cor:flow_alt},
\[
\Phi_0 = \exp(\theta J_{\bar\mu}) \exp(-i\theta|J_{\bar\mu}|) \left( I + i \theta 	\left(|J_{\bar\mu}|+i J_{\bar\mu} \right)\bar\xi \, \bar\xi^T \right);
\]
in particular
\begin{equation}\label{eq:det_2step}
\begin{gathered}
\det \Phi = \det\Phi_0 = \exp(-i\theta\tr|J_{\bar\mu}|) (1+i\theta \langle |J_{\bar\mu}| \bar\xi, \bar\xi \rangle ) \neq 0,\\
\Den_\phi = \exp(-i\theta\tr|J_{\bar\mu}|/2) \sqrt{1+i\theta \langle |J_{\bar\mu}| \bar\xi, \bar\xi \rangle}.
\end{gathered}
\end{equation}
Moreover, if $G$ is a group of Heisenberg type, then
\begin{equation}\label{eq:det_Htype}
\det \Phi = e^{-i\theta d_1}(1+i\theta), \qquad \Den_\phi = e^{-i\theta d_1/2} \sqrt{1+i\theta}. 
\end{equation}
\item\label{en:phder_tx} The functions $\partial_t \phi$, $\partial_t^2 \phi$, $\nabla^\opL \phi$ are polynomials of degree $1$ in $x$ and independent of $u$, while $\opL \phi$ is $\bx$-independent; here
\begin{equation}\label{eq:hgradient}
\nabla^\opL f \defeq (X_1 f,\dots,X_{d_1} f)
\end{equation}
is the \emph{horizontal gradient} associated with $\opL = -\sum_{j=1}^{d_1} X_j^2$.
\end{enumerate}
\end{lem}
\begin{proof}
\ref{en:phder_grad}. From \eqref{eq:phase} it is easily seen that
\[
\underline{\phi} = 0, \qquad \underline{\nabla_{\bx} \phi} = \bxi^t,
\]
and moreover
\[\begin{split}
\partial_{\bxi} \phi &= (\bx - \bx^t)^T \partial_{\bxi} \bxi^t - \frac{i}{2} (x-x^t)^T |J_\mu| \partial_{\bxi} x^t\\
 &\quad+ \frac{i}{4} \sum_{j=1}^{d_2} \langle (\partial_{\mu_j} |J_\mu|) (x-x^t), x-x^t \rangle \left( \begin{array}{c} 0 \\\hline e_j \end{array} \right)^T,
\end{split}\]
where we used the fact that $(\bxi^t)^T \partial_{\bxi} \bx^t  = 0$ by Lemma \ref{lem:symplectic}. Consequently,
\begin{align*}
\Phi &= \partial_{\bxi} \nabla_{\bx} \phi = \partial_{\bxi} \bxi^t - \frac{i}{2} \left( \begin{array}{c} |J_\mu| \partial_{\bxi} x^t \\\hline 0 \end{array} \right) + \frac{i}{2} \sum_{j=1}^{d_2} \left( \begin{array}{c} (\partial_{\mu_j} |J_\mu|) (x-x^t) \\\hline 0 \end{array}\right) \left( \begin{array}{c} 0 \\\hline e_j \end{array} \right)^T, \\
\underline{\Phi}  &= \partial_{\bxi} \bxi^t - \frac{i}{2} \left( \begin{array}{c} |J_\mu| \partial_{\bxi} x^t \\\hline 0 \end{array} \right),
\end{align*}
and therefore
\begin{align*}
\partial_{\bxi} \phi &= (\bx - \bx^t)^T \underline{\Phi} + \frac{i}{4} \sum_{j=1}^{d_2} \langle (\partial_{\mu_j} |J_\mu|) (x-x^t), x-x^t \rangle \left( \begin{array}{c} 0 \\\hline e_j \end{array} \right)^T,\\
\underline{\partial_{\bxi} \phi} &= 0,
\end{align*}
as desired.

\smallskip

\ref{en:phder_hess}. The above computations, together with the fact that $\mu^t = \mu$, also give the claimed expressions for $\Phi$ and $\Phi_0$. In addition, as $x^0 = 0$ and $\xi^0 = \xi$, one sees that $\Phi_0|_{t=0} = I$.

\smallskip

\ref{en:phder_hess_nondeg}. 
From Corollary \ref{cor:flow_alt} we know that
\begin{equation}\label{eq:ximux}
\begin{aligned}
\xi^t &= \cosh(\theta J_{\bar\mu}) \exp(\theta J_{\bar\mu}) \xi = \frac{1}{2} \left( \exp(2\theta J_{\bar\mu}) + I\right) \xi, \\
\frac{1}{2}|J_\mu| x^t &= |J_{\bar\mu}| \frac{\sinh(\theta J_{\bar\mu})}{J_{\bar\mu}} \exp(\theta J_{\bar\mu}) \xi = \frac{1}{2} |J_{\bar\mu}|\frac{\exp(2\theta J_{\bar\mu}) - I}{J_{\bar\mu}} \xi, 
\end{aligned}
\end{equation}
where $\theta = t|\mu|/(2|\xi|)$ and $\bar\mu=\mu/|\mu|$. So
\begin{equation}\label{eq:derxitheta}
\nabla_\xi \theta = -\theta |\xi|^{-1} \bar\xi,
\end{equation}
where $\bar\xi = \xi/|\xi|$, and
\begin{equation}\label{eq:Phi0_prelim}
\begin{split}
\Phi_0 
&= \exp(\theta J_{\bar\mu}) \left(\cosh(\theta J_{\bar\mu}) - i |J_{\bar\mu}| \frac{ \sinh(\theta J_{\bar\mu}) }{J_{\bar\mu}} \right)-  \theta \exp(2\theta J_{\bar\mu}) \left(J_{\bar\mu} -i |J_{\bar\mu}| \right)\bar\xi \, \bar\xi^T \\
&= \exp(\theta J_{\bar\mu}) \left(\cosh(\theta J_{\bar\mu}) - i |J_{\bar\mu}| \frac{ \sinh(\theta J_{\bar\mu}) }{J_{\bar\mu}} + i \theta \exp(\theta J_{\bar\mu}) \left(|J_{\bar\mu}|+i J_{\bar\mu} \right)\bar\xi \, \bar\xi^T \right) .
\end{split}
\end{equation}
Notice now that, since $J_{\bar\mu}^2 = -|J_{\bar\mu}|^2$, 
\begin{equation}\label{eq:coshcos}
\cosh(\theta J_{\bar\mu}) - i |J_{\bar\mu}| \frac{ \sinh(\theta J_{\bar\mu})}{J_{\bar\mu}} = \cos(\theta|J_{\bar\mu}|) - i \sin(\theta|J_{\bar\mu}|) = \exp(-i\theta|J_{\bar\mu}|) 
\end{equation}
is a unitary operator, hence invertible; moreover
\[
J_{\bar\mu} \left(|J_{\bar\mu}| +i J_{\bar\mu} \right) = -i|J_{\bar\mu}| \left(|J_{\bar\mu}|+i J_{\bar\mu} \right),
\]
whence
\begin{equation}\label{eq:fctnJ_2step}
F(J_{\bar\mu}) \left(|J_{\bar\mu}|+i J_{\bar\mu} \right) = F(-i|J_{\bar\mu}|) \left(|J_{\bar\mu}|+i J_{\bar\mu} \right) 
\end{equation}
for any polynomial $F$, hence for any function $F : \CC \to \CC$, and in particular
\begin{equation}\label{eq:change_exp}
 \exp(\theta J_{\bar\mu}) \left(|J_{\bar\mu}|+i J_{\bar\mu} \right) = \exp(-i\theta|J_{\bar\mu}|) \left(|J_{\bar\mu}|+i J_{\bar\mu} \right) .
\end{equation}
By plugging the identities \eqref{eq:coshcos} and \eqref{eq:change_exp} into \eqref{eq:Phi0_prelim}, we obtain 
\[
\Phi_0 = \exp(\theta J_{\bar\mu}) \exp(-i\theta|J_{\bar\mu}|) \left( I + i \theta \left(|J_{\bar\mu}|+i J_{\bar\mu} \right)\bar\xi \, \bar\xi^T \right),
\]
and the corresponding formulas \eqref{eq:det_2step} for $\det\Phi_0$ and $\Den_\phi$ follow.

Finally, when $G$ is Heisenberg-type, we have $|J_{\bar\mu}| = I$, so the formulas \eqref{eq:det_2step} reduce to \eqref{eq:det_Htype}.

\smallskip

\ref{en:phder_tx}.
This quickly follows by differentiating \eqref{eq:phase} and using the fact that
\begin{equation}\label{eq:Xfields}
X_j = \partial_{x_j} + \frac{1}{2} [x,e_j] \cdot \nabla_u
\end{equation}
for $j=1,\dots,d_1$.
\end{proof}

\subsection{Wave front and Hamiltonian flow}

From the general theory of oscillatory integrals, we know that the wave front set of FIO kernels of the form \eqref{eq:fio_Iq} is related to the set of critical points of the phase function $\phi$. More precisely, if we define the \emph{stationary set} and the \emph{wave front} 
\begin{align*}
\Sigma_{\phi,t} &\defeq \{ (\bx,\bxi) \in G \times (\dot\RR^{d_1} \times \Omega) \tc \partial_{\bxi} \phi(t,\bx,\bxi) = 0\},\\
\Lambda_{\phi,t} &\defeq \{ (\bx,\nabla_{\bx} \phi(t,\bx,\bxi)) \tc (\bx,\bxi) \in \Lambda_{\phi,t} \}
\end{align*}
of the phase function $\phi$ at time $t$, then from \cite[Theorem 8.1.9]{Ho1} it follows that, under suitable assumptions on the amplitude $q$, the wave front set of the distribution $I[q](t,\cdot)$ on $G$ is contained in $\Lambda_{\phi,t}$.

As we aim to use FIO kernels to give an approximate representation of the convolution kernel of the wave propagator, it is reasonable to expect that the wave front $\Lambda_{\phi,t}$ can be described in terms of the Hamiltonian flow $(\bx^t,\bxi^t)$ for $\Ham$; we now show that this is indeed the case. The following result should be compared with \cite[Lemma 1.2]{LSV}; we point out that the proof in \cite{LSV} does not directly apply here, as our $\Im \phi$ is $u$-independent, and therefore it also vanishes at points $\bx$ that are arbitrarily far from $\bx^t$.

\begin{prp}\label{prp:wf}
Let $t \in \RR$.
\begin{enumerate}[label=(\roman*)]
\item\label{en:wf_sing} For all $\bx \in G$ and $\bxi \in \dot\RR^{d_1} \times \Omega$,
\[
\partial_{\bxi} \phi(t,\bx,\bxi) = 0 \quad\iff\quad \bx = \bx^t(\bxi).
\]
\item\label{en:wf_wf} We have
\[
\Lambda_{\phi,t} = \{ (\bx^t(\bxi), \bxi^t(\bxi)) \tc \bxi \in \dot\RR^{d_1} \times \Omega \}.
\]
\end{enumerate}
\end{prp}
\begin{proof}
\ref{en:wf_sing}. We already know from Lemma \ref{lem:phder}\ref{en:phder_grad} that $\partial_{\bxi} \phi|_{\bx = \bx^t} = 0$. Conversely, if $(\bx,\bxi) \in \dot\RR^{d_1} \times \Omega$ is such that $\partial_{\bxi}\phi(t,\bx,\bxi) = 0$, then by $1$-homogeneity we also have $\phi(t,\bx,\bxi) = 0$, and in particular, by \eqref{eq:phase},
\[
\langle |J_\mu| (x-x^t(\bxi)), x-x^t(\bxi) \rangle = 4\Im \phi(t,\bx,\bxi) = 0.
\]
In other words, the function $(\bx,\bxi) \mapsto \langle |J_\mu| x, x \rangle$ vanishes at the point $(\bx-\bx^t(\bxi),\bxi)$; as this function is smooth and nonnegative on $G \times (\dot\RR^{d_1} \times \Omega)$, its zeros are at least of second order, and therefore $\langle (\partial_{\mu_j} |J_\mu|) (x-x^t(\bxi)), x-x^t(\bxi) \rangle = 0$ for $j=1,\dots,d_2$. Thus, from the expression for $\partial_{\bxi}\phi(t,\bx,\bxi)$ given in Lemma \ref{lem:phder}\ref{en:phder_grad} we conclude that $(\bx-\bx^t(\bxi))^T \underline\Phi(t,\bxi) = 0$. On the other hand, $\underline{\Phi}(t,\bxi)$ is nondegenerate by Lemma \ref{lem:phder}\ref{en:phder_hess}, thus $\bx = \bx^t(\bxi)$, as claimed.

\smallskip

\ref{en:wf_wf}. From part \ref{en:wf_sing} it follows that $\Sigma_{\phi,t} = \{ (\bx^t(\bxi),\bxi) \tc \bxi \in \dot\RR^{d_1} \times \Omega \}$; since we know that $\nabla_{\bx} \phi(t,\bx^t(\bxi),\bxi) = \bxi^t(\bxi)$ by Lemma \ref{lem:phder}\ref{en:phder_grad}, the conclusion follows.
\end{proof}

\begin{rem}
Much as in \cite{LSV}, the proof of Proposition \ref{prp:wf} makes fundamental use of the nondegeneracy at $\bx = \bx^t$ of the mixed hessian $\Phi$ of the phase function $\phi$, which in turn relies crucially on the presence of a suitable imaginary part of the phase function. We remark that the above characterisation of $\Lambda_{\phi,t}$ would in general not hold without the imaginary term in the phase, due to the presence of caustics; this so happens already for Heisenberg groups.
\end{rem}

In our approach, we are strongly making use of the left-invariance of the sub-Laplacian $\opL$ on $G$.
Ignoring this left-invariance, we could as well have tried to construct a parametrix in a similar way as in \cite{LSV} by means of an FIO with a suitable complex phase function $\phi_*(t,\bx,\by,\bxi)$, with $(\bx,\by)\in G \times G$. As in Section \ref{ss:flowy}, denote by $(\bx^t_*(\by,\bxi),\bxi_*^t(\by,\bxi))$ the Hamiltonian flow generated by $\Ham$ with initial datum $(\by,\bxi)$.
By comparison with \eqref{eq:phase}, the natural candidate for the phase function of such an FIO would be
\begin{equation}\label{eq:phasey}
\phi_*(t, \bx,\by,\bxi) 
\defeq (\bx-\bx_*^t(\by,\bxi))\cdot \bxi_*^t(\by,\bxi) + \frac{i}{4} \langle |J_\mu| (x-x_*^t(\by,\bxi)), x-x_*^t(\by,\bxi) \rangle.
\end{equation}
Notice that $\phi_*$ is smooth on $\RR \times G \times \Xi$, where
\begin{align*}
\Xi &\defeq \bigcup_{\by \in G} \{ \by \} \times \Xi_{\by},\\
\Xi_{\by} &\defeq \dltr_{\by}(\dot\RR^{d_1} \times \Omega) = \{ (\xi - J_\mu y/2, \mu) \tc \bxi \in \dot\RR^{d_1} \times \Omega\}.
\end{align*}

From the covariance of the Hamiltonian flow under left translations discussed in Section \ref{ss:flowy}, one easily deduces a relation between $\phi_*$ and $\phi$.

\begin{prp}
For all $t \in \RR$ and $(\bx,\by,\bxi) \in G \times \Xi$,
\begin{equation}\label{eq:phase_ltr}
\phi_*(t, \bx,\by,\bxi)=\phi(t,\ltr_{\by}^{-1} \bx, \dltr_{\by}^{-1}\bxi).
\end{equation}
\end{prp}
\begin{proof}
We know from Section \ref{ss:flowy} that $\ltr_{\by}$ is an affine-linear map on $\RR^d$, whose differential $\ltr_{\by}' = D\ltr_{\by}(\bx)$ is independent of $\bx$, thus
\[
\ltr_{\by} (\bx_1) - \ltr_{\by} (\bx_2) = \ltr_{\by}' (\bx_1 - \bx_2) \qquad \forall \bx_1,\bx_2 \in \RR^d.
\]
In particular, if we set $\bx_0 \defeq \ltr_{\by}^{-1} \bx$ and $\bxi_0 \defeq \dltr_{\by}^{-1} \bxi$, then, by \eqref{eq:flowy}, we see that
\begin{equation}\label{eq:diff_aff}
\bx - \bx_*^t(\by,\bxi) = \ltr_{\by} (\bx_0) - \ltr_{\by} (\bx^t(\bxi_0)) = \ltr_{\by}' (\bx_0 - \bx^t(\bxi_0))
\end{equation}
and therefore
\[
(\bx-\bx_*^t(\by,\bxi)) \cdot \bxi_*^t(\by,\bxi) = (\ltr_{\by}' (\bx_0 - \bx^t(\bxi_0))) \cdot (\dltr_{\by} \bxi^t(\bxi_0)) = (\bx_0 - \bx^t(\bxi_0)) \cdot \bxi^t(\bxi_0),
\]
where we used that, by definition, $\dltr_{\by} = ((\ltr_{\by}')^T)^{-1}$. In light of \eqref{eq:phase} and \eqref{eq:phasey}, this proves that $\Re \phi_*(t,\bx,\by,\bxi) = \Re \phi(t,\bx_0,\bxi_0)$.

Now, by projecting \eqref{eq:diff_aff} onto the $x$-component, and taking the formula \eqref{eq:diff_ltr} for $\ltr_{\by}'$ into consideration, we also see that
\[
x - x_*^t(\by,\bxi) = x_0 - x^t(\bxi_0),
\]
and moreover $\mu_0 = \mu$ by \eqref{eq:dltr},
whence $\Im \phi_*(t,\bx,\by,\bxi) = \Im \phi(t,\bx_0,\bxi_0)$.
\end{proof}

From Proposition \ref{prp:wf} we then deduce an analogous result for $\phi_*$.

\begin{cor}
Let $t \in \RR$.
\begin{enumerate}[label=(\roman*)]
\item\label{en:cwf_sing} For all $\bx,\by \in G$ and $\bxi \in \Xi_{\by}$,
\[
\partial_{\bxi} \phi_*(t,\bx,\by,\bxi) = 0 \quad\iff\quad \bx = \bx_*^t(\by,\bxi).
\]
\item\label{en:cwf_wf} If
\begin{align*}
\Sigma_{\phi_*,t} &\defeq \{ (\bx,\by,\bxi) \in G \times \Xi \tc \partial_{\bxi}\phi_*(t,\bx,\by,\bxi) = 0\},\\
\Lambda_{\phi_*,t} &\defeq \{ ((\bx,\nabla_{\bx} \phi_*(t,\bx,\by,\bxi)),(\by,\nabla_{\by} \phi_*(t,\bx,\by,\bxi)) \tc (\bx,\by,\bxi) \in \Sigma_{\phi_*,t}\},
\end{align*}
then
\[
\Lambda_{\phi_*,t} = \{ ((\bx_*^t(\by,\bxi), \bxi_*^t(\by,\bxi)),(\by,-\bxi)) \tc (\by,\bxi) \in \Xi \}.
\]
In other words, the wave front $\Lambda_{\phi_*,t}$ of the phase function $\phi_*$ at time $t$ is the Lagrangian manifold associated with the canonical transformation $(\by,\bxi) \mapsto (\bx^t_*(\by,\bxi),\bxi^t_*(\by,\bxi))$.
\end{enumerate}
\end{cor}
\begin{proof}
Thanks to the relation \eqref{eq:phase_ltr}, part \ref{en:cwf_sing} directly follows by straightforward changes of variables from Proposition \ref{prp:wf}\ref{en:wf_sing}. Part \ref{en:cwf_wf} then follows from part \ref{en:cwf_sing}, by using the identities
\[
\nabla_{\bx} \phi_*|_{\bx = \bx^t_*(\by,\bxi)} = \bxi^t_*(\by,\bxi), \qquad \nabla_{\by} \phi_*|_{\bx = \bx^t_*(\by,\bxi)} = -\bxi.
\]
The first identity is immediate from the definition of $\phi_*$. As for the second one, much as in \cite{LSV}, it is enough to differentiate the identity
\[
\phi_*|_{\bx = \bx^t_*(\by,\bxi)} = 0
\]
(which follows from part \ref{en:cwf_sing} and homogeneity) with respect to $\by$, and exploit the identity \eqref{eq:symplecticy_xi}.
\end{proof}

\subsection{The action of the wave operator on FIO kernels}

We now describe the action of the wave operator $\partial_t^2 + \opL$ on FIO kernels as in \eqref{eq:fio_Iq}.
To this purpose, we introduce the notation
\[
I[p](t,\bx) \defeq (2\pi)^{-d} \int_{\RR^d} e^{i\phi(t,\bx,\bxi)} \, p(t,\bx,\bxi) \, \Den_\phi(t,\bx,\bxi) \,d\bxi
\]
for the FIO kernel with phase $\phi$ and amplitude $p$, where the amplitude may also depend on $\bx$. For brevity, sometimes we may omit variables and just write
\[
I[p] = (2\pi)^{-d} \int_{\RR^d} e^{i\phi} \, p \, \Den_\phi \,d\bxi.
\]

\begin{lem}\label{lem:dal}
The following hold.
\begin{enumerate}[label=(\roman*)]
\item\label{en:dal_formula} For any smooth symbol $q = q(t,\bxi)$ with $\text{$\bxi$-}\supp q \Subset \dot\RR^{d_1} \times \Omega$,
\begin{align*}
(\partial_t^2 + \opL) I[q] &= \sum_{\substack{j,k \geq 0 \\ j+k \leq 2}} i^{-k} I[F_{kj} \partial_t^j q],\\
\partial_t I[q] &= \sum_{\substack{j,k \geq 0 \\ j+k \leq 1}} 2^{j-1} i^{-k} I[F_{k(j+1)} \partial_t^j q],
\end{align*}
where
\begin{equation}\label{eq:Fkj}
\begin{aligned}
F_{20} &= (\partial_t \phi)^2 - \nabla^\opL \phi \cdot \nabla^\opL \phi,\\
-F_{11} &= 2\partial_t \phi, \\
-F_{10} &= (\partial_t^2 + \opL)\phi + 2 (\partial_t \phi) (\Den_\phi^{-1} \partial_t \Den_\phi) - 2 (\nabla^\opL \phi) \cdot (\Den_\phi^{-1} \nabla^\opL \Den_\phi),\\
F_{02} &=1,\\
F_{01} &= 2 \Den_\phi^{-1} \partial_t \Den_\phi,\\
F_{00} &= \Den_\phi^{-1} (\partial_t^2 + \opL) \Den_\phi,
\end{aligned}
\end{equation}
and $\nabla^\opL$ is the horizontal gradient defined in \eqref{eq:hgradient}.
\item\label{en:dal_Fkj} Each of the terms $F_{kj}$ in \eqref{eq:Fkj} is:
\begin{itemize}
\item smooth as a function of $(t,\bx,\bxi) \in \RR \times \RR^d \times (\dot\RR^{d_1} \times \Omega)$;
\item $k$-homogeneous in $\bxi$;
\item a polynomial of degree $k$ in $x$;
\item independent of $u$.
\end{itemize}
\end{enumerate}
\end{lem}

\begin{rem}
Since $\Den_\phi$ in \eqref{eq:det_2step} is independent of $\bx$ (see Lemma \ref{lem:phder}\ref{en:phder_hess}), some of the  formulas in \eqref{eq:Fkj} could be slightly simplified: for example, 
$2 (\nabla^\opL \phi) \cdot (\Den_\phi^{-1} \nabla^\opL \Den_\phi)=0$ in the formula for $F_{10}$, and moreover $F_{00} = \Den_\phi^{-1} \partial_t^2 \Den_\phi$.
\end{rem}

\begin{proof}
\ref{en:dal_formula}. By differentiating under the integral sign, we are reduced to proving that
\begin{align*}
e^{-i\phi} \Den_\phi^{-1} (\partial_t^2 + \opL) (e^{i\phi} \, q \, \Den_\phi) 
&= \sum_{\substack{j,k \geq 0 \\ j+k \leq 2}} i^{-k} F_{kj} \partial_t^j q,\\
e^{-i\phi} \Den_\phi^{-1} \partial_t (e^{i\phi} \, q \, \Den_\phi) 
&= \sum_{\substack{j,k \geq 0 \\ j+k \leq 1}} 2^{j-1} i^{-k} F_{k(j+1)} \partial_t^j q.
\end{align*}

Now, by the Leibniz rule,
\[
\partial_t (e^{i\phi} \, q \, \Den_\phi) = (\partial_t e^{i\phi}) \, q \, \Den_\phi + e^{i\phi} \, (\partial_t q) \, \Den_\phi + e^{i\phi} \, q \, (\partial_t \Den_\phi)
\]
and
\[\begin{split}
(\partial_t^2 + \opL) (e^{i\phi} \, q \, \Den_\phi) 
&= [(\partial_t^2 + \opL) e^{i\phi}] \, q \, \Den_\phi \\
&+ e^{i\phi} \, [(\partial_t^2 + \opL) q] \, \Den_\phi \\
&+ e^{i\phi} \,  q \, [(\partial_t^2 + \opL) \Den_\phi] \\
&+ 2[(\partial_t e^{i\phi}) (\partial_t q) - (\nabla^\opL e^{i\phi}) \cdot (\nabla^\opL q)]  \, \Den_\phi \\
&+ 2[(\partial_t e^{i\phi}) (\partial_t \Den_\phi) - (\nabla^\opL e^{i\phi}) \cdot (\nabla^\opL \Den_\phi)]  \, q \\
&+ 2e^{i\phi} [(\partial_t q) (\partial_t \Den_\phi) - (\nabla^\opL q) \cdot (\nabla^\opL \Den_\phi)]  \, q,
\end{split}\]
and moreover, by the chain rule,
\begin{gather*}
\partial_t e^{i\phi} = i e^{i\phi} \partial_t \phi , \qquad \nabla^\opL e^{i\phi} = i e^{i\phi}  \nabla^\opL \phi,\\
 (\partial_t^2 + \opL) e^{i\phi} = e^{i\phi} [ - (\partial_t \phi)^2 + (\nabla^\opL \phi) \cdot (\nabla^\opL \phi) + i (\partial_t^2+\opL) \phi];
\end{gather*}
furthermore
\[
\nabla^\opL q = 0, \qquad \opL q = 0
\]
because $q = q(t,\bxi)$ is independent of $\bx$.
The desired formula follows by combining the above identities.

\smallskip

\ref{en:dal_Fkj}. This follows immediately by inspection of the formulas in \eqref{eq:Fkj} via the properties stated in Lemma \ref{lem:phder}, together with the $1$-homogeneity of $\phi$ in $\bxi$.
\end{proof}

\subsection{The amplitude-to-symbol reduction}

Lemma \ref{lem:phder} shows that the application of $\partial_t^2 + \opL$ to  FIO kernels  $I[q]$ with $\bx$-independent symbols $q$ produces new FIO kernels  $I[p]$ with $\bx$-dependent amplitudes. Following \cite[Lemma 1.10]{LSV} (see also \cite[Lemma 2.7.3]{SV} and \cite[Appendix A]{CLV}), we now discuss how such FIO kernels  $I[p]$ can be reverted into FIO kernels  with $\bx$-independent symbols.

Extending the notation of \eqref{eq:def_underlining}, for a function $p = p(t,\bx,\bxi)$ we write
\begin{align*}
\underline{p} &\defeq p|_{\bx = \bx^t} = p(t,\bx^t(\bxi),\bxi),\\
\widetilde{p} &\defeq \int_0^1 p|_{\bx = \bx^t + s(\bx - \bx^t)} \,ds.
\end{align*}
Of course, when $q = q(t,\bxi)$ is $\bx$-independent, one has
\[
\underline{q} = \widetilde{q} = q.
\]

\begin{lem}\label{lem:a2s}
Let $p = p(t,\bx,\bxi)$ be a smooth amplitude with $\text{$\bxi$-}\supp p \Subset \dot\RR^{d_1} \times \Omega$.
\begin{enumerate}[label=(\roman*)]
\item\label{en:a2s_gen} We have
\begin{equation}\label{eq:A2Sred}
I[p] = I[\underline{p}] + i I[Rp],
\end{equation}
where
\begin{equation}\label{eq:A2Sop}
\begin{split}
Rp &\defeq \Den_\phi^{-1} \Div_{\bxi} ( \Den_\phi \widetilde{\Phi}^{-1} \widetilde{\nabla_{\bx} p} ) \\
&= \frac{1}{2} \frac{\partial_{\bxi} \det \Phi}{\det \Phi} \widetilde{\Phi}^{-1} \widetilde{\nabla_{\bx} p} + \Div_{\bxi} (\widetilde{\Phi}^{-1} \widetilde{\nabla_{\bx} p}).
\end{split}
\end{equation}
\item\label{en:a2s_dx} If $\underline{\nabla_{\bx} p} = 0$, then 
\[
\underline{Rp} = \frac{1}{2} \tr (\underline{\Phi}^{-1} \underline{\partial_{\bxi}\nabla_{\bx} p}).
\]
\item\label{en:a2s_leib} If $q = q(t,\bxi)$ is an $\bx$-independent symbol, then
\[
R(qp) = q (Rp) + (\partial_{\bxi} q) \widetilde{\Phi}^{-1} {\widetilde{\nabla_{\bx} p}}.
\]
\item\label{en:a2s_uind} If $p$ is $u$-independent, then
\[\begin{split}
Rp &= \Den_\phi^{-1} \Div_{\xi} ( \Den_\phi \Phi_0^{-1} {\widetilde{\nabla_{x} p}} )\\
&= \frac{1}{2} \frac{\partial_{\xi} \det \Phi_0}{\det \Phi_0} \Phi_0^{-1} \widetilde{\nabla_{x} p} + \Div_{\xi} (\Phi_0^{-1} \widetilde{\nabla_{x} p}),
\end{split}\]
and $Rp$ is also $u$-independent. Moreover, for any $\bx$-independent symbol $q$,
\[
R(qp) = q (Rp) + (\partial_{\xi} q) \Phi_0^{-1} {\widetilde{\nabla_{x} p}}.
\]
\item\label{en:a2s_uindxpol} For any $k \in \NN$, if $p$ is $u$-independent and a polynomial in $x$ of degree $k$, then $Rp$ is a polynomial in $x$ of degree $k-1$ (i.e., it vanishes when $k=0$), and
\[
I[p] = \sum_{\ell=0}^k i^{\ell}I[\underline{R^{\ell} p}],
\]
where $R^\ell$ denotes the $\ell$th  compositional power of $R$.
\end{enumerate}
\end{lem}
\begin{proof}
\ref{en:a2s_gen}. By the fundamental theorem of calculus,
\[
p - \underline{p} = (\bx - \bx^t)^T \widetilde{\nabla_{\bx} p}  
\]
and similarly
\[
\partial_{\bxi} \phi = \partial_{\bxi} \phi - \underline{\partial_{\bxi} \phi} = (\bx-\bx^t)^T \widetilde{\Phi}.
\]
So
\[
e^{i\phi} (p-\underline{p}) 
= e^{i\phi} (\bx - \bx^t)^T \widetilde{\nabla_{\bx} p} 
= e^{i\phi} (\partial_{\bxi} \phi) \widetilde{\Phi}^{-1} \widetilde{\nabla_{\bx} p} 
= -i \partial_{\bxi} (e^{i\phi}) \widetilde{\Phi}^{-1} \widetilde{\nabla_{\bx} p}
\]
and
\[
I[p]-I[\underline{p}] 
= (2\pi)^{-d} \int e^{i\phi} (p-\underline{p}) \Den_\phi \,d\bxi 
= - i (2\pi)^d \int \partial_{\bxi} (e^{i\phi}) \widetilde{\Phi}^{-1} \widetilde{\nabla_{\bx} p} \, \Den_\phi \,d\bxi,
\]
whence the identity \eqref{eq:A2Sred} follows by integration by parts.

The second expression for $Rp$ in \eqref{eq:A2Sop} follows from the first one by applying the Leibniz rule to $\bxi$-differentiation and observing that
\[
\Den_\phi^{-1} \partial_{\bxi} \Den_\phi = \partial_{\bxi} \log\Den_\phi = \frac{1}{2} \frac{\partial_{\bxi} \det\Phi}{\det\Phi}
\]
by \eqref{eq:den}.

\smallskip

\ref{en:a2s_dx}. By our assumption, $\underline{\widetilde{\nabla_{\bx} p}} = \underline{\nabla_{\bx} p} = 0$. So, by applying the Leibniz rule to $\bxi$-differentiation in \eqref{eq:A2Sop}, we obtain that
\begin{equation}\label{eq:a2s_dx_1red}
\underline{Rp} = \tr (\underline{\Phi}^{-1} \underline{\partial_{\bxi} \widetilde{\nabla_{\bx} p}}).
\end{equation}
Notice now that, for any function $f=f(t,\bx,\bxi)$, by the chain rule,
\[
\partial_{\bxi} \widetilde{f} =  \widetilde{\partial_{\bxi} f} + \int_0^1 (1-s)(\partial_{\bx} f)|_{\bx = \bx^t +s(\bx-\bx^t)} \partial_{\bxi} \bx^t \,ds,
\]
so
\[
\underline{\partial_{\bxi} \widetilde{f}} =  \underline{\partial_{\bxi} f} + \frac{1}{2} \underline{\partial_{\bx} f} \,\partial_{\bxi} \bx^t ,
\]
and in particular
\[
\underline{\partial_{\bxi} \widetilde{\nabla_{\bx} p}} =  \underline{\partial_{\bxi}\nabla_{\bx}  p} + \frac{1}{2} \underline{\partial_{\bx} \nabla_{\bx} p} \,\partial_{\bxi} \bx^t.
\]
On the other hand, as $\underline{\partial_{\bx} p} = 0$, again by the chain rule,
\[
0 = \partial_{\bxi} (\underline{\nabla_{\bx} p}) = \underline{\partial_{\bxi} \nabla_{\bx} p} + \underline{\partial_{\bx} \nabla_{\bx} p} \, \partial_{\bxi} \bx^t,
\]
so in conclusion
\[
\underline{\partial_{\bxi} \widetilde{\nabla_{\bx} p}} =  \frac{1}{2} \underline{\partial_{\bxi} \nabla_{\bx} p}.
\]
By combining this with \eqref{eq:a2s_dx_1red} we obtain the desired result.

\smallskip

\ref{en:a2s_leib}. If $q$ is $\bx$-independent, then
\[
R(qp) = \Den_\phi^{-1} \Div_{\bxi} (q \Den_\phi \widetilde{\Phi}^{-1} {\widetilde{\nabla_{\bx} p}} )
\]
and the result follows by the Leibniz rule.

\smallskip

\ref{en:a2s_uind}. If $p$ is $u$-independent, then
\[
\widetilde{\nabla_{\bx} p} = \left(\begin{array}{c}
\widetilde{\nabla_x p} \\\hline 0
\end{array}\right);
\]
moreover, from Lemma \ref{lem:phder}\ref{en:phder_hess} we know that
\[
\widetilde{\Phi}^{-1} = \left(\begin{array}{c|c}
\Phi_0^{-1} & * \\\hline
0 & I
\end{array}\right),
\]
as $\Phi_0$ is $\bx$-independent. Thus
\[
\widetilde{\Phi}^{-1} \widetilde{\nabla_{\bx} p} = \left(\begin{array}{c}
\Phi_0^{-1} \widetilde{\nabla_x p} \\ \hline 0
\end{array}\right),
\]
and the desired formulas for $Rp$ and $R(qp)$ follow from those in parts \ref{en:a2s_gen} and \ref{en:a2s_leib}. As a consequence, due to the fact that $\Phi_0$ and $\Den_\phi$ are $\bx$-independent by Lemma \ref{lem:phder}, we deduce that $Rp$ is $u$-independent.

\smallskip

\ref{en:a2s_uindxpol}. From the formula for $Rp$ in part \ref{en:a2s_uind} we also deduce that, if $p$ is $u$-independent and a polynomial in $x$ of degree $k$, then $\widetilde{\nabla_x p}$ and $Rp$ are also polynomial functions in $x$ of degree $k-1$. Inductively, we then deduce that $R^{(k+1)} p = 0$ and therefore, by iterating the expression for $I[p]$ in part \ref{en:a2s_gen} we obtain the desired formula.
\end{proof}

We now state some crucial properties of the coefficients in \eqref{eq:Fkj}, which are due to the form of the phase function $\phi$.
Analogous results are contained in the proofs of \cite[Lemma 3.1]{LSV} and \cite[Theorem 2.4.16 and Theorem 3.3.2]{SV}; in particular, the second-order vanishing of the coefficient $F_{20}$ at $\bx = \bx^t$ expresses the fact that the phase function $\phi$ satisfies the eikonal equation
\[
(\partial_t \phi)^2 - \nabla^\opL \phi \cdot \nabla^\opL \phi=0
\]
at least along the geodesic flow, while the vanishing of $F_{10} + RF_{20}$ corresponds to the vanishing of the ``sub-principal symbol'' of the sub-Laplacian $\opL$.

\begin{lem}\label{lem:crucial_coeff}
The coefficients defined in \eqref{eq:Fkj} satisfy the following identities:
\begin{align*}
\underline{F_{11}} &= 2 |\xi|, \\
\underline{F_{20}} &= 0,\\
\underline{\partial_{\bx} F_{20}} &= 0,\\
\underline{F_{10} + R F_{20}} &= 0. \\
\end{align*}
\end{lem}
\begin{proof}
From \eqref{eq:phase} it is clear that
\begin{equation}\label{eq:der_t_phi_s}
\underline{\partial_t \phi} = -\dot \bx^t \cdot \bxi^t = - (\partial_{\bxi} \Ham)(\bx^t,\bxi^t) \, \bxi^t = - \Ham(\bx^t,\bxi^t) = -\Ham(\bx^0,\bxi^0) = -|\xi|,
\end{equation}
where we used the Hamilton equations, Euler's identity ($\Ham$ is $1$-homogeneous in $\bxi$) and the fact that $\Ham$ is constant under the flow.

Moreover, by \eqref{eq:phase} and \eqref{eq:Xfields},
\[
\underline{\nabla^\opL \phi} = \xi^t +\frac{1}{2} J_\mu x^t,
\]
so
\[
\underline{(\nabla^\opL \phi) \cdot (\nabla^\opL \phi)} = \Ham(\bx^t,\bxi^t)^2 = |\xi|^2.
\]

The above identities, together with \eqref{eq:Fkj}, give that $\underline{F_{20}} = 0$ and $\underline{F_{11}} = 2|\xi|$.

\smallskip

Next, by differentiating the identity $\underline{\nabla_{\bx} \phi} = \bxi^t$ from Lemma \ref{lem:phder}\ref{en:phder_grad},
\begin{equation}\label{eq:der_tx_phi_s}
\dot\bxi^t = \partial_t \underline{\nabla_{\bx} \phi} = \underline{\partial_t\nabla_{\bx} \phi} + \underline{\partial_{\bx}\nabla_{\bx} \phi} \, \dot\bx.
\end{equation}
Moreover, from \eqref{eq:Xfields}, we deduce that
\[
\nabla^\opL \phi(t,\bx,\bxi) = \nabla_x \phi(t,\bx,\bxi) + \frac{1}{2} J_{\nabla_u \phi(t,\bx,\bxi)} x,
\]
thus
\[
\nabla^\opL \phi \cdot \nabla^\opL \phi = \Ham^2(\bx,\nabla_{\bx} \phi);
\]
notice that $\nabla_{\bx} \phi$ need not be real, however $\Ham^2(\bx,\bxi)$ is a polynomial in $\bxi$, so it makes sense to evaluate it at complex values of $\bxi$ too.
Therefore,
\[
\partial_{\bx} F_{20} = 2 (\partial_t \phi) (\partial_t \partial_{\bx}\phi) - \partial_{\bx} \Ham^2(\bx,\nabla_{\bx} \phi) - \partial_{\bxi} \Ham^2(\bx,\nabla_{\bx} \phi) \partial_{\bx} \nabla_{\bx} \phi
\]
and
\[\begin{split}
\underline{\partial_{\bx} F_{20}} 
&= 2 \underline{\partial_t \phi} \, \underline{\partial_t \partial_{\bx}\phi} - \partial_{\bx} \Ham^2(\bx^t,\underline{\nabla_{\bx} \phi}) - \partial_{\bxi} \Ham^2(\bx^t,\underline{\nabla_{\bx} \phi}) \underline{\partial_{\bx}\nabla_{\bx} \phi} \\
&= 2 \underline{\partial_t \phi} \, \underline{\partial_t \partial_{\bx}\phi} - \partial_{\bx} \Ham^2(\bx^t,\bxi^t) - \partial_{\bxi} \Ham^2(\bx^t,\bxi^t) \underline{\partial_{\bx} \nabla_{\bx} \phi} \\
&= 2 \underline{\partial_t \phi} \, \underline{\partial_t \partial_{\bx}\phi} - 2 \Ham(\bx^t,\bxi^t) \partial_{\bx} \Ham(\bx^t,\bxi^t) - 2 \Ham(\bx^t,\bxi^t) \partial_{\bxi} \Ham(\bx^t,\bxi^t) \underline{\partial_{\bx}\nabla_{\bx} \phi} \\
&= 2 \underline{\partial_t \phi} \, \underline{\partial_t \partial_{\bx}\phi} + 2 \Ham(\bx^t,\bxi^t) [ (\dot\bxi^t)^T -  (\dot\bx^t)^T \underline{\partial_{\bx} \nabla_{\bx} \phi} ]  = 0,
\end{split}\]
where we used Lemma \ref{lem:phder}\ref{en:phder_grad}, the Hamilton equations \eqref{eq:Ham} and the identities \eqref{eq:der_t_phi_s} and \eqref{eq:der_tx_phi_s}.

It remains to prove that $\underline{F_{10} + R F_{20}} = 0$. From Lemma \ref{lem:a2s}\ref{en:a2s_dx} and the fact that $\underline{\partial_{\bx} F_{20}} = 0$, we deduce that
\begin{equation}\label{eq:RF20tr}
\underline{R F_{20}} = \frac{1}{2} \tr \underline{\Phi^{-1} \partial_{\bx}\partial_{\bxi} F_{20}}.
\end{equation}
On the other hand, by the chain rule it is easily seen that, for any function $p = p(t,\bx,\bxi)$,
\[
\frac{1}{2} \tr (\Phi^{-1} \partial_{\bxi} \nabla_{\bx}(p^2)) = p \tr (\Phi^{-1} \partial_{\bxi}\nabla_{\bx} p)
+ (\partial_{\bxi} p) \Phi^{-1} (\nabla_{\bx} p),
\]
thus
\[
\frac{1}{2} \tr (\Phi^{-1} \partial_{\bxi} \nabla_{\bx}((\partial_t \phi)^2)) 
= (\partial_t \phi) \tr (\Phi^{-1} \partial_t \Phi)
+ (\partial_{\bxi} \partial_t\phi) \Phi^{-1} (\nabla_{\bx} \partial_t\phi).
\]
Additionally, if $X = a \cdot \nabla_{\bx}$ is a vector field with coefficients $a = a(\bx)$, then, by the chain rule,
\[
\partial_{\bxi}\nabla_{\bx} (Xp) = (\partial_{\bx} a)^T (\partial_{\bxi} \nabla_{\bx}p) + X\partial_{\bxi}\nabla_{\bx} p,
\]
where $X$ acts on $\partial_{\bxi}\nabla_{\bx} p$ componentwise; thus, when $p = \phi$,
\[
\tr(\Phi^{-1} \partial_{\bxi}\nabla_{\bx} (X\phi))
= \Div_{\bx} a +\tr(\Phi^{-1} X \Phi),
\]
and the term $\Div_{\bx} a$ vanishes when $X$ is divergence-free. As the vector fields $X_j$ defining the sub-Laplacian $\opL$ are divergence-free, we deduce that
\[
\frac{1}{2} \tr (\Phi^{-1} \partial_{\bxi}\nabla_{\bx} ((X_j \phi)^2)) 
= (X_j \phi) \tr (\Phi^{-1} X_j \Phi) + (\partial_{\bxi} X_j \phi) \Phi^{-1} (\nabla_{\bx} X_j \phi).
\]
In conclusion,
\[\begin{split}
\frac{1}{2} \tr (\Phi^{-1} \partial_{\bxi}\nabla_{\bx} F_{20})
&= 
(\partial_t \phi) (\partial_t \log\det\Phi) 
- \sum_{j=1}^{d_1} (X_j \phi) (X_j\log \det\Phi) \\
&\quad+ (\partial_{\bxi} \partial_t\phi) \Phi^{-1} (\nabla_{\bx} \partial_t\phi) 
- \sum_{j=1}^{d_1} (\partial_{\bxi} X_j \phi) \Phi^{-1} (\nabla_{\bx} X_j \phi),
\end{split}\]
where we also used the identities
\[
\partial_t \log \det \Phi = \tr(\Phi^{-1} \partial_t \Phi), \quad X_j \log \det \Phi = \tr(\Phi^{-1} X_j \Phi).
\]

On the other hand, from \eqref{eq:Fkj} we also deduce that
\[\begin{split}
-F_{10} 
&= (\partial_t^2 + \opL)\phi + 2 (\partial_t \phi) (\Den_\phi^{-1} \partial_t \Den_\phi) - 2 (\nabla^\opL \phi) \cdot (\Den_{\phi}^{-1} \nabla^\opL \Den_\phi),\\
&= (\partial_t^2 + \opL)\phi + (\partial_t \phi) (\partial_t \log \det \Phi) - (\nabla^\opL \phi) \cdot (\nabla^\opL \log\det \Phi),
\end{split}\]
and thus 
\begin{multline}\label{eq:F10trF20}
F_{10} + \frac{1}{2} \tr (\Phi^{-1} \partial_{\bxi} \nabla_{\bx} F_{20}) \\
= -(\partial_t^2 + \opL)\phi + (\partial_{\bxi} \partial_t\phi) \Phi^{-1} (\nabla_{\bx} \partial_t\phi) 
- \sum_{j=1}^{d_1} (\partial_{\bxi} X_j \phi) \Phi^{-1} (\nabla_{\bx} X_j \phi).
\end{multline}

Notice now that, for any vector field $X = a \cdot \nabla_{\bx}$ with coefficients $a= a(\bx)$,
\[
X p = a \cdot \nabla_{\bx} p, \qquad X^2 p = a \cdot \nabla_{\bx} (a \cdot \nabla_{\bx} p), \qquad \partial_{\bxi} Xp = a^T  ( \partial_{\bxi} \nabla_{\bx}p),
\]
so
\[
(\partial_{\bxi} X p) ( \partial_{\bxi}\nabla_{\bx} p)^{-1} (\nabla_{\bx} X p) = a^T \nabla_{\bx} (a \cdot \nabla_{\bx} p) = X^2 p.
\]
In light of this, the identity \eqref{eq:F10trF20} reduces to
\begin{equation}\label{eq:F10trF20red}
F_{10} + \frac{1}{2} \tr (\Phi^{-1} \partial_{\bxi}\nabla_{\bx} F_{20}) \\
= -\partial_t^2\phi + (\partial_{\bxi} \partial_t\phi) \Phi^{-1} (\nabla_{\bx} \partial_t\phi) .
\end{equation}

Finally, by differentiating \eqref{eq:der_t_phi_s},
\[
0 = \partial_t \underline{\partial_t \phi} = \underline{\partial_t^2 \phi} + \underline{\partial_{\bx} \partial_t \phi} \, \dot\bx^t;
\]
thus, by \eqref{eq:RF20tr} and \eqref{eq:F10trF20red},
\[\begin{split}
\underline{F_{10}+RF_{20}} &= \underline{F_{10} + \frac{1}{2} \tr (\Phi^{-1} \partial_{\bxi} \nabla_{\bx}F_{20})} \\
&= - \underline{\partial_t^2 \phi} + \underline{\partial_{\bxi} \partial_t\phi} \, \underline{\Phi}^{-1} \underline{\nabla_{\bx} \partial_t\phi} \\
&= ((\dot\bx^t)^T \underline{\Phi} + \underline{\partial_{\bxi}\partial_t \phi}) \underline{\Phi}^{-1} \underline{\nabla_{\bx} \partial_t \phi} \\
&= \partial_t \underline{\partial_{\bxi} \phi} \, \underline{\Phi}^{-1} \underline{\nabla_{\bx} \partial_t \phi} = 0,
\end{split}\]
where we used the fact that $\underline{\partial_{\bxi} \phi} = 0$ by Lemma \ref{lem:phder}\ref{en:phder_grad}.
\end{proof}

We can now prove an enhanced version of Lemma \ref{lem:dal}, showing that the application of the wave operator $\partial_t^2 + \opL$ to an FIO  kernel of the form \eqref{eq:fio_Iq} with $\bx$-independent symbol can  be written again as an FIO kernel with an $\bx$-independent symbol. The analogous results in \cite{CLV,LSV,SV} prove a similar identity only up to smoothing terms, and express the new symbol in terms of an infinite asymptotic expansion, involving iterated applications of the ``amplitude-to-symbol operator'' $R$; here, thanks to the $\bx$-polynomial nature of the coefficients $F_{kj}$ of Lemma \ref{lem:dal}, only finitely many terms of the expansion survive, so we obtain an exact formula.

\begin{prp}\label{prp:dal2}
For any smooth symbol $q = q(t,\bxi)$ with $\text{$\bxi$-}\supp q \Subset \dot\RR^{d_1} \times \Omega$,
\begin{align*}
(\partial_t^2 + \opL) I[q] &= I[-2i|\xi|\partial_t q + \Lambda q] , \\
\partial_t I[q] &= I[-i|\xi| q + \mho q],
\end{align*}
where
\begin{equation}
\label{eq:opLambda}
\Lambda q \defeq \sum_{\substack{j,k \geq 0 \\ j+k \leq 2}} \underline{R^{k}(F_{kj} \partial_t^j q)}, \qquad
\mho q \defeq \sum_{\substack{j,k \geq 0 \\ j+k \leq 1}} 2^{j-1} \underline{R^{k}(F_{k(j+1)} \partial_t^j q)},
\end{equation}
\end{prp}
\begin{proof}
We only prove the formula for $(\partial_t^2 + \opL) I[q]$; the proof for $\partial_t I[q]$ is fully analogous and is omitted.

By Lemma \ref{lem:dal},
\[
(\partial_t^2 + \opL) I[q] = \sum_{\substack{j,k \geq 0 \\ j+k \leq 2}} i^{-k} I[F_{kj} \partial_t^j q], 
\]
and moreover the function $F_{kj} \partial_t^j q$ is a polynomial in $x$ of degree $k$ and independent of $u$; thus, by Lemma \ref{lem:a2s},
\[\begin{split}
(\partial_t^2 + \opL) I[q] 
&= \sum_{\substack{j,k \geq 0 \\ j+k \leq 2}} \sum_{\ell=0}^k i^{\ell-k} I[\underline{R^{\ell} ( F_{kj} \partial_t^j q)}] \\
&= I \left[ \sum_{s=0}^2 \sum_{k=s}^2 \sum_{j=0}^{2-k} i^{-s} \underline{R^{k-s} ( F_{kj} \partial_t^j q)}\right],
\end{split}\]
where the sum was reindexed by setting $s=k-\ell$.

Now, the term with $s=2$ in the latter sum is
\[
i^{-2} \underline{F_{20}} \, q = 0
\]
due to Lemma \ref{lem:crucial_coeff}. So the sum can actually be restricted to $s=0,1$. In addition, for $s=1$, we get
\[\begin{split}
i^{-1} \underline{F_{11} \partial_t q + F_{10} q + R(F_{20} q)} 
&= i^{-1} ( \underline{F_{11}} \partial_t q + \underline{F_{10}+ RF_{20}} \, q + \partial_\xi q \Phi_0^{-1} \underline{\partial_x F_{20}}^T) \\
&= - 2i|\xi| \partial_t q,
\end{split}\]
by Lemmas \ref{lem:a2s}\ref{en:a2s_uind} and \ref{lem:crucial_coeff}. Finally, the part of the sum with $s=0$ gives $\Lambda q$ as defined in \eqref{eq:opLambda}, and we are done.
\end{proof}

\subsection{Construction of the parametrix}

Much as in \cite{CLV,LSV,SV}, the identities in Proposition \ref{prp:dal2} allows us to set up an iterative scheme, based on the solution of suitable ``transport equations'', to obtain an approximate representation of solutions of wave equation in terms of FIO kernels of the form \eqref{eq:fio_Iq}.

According to Proposition \ref{prp:dal2}, for a given symbol $q=q(\bxi)$, ideally we would like to construct a symbol  $p = p(t,\bxi)$ that satisfies the initial condition $p(0,\bxi)=q(\bxi)$ and solves the ``transport equation'' 
\begin{equation}\label{eq:transpeq}
(-2i|\xi|\partial_t + \Lambda) p = 0,
\end{equation} 
for then we would have a solution of the wave equation $(\partial_t^2 + \opL) I[p]=0.$

To this end, if we look at the $\bxi$-homogeneity degrees of the various quantities $\Phi$, $\Den_\phi$ and $F_{kj}$ (see \eqref{eq:den} and Lemma \ref{lem:dal}\ref{en:dal_Fkj}), then we see that the operator $R$ of \eqref{eq:A2Sop} decreases the $\bxi$-degree by one, while the operator $\Lambda$ of \eqref{eq:opLambda} preserves the $\bxi$-degree. On the other hand, the operator $2i|\xi|\partial_t$ increases the $\bxi$-degree by one.
Thus, assuming as usual an asymptotic expansion $p=\sum_{j=0}^\infty p_{-j}$ of $p$ into symbols $p_{-j}$ of order $-j$, at least at a formal level the equation \eqref{eq:transpeq} then transforms into the system of equations
\[
2i|\xi|\partial_t p_{-j}=\Lambda p_{1-j}, \qquad j\geq 0,
\]
which can be solved by setting recursively
\[
p_{-j}(t,\bxi) \defeq \frac{1}{2i|\xi|} \int_0^t \Lambda p_{1-j}(\tau,\bxi) \,d\tau, \qquad j \geq 0.
\]

Let us make these arguments more precise. 
For a symbol $q = q(t,\bxi)$, we introduce the notation
\[
I_0[q](\bx) \defeq I[q](0,\bx)
\]
for the oscillatory integral with symbol $q$ at time $t=0$.
Notice that the functions $H_{m,h}$ of \eqref{eq:initial_amplitude_complex} have the form $I_0[q]$ for an appropriate $q$.
We also define 
\begin{align}
\label{eq:opLambda_int}
\Lambda_I q(t,\bxi) &\defeq \frac{1}{2i|\xi|} \int_0^t \Lambda q(\tau,\bxi) \,d\tau,\\
\label{eq:opMho_int}
\mhoI q(t,\bxi) &\defeq |\xi| q(t,\bxi) +i \mho q(t,\bxi),
\end{align}
where $\Lambda$ and $\mho$ are as in \eqref{eq:opLambda}.

\begin{prp}\label{prp:FIO_repn_wave}
Assume that the symbol $q = q(\bxi)$ is smooth, $(t,\bx)$-independent and compactly supported in $\dot\RR^{d_1} \times \Omega$.
Then, for all $N \in \NN$,
\begin{multline}\label{eq:FIO_repn_wave}
\cos(t\sqrt{\opL}) I_0[q](\bx) = \frac{1}{2} \sum_{j=0}^N ( I[\Lambda_I^j q](t,\bx) + I[\Lambda_I^j q](-t,\bx)) \\
- \frac{1}{2} \int_0^t \int_0^{t-\tau} \cos(s \sqrt{\opL}) \left(I[\Lambda \Lambda_I^N q](\tau,\bx) + I[\Lambda \Lambda_I^N q](-\tau,\bx)\right) \,ds\,d\tau,
\end{multline}
where $\Lambda_I^j$ denotes the $j$th compositional power of $\Lambda_I$.
In addition,
\begin{multline}\label{eq:FIO_repn_wave_der}
\sqrt{\opL} \sin(t\sqrt{\opL}) I_0[q](\bx) = -\frac{1}{2i} \sum_{j=0}^N ( I[ \mhoI \Lambda_I^j q](t,\bx) - I[\mhoI \Lambda_I^j q](-t,\bx)) \\
+ \frac{1}{2} \int_0^t \cos((t-\tau) \sqrt{\opL}) \left( I[\Lambda \Lambda_I^N q](\tau,\bx) + I[\Lambda \Lambda_I^N q](-\tau,\bx) \right) \,d\tau.
\end{multline}
\end{prp}
\begin{proof}
Notice that, by construction,
\[
\partial_t (\Lambda_I^0 q) = 0, \qquad \Lambda_I^{j+1} q|_{t=0} = 0, \qquad \partial_t (\Lambda_I^{j+1} q) = \frac{1}{2i|\xi|} \Lambda \Lambda_I^j q 
\]
for all $j \in \NN$. Thus, if we set, for all $N \in \NN$,
\[
H^{(N)}(t,\bxi) \defeq \sum_{j=0}^N \Lambda_I^j q(t,\bxi),
\]
then it is easily seen that
\begin{align*}
H^{(N)}(0,\bxi) &= q(\bxi),\\
(-2i|\xi| \partial_t+ \Lambda)H^{(N)}(t,\bxi)  &= \Lambda \Lambda_I^N q.
\end{align*}

From Proposition \ref{prp:dal2} we then deduce that
\[
I[H^{(N)}]|_{t=0} = I_0[q], \qquad
(\partial_t^2 + \opL) I[H^{(N)}] = I[\Lambda \Lambda_I^N q].
\]
An application of a Duhamel-type formula (as in, e.g., \cite[eq.\ (4.18)]{MMNG}) then yields
\[\begin{split}
\cos(t\sqrt{\opL}) I_0[q] 
&= \frac{I[H^{(N)}](t,\cdot) + I[H^{(N)}](-t,\cdot)}{2} \\
&\qquad-\int_0^t  \frac{\sin((t-\tau)\sqrt{\opL})}{\sqrt {\opL}}\frac{I[\Lambda \Lambda_I^N q](\tau,\cdot) + I[\Lambda \Lambda_I^N q](-\tau,\cdot)}{2} d\tau\\
&= \frac{I[H^{(N)}](t,\cdot) + I[H^{(N)}](-t,\cdot)}{2} \\
&\qquad-\int_0^t \int_0^{t-\tau} \cos(s\sqrt{\opL}) \frac{I[\Lambda \Lambda_I^N q](\tau,\cdot) + I[\Lambda \Lambda_I^N q](-\tau,\cdot)}{2} \,ds\,d\tau, 
\end{split}\]
thus proving \eqref{eq:FIO_repn_wave}.

\smallskip

Finally, by noticing that
\[
\sqrt{\opL} \sin(t\sqrt{\opL}) I_0[q] = -\partial_t \cos(t \sqrt{\opL}) I_0[q],
\]
then \eqref{eq:FIO_repn_wave_der} follows by differentiating \eqref{eq:FIO_repn_wave} with respect to $t$, and applying the formula $-\partial_t I[\Lambda_I^j q] = i I[\mhoI \Lambda_I^j q]$ from Proposition \ref{prp:dal2}.
\end{proof}

According to Proposition \ref{prp:FIO_repn_wave}, oscillatory integrals of the form \eqref{eq:fio_Iq} can be used to give an approximate FIO representation (or ``parametrix'') for solutions of the wave equation associated with the sub-Laplacian $\opL$.
In order to use this representation effectively, we actually need to show that the summands in the right-hand side of \eqref{eq:FIO_repn_wave} become of lower and lower order (in a suitable sense) as $j$ grows, and correspondingly the last summand can be treated as a negligible error term when $N$ is sufficiently large.

As we have already seen, the operator $\Lambda$ of \eqref{eq:opLambda} indeed preserves the $\bxi$-degree of our symbols, and thus the  operator $\Lambda_I$ of \eqref{eq:opLambda_int} decreases the $\bxi$-degree by one. In this sense the right-hand side of \eqref{eq:FIO_repn_wave} can be thought of as an expansion of $\cos(t\sqrt{\opL}) I_0[q]$ into oscillatory integrals with symbols of lower and lower $\bxi$-order.

In a similar way, the right-hand side of \eqref{eq:FIO_repn_wave_der} can be thought of as an expansion of $\sqrt{\opL}\sin(t\sqrt{\opL}) I_0[q]$ into oscillatory integrals with symbols of lower and lower $\bxi$-order. Note, however, that $\mhoI$ increases the $\bxi$-degree by one.

However, these $\bxi$-homogeneity considerations are not enough to guarantee that there is no blow-up in time of these expansions. Indeed, as the definition of $\Lambda_I$ in \eqref{eq:opLambda_int} involves integration on an interval of length $t$, and the lower order terms in these expansions involve iterated applications of $\Lambda_I$, more precise information is needed to ensure that the lower order terms remain negligible with respect to the main term as $t$ grows. To gain the appropriate control for large time, we shall need a more explicit expression for the operator $\Lambda$.

\begin{prp}\label{prp:Lambdaop_2step}
With the notation of Corollary \ref{cor:flow_alt},
the operators $\Lambda$ and $\mho$ from \eqref{eq:opLambda} are given by
\begin{align}
\label{eq:opLambda_2step}
\Lambda q 
&= \sum_{j=0}^2 \Lambda_{j0} \partial_t^j q + \sum_{j=0}^1 (\partial_t^j \partial_{\xi} q) \Lambda_{j1} + \tr(\Lambda_{02} \partial_\xi \nabla_\xi q),\\
\label{eq:opMho_2step}
\mho q 
&= \sum_{j=0}^1 2^{j-1} \Lambda_{(j+1)0} \partial_t^j q + 2^{-1} (\partial_{\xi} q) \Lambda_{11},
\end{align}
where
{\allowdisplaybreaks\begin{align*}\label{eq:opLambda_coeff_2step}
\Lambda_{00} &=
-\frac{|\mu|^2}{16|\xi|^2} \frac{1}{(1+ i\theta \langle |J_{\bar\mu}| \bar\xi,\bar\xi \rangle)^2} 
\Biggl[ 
 \tr^2|J_{\bar\mu}| - 6 \langle|J_{\bar\mu}|\bar\xi,\bar\xi\rangle \tr|J_{\bar\mu}| \\
&\quad -12\langle|J_{\bar\mu}|^2 \bar\xi,\bar\xi\rangle + 21\langle|J_{\bar\mu}|\bar\xi,\bar\xi\rangle^2
- 4i\theta \tr|J_{\bar\mu}| \frac{2\langle|J_{\bar\mu}|^2\bar\xi,\bar\xi\rangle-3\langle|J_{\bar\mu}|\bar\xi,\bar\xi\rangle^2}{1+i\theta\langle|J_{\bar\mu}|\bar\xi,\bar\xi\rangle}\\
&\quad- 2i\theta \frac{4\langle|J_{\bar\mu}|^3\bar\xi,\bar\xi\rangle-30\langle|J_{\bar\mu}|^2\bar\xi,\bar\xi\rangle \langle|J_{\bar\mu}|\bar\xi,\bar\xi\rangle+33\langle|J_{\bar\mu}|\bar\xi,\bar\xi\rangle^3}{1+i\theta\langle|J_{\bar\mu}|\bar\xi,\bar\xi\rangle}\\
&\quad+ 5(i\theta)^2 \frac{(2\langle|J_{\bar\mu}|^2\bar\xi,\bar\xi\rangle-3\langle|J_{\bar\mu}|\bar\xi,\bar\xi\rangle^2)^2}{(1+i\theta\langle|J_{\bar\mu}|\bar\xi,\bar\xi\rangle)^2} 
+ 2 \frac{2\langle|J_{\bar\mu}|^2\bar\xi,\bar\xi\rangle-3\langle|J_{\bar\mu}|\bar\xi,\bar\xi\rangle^2}{1+i\theta\langle|J_{\bar\mu}|\bar\xi,\bar\xi\rangle}
\Biggr],\\
\Lambda_{01} &= 
 -\frac{|\mu|^2}{4|\xi|^2} \frac{1}{(1+i\theta\langle|J_{\bar\mu}| \bar\xi,\bar\xi\rangle)^2}\\
&\quad\times \Biggl[ \left[  \tr|J_{\bar\mu}| -3 \langle |J_{\bar\mu}| \bar\xi, \bar\xi \rangle +2i\theta \frac{3\langle |J_{\bar\mu}| \bar\xi, \bar\xi \rangle^2 - 2\langle |J_{\bar\mu}|^2 \bar\xi,\bar\xi\rangle}{1+i\theta \langle |J_{\bar\mu}| \bar\xi, \bar\xi \rangle} \right] I + 2|J_{\bar\mu}| \Biggr]\\ 
&\quad\times \left(|J_{\bar\mu}|+iJ_{\bar\mu}\right) \xi ,\\
\Lambda_{02} &= - \frac{|\mu|^2}{4|\xi|^2} \frac{1}{(1+i\theta\langle|J_{\bar\mu}| \bar\xi,\bar\xi\rangle)^2} [\left(|J_{\bar\mu}|+iJ_{\bar\mu}\right) \xi] [\left(|J_{\bar\mu}| +iJ_{\bar\mu}\right) \xi]^T,\\
\Lambda_{10} &= \frac{i|\mu|}{2|\xi|} \frac{1}{1+i\theta \langle |J_{\bar\mu}| \bar\xi, \bar\xi \rangle} \left[ \tr|J_{\bar\mu}| - \langle |J_{\bar\mu}| \bar\xi, \bar\xi \rangle  +i\theta \frac{ 3\langle |J_{\bar\mu}| \bar\xi, \bar\xi \rangle^2 - 2\langle |J_{\bar\mu}|^2 \bar\xi,\bar\xi\rangle }{1+i\theta \langle |J_{\bar\mu}| \bar\xi, \bar\xi \rangle} \right],\\
\Lambda_{11} &= i\frac{|\mu|}{|\xi|} \frac{1}{1+ i\theta \langle |J_{\bar\mu}| \bar\xi,\bar\xi\rangle} (|J_{\bar\mu}| + i J_{\bar\mu}) \xi,\\
\Lambda_{20} &= 1. \stepcounter{equation}\tag{\theequation}
\end{align*}}
\end{prp}

The proof of Proposition \ref{prp:Lambdaop_2step} requires a number of lengthy computations and is postponed to Section \ref{s:opLambda_2step}.

\begin{rem}
Assume that $G$ is a Heisenberg-type group and $\opL$ is its distinguished sub-Laplacian.
Then $|J_{\bar\mu}| = I$ and the formulas \eqref{eq:opLambda_coeff_2step} take a simpler form:
\[
\begin{aligned}
\Lambda_{00} &= -\frac{|\mu|^2}{16 |\xi|^2} \frac{1}{(1+ i\theta)^2} \left[ d_1 (d_1-2) - \frac{2 (2d_1-1)}{1+ i\theta} + \frac{5}{(1+ i\theta)^2} \right],\\
\Lambda_{01} &= -\frac{|\mu|^2}{4 |\xi|^2} \frac{1}{(1+ i\theta)^2} \left[ d_1 +1 - \frac{2}{1+ i\theta} \right] (I + i J_{\bar\mu}) \xi,\\
\Lambda_{02} &= -\frac{|\mu|^2}{4 |\xi|^2} \frac{1}{(1+ i\theta)^2}  [(I+ i J_{\bar\mu}) \xi ] [ (I+ i J_{\bar\mu}) \xi]^T,\\
\Lambda_{10} &= i\frac{|\mu|}{2 |\xi|} \frac{1}{1+ i\theta} \left[ d_1 - \frac{1}{1+ i\theta} \right],\\
\Lambda_{11} &= i\frac{|\mu|}{|\xi|} \frac{1}{1+ i\theta} (I + i J_{\bar\mu}) \xi,\\
\Lambda_{20} &= 1.
\end{aligned}
\]
\end{rem}

Let $\kappa > 1$, and define
\begin{equation}\label{eq:Omega_kappa}
\Omega_\kappa = \{ \bxi \in \RR^{d} \tc |\xi| \geq \kappa^{-1}, \  \kappa^{-1} \leq |\mu|/|\xi| \leq \kappa \}.
\end{equation}
Given $a,b \in \RR,$ we say that a family $\cQ$ of smooth symbols $q = q(t,\bxi)$ is \emph{$(a,b,\kappa)$-bounded} if
\begin{equation}\label{eq:Q_bound}
\sup_{q \in \cQ} \sup_{(t,\bxi) \in \RR \times \RR^{d}} (1+|t|)^{a-k} (1+|\bxi|)^{|\alpha|-b} |\partial_t^k \partial_{\bxi}^\alpha q(t,\bxi)| < \infty \qquad\forall k \in \NN, \ \alpha \in \NN^{d}
\end{equation}
(i.e., if $\cQ$ is bounded in the Fr\'echet space $S^a(\RR) \otimes S^{b}(\RR^d)$) and
\begin{equation}\label{eq:Q_supp}
\bigcup_{q \in \cQ} \bigcup_{t \in \RR} \supp q(t,\cdot) \subseteq \Omega_\kappa.
\end{equation}

\begin{lem}\label{lem:symbols_opLambda}
Assume that $G$ is a M\'etivier group.
Let $\kappa > 1$ and $m \in \RR$. 
If $\cQ$ is a $(0,m,\kappa)$-bounded family of symbols, then
\begin{itemize}
\item $\Lambda(\cQ) \defeq \{ \Lambda q \tc q \in \cQ\}$ is $(-2,m,\kappa)$-bounded, 
\item $\Lambda_I(\cQ) \defeq \{ \Lambda_I q \tc q \in \cQ \}$ is $(0,m-1,\kappa)$-bounded,
\item $\mho(\cQ) \defeq \{ \mho q \tc q \in \cQ\}$ is $(-1,m,\kappa)$-bounded, 
\item $\mhoI(\cQ) \defeq \{ \mhoI q \tc q \in \cQ\}$ is $(0,m+1,\kappa)$-bounded.
\end{itemize}
Moreover,
\begin{gather*}
\bigcup_{t \in \RR} \supp \Lambda_I q(t,\cdot) \subseteq \bigcup_{t \in \RR} \supp \Lambda q(t,\cdot) \subseteq \bigcup_{t \in \RR} \supp q(t,\cdot),\\
\bigcup_{t \in \RR} \supp \mhoI q(t,\cdot) \cup \bigcup_{t \in \RR} \supp \mho q(t,\cdot) \subseteq \bigcup_{t \in \RR} \supp q(t,\cdot).
\end{gather*}
\end{lem}
\begin{proof}
This follows easily by inspection of the formulas \eqref{eq:opLambda_2step}, \eqref{eq:opMho_2step}, \eqref{eq:opLambda_coeff_2step}, \eqref{eq:opLambda_int} and \eqref{eq:opMho_int}.

Indeed, since $G$ is M\'etivier, the map $\mu \mapsto |J_{\bar\mu}|$ is smooth on $\dot\RR^{d_2}$ (see Proposition \ref{prp:modJmu_analytic}). So the coefficients in \eqref{eq:opLambda_coeff_2step} are smooth where $|\xi|\neq 0$ and $|\mu|\neq 0$, thus the support condition \eqref{eq:Q_supp} ensures that the symbols $\Lambda q$ are also smooth for $q \in \cQ$. Moreover clearly $\supp \Lambda q \subseteq \supp q$ and $\bigcup_{t \in \RR} \supp \Lambda q(t,\cdot) \subseteq \bigcup_{t \in \RR} \supp q(t,\cdot)$, and similar support considerations apply to $\mho$ and $\mhoI$.

\smallskip

Additionally, again by the M\'etivier assumption,
\[
\tr |J_{\bar\mu}| \simeq 1, \qquad \langle |J_{\bar\mu}|^N \bar\xi,\bar\xi \rangle \simeq_N 1 
\]
for all $\xi \neq 0$ and $\mu \neq 0$. As $\theta = t|\mu|/(2|\xi|)$ and $|\mu|/|\xi| \simeq 1$ on $\Omega_\kappa$, it is clear that the coefficients $\Lambda_{jr}$ in \eqref{eq:opLambda_coeff_2step} behave on $\Omega_\kappa$ as symbols in $S^{j-2}(\RR) \otimes S^{r}(\RR^d)$; so from \eqref{eq:opLambda_2step} it follows that, if $q$ ranges in a bounded set of $S^{0}(\RR) \otimes S^{m}(\RR^d)$, then $\Lambda q$ ranges in a bounded set of $S^{-2}(\RR) \otimes S^{m}(\RR^d)$. This shows that $\Lambda(\cQ)$ is $(-2,m,\kappa)$-bounded. In a similar way, from \eqref{eq:opMho_2step} and \eqref{eq:opMho_int} one sees that $\mho(\cQ)$ is $(-1,m,\kappa)$-bounded and $\mhoI(\cQ)$ is $(0,m+1,\kappa)$-bounded.

\smallskip

Now, if we define $\Lambda_I' q(t,\bxi) \defeq \int_0^t \Lambda q(\tau,\bxi) \,d\tau$, then clearly the $\Lambda_I' q$ are also smooth, and moreover, for any $\alpha \in \NN^d$,
\[\begin{split}
|\partial_{\bxi}^\alpha \Lambda_I'q(t,\bxi)| 
= \left|\int_0^t \partial_{\bxi}^\alpha \Lambda q(\tau,\bxi) \,d\tau\right| 
&\lesssim_{\cQ,\alpha} (1+|\bxi|)^{m-|\alpha|} \int_0^{\infty} (1+\tau)^{-2} \,d\tau \\
&\lesssim (1+|\bxi|)^{m-|\alpha|},
\end{split}\]
while, for any $k \in \Npos$,
\[\begin{split}
|\partial_t^k \partial_{\bxi}^\alpha \Lambda_I'q(t,\bxi)| 
= |\partial_t^{k-1} \partial_{\bxi}^\alpha \Lambda q(\tau,\bxi) | 
&\lesssim_{\cQ,\alpha,k} (1+|\bxi|)^{m-|\alpha|} (1+|t|)^{-2-(k-1)} \\
&\leq (1+|\bxi|)^{m-|\alpha|} (1+|t|)^{-k};
\end{split}\]
this shows that $\Lambda_I' q$ ranges in a bounded subset of $S^{0}(\RR) \otimes S^{m}(\RR^d)$ if $q \in \cQ$. As clearly $\bigcup_{t \in \RR} \supp \Lambda_I' q(t,\cdot) \subseteq \bigcup_{t \in \RR} \supp \Lambda q(t,\cdot)$, we conclude that $\Lambda_I'(\cQ)$ is $(0,m,\kappa)$-bounded.

Finally, the division by $|\xi|$ in \eqref{eq:opLambda_int} does not create singularities due to the support condition, and only has the effect of decreasing by one the order in $\bxi$ of the symbol, thus resulting in $\Lambda_I(\cQ)$ being $(0,m-1,\kappa)$-bounded.
\end{proof}

As a consequence of Lemma \ref{lem:symbols_opLambda}, we see that, if in Proposition \ref{prp:FIO_repn_wave} we start with a smooth symbol $q = q(\bxi)$ of order $0$ supported in some $\Omega_\kappa$, then the symbols $\Lambda_I^j q(t,\bxi)$ and $\Lambda \Lambda_I^j q(t,\bxi)$ appearing in the FIO representation of $\cos(t\sqrt{L}) I_0[q]$ are of lower and lower $\bxi$-order as $j$ increases, while staying uniformly bounded in $t$ and having $\bxi$-supports contained in that of $q$. A similar observation applies to the FIO representation of $\sqrt{\opL} \sin(t\sqrt{\opL}) I_0[q]$.

\section{\texorpdfstring{$L^1$}{L1}-estimates for FIO kernels}\label{s:sss}

In light of Proposition \ref{prp:FIO_repn_wave}, the action of $\cos(t\sqrt{\opL})$ and $\sqrt{\opL} \sin(t\sqrt{\opL})$ on the functions $H_{m,h}$ of \eqref{eq:initial_amplitude_complex} can be described in terms of FIO kernels $I[q]$ for appropriate choices of $q$. As a consequence, the crucial estimate \eqref{eq:target_est}, to which we had reduced the proof of the estimate \eqref{eq:non-elliptic_region} for the non-elliptic region, can be in turn reduced to spatial $L^1$-estimates for FIO kernels, which are discussed in this section in the case where $G$ is a M\'etivier group.

\subsection{Local-in-time \texorpdfstring{$L^1$}{L1}-estimates}

By means of a modification of the key method of Seeger, Sogge and Stein in \cite{SSS}, we shall first obtain a local $L^1$-estimate for oscillatory integrals of the form $I[q](t,\cdot)$, which is uniform as long as $|t|$ is bounded.

One key difference between our approach and that of \cite{SSS} is due to the fact that the phase function $\phi$ that we use here  (see \eqref{eq:phase}) has a nontrivial imaginary part. By taking advantage of the Gaussian decay of $e^{i\phi}$ in $x$ away from $x=x^t$ that is granted by $\Im \phi$ (compare \eqref{eq:useIm} below), we can effectively avoid the use of integration by parts in $\bxi$ to gain a suitable localisation in $x$, and 
we shall only need to integrate by parts in $\bxi$ in order to localise in $u$.

\begin{prp}\label{prp:sdd_0}
Assume that $G$ is a M\'etivier group.
Let $\kappa > 1$ and $a \in \RR$.
Let $\cQ$ be a $(0,a,\kappa)$-bounded family of symbols.
Let $r,T,\delta_0 > 0$. For all $m \in \NN$, all $q \in \cQ$ such that
\begin{equation}\label{eq:Q_supp_m}
\bigcup_{t \in \RR} \supp q(t,\cdot) \subseteq \{ \bxi \tc \kappa^{-1} \leq |\xi|/2^m \leq \kappa \},
\end{equation}
and all $t \in [-T,T]$,
\[
\|\chr_{\overline{B}(0,r)} I[q](t,\cdot) \|_1 \lesssim_{\cQ,\kappa,r,T,\delta_0} 2^{m[\delta_0+a+(d-1)/2]}.
\]
\end{prp}
\begin{proof}
Let $q \in \cQ$ and $m \in \NN$ satisfying \eqref{eq:Q_supp_m}.

By Proposition \ref{prp:modJmu_analytic}, due to our constant rank assumption, the function $\mu \mapsto |J_\mu|$ is smooth on $\dot\RR^{d_2}$.
Thus, by the expression \eqref{eq:phase} for $\phi$ and the formulas \eqref{eq:flow_2step}, it is clear that $\phi$ is a smooth function on $\RR \times \RR^d \times \{ \bxi \tc |\xi| \neq 0 \neq |\mu| \}$, so the same is true for $\Phi$ and $\Den_\phi$; moreover $\phi$ and $\Den_\phi$ are $1$-homogeneous and $0$-homogeneous in $\bxi$ respectively. Thus,
if $\bx \in \overline{B}(0,r)$ and $|t| \leq T$,
\begin{equation}\label{eq:est_phi_unif}
|\partial_{\bxi}^\alpha \phi(t,\bx,\bxi)| \lesssim_{\kappa,r,T,\alpha} |\bxi|^{1-|\alpha|}, \quad
|\partial_{\bxi}^\alpha \Den_\phi(t,\bxi)| \lesssim_{\kappa,r,T,\alpha} |\bxi|^{-|\alpha|}
\end{equation}
whenever $\bxi = (\xi,\mu) \in \Omega_\kappa$ (see \eqref{eq:Omega_kappa}); as a consequence, by the support conditions \eqref{eq:Q_supp} and \eqref{eq:Q_supp_m}, we have
\begin{equation}\label{eq:first_dyadic}
|\bxi| \simeq_\kappa 2^m
\end{equation}
and
\begin{equation}\label{eq:est_simb_unif}
|\partial_{\bxi}^\alpha \phi(t,\bx,\bxi)| \lesssim_{\kappa,r,T,\alpha} 2^{m(1-|\alpha|)}, \quad
|\partial_{\bxi}^\alpha (q(t,\bxi) \Den_\phi(t,\bxi))| \lesssim_{\cQ,\kappa,r,T,\alpha} \lesssim 2^{m(a-|\alpha|)}.
\end{equation}
when $\bxi \in \supp q(t,\cdot)$. 

\smallskip

As just mentioned, our support conditions enforce the dyadic frequency localisation \eqref{eq:first_dyadic}.
We now proceed with a ``second dyadic decomposition'' according to the direction of $\bxi$, analogous to that used in \cite{SSS} for estimates for oscillatory integrals with real phase functions. As we shall see, while this decomposition is essentially the same as that in \cite{SSS}, the way one estimates the resulting pieces is partly different, due to the presence of an imaginary part in the phase function $\phi$.

\smallskip

Namely, we can find (see, e.g., \cite[Section IX.4]{St} or \cite[Lemma 14]{MM_newclasses}) a finite set $Z_m \subseteq \{ \bxi \tc |\bxi| = 1\}$ of directions in $\RR^d$ with
\begin{equation}\label{eq:count_sdd_0}
\sharp Z_m \lesssim 2^{m(d-1)/2}
\end{equation}
and a smooth homogeneous partition of unity $\{\chi_{m,\breve{\bxi}}\}_{\breve{\bxi} \in Z_m}$ on $\dot\RR^d$ such that
\begin{equation}\label{eq:supp_sdd_0}
\supp \chi_{m,\breve{\bxi}} \subseteq \{ \bxi \tc |\bxi/|\bxi|-\breve{\bxi}| \leq c \, 2^{-m/2} \}
\end{equation}
for an appropriately small $c>0$, and
\begin{equation}\label{eq:der_sdd_0}
|\partial_{\bxi}^\alpha (\breve{\bxi} \cdot \nabla_{\bxi})^N \chi_{m,\breve{\bxi}} | \lesssim_{N,\alpha} 2^{|\alpha|m/2} |\bxi|^{-N-|\alpha|}
\end{equation}
for all $\breve{\bxi} \in Z_m$, $\alpha \in \NN^d$ and $N \in \NN$, where the implicit constants do not depend on $m$. Accordingly, we split
\begin{equation}\label{eq:dec_int_sdd_0}
I[q] = \sum_{\breve{\bxi} \in \tilde  Z_m} I[q_{\breve{\bxi}}],
\end{equation}
where $q_{\breve{\bxi}}(t,\bxi) \defeq q(t,\bxi) \chi_{m,\breve{\bxi}}(\bxi)$
and, due to the support condition \eqref{eq:Q_supp}, the sum is restricted to
\[
\tilde Z_m = \{ \breve{\bxi} \in Z_m \tc \supp \chi_{m,\breve{\bxi}} \cap \Omega_{\kappa} \neq \emptyset\}.
\]
As $|\mu| \simeq_\kappa |\xi|$ for all $\bxi \in \Omega_{\kappa}$, by choosing a small enough $c=c(\kappa)$ in \eqref{eq:supp_sdd_0} we may assume that
\begin{equation}\label{eq:normal_dir_0}
|\breve{\xi}| \simeq |\breve{\mu}| \simeq 1 \qquad\forall \breve{\bxi} = (\breve{\xi},\breve{\mu}) \in \tilde Z_m.
\end{equation}

We now fix a direction $\breve{\bxi} \in \tilde Z_m$ and decompose
\[
\phi(t,\bx,\bxi) = \partial_{\bxi}\phi(t,\bx,\bxi/|\bxi|) \bxi = \partial_{\bxi}\phi(t,\bx,\breve{\bxi}) \bxi + h_{\breve{\bxi}}(t,\bx,\bxi),
\]
where the first equality is due to Euler's identity. By construction, $h_{\breve{\bxi}}(t,\bx,\cdot)$ is $1$-homogeneous and vanishes of second order on the half-line $\Rpos \breve{\bxi}$, as
\[
\partial_{\bxi} h_{\breve{\bxi}}(t,\bx,\bxi) = \partial_{\bxi}\phi(t,\bx,\bxi/|\bxi|) - \partial_{\bxi}\phi(t,\bx,\breve{\bxi});
\]
from this and \eqref{eq:est_phi_unif} one then deduces, for all $N \in \NN$ and $\alpha \in \NN^d$, the estimates
\begin{align*}
|(\breve{\bxi} \cdot \nabla_{\bxi})^N h_{\breve{\bxi}}(t,\bx,\bxi)| 
&\lesssim_{\kappa,r,T,N} \left|\frac{\bxi}{|\bxi|}-\breve{\bxi}\right|^2 |\bxi|^{1-N}, \\
|\partial_{\bxi} (\breve{\bxi} \cdot \nabla_{\bxi})^N h_{\breve{\bxi}}(t,\bx,\bxi)| 
&\lesssim_{\kappa,r,T,N} \left|\frac{\bxi}{|\bxi|}-\breve{\bxi}\right| |\bxi|^{-N}, \\
|\partial_{\bxi}^\alpha h_{\breve{\bxi}}(t,\bx,\bxi)|
&\lesssim_{\kappa,r,T,\alpha}  |\bxi|^{1-|\alpha|} 
\end{align*}
under the assumptions $|t|\leq T$, $\bx \in \overline{B}(0,r)$ and $\bxi \in \Omega_\kappa$. Thus
\begin{equation}\label{eq:SSSder}
|(\breve{\bxi} \cdot \nabla_{\bxi})^N h_{\breve{\bxi}}(t,\bx,\bxi)| 
\lesssim_{\kappa,r,T,N} 2^{-mN} , \qquad
|\partial_{\bxi}^\alpha h_{\breve{\bxi}}(t,\bx,\bxi)| 
\lesssim_{\kappa,r,T,\alpha} 2^{-m|\alpha|/2} 
\end{equation}
whenever $|t|\leq T$, $\bx \in \overline{B}(0,r)$ and $\bxi \in \supp q(t,\cdot) \cap \supp \chi_{m,\breve{\bxi}}$.

\smallskip

As a consequence, we also deduce that
\[
|e^{-ih_{\breve{\bxi}}} (\breve{\bxi} \cdot \nabla_{\bxi})^N e^{ih_{\breve{\bxi}}}| \lesssim_{\kappa,r,T,N} 2^{-mN}, \quad
|e^{-ih_{\breve{\bxi}}} \partial_{\bxi}^\alpha e^{ih_{\breve{\bxi}}}| \lesssim_{\kappa,r,T,\alpha} 2^{-m|\alpha|/2},
\]
whence also, by \eqref{eq:est_simb_unif} and \eqref{eq:der_sdd_0},
\begin{equation}\label{eq:est_ampl_SSS_0}
\begin{aligned}
|e^{-ih_{\breve{\bxi}}} (\breve{\bxi} \cdot \nabla_{\bxi})^N (e^{ih_{\breve{\bxi}}} \tilde q_{\breve{\bxi}})|
&\lesssim_{\cQ,\kappa,r,T,N}  2^{m(a-N)}, \\
|e^{-ih_{\breve{\bxi}}} \partial_{\bxi}^{\alpha} (e^{ih_{\breve{\bxi}}} \tilde q_{\breve{\bxi}})| 
&\lesssim_{\cQ,\kappa,r,T,\alpha}  2^{m(a-|\alpha|/2)}
\end{aligned}
\end{equation}
for all $N \in \NN$, where $\tilde q_{\breve{\bxi}} \defeq (2\pi)^{-n} q_{\breve{\bxi}} \Den_\phi$.

\smallskip

Notice now that
\[
(\breve{\bxi} \cdot \nabla_{\bxi})[\partial_{\bxi}\phi(t,\bx,\breve{\bxi})\bxi] = \partial_{\bxi}\phi(t,\bx,\breve{\bxi})\breve{\bxi} = \phi(t,\bx,\breve{\bxi})
\]
by $1$-homogeneity. Hence, by iterated integration by parts in the direction $\breve{\bxi}$,
\[\begin{split}
&I[q_{\breve{\bxi}}](t,\bx) \\
&= \int e^{i \partial_{\bxi}\phi(t,\bx,\breve{\bxi})\bxi} [e^{ih_{\breve{\bxi}}(t,\bx,\bxi)} \tilde q_{\breve{\bxi}}(t,\bxi)] \,d\bxi \\
&= (-i\phi(t,\bx,\breve{\bxi}))^{-N}  \int e^{i \partial_{\bxi}\phi(t,\bx,\breve{\bxi})\bxi} (\breve{\bxi} \cdot \nabla_{\bxi})^N [e^{ih_{\breve{\bxi}}(t,\bx,\bxi)} \tilde q_{\breve{\bxi}}(t,\bxi)] \,d\bxi \\
&= (-i\phi(t,\bx,\breve{\bxi}))^{-N} 
\int e^{i \phi(t,\bx,\bxi)} [e^{-ih_{\breve{\bxi}}(t,\bx,\bxi)} (\breve{\bxi} \cdot \nabla_{\bxi})^N [e^{ih_{\breve{\bxi}}(t,\bx,\bxi)} \tilde q_{\breve{\bxi}}(t,\bxi)]] \,d\bxi.
\end{split}\]
As $\Im\phi \geq 0$, from this expression, the estimate \eqref{eq:est_ampl_SSS_0} and the support conditions \eqref{eq:Q_supp}, \eqref{eq:Q_supp_m} and \eqref{eq:supp_sdd_0}, we finally deduce that
\begin{equation}\label{eq:est_SSS_rad_0}
\begin{split}
|I[q_{\breve{\bxi}}](t,\bx)| 
&\lesssim_{\cQ,\kappa,r,T,N} 2^{m[(d+1)/2+a]} 2^{-mN}  |\phi(t,\bx,\breve{\bxi})|^{-N} \\
&\leq 2^{m[(d+1)/2+a]} (2^{m} |(x-\breve{x}^{t}) \cdot \breve{\xi}^{t} + (u-\breve{u}^{t}) \cdot \breve{\mu}|)^{-N}
\end{split}
\end{equation}
where $\breve{x}^{t},\breve{u}^{t}, \breve{\xi}^{t}$ are the values of $x^{t},u^{t},\xi^{t}$ at $\bxi=\breve{\bxi}$, 
and we used that $|\phi| \geq \left|\Re\phi\right|$.

\smallskip

Recall now that $|\breve{\mu}| \simeq 1$ by \eqref{eq:normal_dir_0}. So we can complete $\breve{\mu}/|\breve{\mu}|$ to an orthonormal basis $\breve{\mu}/|\breve{\mu}|,\breve{\mu}^\perp_{1},\dots,\breve{\mu}^\perp_{d_2-1}$ of $\RR^{d_2}$. Thus, if we set $\breve{\bxi}^\perp_j \defeq (0,\breve{\mu}^\perp_j)$ for any $j=1,\dots,{d_2-1}$, then, by Lemma \ref{lem:phder}\ref{en:phder_grad}-\ref{en:phder_hess},
\[\begin{split}
|(\breve{\bxi}^\perp_j \cdot \nabla_{\bxi}) [ \partial_{\bxi} \phi(t,x,\breve{\bxi}) \bxi]|
&\geq |\Re (\breve{\bxi}^\perp_j \cdot \nabla_{\bxi}) [ \partial_{\bxi} \phi(t,x,\breve{\bxi}) \bxi] |\\
&= |\Re\partial_{\bxi} \phi(t,x,\breve{\bxi}) \breve{\bxi}^\perp_j| \\
&= |(\bx-\breve{\bx}^t)^T \Re\underline{\Phi}|_{\bxi=\breve{\bxi}} \, \breve{\bxi}^\perp_j| \\
&= |(x - \breve{x}^t) \cdot \breve{A}^t_j + (u-\breve{u}^t) \cdot \breve{\mu}^\perp_j|,
\end{split}\]
where $\breve{\bx}^t \defeq (\breve{x}^t,\breve{u}^t)$ and $\breve{A}^t_j \defeq \Re (\underline{\partial_x \partial_\mu \phi}|_{\bxi=\breve{\bxi}} \, \breve{\mu}^\perp_j)^T$. Arguing as before, by repeated integration by parts in the direction $\breve{\bxi}_j^\perp$ and using the second set of estimates in  \eqref{eq:SSSder} in place of the first ones, we also deduce
\begin{equation}\label{eq:est_SSS_ort_0}
|I[q_{\breve{\bxi}}](t,\bx)| 
\leq_{\cQ,\kappa,r,T,N} 2^{m[(d+1)/2+a]} (2^{m/2} |(x-\breve{x}^{t})\cdot\breve{A}^t_j   + (u-\breve{u}^{t})\cdot \breve{\mu}_j^\perp|)^{-N}
\end{equation}

We now notice that, as $\bxi$ ranges in $\Omega_{\kappa} \cap \supp \chi_{m,\breve{\bxi}}$, its direction $\bxi/|\bxi|$ stays in a $2^{-m/2}$-neighbourhood of $\breve{\bxi}$. Moreover, as $x^{t}$ is $0$-homogeneous in $\bxi$, its dependence on $\bxi$ is only through $\bxi/|\bxi|$, and $|\partial_{\bxi} x^{t}| \lesssim_{\kappa,T} 1$ on the unit sphere cap $\{ \bxi/|\bxi| \tc \bxi \in \Omega_{\kappa}\}$ provided $|t| \leq T$. From this we deduce that
\[
|x^{t}-\breve{x}^{t}| \lesssim_{\kappa,T} 2^{-m/2}
\]
whenever $\bxi \in \Omega_{\kappa} \cap \supp \chi_{m,\breve{\bxi}}$ and $|t| \leq T$.
Consequently, on the support of $q_{\breve{\bxi}}$,
\begin{equation}\label{eq:useIm}
\Im\phi \gtrsim |\mu| |x-x^{t}|^2 \gtrsim_{\kappa,T} 2^{\delta_0 m}
\end{equation}
whenever $|x-\breve{x}^{t}| > b_{\kappa,T} 2^{m(\delta_0-1)/2}$ for a sufficiently large $b_{\kappa,T} \in \Rpos$ and any $\delta_0 > 0$; in the first inequality in \eqref{eq:useIm}, the M\'etivier condition \eqref{eq:metivier} was used.
Hence
\[\begin{split}
|I[q_{\breve{\bxi}}](t,\bx)| \leq \int_{\RR^d} e^{-\Im \phi(t,\bx,\bxi)} |\tilde q_{\breve{\bxi}}(t,\bxi)| \,d\bxi
&\lesssim_{\cQ,\kappa,r,T} (2^m)^{a+d} e^{-c_{\kappa,T} 2^{\delta_0 m}} \\
&\lesssim_{a,\kappa,T,\delta_0,N} 2^{-mN}
\end{split}\]
for an appropriate $c_{\kappa,T} \in \Rpos$ and all $N \in \NN$; thus 
\begin{equation}\label{eq:local1est}
\|\chr_{\overline{B}(0,r) \setminus (\overline{B}_{\RR^{d_1}}(\breve{x}^{t},b_{\kappa,T} 2^{m(\delta_0-1)/2}) \times \RR^{d_2})} I[q_{\breve{\bxi}}](t,\cdot)\|_1 \lesssim_{\cQ,\kappa,r,T,\delta_0,N} 2^{-mN}
\end{equation}
for all $N \in \NN$, showing that the contribution of this part is negligible.

\smallskip

As for the remaining part, we can use the estimates \eqref{eq:est_SSS_rad_0} and \eqref{eq:est_SSS_ort_0} with $N=0$ and $N=d_2+1$ to deduce
\begin{equation}\label{eq:nonlocal1est}
\begin{split}
&\|\chr_{\overline{B}(0,r)} \chr_{\overline{B}_{\RR^{d_1}}(\breve{x}^{t},b_{\kappa,T} 2^{m(\delta_0-1)/2}) \times \RR^{d_2}} 
I[q_{\breve{\bxi}}](t,\cdot)\|_1 \\ 
&\lesssim_{\cQ,\kappa,r,T} 2^{m[(d+1)/2+a]} \\
&\qquad\times \int_{\RR^{d_2}} \int_{|x-\breve{x}^{t}| \leq b_{\kappa,T} 2^{m(\delta_0-1)/2}} \Biggl( 1+2^m|(x-\breve{x}^{t}) \cdot \breve{\xi}^{t} + (u-\breve{u}^{t}) \cdot \breve{\mu}| \\
&\qquad\qquad+ 2^{m/2}\sum_{j=1}^{d_2-1} |(x-\breve{x}^t) \breve{A}^t_j + (u-\breve{u}^t)\cdot \breve{\mu}^\perp_j|\Biggr)^{-d_2-1} \,dx \,du \\
&\lesssim_{\kappa,T}  2^{m[(d+1)/2+a]}2^{md_1(\delta_0-1)/2} 2^{-m(d_2+1)/2}
= 2^{m(a+\delta_0 d_1/2)}, 
\end{split}
\end{equation}
where we used the fact that, as $|\breve{\mu}| \simeq_\kappa  1$ by \eqref{eq:normal_dir_0}, the change of variables
\[
u \mapsto (\breve{\mu} \cdot (u-\breve{u}^t),\breve{\mu}^\perp_1 \cdot (u-\breve{u}^t),\dots,\breve{\mu}^\perp_{d_2-1} \cdot (u-\breve{u}^t))\
\]
has Jacobian determinant $\simeq_\kappa 1$ in absolute value.

\smallskip

In light of \eqref{eq:count_sdd_0} and \eqref{eq:dec_int_sdd_0}, we can now sum the above $L^1$-norm estimates \eqref{eq:local1est} and
\eqref{eq:nonlocal1est} for the FIO kernels $I[q_{\breve{\bxi}}]$ to obtain, up to a relabelling of $\delta_0$, the desired estimate for $I[q]$.
\end{proof}

\subsection{Large-time \texorpdfstring{$L^1$}{L1}-estimates}

The estimate in Proposition \ref{prp:sdd_0} is not sufficient for our purposes, as the implicit constant depends on the bound $T$ on $|t|$ and therefore may blow up for large time.

Indeed, recall from Corollary \ref{cor:flow_alt} that the Hamiltonian flow, hence also our phase function $\phi$, depends in a fundamental way on the quantity $\theta = t|\mu|/(2|\xi|)$, and derivatives of $\theta$ with respect to $\bxi$ (as needed in integration-by-parts arguments) would lead to unfavourable losses of powers of $t$, eventually yielding unsuitable estimates for our FIO kernels  for large time.

We will now show how this problem can be overcome, by exploiting a certain ``quasi-periodicity'' in $\theta$ of crucial terms appearing in the phase function $\phi$ and the density $\Den_\phi$ (see the formulas in \eqref{eq:flow_2step} and \eqref{eq:det_2step}).
More precisely, we shall make fundamental use of the following uniform bounds.

\begin{lem}\label{lem:der_modJmu}
The following estimates hold.
\begin{enumerate}[label=(\roman*)]
\item\label{en:der_modJmu_nonmod} For any $v,\mu \in \RR^{d_2}$ and all $\alpha \in \NN^{d_2}$,
\begin{align*}
\|\partial_\mu^\alpha \exp(J_v + J_\mu)\| \lesssim_\alpha 1,\\
\|\partial_\mu^\alpha \cosh(J_v + J_\mu)\| \lesssim_\alpha 1,\\
\|\partial_\mu^\alpha \sinh(J_v + J_\mu)\| \lesssim_\alpha 1,
\end{align*}
where the implicit constants are independent of $\mu$ and $v$.
\item\label{en:der_modJmu_mod} Assume that $\rk J_\mu$ is constant for $\mu \neq 0$. For any $v,\mu \in \RR^{d_2}$ with $v+\mu \neq 0$ and all $\alpha \in \NN^{d_2}$,
\begin{align*}
\|\partial_\mu^\alpha |J_v + J_\mu|\| &\lesssim_{\alpha} |v+\mu|^{1-|\alpha|},\\
|\partial_\mu^{\alpha} \tr |J_v+J_\mu|| &\lesssim_{\alpha} |v+\mu|^{1-|\alpha|},\\
|\partial_\mu^{\alpha} \exp(\pm i\tr|J_v+J_\mu|)| &\lesssim_{\alpha} 1 + \min\{1,|v+\mu|\}^{1-|\alpha|},
\end{align*}
where the implicit constants are independent of $\mu$ and $v$.
\end{enumerate}
\end{lem}
\begin{proof}
\ref{en:der_modJmu_nonmod}. As $\cosh$ and $\sinh$ can be expressed in terms of $\exp$, it is enough to prove the first estimate. Moreover, as $J_\mu$ is linear in $\mu$, by a change of variables we may assume $v = 0$. A well-known formula for the derivative of the matrix exponential (see, e.g., \cite[Example IX.4.2(v)]{Bhatia} or \cite[Section IX.2.1]{Kato}) gives that, for all $t \in \RR$, $\mu \in \RR^{d_1}$ and $j=1,\dots,d_2$,
\[
\partial_{\mu_j} \exp(t J_\mu) = \int_0^t \exp((t-s)J_\mu) (\partial_{\mu_j} J_{\mu}) \exp(s J_\mu) \,ds;
\]
this formula can clearly be iterated, thus obtaining an expression for $\partial^\alpha_\mu \exp(J_\mu)$ as a sum of multiple integrals of products of exponentials $\exp(sJ_\mu)$ and derivatives of $J_\mu$. Notice now that $\exp(sJ_{\mu})$ is an isometry for any $s \in \RR$, as $J_{\mu}$ is skew-symmetric; moreover, $\partial_{\mu_j} J_\mu$ is constant in $\mu$, as $J_\mu$ is linear in $\mu$, and therefore higher-order derivatives of $J_\mu$ vanish. This readily gives the desired estimate for $\partial_\mu^\alpha \exp(J_\mu)$.

\smallskip

\ref{en:der_modJmu_mod}. In the case $v=0$, the above estimates are easy consequences of the $1$-homogeneity and smoothness of $|J_\mu|$ and $\tr|J_\mu|$ as functions of $\mu \in \dot\RR^{d_2}$ (see Proposition \ref{prp:modJmu_analytic}). The general case follows from a change of variables and the linearity of $\mu \mapsto J_\mu$.  
\end{proof}

Due to the mentioned dependence on $\theta = t|\mu|/(2|\xi|)$ of the phase function and the density, it is natural to perform the change of variables $t\mu \mapsto \mu$ in the oscillatory integrals $I[q]$ under consideration, which makes the ``harmful factor'' $t$ disappear. This scaling in $\mu$, however, results in a change of frequency localisation: namely, assuming $|t| \simeq T$, if the original frequency localisation was $|\xi| \simeq |\mu| \simeq 2^m$, after the scaling the localisation becomes $|\xi| \simeq 2^m$, $|\mu| \simeq 2^m T$.

In order to restore the frequency localisation to ``standard'' $\bxi$-boxes of size $2^m$, we perform a splitting in the variable $\mu$ (see Figure \ref{fig:freq_mushear}). Specifically, we first split the annulus $|\mu| \simeq T2^m$ into sectors of angular aperture $T^{-1}$ (so the transversal size of each sector is $2^m$); each of these sectors is then further split in the radial direction into boxes of radial length $2^m$. Finally, the centre of each $\mu$-box is translated to the origin, eventually resulting in a frequency localisation $|\mu| \lesssim 2^m \simeq |\xi|$.

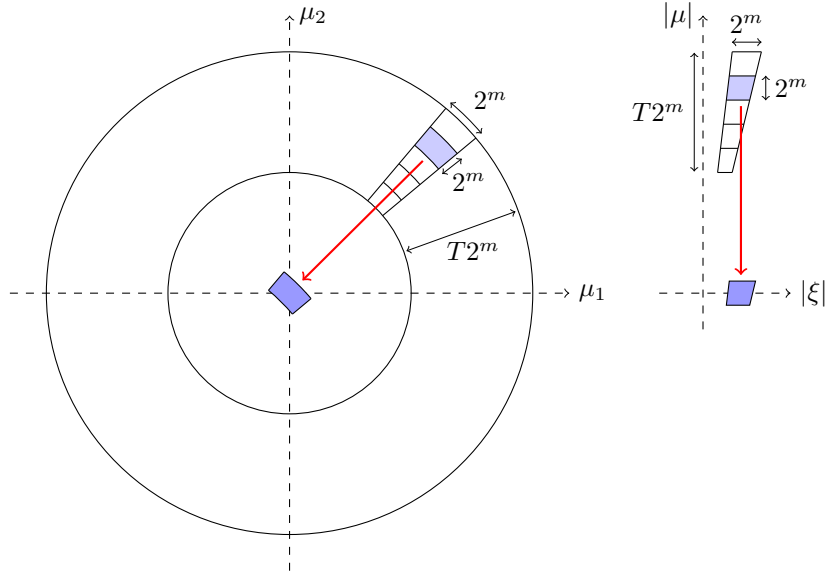
\begin{figure}

\begin{tikzpicture}[scale=1.6]

\draw[dashed,->] (-2.3,0) -- (2.3,0) node[right] {$\mu_1$};
\draw[dashed,->] (0,-2.3) -- (0,2.3) node[right] {$\mu_2$};

\fill[fill=blue!20] (40:1.6) arc (40:50:1.6) -- (50:1.8) arc (50:40:1.8) -- cycle;

\begin{scope}[shift={(-135:1.7)}]
\fill[fill=blue!40] (40:1.6) arc (40:50:1.6) -- (50:1.8) arc (50:40:1.8) -- cycle;
\draw ([shift={(40:1.6)}]0,0) arc (40:50:1.6);
\draw ([shift={(40:1.8)}]0,0) arc (40:50:1.8);
\draw (40:1.6) -- (40:1.8);
\draw (50:1.8) -- (50:1.6);
\end{scope}

\draw (0,0) circle (2);
\draw (0,0) circle (1);

\draw (40:1) -- (40:2);
\draw (50:1) -- (50:2);

\draw ([shift={(40:1.2)}]0,0) arc (40:50:1.2);
\draw ([shift={(40:1.4)}]0,0) arc (40:50:1.4);
\draw ([shift={(40:1.6)}]0,0) arc (40:50:1.6);
\draw ([shift={(40:1.8)}]0,0) arc (40:50:1.8);

\draw[very thin,<->] ([shift={(40:2.06)}]0,0) arc (40:50:2.06);
\node at (44:2.3) {$2^m$}; 
\draw[very thin,<->] (38.2:1.6) -- (38.2:1.8);
\node at (32:1.74) {$2^m$}; 
\draw[very thin,<->] (20:1.025) -- (20:1.975);
\node at (13:1.55) {$T 2^m$};

\draw[thick,->,red] (45:1.55) -- (45:0.15);

\begin{scope}[shift={(3.4,0)},xscale=1.2]

\draw[dashed,->] (0,-0.3) -- (0,2.3) node[left] {$|\mu|$};
\draw[dashed,->] (-0.3,0) -- (0.6,0) node[right] {$|\xi|$};

\fill[fill=blue!20] (.16,1.6) -- (.32,1.6) -- (.36,1.8) -- (.18,1.8) -- cycle;

\begin{scope}[shift={(0,-1.7)}]
\fill[fill=blue!40] (.16,1.6) -- (.32,1.6) -- (.36,1.8) -- (.18,1.8) -- cycle;
\draw (.16,1.6) -- (.32,1.6) -- (.36,1.8) -- (.18,1.8) -- cycle;
\end{scope}

\draw (.1,1) -- (.2,2);
\draw (.2,1) -- (.4,2);
\draw (.1,1) -- (.2,1);
\draw (.12,1.2) -- (.24,1.2);
\draw (.14,1.4) -- (.28,1.4);
\draw (.16,1.6) -- (.32,1.6);
\draw (.18,1.8) -- (.36,1.8);
\draw (.2,2) -- (.4,2);

\draw[very thin,<->] (-.06,1.01) -- (-.06,1.99) node[midway,left] {$T 2^m$};
\draw[very thin,<->] (.43,1.6) -- (.43,1.8) node[midway,right] {$2^m$}; 

\draw[very thin,<->] (.2,2.08) -- (.4,2.08) node[midway,above] {$2^m$}; 

\draw[thick,->,red] (.26,1.55) -- (.26,0.15);

\end{scope}

\end{tikzpicture}

\caption{Splitting and shifting of $\mu$ after scaling by $t$.}
\label{fig:freq_mushear}

\end{figure}

The decomposition of the oscillatory integral $I[q]$ corresponding to the above change of variables and splitting in $\mu$ is expressed by \eqref{eq:dec_periodic} below. Each of the resulting pieces $I_{k,v}[q_{k,v,T,\kappa}]$ turns out to be amenable to an analysis analogous to that in Proposition \ref{prp:sdd_0} above, eventually yielding the desired $L^1$-estimates.

Notice that the number of pieces resulting from this splitting is comparable to $T^{d_2}$, and therefore is compensated by the Jacobian factor $|t|^{-d_2}$ due to the change of variables in $\mu$. Moreover, in light of Lemma \ref{lem:der_modJmu}, the aforementioned translation in $\mu$ does not affect the uniformity of the derivative estimates for the crucial terms in the phase function $\phi$ and the density $\Den_\phi$.

Our analysis does actually lead to a $|t|^{d_2-1/2}$ growth of the resulting estimate for $I[q]$: this is due, on the one hand, to the $|t|^{1/2}$ growth of the density $\Den_\phi$, and on the other hand, to a coarser spatial localisation in $u$ corresponding to the aforementioned $\mu$-splitting (specifically, $|u^\perp_v| \lesssim |t|$, where $u^\perp_v$ denotes the component of $u$ transversal to the direction $v$ of the $\mu$-box). Nevertheless, thanks to our assumption \eqref{eq:cond_lm} on the time and frequency localisation parameters, this growth in time is still compatible with the desired wave propagator estimate, as $d_2-1/2 \leq (d-1)/2$ on a M\'etivier group (see \eqref{eq:metivier_dim_ineq}).

\smallskip

For any unit vector $v \in \RR^{d_2}$ and all $\mu \in \RR^{d_2}$, we write
\[
\mu_v^\parallel \defeq \mu \cdot v, \qquad \mu_v^\perp \defeq \mu - \mu_v^\parallel v,
\]
so we have the decomposition $\mu = \mu_v^\parallel v + \mu_v^\perp$.

\begin{prp}\label{prp:dec_periodic}
Assume that $G$ is a M\'etivier group.
Let $\kappa > 1$, $T \geq 16\kappa^2$ be given. There exists a finite set $V_{T,\kappa}$ of unit vectors in $\RR^{d_2}$ such that
\begin{equation}\label{eq:count_sdd_2ndlayer}
\sharp V_{T,\kappa} \lesssim_\kappa T^{d_2-1} 
\end{equation}
and the following hold.
Let $a \in \RR$, and let $\cQ$ be a $(0,a,\kappa)$-bounded family of symbols. 
Then, for any $q \in \cQ$, and any $t \in \RR$ with $\kappa^{-1} \leq |t/T| \leq \kappa$,
\begin{equation}\label{eq:dec_periodic}
I[q](t,\bx) = |t|^{1/2-d_2} \sum_{k \geq 2} \sum_{v \in V_{T,\kappa}} I_{k,v}[q_{k,v,T,\kappa}](t,\bx),
\end{equation}
where
\begin{align}
\label{eq:int_kv}
I_{k,v}[A](t,\bx) &\defeq \int_{\RR^d} e^{i \phi_{k,v}(t,\bx,\bxi)} A(t,\bxi) \,d\bxi, \\ 
\label{eq:phase_tk_orig}
\phi_{k,v}(t,\bx,\bxi) &\defeq \phi(t,\bx,(\xi,\mu^{t,k,v}(\bxi)), \\
\label{eq:mu_tk}
\mu^{t,k,v}(\bxi) &\defeq \frac{2k|\xi|v + \mu}{t} = \frac{\mu_v^\parallel + 2 k|\xi|}{t} v+ \frac{\mu_v^\perp}{t}, 
\end{align}
while the symbols $q_{k,v,T,\kappa}$ for $k \geq 2$ and $v \in V_{T}$ satisfy the bounds
\begin{equation}\label{eq:est_symb_sdd_2ndlayer}
|\partial_t^h \partial_{\bxi}^\alpha q_{k,v,T,\kappa}(t,\bxi)| \lesssim_{\cQ,\kappa,h,\alpha} (1+|t|)^{-h} (1+|\bxi|)^{a-|\alpha|}
\end{equation}
for all $h \in \NN$, $\alpha \in \NN^{d}$, and moreover
\begin{equation}\label{eq:supp_symb_sdd_2ndlayer}
\supp q_{k,v,T,\kappa}(t,\cdot) \subseteq (\overline{\pi_{\RR^{d_1}} \supp q(t,\cdot)} \times \RR^{d_2}) \cap \Omega_{\kappa,0},
\end{equation}
where $\pi_{\RR^{d_1}} : \RR^{d_1} \times \RR^{d_2} \to \RR^{d_1}$ is the projection onto the first factor, and
\begin{equation}\label{eq:Omegakappa0}
\Omega_{\kappa,0} \defeq \{ \bxi \tc |\xi| \geq \kappa^{-1}, \ |\mu|/|\xi| \leq 5/2 \}.
\end{equation}
In addition, $q_{k,v,T,\kappa}(t,\cdot)$ and $I_{k,v}[q_{k,v,T,\kappa}](t,\cdot)$ vanish identically unless
\begin{equation}\label{eq:cond_tk_vanish}
\frac{|t|}{8 \kappa} < k < \kappa |t|.
\end{equation}
\end{prp}
\begin{proof}
Much as in the proof of Proposition \ref{prp:sdd_0},
we can split $\dot\RR^{d_2}$ into sectors of aperture $\simeq T^{-1}$ by means of a partition of unity. Namely, we can find a finite set $V_{T,\kappa}$ of unit vectors in $\RR^{d_2}$ such that \eqref{eq:count_sdd_2ndlayer} holds, and
and a smooth $0$-homogeneous partition of unity $\{\chi_{T,\kappa,v}\}_{v \in V_{T,\kappa}}$ on $\dot\RR^{d_2}$ such that
\begin{equation}\label{eq:supp_sdd_2ndlayer}
\supp \chi_{T,\kappa,v} \subseteq \{ \mu \tc \mu_v^\parallel > 0, \  |\mu_v^\perp|/|\mu| \leq c_\kappa T^{-1} \}
\end{equation}
for an appropriately small $c_\kappa \in (0,1/2]$, to be fixed later, and
\begin{equation}\label{eq:der_sdd_2ndlayer}
|(v \cdot \nabla_\mu)^N \partial_{\mu}^\alpha \chi_{T,\kappa,v}(\mu) | \lesssim_{\kappa,N,\alpha} T^{|\alpha|} |\mu|^{-|\alpha|-N}
\end{equation}
for all $v \in V_{T,\kappa}$, $\alpha \in \NN^{d_2}$ and $N \in \NN$, where the implicit constants do not depend on $T$.
Accordingly, we can split, for any $t \in \RR$ with $\kappa^{-1} \leq |t/T| \leq \kappa$,
\begin{equation}\label{eq:dec_int_sdd_2ndlayer_prelim}
I[q](t,\bx) = \sum_{v \in V_{T,\kappa}} I[q_{v,T,\kappa}](t,\bx),
\end{equation}
where
\[
q_{v,T,\kappa}(t,\bxi) \defeq q(t,\bxi) \chi_{T,\kappa,v}(t\mu) \chi_\kappa(|t|/T)=q(t,\bxi) \chi_{T,\kappa,v}(\pm\mu) \chi_\kappa(|t|/T),
\]
with $\pm \defeq \sgn t$, and $\chi_\kappa \in C^\infty_c(\Rpos)$ is a smooth cutoff such that
\begin{equation}\label{eq:time_cutoff}
\chi_\kappa|_{[\kappa^{-1},\kappa]} = 1, \qquad \supp \chi_\kappa \subseteq [(2\kappa)^{-1},2\kappa].
\end{equation}

Choose now an even cutoff function $\chi_+ \in C^\infty_c(\RR)$ with the properties
\begin{equation}\label{eq:additivecutoff}
0 \leq \chi_+ \leq 1, \qquad \supp \chi_+ \subseteq [-1,1], \qquad \sum_{k \in \ZZ} \chi_+(s-k) = 1 \quad\forall s \in \RR. 
\end{equation}
Accordingly, we split
\begin{multline}\label{eq:dec_int_sdd_2ndlayer}
I[q_{v,T,\kappa}](t,\bx) \\ = \sum_{k \in \ZZ}
(2\pi)^{-d} \int_{\RR^d} e^{i \phi(t,\bx,\bxi)}  q_{v,T,\kappa}(t,\bxi) \,  \chi_+(t\mu_v^\parallel/(2|\xi|)-k) \, \Den_\phi(t,\bxi) \,d\bxi.
\end{multline}

In each of the integrals in the right-hand side of \eqref{eq:dec_int_sdd_2ndlayer} we now make a different ($\xi$-dependent) change of variables from $\mu$ to  $\tilde\mu$, given by
\[
\mu = \frac{\tilde\mu + 2k|\xi|v}{t}, \qquad\text{i.e.,}\qquad \mu_v^\parallel = \frac{\tilde\mu_v^\parallel + 2k|\xi|}{t}, \qquad \mu_v^\perp = \frac{\tilde\mu_v^\perp}{t};
\]
this leads to a change of variable from $\bxi=(\xi,\mu)$ to $\tilde\bxi \defeq (\xi,\tilde\mu)$.
Then in particular
\begin{equation}\label{eq:jacobiantildemu}
d\mu = |t|^{-d_2} \,d\tilde\mu.
\end{equation}
In analogy with our previous quantity $\theta=t|\mu|/(2|\xi|)$, here we define
\[
\vartheta_v=\vartheta_v(\tilde\bxi) \defeq \frac{\tilde\mu_v^\parallel}{2|\xi|}.
\]
We then may write
\[
\frac{t\mu_v^\parallel}{2|\xi|} = \frac{\tilde\mu_v^\parallel}{2|\xi|} + k= \vartheta_v + k,
\]
i.e.,
\[
\mu=2\frac{\vartheta_v + k}{t} |\xi| v+\frac{1}{t} \tilde\mu_v^\perp.
\]

If we now define
\begin{equation}\label{eq:mutkv_formula}
\mu^{t,k,v}(\tilde\bxi) \defeq 2\frac{\vartheta_v + k}{t} |\xi| v+\frac{1}{t} \tilde\mu_v^\perp, \qquad\pm \defeq \sgn t,
\end{equation}
then, by using the notation \eqref{eq:int_kv} and taking \eqref{eq:jacobiantildemu} into consideration, we can rewrite \eqref{eq:dec_int_sdd_2ndlayer} as
\begin{equation}\label{eq:dec_int_sdd_2ndlayer_chvar}
I[q_{v,T,\kappa}](t,\bx) = |t|^{1/2-d_2} \sum_{k \in \ZZ} I_{k,v}[q_{k,v,T,\kappa}],
\end{equation}
where
\begin{equation}\label{eq:new_symb_all}
q_{k,v,T,\kappa}(t,\tilde\bxi) 
\defeq (2\pi)^{-d}  \tilde q_{k,v}(t,\tilde\bxi)
\chi_{\kappa}(|t|/T) \, \tilde\chi_{k,T,\kappa,v}(t,\tilde\bxi) \, \chi_+(\vartheta_v) \, \Den_{\phi,k,v}(t,\tilde\bxi)
\end{equation}
and
\begin{align}
\label{eq:new_symb_main}
\tilde q_{k,v}(t,\tilde\bxi) &\defeq q(t,(\xi,\mu^{t,k,v}(\tilde\bxi))),\\
\label{eq:new_symb_sddcutoff}
\tilde\chi_{k,T,\kappa,v}(t,\tilde\bxi) &\defeq \chi_{T,\kappa,v}(\pm \mu^{t,k,v}(\tilde\bxi)),\\
\label{eq:new_symb_density}
\Den_{\phi,k,v}(t,\tilde\bxi) &\defeq |t|^{-1/2} \Den_\phi(t,(\xi,\mu^{t,k,v}(\tilde\bxi))).
\end{align}

Notice that \eqref{eq:mutkv_formula} is consistent with \eqref{eq:mu_tk}: in the latter, the coordinates $\tilde \mu$ were renamed back to $\mu$. Moreover, due to the support conditions \eqref{eq:Q_supp}, \eqref{eq:supp_sdd_2ndlayer}, \eqref{eq:time_cutoff} and \eqref{eq:additivecutoff}, the quantity $q_{k,v,T,\kappa}(t,\tilde\bxi)$ vanishes unless
\begin{gather*}
(2\kappa)^{-1} \leq |t/T| \leq 2\kappa , \quad |\xi| \geq \kappa^{-1}, \quad |\vartheta_v| = |\tilde\mu_v^\parallel|/(2|\xi|) < 1, \quad \vartheta_v +  k > 0, \\
\kappa^{-1} |\xi| \leq |\mu^{t,k,v}| \leq \kappa |\xi|,\\
\left|\frac{1}{t} \tilde\mu_v^\perp\right| \leq c_\kappa T^{-1} \left|\mu^{t,k,v}\right|.
\end{gather*}
As $c_\kappa T^{-1/2} < 1/2$, from the above we deduce that
\[
\left|2\frac{\vartheta_v +  k}{t} |\xi|\right| \leq |\mu^{t,k,v}| \leq 2 \left|2\frac{\vartheta_v +  k}{t} |\xi|\right|
\]
and
\[
\frac{1}{4\kappa} \leq \frac{\vartheta_v+ k}{|t|} \leq \frac{\kappa}{2}.
\]
Since $T \geq 16\kappa^2$, under the above conditions we also have $|t| \geq 8\kappa$ and $|\vartheta_v| < 1$; the above inequalities thus imply the constraint \eqref{eq:cond_tk_vanish}, and in particular the bound $k \geq 2$. Combining this restriction with \eqref{eq:dec_int_sdd_2ndlayer_chvar} and \eqref{eq:dec_int_sdd_2ndlayer_prelim} completes the proof of the decomposition \eqref{eq:dec_periodic}. In addition, the above conditions also imply that
\[
|\tilde\mu_v^\perp| \leq 2 c_\kappa \kappa |\xi|;
\]
as we already know that $|\tilde\mu_v^\parallel| \leq 2|\xi|$,
by choosing $c_\kappa$ small enough we may ensure that 
\[
|\tilde \mu| \leq 5 |\xi|/2,
\]
thus obtaining the support condition \eqref{eq:supp_symb_sdd_2ndlayer}.

\smallskip

Now, from the above discussion it follows that, on the support of $q_{k,v,T,\kappa}$,
\begin{equation}\label{eq:rough_tk_cond}
1 \lesssim_\kappa T \simeq_\kappa |t| \simeq_\kappa k \simeq_\kappa \vartheta_v + k.
\end{equation}
By using the bounds \eqref{eq:rough_tk_cond}, 
one readily sees from \eqref{eq:mu_tk} that, on the support of $q_{k,v,T,\kappa}$,
\[
|\partial_t^h \partial_{\tilde\bxi}^\alpha  \mu^{t,k,v}(\tilde\bxi)| \lesssim_{\kappa,h,\alpha} (1+|t|)^{-h} (1+|\xi|)^{1-|\alpha|}
\]
for all $\alpha \in \NN^{d}$, $h \in \NN$. Thus, 
from \eqref{eq:new_symb_main} and \eqref{eq:Q_bound} one readily derives that, on the support of $q_{k,v,T,\kappa}$,
\begin{equation}\label{eq:new_symb_main_est}
|\partial_t^h \partial_{\tilde\bxi}^\alpha \tilde q_{k,v}(t,\tilde\bxi)| \lesssim_{\cQ,\kappa,h,\alpha} (1+|t|)^{-h} (1+|\xi|)^{a-|\alpha|}
\end{equation}
for all $\alpha \in \NN^{d}$, $h \in \NN$. 
Moreover, by \eqref{eq:new_symb_sddcutoff} and \eqref{eq:mu_tk},
\begin{equation}\label{eq:new_symb_sddcutoff_expl}
\tilde\chi_{k,T,\kappa,v}(t,\tilde\bxi) = \chi_{T,\kappa,v}\left(\frac{\mu + 2 k|\xi|v}{|t|}\right),
\end{equation}
and from \eqref{eq:der_sdd_2ndlayer} one deduces that, on the support of $q_{k,v,T,\kappa}$,
\begin{equation}\label{eq:new_symb_sddcutoff_est}
|\partial_t^h \partial_{\tilde\bxi}^\alpha \tilde\chi_{k,T,\kappa,v}(t,\tilde\bxi)| \lesssim_{\kappa,h,\alpha} (1+|t|)^{-h} (1+|\xi|)^{a-|\alpha|}
\end{equation}
for all $\alpha \in \NN^{d}$, $h \in \NN$: we point out that, when differentiating \eqref{eq:new_symb_sddcutoff_expl} with respect to $\xi$, the Chain Rule only produces derivatives of $\chi_{T,\kappa,v}$ in the direction of $v$, so by \eqref{eq:der_sdd_2ndlayer} there are no extra $T$ factors; when differentiating \eqref{eq:new_symb_sddcutoff_expl} with respect to $\mu$ or $t$, instead, the factors $T$ from \eqref{eq:der_sdd_2ndlayer} are compensated by factors $|t|^{-1}$ from the corresponding derivatives of the argument of $\chi_{T,\kappa,v}$ in \eqref{eq:new_symb_sddcutoff_expl}. 

\smallskip

Notice now that, by \eqref{eq:mu_tk},
\[
\frac{t\mu^{t,k,v}(\tilde\bxi)}{2|\xi|} = kv+\frac{\tilde\mu}{2|\xi|};
\]
so, by \eqref{eq:det_2step} and \eqref{eq:new_symb_density},
\[
\Den_{\phi,k,v}(t,\tilde\bxi) = \exp(\mp i\tr|J_{kv} + J_{\tilde\mu/(2|\xi|)}|/2) \sqrt{\frac{1\pm i \langle |J_{kv} + J_{\tilde\mu/(2|\xi|)}| \bar\xi, \bar\xi \rangle}{|t|}}.
\]
From the above discussion it follows that $|\xi| \gtrsim_\kappa 1$, $|\tilde\mu| \lesssim |\xi|$ and $|kv+\tilde\mu/(2|\xi|)| \simeq_\kappa k \simeq |t| \gtrsim_\kappa 1$ on the support of $q_{k,v,T,\kappa}$; these estimates, together with Lemma \ref{lem:der_modJmu}\ref{en:der_modJmu_mod} and the Chain Rule, readily imply that, on the support of $q_{k,v,T,\kappa}$,
\begin{equation}\label{eq:new_symb_density_est}
|\partial_t^h \partial_{\tilde\bxi}^\alpha \Den_{\phi,k,v}(t,\tilde\bxi)| \lesssim_{\kappa,h,\alpha} (1+|t|)^{-h} (1+|\xi|)^{-|\alpha|}
\end{equation}
for all $\alpha \in \NN^{d}$ and $h \in \NN$.

\smallskip

In light of \eqref{eq:new_symb_all}, combining the estimates \eqref{eq:new_symb_main_est}, \eqref{eq:new_symb_sddcutoff_est} and \eqref{eq:new_symb_density_est} finally gives the estimates \eqref{eq:est_symb_sdd_2ndlayer}.
\end{proof}

According to Proposition \ref{prp:dec_periodic}, we can split, as in \eqref{eq:dec_periodic}, the FIO kernel $I[q]$ into a sum of oscillatory integrals $I_{k,v}[q_{k,v,T,\kappa}]$, whose amplitudes $q_{k,v,T,\kappa}$, according to \eqref{eq:est_symb_sdd_2ndlayer}, constitute a bounded subset of $S^0(\RR) \otimes S^{a}(\RR^d)$, provided $q$ ranges in a $(0,a,\kappa)$-bounded family of symbols.

In order to obtain spatial $L^1$-bounds for the $I_{k,v}[q_{k,v,T,\kappa}]$, ideally we would also like that the derivatives $\partial_{\bxi} \phi_{k,v}$ of the phase functions defined in \eqref{eq:phase_tk_orig} behave like a bounded family of $S^0(\RR^d)$-symbols with respect to the frequency variable $\bxi$, at least on the supports of the corresponding $q_{k,v,T,\kappa}$. The following proposition shows that the $\bx$-independent terms in these phase functions indeed satisfy such bounds.

In addition, the identity \eqref{eq:dxi_re_phase_tk} below, which originates from the properties of the geodesic flow discussed in Lemma \ref{lem:symplectic}, will greatly simplify computations of $\bxi$-derivatives of these phase functions.

\begin{prp}\label{prp:phase_tkv_bounds}
Assume that $G$ is a M\'etivier group.
With the notation of Proposition \ref{prp:dec_periodic}, we can write
\begin{multline}\label{eq:phase_tk}
\phi_{k,v} = (x - x^{t,k,v}) \cdot \xi^{t,k,v} + (u-u^{t,k,v}) \cdot \mu^{t,k,v} \\
+ \frac{i}{4} \langle|J_{\mu^{t,k,v}}| (x-x^{t,k,v}), x-x^{t,k,v}\rangle,
\end{multline}
where
\[
x^{t,k,v}(\bxi) \defeq x^t(\xi,\mu^{t,k,v}(\bxi)), \  u^{t,k,v}(\bxi) \defeq u^t(\xi,\mu^{t,k,v}(\bxi)), \  \xi^{t,k,v}(\bxi) \defeq \xi^t(\xi,\mu^{t,k,v}(\bxi)).
\]
Moreover, under the assumptions
\begin{equation}\label{eq:cond_tkv}
\begin{gathered}
\bxi \in \tilde\Omega_{\kappa,0} \defeq \{ \bxi \tc |\xi| \geq \kappa^{-1},\  |\mu|/|\xi| \leq 3 \} , \\
k \geq 2, \qquad \kappa^{-1} \leq |t|/k \leq \kappa,
\end{gathered}
\end{equation}
we have
\begin{equation}
\label{eq:bd_mu_tk}
|v \cdot \mu^{t,k,v}(\bxi)| \simeq_\kappa |\mu^{t,k,v}(\bxi)| \simeq_\kappa |\xi| \simeq_\kappa 1+|\bxi|
\end{equation}
and
\begin{align}
\label{eq:bd_mu_tk_der}
|\partial_t^{h} \partial_{\bxi}^\alpha \mu^{t,k,v}(\bxi)| &\lesssim_{\kappa,h,\alpha} (1+|t|)^{-h} (1+|\bxi|)^{1-|\alpha|},\\
\label{eq:bd_x_tk_der}
|\partial_t^{h} \partial_{\bxi}^\alpha x^{t,k,v}(\bxi)| &\lesssim_{\kappa,h,\alpha} (1+|t|)^{-h} (1+|\bxi|)^{-|\alpha|},\\
\label{eq:bd_xi_tk_der}
|\partial_t^{h} \partial_{\bxi}^\alpha \xi^{t,k,v}(\bxi)| &\lesssim_{\kappa,h,\alpha} (1+|t|)^{-h} (1+|\bxi|)^{1-|\alpha|}
\end{align}
for all $\alpha \in \NN^d$, $h \in \NN$. Furthermore, we can write
\begin{equation}\label{eq:umutkv}
(u-u^{t,k,v}) \cdot \mu^{t,k,v} = \frac{u_v^\perp}{t} \cdot \mu_v^\perp + 2|\xi| \frac{\vartheta_v + k}{t} \left(u_v^\parallel - \frac{t^2}{4 k}\right) + \rho^{t,k,v},
\end{equation}
where $\vartheta_v \defeq \mu^\parallel_v/(2|\xi|)$ and, under the assumptions \eqref{eq:cond_tkv},
\begin{equation}\label{eq:bd_rho_tk_der}
|\partial_t^{h} \partial_{\bxi}^\alpha \rho^{t,k,v}(\bxi)| \lesssim_{\kappa,h,\alpha} (1+|t|)^{-h} (1+|\bxi|)^{1-|\alpha|}
\end{equation}
for all $\alpha \in \NN^d$, $h \in \NN$.
In addition, with the notation $\bx^{t,k,v} = (x^{t,k,v},u^{t,k,v})$, $\bxi^{t,k,v} = (\xi^{t,k,v},\mu^{t,k,v})$,
\begin{equation}\label{eq:dxi_re_phase_tk}
\Re \partial_{\bxi} \phi_{k,v}(t,\bx,\bxi) = (\bx-\bx^{t,k,v})^T \partial_{\bxi} \bxi^{t,k,v}.
\end{equation}
\end{prp}
\begin{proof}
The formula \eqref{eq:phase_tk} immediately follows from \eqref{eq:phase} and \eqref{eq:phase_tk_orig}.
Notice now that, under the assumptions \eqref{eq:cond_tkv}, we have
\begin{equation}\label{eq:rough_tk_cond_bis}
|\vartheta_v| \leq 3/2, \qquad
1 \lesssim_\kappa |t| \simeq_\kappa k \simeq_\kappa \vartheta_v + k,
\end{equation}
so \eqref{eq:bd_mu_tk} and \eqref{eq:bd_mu_tk_der} easily follow from \eqref{eq:mu_tk}.

\smallskip

Now, from the formulas \eqref{eq:flow_2step} and \eqref{eq:mu_tk} we also deduce that
\begin{equation}\label{eq:xtkv}
\begin{aligned}
x^{t,k,v} &= 2 J_{\mu^{t,k,v}/|\mu^{t,k,v}|}^{-1} \sinh(J_{kv} + J_{\mu/(2|\xi|)}) \exp(J_{kv} + J_{\mu/(2|\xi|)}) \frac{\xi}{|\mu^{t,k,v}|}, \\
\xi^{t,k,v} &= \cosh(J_{kv} + J_{\mu/(2|\xi|)}) \exp(J_{kv} + J_{\mu/(2|\xi|)})  \xi,\\
\end{aligned}
\end{equation}
as well as
\[
\mu^{t,k,v} \cdot u^{t,k,v} 
= \frac{t|\xi|}{2} - \frac{|\xi|^2}{2|\mu^{t,k,v}|} \left\langle J_{\mu^{t,k,v}/|\mu^{t,k,v}|}^{-1} \sinh(J_{2kv} + J_{\mu/|\xi|}) \bar\xi,\bar\xi \right\rangle,
\]
so the formula \eqref{eq:umutkv} holds, with
\[
\rho^{t,k,v} 
\defeq \frac{|\xi|}{2} \left[\frac{t}{k} \vartheta_v  
+ \frac{|\xi|}{|\mu^{t,k,v}|} \left\langle J_{\mu^{t,k,v}/|\mu^{t,k,v}|}^{-1} \sinh(J_{2kv} + J_{\mu/|\xi|}) \bar\xi,\bar\xi \right\rangle \right].
\]
Now, as $G$ is M\'etivier, the function $\mu \mapsto J_\mu^{-1}$ is smooth and $(-1)$-homogeneous on $\dot\RR^{d_2}$. So, from the above formulas, together with the bounds \eqref{eq:rough_tk_cond_bis}, \eqref{eq:bd_mu_tk}, \eqref{eq:bd_mu_tk_der} and Lemma \ref{lem:der_modJmu}, it is now readily deduced that the estimates \eqref{eq:bd_x_tk_der}, \eqref{eq:bd_xi_tk_der}, \eqref{eq:bd_rho_tk_der} hold.

\smallskip

Finally, in light of \eqref{eq:phase_tk}, we have $\Re\phi_{k,v} = (\bx-\bx^{t,k,v})^T \bxi^{t,k,v}$, so the proof of \eqref{eq:dxi_re_phase_tk} reduces to showing that
\[
(\partial_{\bxi} \bx^{t,k,v})^T \bxi^{t,k,v} = 0.
\]
However, as
\[
\bxi^{t,k,v} = \bxi^t|_{\mu = \mu^{t,k,v}}, \qquad \partial_{\bxi} \bx^{t,k,v} = (\partial_{\bxi}\bx^t)|_{\mu = \mu^{t,k,v}} \, \partial_{\bxi} \left(\begin{array}{c} \xi \\\hline \mu^{t,k,v} \end{array}\right),
\]
the desired identity follows immediately from \eqref{eq:symplectic_ort}.
\end{proof}

The estimates in Proposition \ref{prp:phase_tkv_bounds} show that the only possible sources of unboundedness of the derivatives $\partial_{\bxi} \phi_{k,v}$ as a family of $S^0(\RR^d)$-symbols in the variable $\bxi$ are the $\bx$-dependent terms $x - x^{t,k,v}$, $u_v^\perp/t$ and $u_v^\parallel - t^2/(4 k)$ appearing in \eqref{eq:phase_tk} and \eqref{eq:umutkv}. We shall now see that these terms can effectively be assumed to be bounded, in the sense that the remainder is negligible for our purposes. This fact is contained in the following proposition, which shows that, in the study of oscillatory integrals of the form $I_{k,v}[q]$, we can achieve a tighter spatial localisation than that obtained in Section \ref{s:spatialloc} from finite propagation speed.

\begin{prp}\label{prp:first_space_loc}
Assume that $G$ is a M\'etivier group.
Let $a \in \RR$, $\kappa > 1$, $\vec\epsilon \in (\Rpos)^4$ with $\epsilon_4>\epsilon_3$, $m,\ell \in \NN$, $\lambda \geq 1$ satisfying \eqref{eq:cond_lm}. Let $v \in \RR^{d_2}$ be a unit vector. Let $\cQ$ be a bounded family of symbols in $S^0(\RR) \otimes S^a(\RR^d)$.
Then, for all $q \in \cQ$ such that
\begin{equation}\label{eq:supp_Qk}
\bigcup_{t \in \RR} \supp q(t,\cdot) \subseteq \Omega_{\kappa,0} \cap \{ \bxi \tc \kappa^{-1} \leq |\xi|/2^m \leq \kappa \},
\end{equation}
for all $t \in \RR$ and $k \in \NN$ such that
\begin{equation}\label{eq:cond_tk}
2^{\ell-1} \leq |t| \leq 2^{\ell+1}, \qquad
k \geq 2, \qquad \kappa^{-1} \leq |t|/k \leq \kappa,
\end{equation}
and for all $b,\delta_0>0$, we have
\[
I_{k,v}[q](t,\cdot) 
= \chr_{\{ \bx \tc |x| \leq r_k, \, |u_v^\parallel-t^2/(4 k)| \leq b\lambda^{\delta_0}, \, |u_v^\perp| \leq |t| b \lambda^{\delta_0}\}} I_{k,v}[q](t,\cdot) 
+ R_{k,v,t,b,\delta_0,\lambda,q}
\]
for some $r_\kappa \in \Rpos$, where,
for all $N \in \NN$,
\[
\|\chr_{\overline{B}(0,2^{2+\ell})} R_{k,v,t,b,\delta_0,\lambda,q}\|_1 \lesssim_{\cQ,\kappa,\vec\epsilon,b,\delta_0,N} \lambda^{-N}.
\]
\end{prp}
\begin{proof}
The spatial localisation in $x$ is again easily achieved by taking advantage of the Gaussian decay of $e^{i\phi}$ in $x$ away from $x=x^t$. Indeed, under the conditions \eqref{eq:supp_Qk} and \eqref{eq:cond_tk}, we have,
by \eqref{eq:bd_mu_tk}, \eqref{eq:bd_x_tk_der} and \eqref{eq:xtkv},
\[
|\mu^{t,k,v}| \simeq_\kappa |\xi| \simeq_\kappa 2^m, \qquad |x^{t,k,v}| \lesssim_\kappa 1, \qquad |t| \simeq_\kappa \vartheta_v + k \simeq  k \simeq 2^\ell
\]
on the support of $q$. Thus also, by \eqref{eq:phase_tk} and \eqref{eq:metivier_nondeg},
\[
\Im \phi_{k,v}(t,\bx,\bxi) \simeq_\kappa 2^m |x-x^{t,k,v}|^2 \gtrsim_\kappa 2^m
\]
whenever $|x| > r_\kappa$ for a sufficiently large $r_\kappa \in \Rpos$, and, by \eqref{eq:int_kv},
\[
|I_{k,v}[q](t,\bx)| \leq \int_{\RR^d} e^{-\Im \phi_{k,v}(t,\bx,\bxi)} |q(t,\bxi)| \,d\bxi \lesssim_{\cQ,\kappa} (2^m)^{a+d} e^{-c_\kappa 2^m} \lesssim_{a,\kappa,N,\vec\epsilon} \lambda^{-N}
\]
for an appropriate $c_\kappa \in \Rpos$ and all $N \in \NN$, where we used \eqref{eq:cond_lm}. As $|\overline{B}(0,2^{2+\ell})| \lesssim \lambda^Q$, we readily deduce that
\[
\| \chr_{\overline{B}(0,2^{2+\ell})} \chr_{(\RR^{d_1} \setminus \overline{B}_{\RR^{d_1}}(0,r_\kappa)) \times \RR^{d_2}} I_{k,v}[q](t,\cdot)\|_1 \lesssim_{\cQ,\kappa,\vec\epsilon,N} \lambda^{-N}
\]
for all $N \in \NN$.

\smallskip

We now prove the localisation in $u_v^\parallel$.
According to Proposition \ref{prp:phase_tkv_bounds} and \eqref{eq:umutkv}, we can decompose
\[
\phi_{k,v} = (v \cdot \mu^{t,k,v}) \left( u_v^\parallel -\frac{t^2}{4 k}\right) + \psi_{k,v},
\]
where
\[
\psi_{k,v} \defeq (x - x^{t,k,v}) \cdot \xi^{t,k,v} +  \frac{ u_v^\perp}{t} \cdot \mu_v^\perp + \rho^{t,k,v} + \frac{i}{4} \langle|J_{\mu^{t,k,v}}| (x-x^{t,k,v}), x-x^{t,k,v}\rangle.
\]
Thanks to the bounds in Lemma \ref{lem:der_modJmu}, Proposition \ref{prp:phase_tkv_bounds} and \eqref{eq:xtkv}, together with the fact that the only summand in $\psi^{t,k}$ containing $u$ is independent of $\xi$, it is easily seen that, under the assumption $|x| \leq r_\kappa$,
\[
|\partial_\xi^\alpha \psi_{t,k}| \lesssim_{\kappa,\alpha} 2^{m(1-|\alpha|)}
\]
for all $\alpha \in \NN^{d_1} \setminus \{0\}$, from which it readily follows that
\[
|\partial_\xi^\alpha (e^{i\psi_{t,k}} q)| \lesssim_{\cQ,\kappa,\alpha} 2^{ma}
\]
for all $\alpha \in \NN^{d_1}$ (notice that $\Im\psi_{t,k} \geq 0$); similarly, from \eqref{eq:mu_tk} one deduces that
\[
|\partial_\xi^\alpha (v \cdot \mu^{t,k,v})| \lesssim_{\kappa,\alpha} 2^{m(1-|\alpha|)}, \qquad |\partial_\xi (v \cdot \mu^{t,k,v})| = \left|\frac{2 k}{t}\right| \simeq_\kappa 1.
\]
Therefore, if we write
\[
I_{k,v}[q](t,\bx) = \int_{\RR^d} e^{i (u_v^\parallel-t^2/(4 k)) (v\cdot\mu^{t,k,v})} [e^{i\psi_{k,v}} q(t,\bxi)] \,d\bxi
\]
(i.e., consider $e^{i\psi_{k,v}}$ as part of the amplitude), then repeated integration by parts in $\xi$ (see, e.g., \cite[Theorem 7.7.1]{Ho1}) yields, for $|x| \leq r_\kappa$, the estimate
\[
|I_{k,v}[q](t,\bx)| \lesssim_{\cQ,\kappa,N} |u_v^\parallel-t^2/(4 k)|^{-N} 2^{m(a+d)}
\]
for all $N \in \NN$; thus, for $|x| \leq r_\kappa$ and $|u_v^\parallel-t^2/(4 k)| > b\lambda^{\delta_0}$, we also get
\[
|I_{k,v}[q](t,\bx)| \lesssim_{\cQ,\kappa,b,\delta_0,N} \lambda^{-N}
\]
for all $N \in \NN$. As before, combining this with the volume bound $|\overline{B}(0,2^{2+\ell})| \lesssim \lambda^Q$ yields the estimate
\[
\| \chr_{\overline{B}(0,2^{2+\ell})} \chr_{\{ \bx \tc |x| \leq r_\kappa, \, |u_v^\parallel-t^2/(4 k)| > b\lambda^{\delta_0}\}} I_{k,v}[q](t,\cdot)\|_1 \lesssim_{\cQ,\kappa,\vec\epsilon,b,\delta_0,N} \lambda^{-N}
\]
for all $N \in \NN$.

\smallskip

It remains to prove the localisation in $u_v^\perp$. Here instead we decompose
\[
\phi_{k,v} =  \frac{u_v^\perp}{t} \cdot \mu_v^\perp + \tilde\psi_{k,v},
\]
with
\begin{multline*}
\tilde\psi_{k,v} \defeq (x - x^{t,k,v}) \cdot \xi^{t,k,v} + (v \cdot \mu^{t,k,v}) \left( u_v^\parallel -\frac{t^2}{4 k}\right) + \rho^{t,k,v} \\
+ \frac{i}{4} \langle |J_{\mu^{t,k,v}}| (x-x^{t,k,v}),x-x^{t,k,v}\rangle .
\end{multline*}
Thanks to the localisation $|x| \leq r_\kappa$, and due to the fact that the only summand in $\tilde\psi_{k,v}$ depending on $u$ is $\mu_v^\perp$-independent (see \eqref{eq:mu_tk}), 
as before one deduces that
\begin{align*}
|\partial_{\mu^\perp_v}^\alpha \tilde\psi_{k,v}| &\lesssim_{\kappa,b,\alpha} 2^{m(1-|\alpha|)} \quad\text{for } \alpha \neq 0, \\
|\partial_{\mu^\perp_v}^\alpha (e^{i\tilde\psi_{k,v}} q)| &\lesssim_{\cQ,\kappa,b,\alpha}  2^{ma} \quad\quad\ \text{for all } \alpha.
\end{align*}
Thus, iterated integration by parts in $\mu_v^\perp$ gives, for all $N \in \NN$, the estimate
\[
|I_{k,v}[q](t,\bx)| \lesssim_{\cQ,\kappa,b,N}  |u_v^\perp/t|^{-N} 2^{m(a+d)},
\]
and therefore, under the assumption  $|u_v^\perp| > b\lambda^{\delta_0}|t|$, also
\[
|I_{k,v}[q](t,\bx)| \lesssim_{\cQ,\kappa,b,\delta_0,N} \lambda^{-N}
\]
for all $N \in \NN$. In conclusion, as before,
\[
\| \chr_{\overline{B}(0,2^{2+\ell})} \chr_{\{ \bx \tc |x| \leq r_\kappa, \ |u_v^\perp| > |t| b \lambda^{\delta_0}\}} I_{k,v}[q](t,\cdot)\|_1 \\
 \lesssim_{\cQ,\kappa,\vec\epsilon,b,\delta_0,N} \lambda^{-N}
\]
for all $N \in \NN$.

\smallskip

The desired bound follows by combining the above estimates.
\end{proof}

We are now ready to treat the main contribution of the oscillatory integrals of the form $I_{k,v}[q]$.

\begin{prp}\label{prp:sdd}
Assume that $G$ is a M\'etivier group.
Let $\kappa > 1$, $r>0$, $\delta_0 > 0$, $\vec\epsilon \in (\Rpos)^4$ with $\epsilon_4>\epsilon_3$, $m,\ell \in \NN$, $\lambda \geq 1$ satisfying \eqref{eq:cond_lm}, and $a \in \RR$.
Let $\cQ$ be a bounded family of symbols in $S^0(\RR) \otimes S^a(\RR^d)$. 
For all $t \in \RR$ and $k \in \ZZ$ satisfying \eqref{eq:cond_tk}, and all $q \in \cQ$ satisfying the support condition \eqref{eq:supp_Qk}, we have
\begin{multline*}
\|\chr_{\{\bx \tc |x| \leq r, \, |u_v^\parallel-t^2/(4 k)| \leq r\lambda^{\delta_0}, |u_v^\perp| \leq |t| r \lambda^{\delta_0}\} }
I_{k,v}[q](t,\cdot) \|_1 
\\
\lesssim_{\cQ,\kappa,r,\delta_0,\vec\epsilon} 
\lambda^{(d+1)\delta_0} 2^{\ell(d_2-1)} 2^{m[a+(d-1)/2]}.
\end{multline*}
\end{prp}
\begin{proof}
We preliminarily observe that, under the spatial localisation
\begin{equation}\label{eq:spatial_loc}
|x| \leq r, \qquad |u_v^\parallel-t^2/(4 k)| \leq r \lambda^{\delta_0}, \qquad |u_v^\perp| \leq |t| r \lambda^{\delta_0},
\end{equation}
from the formulas and bounds of Proposition \ref{prp:phase_tkv_bounds} and Lemma \ref{lem:der_modJmu}, one readily sees that
\begin{equation}\label{eq:phase_tkv_bound_sploc}
|\partial_{\bxi}^\alpha \phi_{k,v}(t,\bx,\bxi)| \lesssim_{\kappa,r,\alpha} \lambda^{\delta_0} |\bxi|^{1-|\alpha|}
\end{equation}
whenever $\bxi \in \tilde\Omega_{\kappa,0}$, for all $\alpha \in \NN^d$.

Due to the support condition \eqref{eq:supp_Qk}, we know that in the oscillatory integral
\[
I_{k,v}[q](t,\bx) = \int_{\RR^d} e^{i\phi_{k,v}(t,\bx,\bxi)} q(t,\bxi) \,d\bxi
\]
we already have the dyadic frequency localisation $|\bxi| \simeq_\kappa 2^m$. Much as in the proof of Proposition \ref{prp:sdd_0}, we now proceed with a ``second dyadic decomposition'' in the spirit of \cite{SSS}.

\smallskip

Namely, we can find a finite set $Z_m \subseteq \{ \bxi \tc |\bxi| = 1\}$ of directions in $\RR^d$ with
\begin{equation}\label{eq:count_sdd}
\sharp Z_m \lesssim 2^{m(d-1)/2}
\end{equation}
and a smooth $0$-homogeneous partition of unity $\{\chi_{m,\breve{\bxi}}\}_{\breve{\bxi} \in Z_m}$ on $\dot\RR^d$ such that
\begin{equation}\label{eq:supp_sdd}
\supp \chi_{m,\breve{\bxi}} \subseteq C_{m,\breve{\bxi}} \defeq \{ \bxi \tc |\bxi/|\bxi|-\breve{\bxi}| \leq c \, 2^{-m/2} \}
\end{equation}
for an appropriately small $c>0$, and
\begin{equation}\label{eq:der_sdd}
|\partial_{\bxi}^\alpha (\breve{\bxi} \cdot \nabla_{\bxi})^N \chi_{m,\breve{\bxi}} | \lesssim_{N,\alpha} 2^{m|\alpha|/2} |\bxi|^{-N-|\alpha|}
\end{equation}
for all $\breve{\bxi} \in Z_m$ and $N \in \NN$, where the implicit constants do not depend on $m$. Accordingly, we split
\begin{equation}\label{eq:dec_int_sdd}
I_{k,v}[q] = \sum_{\breve{\bxi} \in \tilde  Z_m} I_{k,v}[q_{\breve{\bxi}}],
\end{equation}
where $q_{\breve{\bxi}}(t,\bxi) \defeq q(t,\bxi) \chi_{m,\breve{\bxi}}(\bxi)$ and, due to the support condition \eqref{eq:supp_Qk}, the sum is restricted to
\[
\tilde Z_m = \{ \breve{\bxi} \in Z_m \tc \supp \chi_{m,\breve{\bxi}} \cap \Omega_{\kappa,0} \neq \emptyset\}.
\]
As $|\mu| \leq 5 |\xi|/2$ for all $\bxi \in \Omega_{\kappa,0}$ (see \eqref{eq:Omegakappa0}), by choosing a small enough $c$ in \eqref{eq:supp_sdd} we may assume that
\begin{equation}\label{eq:normal_dir}
C_{m,\breve{\bxi}} \subseteq \Rpos \tilde \Omega_{\kappa,0}
\ \text{ and } \
|\breve{\xi}| \simeq 1 \qquad\forall \breve{\bxi} = (\breve{\xi},\breve{\mu}) \in \tilde Z_m,
\end{equation}
where $\tilde \Omega_{\kappa,0}$ is as in \eqref{eq:cond_tkv}.

\smallskip

We now fix a direction $\breve{\bxi} \in \tilde Z_m$ and decompose
\[
\phi_{k,v}(t,\bx,\bxi) = \partial_{\bxi}\phi_{k,v}(t,\bx,\bxi/|\bxi|) \bxi = \partial_{\bxi}\phi_{k,v}(t,\bx,\breve{\bxi}) \bxi + h_{k,v,\breve{\bxi}}(t,\bx,\bxi),
\]
where the first equality is due to Euler's identity. By construction, $h_{k,v,\breve{\bxi}}(t,\bx,\cdot)$ is $1$-homogeneous and vanishes of second order on the half-line $\Rpos \breve{\bxi}$, as
\[
\partial_{\bxi} h_{k,v,\breve{\bxi}}(t,\bx,\bxi) = \partial_{\bxi}\phi_{k,v}(t,\bx,\bxi/|\bxi|) - \partial_{\bxi}\phi_{k,v}(t,\bx,\breve{\bxi}).
\]
From this and the bounds \eqref{eq:phase_tkv_bound_sploc} one deduces the estimates
\begin{align*}
|(\breve{\bxi} \cdot \nabla_{\bxi})^N h_{k,v,\breve{\bxi}}(t,\bx,\bxi)| 
&\lesssim_{\kappa,r,N} \lambda^{\delta_0} \left|\frac{\bxi}{|\bxi|}-\bxi\right|^2 |\bxi|^{1-N}, \\
|\partial_{\bxi} (\breve{\bxi} \cdot \nabla_{\bxi})^N h_{k,v,\breve{\bxi}}(t,\bx,\bxi)| 
&\lesssim_{\kappa,r,N} \lambda^{\delta_0} \left|\frac{\bxi}{|\bxi|}-\bxi\right| |\bxi|^{-N} , \\
|\partial_{\bxi}^\alpha h_{k,v,\breve{\bxi}}(t,\bx,\bxi)|
&\lesssim_{\kappa,r,\alpha} \lambda^{\delta_0} |\bxi|^{1-|\alpha|}
\end{align*}
for all $N \in \NN$ and $\alpha \in \NN^{d}$, whenever $\bxi \in \tilde\Omega_{\kappa,0}$ and \eqref{eq:spatial_loc} holds. In particular, if $\bxi \in \supp q_{\breve{\bxi}}$, then, by the support conditions \eqref{eq:supp_Qk} and \eqref{eq:supp_sdd} and \eqref{eq:spatial_loc},
\begin{align*}
|(\breve{\bxi} \cdot \nabla_{\bxi})^N h_{k,v,\breve{\bxi}}(t,\bx,\bxi)| 
&\lesssim_{\kappa,r,N} \lambda^{\delta_0} 2^{-mN},\\
|\partial_{\bxi}^\alpha h_{k,v,\breve{\bxi}}(t,\bx,\bxi)| 
&\lesssim_{\kappa,r,\alpha} \lambda^{\delta_0} 2^{-m|\alpha|/2}.
\end{align*}

As a consequence, we also deduce that
\begin{align*}
|e^{-ih_{k,v,\breve{\bxi}}} (\breve{\bxi} \cdot \nabla_{\bxi})^N e^{ih_{k,v,\breve{\bxi}}}| 
&\lesssim_{\kappa,r,N} \lambda^{N\delta_0} 2^{-mN},\\
|e^{-ih_{k,v,\breve{\bxi}}} \partial_{\bxi}^\alpha e^{ih_{k,v,\breve{\bxi}}}| 
&\lesssim_{\kappa,r,\alpha} \lambda^{|\alpha|\delta_0} 2^{-m|\alpha|/2},
\end{align*}
whence also, by \eqref{eq:der_sdd} and the boundedness of $\cQ$ in $S^0(\RR) \otimes S^a(\RR^d)$,
\begin{equation}\label{eq:est_ampl_SSS}
\begin{aligned}
|e^{-ih_{k,v,\breve{\bxi}}} (\breve{\bxi} \cdot \nabla_{\bxi})^N (e^{ih_{k,v,\breve{\bxi}}} q_{\breve{\bxi}})| 
&\lesssim_{\cQ,\kappa,r,N} \lambda^{N\delta_0} 2^{m(a-N)},\\
|e^{-ih_{k,v,\breve{\bxi}}} \partial_{\bxi}^\alpha (e^{ih_{k,v,\breve{\bxi}}} q_{\breve{\bxi}})| 
&\lesssim_{\cQ,\kappa,r,\alpha} \lambda^{|\alpha|\delta_0} 2^{m(a-|\alpha|/2)},
\end{aligned}
\end{equation}
for all $N \in \NN$.

\smallskip

Notice now that
\[
(\breve{\bxi} \cdot \nabla_{\bxi})[\partial_{\bxi}\phi_{k,v}(t,\bx,\breve{\bxi})\bxi] = \partial_{\bxi}\phi_{k,v}(t,\bx,\breve{\bxi})\breve{\bxi} = \phi_{k,v}(t,\bx,\breve{\bxi})
\]
by $1$-homogeneity. Hence, by iterated integration by parts in the direction $\breve{\bxi}$,
\[\begin{split}
&I_{k,v}[q_{\breve{\bxi}}](t,\bx) \\
&= \int_{\RR^d} e^{i \partial_{\bxi}\phi_{k,v}(t,\bx,\breve{\bxi})\bxi} [e^{ih_{k,v,\breve{\bxi}}(t,\bx,\bxi)} q_{\breve{\bxi}}(t,\bxi)] \,d\bxi \\
&= (-i\phi_{k,v}(t,\bx,\breve{\bxi}))^{-N}  \int_{\RR^d} e^{i \partial_{\bxi}\phi_{k,v}(t,\bx,\breve{\bxi})\bxi}  (\breve{\bxi} \cdot \nabla_{\bxi})^N [e^{ih_{k,v,\breve{\bxi}}(t,\bx,\bxi)} q_{\breve{\bxi}}(t,\bxi) ] \,d\bxi \\
&= (-i\phi_{k,v}(t,\bx,\breve{\bxi}))^{-N} \\
&\qquad\times \int_{\RR^d} e^{i \phi_{k,v}(t,\bx,\bxi)} [e^{-ih_{k,v,\breve{\bxi}}(t,\bx,\bxi)} (\breve{\bxi} \cdot \nabla_{\bxi})^N [e^{ih_{k,v,\breve{\bxi}}(t,\bx,\bxi)} q_{\breve{\bxi}}(t,\bxi) ]] \,d\bxi.
\end{split}\]
 As $\Im\phi_{k,v} \geq 0$, from this expression, the estimate \eqref{eq:est_ampl_SSS} and the support conditions \eqref{eq:supp_Qk} and \eqref{eq:supp_sdd}, we finally deduce that
\begin{multline}\label{eq:est_SSS_rad}
|I_{k,v}[q_{\breve{\bxi}}](t,\bx)| 
\lesssim_{\cQ,\kappa,r,N} \lambda^{N\delta_0} 2^{m[(d+1)/2+a]} 2^{-mN} |\phi_{k,v}(t,\bx,\breve{\bxi})|^{-N} \\
\leq \lambda^{N\delta_0} 2^{m[(d+1)/2+a]} 2^{-mN} |(x-\breve{x}^{t,k,v}) \cdot \breve{\xi}^{t,k,v} + (u-\breve{u}^{t,k,v}) \cdot \breve{\mu}^{t,k,v}|^{-N},
\end{multline}
where $\breve{x}^{t,k,v},\breve{u}^{t,k,v}, \breve{\xi}^{t,k,v},\breve{\mu}^{t,k,v}$ are the values of $x^{t,k,v},u^{t,k,v},\xi^{t,k,v},\mu^{t,k,v}$ at $\bxi=\breve{\bxi}$, 
and we used that $|\phi_{k,v}| \geq |\Re\phi_{k,v}|$.

\smallskip

Choose now an orthonormal basis $v_1^\perp,\dots,v_{d_2-1}^\perp$ of the orthogonal complement of $\RR v$ in $\RR^{d_2}$, and set $\bxi^v_j = (0,v_j^\perp)$. Then, by \eqref{eq:dxi_re_phase_tk},
\begin{multline*}
|\Re (\bxi^v_j \cdot \nabla_{\bxi})[\partial_{\bxi}\phi_{k,v}(t,\bx,\breve{\bxi})\bxi]| = |\Re \partial_{\bxi}\phi_{k,v}(t,\bx,\breve{\bxi})\bxi^v_j|  \\
= |(x-\breve{x}^{t,k,v}) \cdot \breve{A}^{t,k,v}_j + (u-\breve{u}^{t,k,v}) \cdot t^{-1} v_j^\perp|,
\end{multline*}
where $\breve{A}^{t,k,v}_j \defeq \partial_\mu \xi^{t,k,v}|_{\bxi = \breve{\bxi}} v_j^\perp,$ and where we used that
\[
\partial_\mu \mu^{t,k,v} v_j^\perp =  t^{-1} v_j^\perp
\]
by \eqref{eq:mu_tk}. Arguing much as above, repeated integration by parts in the direction $\bxi_j^v$ yields the estimate
\begin{multline}\label{eq:est_SSS_ort}
|I_{k,v}[q_{\breve{\bxi}}](t,\bx)| \\
\lesssim_{\cQ,\kappa,r,N}
\lambda^{N\delta_0} 2^{m[(d+1)/2+a]} 2^{-mN/2} |(x-\breve{x}^{t,k,v}) \cdot \breve{A}^{t,k,v}_j + (u-\breve{u}^{t,k,v}) \cdot t^{-1} v_j^\perp|^{-N}
\end{multline}
for all $N \in \NN$.

\smallskip

The above estimates per se are not enough to deduce the required $L^1$-bound for $I_{k,v}[q_{\breve{\bxi}}]$. 
 Here, much as in the proof of Proposition \ref{prp:first_space_loc}, we shall take advantage of the imaginary part of the phase to deduce a sharper localisation in $x$, due to the finer localisation in $\bxi$ achieved through the introduction of the cutoff $\chi_{m,\breve{\bxi}}$.

\smallskip

Namely, let us notice that, as $\bxi$ ranges in $C_{m,\breve{\bxi}}$, its direction $\bxi/|\bxi|$ stays in a $2^{-m/2}$-neighbourhood of $\breve{\bxi}$ contained in $\tilde\Omega_{\kappa,0}$, by \eqref{eq:supp_sdd} and \eqref{eq:normal_dir}. Moreover, as $x^{t,k,v}$ is $0$-homogeneous in $\bxi$, its dependence on $\bxi$ is only through $\bxi/|\bxi|$; finally, from \eqref{eq:bd_x_tk_der} one sees that $|\partial_{\bxi} x^{t,k,v}| \lesssim_\kappa 1$ on the unit sphere cap $\{ \bxi/|\bxi| \tc \bxi \in \tilde\Omega_{\kappa,0}\}$. From this we deduce that
\[
|x^{t,k,v}-\breve{x}^{t,k,v}| \lesssim_\kappa 2^{-m/2}
\]
whenever $\bxi \in C_{m,\breve{\bxi}}$; moreover, from \eqref{eq:bd_mu_tk} we deduce that, on the support of $q_{\breve{\bxi}}$,
\[
|\mu^{t,k,v}| \simeq_\kappa 2^m.
\]
Consequently, by \eqref{eq:phase_tk} and the M\'etivier condition \eqref{eq:metivier},
\[
\Im\phi_{k,v} \gtrsim |\mu^{t,k,v}| |x-x^{t,k,v}|^2 \gtrsim_\kappa \lambda^{2\delta_0}
\]
whenever $|x-\breve{x}^{t,k,v}| > b_\kappa \lambda^{\delta_0} 2^{-m/2}$ for a sufficiently large $b_\kappa \in \Rpos$.
Hence
\[\begin{split}
|I_{k,v}[q_{\breve{\bxi}}](t,\bx)| 
&\leq \int_{\RR^d} e^{-\Im \phi_{k,v}(t,\bx,\bxi)} |q_{\breve{\bxi}}(t,\bxi)| \,d\bxi \lesssim_{\cQ,\kappa} (2^m)^{a+d} e^{-c_\kappa \lambda^{2\delta_0}} \\
&\lesssim_{a,\kappa,N,\delta_0,\vec\epsilon} \lambda^{-N}
\end{split}\]
for an appropriate $c_\kappa \in \Rpos$ and all $N \in \NN$, where we used \eqref{eq:cond_lm}; thus also
\begin{multline*}
\|\chr_{\{\bx \tc |x| \leq r, \, |u_v^\parallel-t^2/(4 k)| \leq r\lambda^{\delta_0}, \, |u_v^\perp| \leq |t| r \lambda^{\delta_0}, \, |x-\breve{x}^{t,k,v}| > b_\kappa \lambda^{\delta_0} 2^{-m/2}\}} I_{k,v}[q_{\breve{\bxi}}](t,\cdot)\|_1 \\
\lesssim_{\cQ,\kappa,r,\vec\epsilon,N} \lambda^{-N}
\end{multline*}
for all $N \in \NN$, showing that the contribution of this part is negligible.

\smallskip

As for the remaining part, we can finally use the estimates \eqref{eq:est_SSS_rad} and \eqref{eq:est_SSS_ort} with $N=0$ and $N=d_2+1$ to deduce
\[\begin{split}
&\|\chr_{\{\bx \tc |x-\breve{x}^{t,k,v}| \leq b_\kappa \lambda^{\delta_0} 2^{-m/2}, \, |u_v^\parallel-t^2/(4 k)| \leq r\lambda^{\delta_0}, \, |u_v^\perp| \leq |t| r \lambda^{\delta_0}\}} I_{k,v}[q_{\breve{\bxi}}](t,\cdot)\|_1 \\
&\lesssim_{\cQ,\kappa,r} \lambda^{(d_2+1)\delta_0} 2^{m[(d+1)/2+a]} \\
&\quad\times \int_{\RR^{d_2}} \int_{|x-\breve{x}^{t,k,v}| \leq b_\kappa \lambda^{\delta_0} 2^{-m/2}} \Biggl(1+2^m|(x-\breve{x}^{t,k,v}) \cdot \breve{\xi}^{t,k,v} + (u-\breve{u}^{t,k,v}) \cdot \breve{\mu}^{t,k,v}| \\
&\qquad\qquad+ \sum_{j=1}^{d_2-1} 2^{m/2} |(x-\breve{x}^{t,k,v}) \breve{A}^{t,k,v}_j + (u-\breve{u}^{t,k,v}) \cdot t^{-1} v_j^\perp | \Biggr)^{-(1+d_2)} \,dx \,du \\
&\lesssim_{\kappa} \lambda^{(d+1)\delta_0} |t|^{d_2-1} 2^{ma};
\end{split}
\]
in the last step we used the facts that
\[
|v \cdot \breve{\mu}^{t,k,v}| \simeq_\kappa |\breve{\xi}| \simeq_\kappa 1
\]
by \eqref{eq:bd_mu_tk} and \eqref{eq:normal_dir}, whence one deduces that the linear map
\[
u \mapsto (u \cdot \breve{\mu}^{t,k,v},u \cdot t^{-1} v_1^\perp,\dots,u \cdot t^{-1} v_{d_2-1}^\perp )
\]
has determinant $\simeq_\kappa |t|^{-(d_2-1)}$ in absolute value.

\smallskip

In light of \eqref{eq:count_sdd}, \eqref{eq:dec_int_sdd} and \eqref{eq:cond_tk}, we can now sum the above $L^1$-norm estimates for the $I_{k,v}[q_{\breve{\bxi}}]$ to obtain the desired estimate for $I_{k,v}[q]$.
\end{proof}

We can now sum the previously obtained estimates for the oscillatory integrals $I_{k,v}[q_{k,v,T,\kappa}]$ and obtain the desired spatial $L^1$-estimates for the FIO kernels $I[q]$.

\begin{prp}\label{prp:sss_sum}
Assume that $G$ is a M\'etivier group.
Let $a \in \RR$, $\kappa > 1$, $m \in \NN$.
Let $\vec\epsilon \in (\Rpos)^4$ and $\delta_0 > 0$ with $\epsilon_4 > \epsilon_3$.
Let $\cQ$ be a $(0,a,\kappa)$-bounded family of symbols.
Then, for all $q \in \cQ$ such that
\begin{equation}\label{eq:qsupp_m}
\bigcup_{t \in \RR} \supp q(t,\cdot) \subseteq \{ \bxi \tc \kappa^{-1} \leq |\xi|/2^m \leq \kappa \},
\end{equation}
for all $\lambda \geq 1$ and $\ell \in \NN$ satisfying \eqref{eq:cond_lm},
\begin{equation}\label{eq:est_Iq_aarb}
\| \chr_{\overline{B}(0,2^{2+\ell})} I[q](\pm 2^\ell,\cdot) \|_1 
\lesssim_{\cQ,\kappa,\vec\epsilon,\delta_0} 2^{m a} \lambda^{\delta_0 + (1+\epsilon_2)(d-1)/2};
\end{equation}
in particular, if $a \geq 0$, then
\begin{equation}\label{eq:est_Iq_apos}
\| \chr_{\overline{B}(0,2^{2+\ell})} I[q](\pm 2^\ell,\cdot) \|_1 
\lesssim_{\cQ,\kappa,\vec\epsilon,\delta_0} 2^{-\ell a} \lambda^{\delta_0 + (1+\epsilon_2)(a+(d-1)/2)}.
\end{equation}
\end{prp}
\begin{proof}
Let us first notice that, when $a \geq 0$, the estimate \eqref{eq:est_Iq_apos} follows from \eqref{eq:est_Iq_aarb} and the condition $2^{\ell+m} \leq \lambda^{1+\epsilon_2}$ from \eqref{eq:cond_lm}. Hence, we only need to prove \eqref{eq:est_Iq_aarb}.

In the case $2^\ell < 16 \kappa^2$,
by Proposition \ref{prp:sdd_0},
\[
\| \chr_{\overline{B}(0,2^{2+\ell})} I[q](\pm 2^\ell,\cdot) \|_1 
\lesssim_{\cQ,\kappa,\delta_0} 2^{m[\delta_0+a+(d-1)/2]} \lesssim 2^{m a} \lambda^{(1+\epsilon_2)\delta_0 + (1+\epsilon_2)(d-1)/2},
\]
where we used \eqref{eq:cond_lm}, and the fact that $\delta_0 + (d-1)/2 \geq 0$; up to a relabelling of $\delta_0$, this implies the desired estimate \eqref{eq:est_Iq_aarb}.

\smallskip

It remains to consider the case where $2^\ell \geq 16 \kappa^2$.
In this case,
we can apply the decomposition \eqref{eq:dec_periodic} of $I[q](t,\cdot)$ from Proposition \ref{prp:dec_periodic} with $T = 2^\ell$, i.e.,
\[
I[q](\pm 2^\ell,\cdot) 
= \sum_{\substack{ k \geq 2 \\ k \simeq_\kappa 2^\ell}} \sum_{v \in V_{2^\ell,\kappa}} 2^{\ell(1/2- d_2)} I_{k,v}[q_{k,v,\ell}](\pm 2^\ell,\cdot) ,
\]
where we write $q_{k,v,\ell}$ instead of $q_{k,v,2^\ell,\kappa}$. In addition
\[
\tilde\cQ = \{ q_{k,v,\ell} \tc q \in \cQ, \ k \geq 2, \ \ell \in \NN, \ 2^\ell \geq 16 \kappa^2, \  v \in V_{2^\ell,\kappa} \}
\]
is a bounded family of symbols in $S^0(\RR) \otimes S^a(\RR^d)$ with
\[
\supp q_{k,v,\ell}(t,\cdot) \subseteq \Omega_{\kappa,0} \cap \{ \bxi \tc \kappa^{-1} \leq |\xi|/2^m \leq \kappa \}
\]
whenever $q \in \cQ$ satisfies \eqref{eq:qsupp_m}.

Hence, by applying Propositions \ref{prp:first_space_loc} and \ref{prp:sdd} to the family of symbols $\tilde\cQ$, and taking \eqref{eq:count_sdd_2ndlayer} into consideration, we obtain that, for all $N \in \NN$,
\[
\begin{split}
&\| \chr_{\overline{B}(0,2^{2+\ell})} I[q](\pm 2^\ell,\cdot) \|_1 \\
&\lesssim_{\cQ,\kappa,\vec\epsilon,\delta_0,N} 
2^{\ell(1/2- d_2)} \\
&\times\sum_{\substack{ |k| \geq 2 \\ |k| \simeq_\kappa 2^\ell}} \sum_{v \in V_{2^\ell,\kappa}} 
\| \chr_{\{ \bx \tc |x| \leq r_\kappa, \, |u_v^\parallel-2^{2\ell}/(4 k)| \leq \lambda^{\delta_0/2}, \, |u_v^\perp| \leq 2^{\ell} \lambda^{\delta_0} \}}  I_{k,v}[q_{k,v,\ell}](\pm 2^\ell,\cdot) \|_1 \\
&\qquad+ \lambda^{-N} \\
&\lesssim_{\cQ,\kappa,\delta_0,\vec\epsilon,N} 
\lambda^{(d+1)\delta_0}  2^{\ell (d_2-1/2)} 2^{m[a+(d-1)/2]} + \lambda^{-N} \\
&= \lambda^{(d+1)\delta_0}  2^{(\ell+m) (d_2-1/2)} 2^{m[a+(d_1-d_2)/2]} + \lambda^{-N}\\
&\lesssim_{\vec\epsilon,N} 
2^{m a} \lambda^{(d+1)\delta_0 + (1+\epsilon_2)[a+(d-1)/2]};
\end{split}
\]
in the last inequality, we used \eqref{eq:cond_lm} and the fact that $(d_1-d_2)/2 \geq 0$, as $G$ is M\'etivier.
This is the desired estimate \eqref{eq:est_Iq_aarb}, up to a relabelling of $\delta_0$.
\end{proof}

\begin{rem}\label{rem:notMS}
One of the crucial tools in our approach has been Lemma \ref{lem:der_modJmu}, which can be interpreted as expressing a very weak form of almost-periodicity in $\theta$ of $\exp(\theta J_{\bar\mu})$ and related expressions which come up in the Hamiltonian flow in Corollary \ref{cor:flow_alt}, and hence also in our phase $\phi$.
Thanks to this lemma, we were able to avoid computations of ``explicit'' formulas for quantities like $\exp(\theta J_{\bar\mu})$.
Such explicit formulas, however, would be needed in order to extend the approach in \cite{MSe} from Heisenberg-type groups to M\'etivier groups, since that approach was heavily based on Mehler-type formulas, which again contain terms like $\exp(\theta J_{\bar\mu})$.
This would require to apply suitable orthogonal linear changes of coordinates (depending on $\bar\mu$), which put the skew-symmetric matrices $J_{\bar\mu}$ into normal forms. Such kind of approach has also been used in numerous proofs of spectral multiplier theorems for sub-Laplacians on certain classes of $2$-step Carnot groups. However, the eigenvalues of $J_{\bar\mu}$ are just algebraic functions of $\bar\mu$, which often have singularities at which these functions are typically not differentiable, so that integrations by parts in $\mu$ (which again play a crucial role in \cite{MSe}) will in general break down. Related obstacles would also show up with the representation-theoretic approach in \cite{MSt}.
\end{rem}

\section{Proof of the main results}\label{s:mainproofs}

We are now able to prove estimate \eqref{eq:target_est} in the case where $G$ is a M\'etivier group.

\begin{prp}\label{prp:non-elliptic_est}
Assume that $G$ is a M\'etivier group.
Let $\vec\epsilon \in (\Rpos)^4$ and $\delta_0 > 0$ with $\epsilon_4 > \epsilon_3$.
Let $\lambda \geq 1$ and $m,\ell \in \NN$ satisfy \eqref{eq:cond_lm}. Let $h \in \NN$. Then
\begin{align*}
\|\chr_{\overline{B}(0,2^{2+\ell})} \cos(2^\ell \sqrt{\opL}) H_{m,h} \|_1 
&\lesssim_{\vec\epsilon,\delta_0,h} \lambda^{(d-1)/2 + \epsilon_2 (d-1)/2 + \delta_0},\\
\|\chr_{\overline{B}(0,2^{2+\ell})} 2^\ell \sqrt{\opL} \sin(2^\ell \sqrt{\opL}) H_{m,h} \|_1 
&\lesssim_{\vec\epsilon,\delta_0,h} \lambda^{1+(d-1)/2 + \epsilon_2 (1+(d-1)/2) + \delta_0},
\end{align*}
where $H_{m,h}$ is as in \eqref{eq:initial_amplitude_complex}.
\end{prp}
\begin{proof}
By comparing \eqref{eq:initial_amplitude_complex} and \eqref{eq:fio_Iq}, we see that
\[
H_{m,h} = I_0 [ q_{m,h} ],
\]
where the functions $q_{m,h}(\bxi) \defeq \eta_{h} (2^{-m}|\bxi|)$ are time-independent symbols, which are uniformly bounded in $S^0(\RR^d)$ and satisfy the support condition
\[
\supp q_{m,h} \subseteq \{ \bxi \tc 1/2 \leq |\xi|/2^m \leq 2, \ 1/2 \leq |\mu|/2^m \leq 2 \}.
\]
By applying Proposition \ref{prp:FIO_repn_wave} we then obtain, for all $N \in \NN$,
\begin{equation}\label{eq:transport_dec}
\begin{split}
&\cos(t\sqrt{\opL}) H_{m,h}(\bx) \\
&= \frac{1}{2} \sum_{j=0}^N ( I[\Lambda_I^{j} q_{m,h}](t,\bx) + I[\Lambda_I^{j} q_{m,h}](-t,\bx)) \\
&- \frac{1}{2} \int_0^t \int_0^{t-\tau} \cos(s \sqrt{\opL}) \, (I[\Lambda \Lambda_I^{N} q_{m,h}](\tau,\bx) + I[\Lambda \Lambda_I^{N} q_{m,h}](-\tau,\bx)) \,ds\,d\tau,
\end{split}
\end{equation}
and
\begin{equation}\label{eq:transport_dec_der}
\begin{split}
&t\sqrt{\opL} \sin(t\sqrt{\opL}) H_{m,h}(\bx) \\
& = -\frac{t}{2i} \sum_{j=0}^N ( I[ \mhoI \Lambda_I^j q_{m,h}](t,\bx) - I[\mhoI \Lambda_I^j q_{m,h}](-t,\bx)) \\
&+ \frac{t}{2} \int_0^t \cos((t-\tau) \sqrt{\opL}) \, ( I[\Lambda \Lambda_I^N q_{m,h}](\tau,\bx) + I[\Lambda \Lambda_I^N q_{m,h}](-\tau,\bx) ) \,d\tau.
\end{split}
\end{equation}
Moreover, by Lemma \ref{lem:symbols_opLambda}, for any $j,h \in \NN$, the families
\[
\cQ_{j,h} = \{ \Lambda_I^{j} q_{m,h} \tc m \in \NN\}, \qquad \dot\cQ_{j,h} = \{ \mhoI \Lambda_I^{j} q_{m,h} \tc m \in \NN\}
\]
are $(0,-j,\kappa)$-bounded and $(0,1-j,\kappa)$-bounded respectively, while
\[
\tilde\cQ_{j,h} = \{ \Lambda \Lambda_I^{j} q_{m,h} \tc m \in \NN\}
\]
is $(-2,-j,\kappa)$-bounded, with $\kappa=4$; in addition,
\begin{multline*}
\supp (\Lambda_I^{j} q_{m,h})(t,\cdot) \cup \supp (\mho \Lambda_I^{j} q_{m,h})(t,\cdot) \cup \supp (\Lambda \Lambda_I^{j} q_{m,h})(t,\cdot) \\
 \subseteq \{ \bxi \tc 1/2 \leq |\xi|/2^m \leq 2, \ 1/2 \leq |\mu|/2^m \leq 2 \}.
\end{multline*}
By applying Proposition \ref{prp:sss_sum} with $a=0$ and $a=1$ to the families $\cQ_{j,h}$ and $\dot\cQ_{j,h}$ respectively, we then deduce that
\begin{equation}\label{eq:est_td_main}
\begin{aligned}
\|\chr_{\overline{B}(0,2^{2+\ell})} I[\Lambda_I^{j} q_{m,h}](\pm 2^\ell,\cdot) \|_1 
&\lesssim_{\vec\epsilon,\delta_0,h,j} \lambda^{(d-1)/2  + \delta_0 + \epsilon_2 (d-1)/2} \\
\|2^\ell \chr_{\overline{B}(0,2^{2+\ell})} I[\mhoI \Lambda_I^{j} q_{m,h}](\pm 2^\ell,\cdot) \|_1 
&\lesssim_{\vec\epsilon,\delta_0,h,j} \lambda^{1+(d-1)/2  + \delta_0 + \epsilon_2 (1+(d-1)/2)} 
\end{aligned}
\end{equation}
(for $j>0$, Proposition \ref{prp:sss_sum} would give an even better estimate in terms of powers of $\lambda$, but we shall not use it).

\smallskip

On the other hand, for all $\tau \in \RR$, by \eqref{eq:det_2step},
\[\begin{split}
|I[\Lambda \Lambda_I^{N} q_{m,h}](\tau,\bx)| 
&\leq (2\pi)^{-d} \int_{\RR^d} |\Lambda \Lambda_I^{N} q_{m,h}(\tau,\bxi)| \,|\Den_\phi(\tau,\bxi)| \,d\bxi \\
&\lesssim_{N,h} (1+|\tau|)^{-3/2} 2^{m(d-N)},
\end{split}\]
and, by the finite propagation speed property \eqref{eq:fps}, if $|s| \leq 2^{\ell}$, then
\[\begin{split}
&\|\chr_{\overline{B}(0,2^{2+\ell})} \cos(s\sqrt{\opL}) I[\Lambda \Lambda_I^{N} H_{m,h}](\tau,\cdot) \|_1 \\
&\lesssim 2^{Q\ell/2} \|\chr_{\overline{B}(0,2^{2+\ell})} \cos(s\sqrt{\opL}) I[\Lambda \Lambda_I^{N} H_{m,h}](\tau,\cdot) \|_2\\
&\leq 2^{Q\ell/2} \|\chr_{\overline{B}(0,2^{3+\ell})} I[\Lambda \Lambda_I^{N} H_{m,h}](\tau,\cdot) \|_2 \\
&\lesssim_{N,h} 2^{Q\ell}(1+|\tau|)^{-3/2} 2^{m(d-N)},
\end{split}\]
thus also, if $|t| \leq 2^\ell$,
\begin{multline}\label{eq:est_td_rem}
\left\|t \chr_{\overline{B}(0,2^{2+\ell})} \int_0^{t} \cos((t-\tau) \sqrt{\opL}) I[\Lambda \Lambda_I^{N} H_{m,h}](\pm\tau,\bx) \,d\tau \right\|_1 \\
+ \left\|\chr_{\overline{B}(0,2^{2+\ell})} \int_0^{t} \int_0^{t-\tau} \cos(s \sqrt{\opL}) I[\Lambda \Lambda_I^{N} H_{m,h}](\pm\tau,\bx) \,ds\,d\tau \right\|_1 \\
\lesssim_{N,h} 2^{(Q+1)\ell+m(d-N)} \leq \lambda^{(Q+1) - (\epsilon_4-\epsilon_3)(N-d)}
\end{multline}
by \eqref{eq:cond_lm}, whenever $N \geq d$. As $\epsilon_4 > \epsilon_3$, by choosing a sufficiently large $N$, we can ensure that $(Q+1) - (\epsilon_4-\epsilon_3)(N-d) \leq (d-1)/2$. With this choice of $N$, in light of the decompositions \eqref{eq:transport_dec} and \eqref{eq:transport_dec_der}, we can combine the bounds \eqref{eq:est_td_main} and \eqref{eq:est_td_rem} to obtain the desired estimate. 
\end{proof}

We can finally obtain the spectrally localised $L^1$-estimate \eqref{eq:target_first} for the wave propagator.

\begin{prp}\label{prp:main_est}
Assume that $G$ is a M\'etivier group. For all $s > (d-1)/2$ and $\lambda \geq 1$,
\begin{equation}\label{eq:sploc_wave_est_bis}
\|\exp(\pm i \sqrt{\opL}) \, \chi_1(\sqrt{\opL}/\lambda)\|_{1 \to 1} \lesssim_s \lambda^s.
\end{equation}
\end{prp}
\begin{proof}
The proof summarizes the various reduction steps discussed throughout.

Since the operator $\exp(\pm i \sqrt{\opL}) \, \chi_1(\sqrt{\opL}/\lambda)$ is given by right convolution with the kernel $\exp(\pm i \sqrt{\opL})  k_{\chi_1(\sqrt{\opL}/\lambda)}$, the estimate \eqref{eq:sploc_wave_est_bis} is  equivalent to 
\[
\|\exp(\pm i \sqrt{\opL}) k_{\chi_1(\sqrt{\opL}/\lambda)}\|_1 \lesssim_s \lambda^s \qquad \text{for } s>(d-1)/2.
\]
By Proposition \ref{prp:spatial_loc}, it even suffices to estimate the $L^1$-norm over the ball $\overline{B}(0,2)$, i.e., to show that
\begin{equation}\label{eq:sploc_wave_est2}
\|\chr_{\overline{B}(0,2)}\exp(\pm i \sqrt{\opL}) k_{\chi_1(\sqrt{\opL}/\lambda)}\|_1 \lesssim_s \lambda^s \qquad \text{for } s>(d-1)/2,
\end{equation}
as the remainder is negligible.
To this end, let $\vec\epsilon \in (\Rpos)^4$ and $\delta_0 > 0$ with $\epsilon_4 > \epsilon_3.$  We shall show how the estimate \eqref{eq:sploc_wave_est2} follows from what we have already proved, assuming implicitly that all constants $\epsilon_1,\epsilon_2,\epsilon_3,\epsilon_4,\delta_0 > 0$ which come up in our estimates are chosen sufficiently small.

First recall that in Proposition \ref{prp:spectrum_frequency} we decomposed
\[
k_{\chi_1(\sqrt{\opL}/\lambda)} = A_{\vec\epsilon,\lambda} + B_{\vec\epsilon,\lambda} + R_{\vec\epsilon,\lambda},
\]
and showed in \eqref{eq:frequency_remainder} that the contribution by $R_{\vec\epsilon,\lambda}$ to $\exp(\pm i \sqrt{\opL})k_{\chi_1(\sqrt{\opL}/\lambda)}$ adds only a negligible term to \eqref{eq:sploc_wave_est2}.

Next, in Proposition \ref{prp:fbi}, we saw that the contribution by $A_{\vec\epsilon,\lambda} $ satisfies the estimate 
\[
\|\chr_{\overline{B}(0,2)} \exp(\pm i \sqrt{\opL}) A_{\vec\epsilon,\lambda} \|_1 \lesssim_{\vec\epsilon,s} \lambda^s
\]
for all $s > d_2+d_1(\epsilon_3+\epsilon_4),$ which is sufficient for the purpose of proving \eqref{eq:sploc_wave_est2}, since $d_2 \leq (d-1)/2$ for M\'etivier groups, see \eqref{eq:metivier_dim_ineq}.

Finally, we decomposed $B_{\vec\epsilon,\lambda} = B^{(0)}_{\vec\epsilon,\lambda} + B^{(1)}_{\vec\epsilon,\lambda}$ in \eqref{eq:ell_nonell} into the contribution $B^{(0)}_{\vec\epsilon,\lambda}$ by the elliptic region, and the contribution $B^{(1)}_{\vec\epsilon,\lambda}$ by the non-elliptic region. 

For the elliptic region, the desired $L^1$-estimate for $\exp(\pm i\sqrt{\opL}) B^{(0)}_{\vec\epsilon,\lambda}$ were reduced to proving the estimates \eqref{eq:elliptic_region} under the assumption \eqref{eq:cond_ell_j}, i.e., $\lambda^{1-\epsilon_3} \leq 2^j < \lambda^{1+\epsilon_2}.$ These estimates, in turn, were reduced by Proposition \ref{prp:sp_loc_removal} to the estimates \eqref{eq:elliptic_region2}, which are indeed a consequence of Proposition \ref{prp:elliptic_region} because of \eqref{eq:cond_ell_j}.

As for the contribution by the non-elliptic region, the desired $L^1$-estimate for $\exp(\pm i \sqrt{\opL}) B^{(1)}_{\vec\epsilon,\lambda}$ was reduced in a similar way through \eqref{eq:non-elliptic_region} and  \eqref{eq:non-elliptic_region_nosploc} after automorphic scaling eventually to the estimates \eqref{eq:non-elliptic_region2}, under the conditions \eqref{eq:cond_lm}. By means of Proposition \ref{prp:introd_complexphase}, these estimates finally were reduced to the estimates \eqref{eq:target_est}, which are proved in Proposition \ref{prp:non-elliptic_est}.
\end{proof}

We are now in the position of proving the main result.

\begin{proof}[Proof of Theorem \ref{thm:main}]
We shall only prove the claimed estimates for $t=\pm 1$; the general case follows by automorphic scaling.

\smallskip

In the case where $\chi$ is the cutoff $\chi_1$ from Section \ref{ss:dyadicpartition}, the estimate in part \ref{en:main_splocest} follows from Proposition \ref{prp:main_est} (which proves the case $p=1$), combined with interpolation and duality (the case $p=2$ is trivially true since $\opL$ is selfadjoint). In other words, for all $p \in [1,\infty]$, $s > (d-1)|1/p-1/2|$ and $\lambda \geq 1$,
\begin{equation}\label{eq:sploc_chi1_p}
\|\exp(\pm i \sqrt{\opL}) \, \chi_1(\sqrt{\opL}/\lambda)\|_{p \to p} \lesssim_s \lambda^s,
\end{equation}
and, as a consequence, also
\begin{equation}\label{eq:sploc_cossin_chi1_p}
\|\cos(\sqrt{\opL}) \, \chi_1(\sqrt{\opL}/\lambda)\|_{p \to p} \lesssim_s \lambda^s, \qquad
\|\sin(\sqrt{\opL}) \, \chi_1(\sqrt{\opL}/\lambda)\|_{p \to p} \lesssim_s \lambda^s.
\end{equation}

We now prove part \ref{en:main_symbol}. If $\cA \subseteq S^{-s}(\RR)$ is bounded for some $s > (d-1)|1/p-1/2|$, then, for all $a \in \cA$, we can write, by \eqref{eq:dyadicpartition},
\[\begin{split}
a(\sqrt{\opL}) \exp(\pm i \sqrt{\opL}) 
&= a(\sqrt{\opL}) \chi_0(\sqrt{\opL}) \exp(\pm i \sqrt{\opL}) \\
&\qquad+ \sum_{k=0}^\infty a(\sqrt{\opL}) \tilde\chi_1(\sqrt{\opL}/2^k) \chi_1(\sqrt{\opL}/2^k) \exp(\pm i \sqrt{\opL})
\end{split}\]
and observe that the family of functions
\[
\{ a \chi_0 \exp(\pm i \cdot) \tc a \in \cA \} \cup \{ 2^{ks} a \, \tilde\chi_1(\cdot/2^k) \tc k \in \NN, \ a \in \cA \}
\]
is bounded in $\Sz(\RR)$. Thus, by Proposition \ref{prp:hulanicki},
\[
\|a(\sqrt{\opL}) \exp(\pm i \sqrt{\opL})\|_{p \to p} \lesssim_{\cA} 1, \qquad \| a(\sqrt{\opL}) \tilde\chi_1(\sqrt{\opL}/2^k) \|_{p \to p} \lesssim_{\cA} 2^{-ks} 
\]
for all $k \in \NN$, and therefore
\[
\|a(\sqrt{\opL}) \exp(\pm i \sqrt{\opL})\|_{p \to p} \lesssim_{\cA} 1 + \sum_{k=0}^\infty 2^{-ks} \|\chi_1(\sqrt{\opL}/2^k) \exp(\pm i \sqrt{\opL})\|_{p \to p}.
\]
By choosing $s_*$ between $s$ and $(d-1)|1/p-1/2|$ and applying \eqref{eq:sploc_chi1_p} with $\lambda=2^k$ and $s_*$ in place of $s$, we see that the series in the right-hand side converges, and we obtain the desired uniform $L^p$-boundedness of $a(\sqrt{\opL}) \exp(\pm i \sqrt{\opL})$, where $a \in \cA$. A completely analogous argument, using \eqref{eq:sploc_cossin_chi1_p} in place of \eqref{eq:sploc_chi1_p}, proves the uniform $L^p$-boundedness of $a(\sqrt{\opL}) \cos(\sqrt{\opL})$ and $b(\sqrt{\opL}) \sin(\sqrt{\opL})/\sqrt{\opL}$ with $a \in \cA$ and $b \in \cB$.

\smallskip

Now, for any $\chi \in C^\infty_c(\Rpos)$, the set $\{ \lambda^{-s} \chi(\cdot/\lambda) \tc \lambda \geq 1\}$ is a bounded subset of $S^{-s}(\RR)$. Thus, by applying part \ref{en:main_symbol} to this family of symbols, we obtain the estimate in part \ref{en:main_splocest} for an arbitrary cutoff $\chi$.

\smallskip

As for part \ref{en:main_cauchy}, it is sufficient to observe that the solution of the given Cauchy problem for the wave equation is given by
\[
u(t,\cdot) = \cos(t\sqrt{\opL}) f + \frac{\sin(t\sqrt{\opL})}{\sqrt{\opL}} g,
\]
whence
\[\begin{split}
\|u(\pm 1,\cdot)\|_{p} 
&\leq \|(1+\opL)^{-s/2} \cos(\sqrt{\opL})\|_{p \to p} \|(1+\opL)^{s/2} f\|_{p} \\
&\qquad+  \left\| (1+\opL)^{(1-s)/2} \frac{\sin(\sqrt{\opL})}{\sqrt{\opL}} \right\|_{p \to p} \|(1+\opL)^{(s-1)/2} g\|_{p},
\end{split}\]
and the desired estimate follows by applying part \ref{en:main_symbol} with $a(\zeta) = (1+\zeta^2)^{-s/2}$ and $b(\zeta) = (1+\zeta^2)^{(1-s)/2}$.
\end{proof}

\section{Solution of the transport equation}\label{s:opLambda_2step}
The purpose of this section is proving Proposition \ref{prp:Lambdaop_2step}.
Recall the operator 
\[
R p 
\defeq \Den_\phi^{-1} \Div_{\bxi} ( \Den_\phi \widetilde{\Phi}^{-1} \widetilde{\nabla_{\bx} p} ) 
= \frac{1}{2} \frac{\partial_{\bxi} \det \Phi}{\det \Phi} \widetilde{\Phi}^{-1} \widetilde{\nabla_{\bx} p} + \Div_{\bxi} (\widetilde{\Phi}^{-1} \widetilde{\nabla_{\bx} p})
\]
from \eqref{eq:A2Sop}, and the operators
\[
\Lambda q \defeq \sum_{\substack{j,k \geq 0 \\ j+k \leq 2}} \underline{R^{k}(F_{kj} \partial_t^j q)}, \qquad \mho q \defeq \sum_{\substack{j,k \geq 0 \\ j+k \leq 1}} 2^{j-1} \underline{R^{k}(F_{k(j+1)} \partial_t^j q)},
\]
from \eqref{eq:opLambda}. We shall also write $\ldbrack v\rdbrack_\ell$ for the $\ell$th component of a vector $v \in \CC^n$.

\begin{lem}
Recall the coefficients $F_{kj}$ from \eqref{eq:Fkj}, and set 
\begin{equation}\label{eq:coeff_K}
K \defeq F_{10} + R F_{20}.
\end{equation}
Then the operators $\Lambda$ and $\mho$ are given by the expressions \eqref{eq:opLambda_2step}-\eqref{eq:opMho_2step}, i.e.,
\begin{align*}
\Lambda q 
&= \sum_{j=0}^2 \Lambda_{j0} \partial_t^j q + \sum_{j=0}^1 (\partial_t^j \partial_{\xi} q) \Lambda_{j1} + \tr(\Lambda_{02} \partial_\xi \nabla_\xi q),\\
\mho q 
&= \sum_{j=0}^1 2^{j-1} \Lambda_{(j+1)0} \partial_t^j q + 2^{-1} (\partial_{\xi} q) \Lambda_{11},
\end{align*}
where the coefficients $\Lambda_{jr}$ (which are scalar-, vector- or matrix-valued according to whether $r=0,1,2$) are given by
\begin{equation}\label{eq:opLambdacoeff_prelim}
\begin{aligned}
\Lambda_{00} &= F_{00} + RK,\\
\Lambda_{01} &= \Phi_0^{-1} \nabla_x K + (R\ldbrack\Phi_0^{-1} \widetilde{\nabla_x F_{20}}\rdbrack_\ell)_{\ell=1}^{d_1}, \\
\Lambda_{02} &= (\ldbrack\Phi_0^{-1} \nabla_x\ldbrack\Phi_0^{-1} \widetilde{\nabla_x F_{20}}\rdbrack_\ell\rdbrack_m)_{\ell,m=1}^{d_1},\\
\Lambda_{10} &= F_{01} + RF_{11},\\
\Lambda_{11} &= \Phi_0^{-1} \nabla_x F_{11},\\
\Lambda_{20} &= F_{02} = 1.
\end{aligned}
\end{equation}
\end{lem}

\begin{proof}
We only prove the formula for $\Lambda$; the proof for $\mho$ is fully analogous.

\smallskip

Notice first that the underlining in the right-hand side of \eqref{eq:opLambda} can be omitted, i.e.,
\begin{equation}\label{eq:opLambda_nound}
\Lambda q = \sum_{\substack{j,k \geq 0 \\ j+k \leq 2}} R^{k}(F_{kj} \partial_t^j q),
\end{equation}
as the coefficients $F_{kj}$ are $u$-independent and $\bx$-polynomials of degree $k$ (see Lemma \ref{lem:dal}\ref{en:dal_Fkj}), thus the terms $R^{k}(F_{kj} \partial_t^j q)$ are $\bx$-independent by Lemma \ref{lem:a2s}\ref{en:a2s_uindxpol}.

Moreover, by Lemma \ref{lem:a2s}\ref{en:a2s_uind}, the summands with $k=1$ can be rewritten as
\begin{align*}
R (F_{10}q) &= q (R F_{10}) + (\partial_{\xi} q) \Phi_0^{-1} (\nabla_x F_{10}), \\
R (F_{11} \partial_t q) &= (\partial_t q) (R F_{11}) + (\partial_{\xi} \partial_t q) \Phi_0^{-1} (\nabla_x F_{11}),
\end{align*}
where again we used that the $F_{1j}$ are $u$-independent and polynomials in $x$ of degree $1$, so the $\nabla_x F_{1j}$ are $\bx$-independent.
As for the only summand with $k=2$, by repeatedly applying Lemma \ref{lem:a2s}\ref{en:a2s_uind}, we see that
\[
R (F_{20}q) = q (R F_{20}) + (\partial_{\xi} q) \Phi_0^{-1} \widetilde{\nabla_x F_{20}},
\]
and
\[\begin{split}
R^{2} (F_{20}q) &= q (R^{2} F_{20}) + (\partial_{\xi} q) \Phi_0^{-1} (\nabla_x R F_{20}) \\
&\quad+ \sum_\ell (\partial_{\xi_\ell} q) R \ldbrack\Phi_0^{-1} \widetilde{\nabla_x F_{20}}\rdbrack_\ell \\
&\quad+ \sum_{\ell,m} (\partial_{\xi_m}\partial_{\xi_\ell} q)  \ldbrack\Phi_0^{-1} \nabla_x\ldbrack\Phi_0^{-1} \widetilde{\nabla_x F_{20}}\rdbrack_\ell\rdbrack_m.
\end{split}\]
By plugging these formulas into \eqref{eq:opLambda_nound} and rearranging the summands according to the orders of $t$- and $\xi$-differentiation of $q$, we obtain the desired expressions for the coefficients $\Lambda_{jr}$.
\end{proof}

The proof of Proposition \ref{prp:Lambdaop_2step} is thus reduced to obtaining more explicit expressions for the terms \eqref{eq:opLambdacoeff_prelim}.
To this purpose, we first compute the coefficients $F_{kj}$.

\begin{lem}\label{lem:Fkj_2step}
With the notation of Corollary \ref{cor:flow_alt},
the coefficients $F_{kj}$ of \eqref{eq:Fkj} are given by
\begin{equation}\label{eq:Fkj_2step}
\begin{aligned}
F_{20} &= -\frac{|\mu|^2}{4|\xi|^2} \langle \exp(-2i\theta |J_{\bar\mu}|) \left(|J_{\bar\mu}|+i J_{\bar\mu}\right) \xi, x-x^t \rangle^2,\\
F_{11} &= 2|\xi| + \frac{i |\mu|}{|\xi|} \langle \exp(-2i\theta |J_{\bar\mu}|) \left(|J_{\bar\mu}|+iJ_{\bar\mu}\right) \xi, x-x^t \rangle, \\
F_{10} &= \frac{|\mu|}{2} \frac{\theta \langle |J_{\bar\mu}| \bar\xi,\bar\xi \rangle^2}{1+i\theta\langle |J_{\bar\mu}| \bar\xi,\bar\xi \rangle} \\
&\qquad+ \frac{|\mu|^2}{4|\xi|^2} \Biggl[ \left( \tr|J_{\bar\mu}| - \frac{\langle |J_{\bar\mu}| \bar\xi,\bar\xi \rangle}{1+i\theta\langle |J_{\bar\mu}| \bar\xi,\bar\xi \rangle} \right) \langle \exp(-2i\theta|J_{\bar\mu}|) \left(|J_{\bar\mu}|+iJ_{\bar\mu}\right)\xi,x-x^t\rangle \\
&\qquad\qquad\qquad+2 \langle |J_{\bar\mu}| \exp(-2i\theta|J_{\bar\mu}|) \left(|J_{\bar\mu}|+iJ_{\bar\mu}\right)\xi,x-x^t\rangle \Biggr],\\ 
F_{02} &= 1,\\
F_{01} &= -\frac{i|\mu|}{2|\xi|} \left(\tr|J_{\bar\mu}| - \frac{\langle |J_{\bar\mu}| \bar\xi,\bar\xi \rangle}{1+i\theta\langle |J_{\bar\mu}| \bar\xi,\bar\xi \rangle} \right),\\
F_{00} &= -\frac{|\mu|^2}{16|\xi|^2} \left(\tr^2|J_{\bar\mu}| - 2\tr|J_{\bar\mu}| \frac{\langle |J_{\bar\mu}| \bar\xi,\bar\xi \rangle}{1+i\theta\langle |J_{\bar\mu}| \bar\xi,\bar\xi \rangle} -\left(\frac{\langle |J_{\bar\mu}| \bar\xi,\bar\xi \rangle}{1+i\theta\langle |J_{\bar\mu}| \bar\xi,\bar\xi \rangle}\right)^2 \right).
\end{aligned}
\end{equation}
\end{lem}
\begin{proof}
By \eqref{eq:ximux} and \eqref{eq:coshcos},
\begin{align*}
\beta^t &\defeq \xi^t-\frac{i}{2} |J_{\mu}| x^t = \exp(\theta J_{\bar\mu}) \exp(-i\theta|J_{\bar\mu}|) \xi,\\
\gamma^t &\defeq \xi^t + \frac{|\mu|}{2} J_{\bar\mu} x^t = \exp(2\theta J_{\bar\mu}) \xi.
\end{align*}
In particular, by \eqref{eq:theta_mu_xi} and \eqref{eq:change_exp},
\[\begin{split}
\partial_t \beta^t 
&= -i\frac{|\mu|}{2|\xi|} \exp(\theta J_{\bar\mu}) \exp(-i\theta|J_{\bar\mu}|) \left(|J_{\bar\mu}|+iJ_{\bar\mu}\right) \xi \\
&= -i\frac{|\mu|}{2|\xi|}  \exp(-2i\theta|J_{\bar\mu}|) \left(|J_{\bar\mu}|+iJ_{\bar\mu}\right) \xi.
\end{split}\]

Now, from \eqref{eq:phase} it follows that
\begin{equation}\label{eq:dt_phi_2step}
\begin{split}
\partial_t \phi 
&= -\dot\bx^t \cdot \bxi^t + \langle \partial_t \beta^t , x-x^t \rangle\\
&= -|\xi|-i\frac{|\mu|}{2|\xi|} \langle \exp(-2i\theta|J_{\bar\mu}|) \left(|J_{\bar\mu}|+iJ_{\bar\mu}\right) \xi, x-x^t\rangle.
\end{split}
\end{equation}
where we also used \eqref{eq:der_t_phi_s}. In light of \eqref{eq:Fkj}, this proves the formula for $F_{11}$ in \eqref{eq:Fkj_2step}.

\smallskip

Notice also that from \eqref{eq:phase} and \eqref{eq:Xfields} it follows that
\begin{equation}\label{eq:nablaphi_2step}
\nabla^\opL \phi = \xi^t + \frac{|\mu|}{2} J_{\bar\mu} x + \frac{i}{2} |J_{\mu}| (x-x^t) = \gamma^t + i \frac{|\mu|}{2} (|J_{\bar\mu}|-iJ_{\bar\mu})(x-x^t),
\end{equation}
thus
\begin{equation}\label{eq:nablaphi2_2step}
\begin{split}
(\nabla^\opL \phi) \cdot (\nabla^\opL \phi) 
&= |\gamma^t|^2 + i|\mu| \langle (|J_{\bar\mu}|-iJ_{\bar\mu})(x-x^t),\gamma^t \rangle \\
&\qquad- \frac{|\mu|^2}{4} \langle (|J_{\bar\mu}|-iJ_{\bar\mu})(x-x^t), (|J_{\bar\mu}|+iJ_{\bar\mu})(x-x^t) \rangle \\
&= |\xi|^2 + i|\mu| \langle (|J_{\bar\mu}|+iJ_{\bar\mu})\gamma^t,x-x^t \rangle \\
&\qquad- \frac{|\mu|^2}{4} \langle (|J_{\bar\mu}|+iJ_{\bar\mu})(|J_{\bar\mu}|-iJ_{\bar\mu})(x-x^t), x-x^t \rangle \\
&= |\xi|^2 + i|\mu| \langle \exp(-2i\theta|J_{\bar\mu}|) \left(|J_{\bar\mu}|+iJ_{\bar\mu}\right)\xi,x-x^t \rangle ,
\end{split}
\end{equation}
where we used that $v \cdot w = \langle v,\overline{w} \rangle$ (see Section \ref{ss:notation}), $|J_{\bar\mu}| \pm i J_{\bar\mu}$ is selfadjoint, $(|J_{\bar\mu}| + iJ_{\bar\mu})(|J_{\bar\mu}|-iJ_{\bar\mu}) = 0$ and
\[
(|J_{\bar\mu}|+iJ_{\bar\mu}) \gamma^t = \exp(-2i\theta|J_{\bar\mu}|) (|J_{\bar\mu}|+iJ_{\bar\mu}) \xi
\]
by \eqref{eq:fctnJ_2step}.

\smallskip

From \eqref{eq:dt_phi_2step} and \eqref{eq:nablaphi2_2step} it is clear that $(\partial_t \phi)^2$ and $\nabla^\opL \phi \cdot \nabla^\opL \phi$ have the same constant and linear parts in $x-x^t$, whence we obtain the formula for $F_{20} = (\partial_t \phi)^2 - \nabla^\opL \phi \cdot \nabla^\opL \phi$ in \eqref{eq:Fkj_2step}.

\smallskip

Recall now from \eqref{eq:det_2step} that
\[
\det\Phi_0 = \exp(-i\theta\tr|J_{\bar\mu}|) (1+i\theta \langle |J_{\bar\mu}| \bar\xi, \bar\xi \rangle ).
\]
So
\begin{align*}
\partial_t \det\Phi_0 &= -i \frac{|\mu|}{2|\xi|} \left[\tr|J_{\bar\mu}| (1+i\theta\langle|J_{\bar\mu}|\bar\xi,\bar\xi\rangle)-\langle|J_{\bar\mu}|\bar\xi,\bar\xi\rangle\right] \exp(-i\theta \tr|J_{\bar\mu}|),\\
\partial_t^2 \det\Phi_0 &= - \tr|J_{\bar\mu}| \frac{|\mu|^2}{4|\xi|^2} \left[\tr|J_{\bar\mu}| (1+i\theta\langle|J_{\bar\mu}|\bar\xi,\bar\xi\rangle)-2\langle|J_{\bar\mu}|\bar\xi,\bar\xi\rangle\right] \exp(-i\theta\tr|J_{\bar\mu}|).
\end{align*}
Consequently, by \eqref{eq:Fkj}, as $\Den_\phi = \sqrt{\det\Phi_0}$,
\begin{equation}\label{eq:F01_2step}
F_{01} = \frac{\partial_t \det\Phi_0}{\det\Phi_0} = -i \frac{|\mu|}{2|\xi|} \left(\tr|J_{\bar\mu}|-\frac{\langle|J_{\bar\mu}|\bar\xi,\bar\xi\rangle}{1+i\theta\langle|J_{\bar\mu}|\bar\xi,\bar\xi\rangle}\right) 
\end{equation}
and
\[
\frac{\partial_t^2 \det\Phi_0}{\det\Phi_0} = - \tr|J_{\bar\mu}| \frac{|\mu|^2}{4|\xi|^2} \left(\tr|J_{\bar\mu}| -\frac{2\langle|J_{\bar\mu}|\bar\xi,\bar\xi\rangle}{1+i\theta\langle|J_{\bar\mu}|\bar\xi,\bar\xi\rangle}\right).
\]
Moreover, as $\Den_\phi = \sqrt{\det\Phi_0}$ is $\bx$-independent, 
\[\begin{split}
F_{00} 
&= \frac{\partial_t^2 \Den_\phi}{\Den_\phi} = \frac{1}{2} \frac{\partial_t^2 \det\Phi_0}{\det \Phi_0} - \frac{1}{4} \left( \frac{\partial_t \det\Phi_0}{\det\Phi_0} \right)^2 \\
&= -\frac{|\mu|^2}{16|\xi|^2} \left(\tr^2 |J_{\bar\mu}| - 2\tr|J_{\bar\mu}| \frac{\langle|J_{\bar\mu}|\bar\xi,\bar\xi\rangle}{1+i\theta\langle|J_{\bar\mu}|\bar\xi,\bar\xi\rangle} - \left(\frac{\langle|J_{\bar\mu}|\bar\xi,\bar\xi\rangle}{(1+i\theta\langle|J_{\bar\mu}|\bar\xi,\bar\xi\rangle}\right)^2\right).
\end{split}\]
This proves the formulas for $F_{01}$ and $F_{00}$ in \eqref{eq:Fkj_2step}.

Now, from \eqref{eq:dt_phi_2step} we also deduce
\[\begin{split}
\partial_t^2 \phi 
&= -\frac{|\mu|^2}{2|\xi|^2} \langle |J_{\bar\mu}| \exp(-2i\theta|J_{\bar\mu}|) \left(|J_{\bar\mu}|+iJ_{\bar\mu}\right) \xi, x-x^t \rangle \\
&\qquad + i \frac{|\mu|}{2|\xi|} \langle \exp(-2i\theta|J_{\bar\mu}|) \left(|J_{\bar\mu}|+i J_{\bar\mu}\right) \xi, \exp(2\theta J_{\bar\mu}) \bar\xi \rangle \\
&= i\frac{|\mu|}{2} \langle |J_{\bar\mu}| \bar\xi,\bar\xi \rangle - \frac{|\mu|^2}{2|\xi|^2}  \langle |J_{\bar\mu}| \exp(-2i\theta|J_{\bar\mu}|) \left(|J_{\bar\mu}|+iJ_{\bar\mu}\right) \xi, x-x^t \rangle ,
\end{split}\]
where we used \eqref{eq:fctnJ_2step} and the fact that, by \eqref{eq:flow_2step},
\[
\dot x^t = \exp(2\theta J_{\bar\mu}) \bar \xi.
\]
Moreover,
\begin{equation}\label{eq:div_modJbarmu}
\Div_x \left[\left( |J_{\bar\mu}| + i J_{\bar\mu} \right) x\right] = \tr |J_{\bar\mu}|,
\end{equation}
because $\tr J_{\bar\mu} = 0$ by skewadjointness; so, by \eqref{eq:nablaphi_2step} and \eqref{eq:Xfields},
\[
\opL \phi = -\Div_x \nabla^\opL \phi = -i \frac{ |\mu|}{2} \tr|J_{\bar\mu}|.
\]
Thus,
\begin{equation}\label{eq:dalphi_2step}
\begin{split}
(\partial_t^2 + \opL) \phi 
&= -i \frac{|\mu|}{2} (\tr|J_{\bar\mu}| - \langle|J_{\bar\mu}| \bar\xi,\bar\xi \rangle) \\
&\qquad -\frac{|\mu|^2}{2|\xi|^2} \langle |J_{\bar\mu}| \exp(-2i\theta|J_{\bar\mu}|) \left(|J_{\bar\mu}|+iJ_{\bar\mu}\right) \xi, x-x^t \rangle.
\end{split}
\end{equation}
Furthermore, by \eqref{eq:dt_phi_2step} and \eqref{eq:F01_2step},
\begin{equation}\label{eq:dtphi_dtDetphi_2step}
\begin{split}
&\partial_t \phi \frac{\partial_t \det\Phi_0}{\det \Phi_0}  \\
&= \left(-|\xi|-i\frac{|\mu|}{2|\xi|} \langle \exp(-2i\theta|J_{\bar\mu}|) \left(|J_{\bar\mu}|+iJ_{\bar\mu}\right) \xi, x-x^t\rangle \right)\\
&\qquad\times \left(-i \frac{|\mu|}{2|\xi|} \left(\tr|J_{\bar\mu}|-\frac{\langle|J_{\bar\mu}|\bar\xi,\bar\xi\rangle}{1+i\theta\langle|J_{\bar\mu}|\bar\xi,\bar\xi\rangle}\right) \right) \\
&= i \frac{|\mu|}{2} \left(\tr|J_{\bar\mu}|-\frac{\langle|J_{\bar\mu}|\bar\xi,\bar\xi\rangle}{1+i\theta\langle|J_{\bar\mu}|\bar\xi,\bar\xi\rangle}\right) \\
&\qquad- \frac{|\mu|^2}{4|\xi|^2} \left(\tr|J_{\bar\mu}|-\frac{\langle|J_{\bar\mu}|\bar\xi,\bar\xi\rangle}{1+i\theta\langle|J_{\bar\mu}|\bar\xi,\bar\xi\rangle}\right) \langle \exp(-2i\theta|J_{\bar\mu}|) \left(|J_{\bar\mu}|+iJ_{\bar\mu}\right) \xi, x-x^t\rangle
\end{split}
\end{equation}

Now, as $\Den_\phi = \sqrt{\det\Phi_0}$ is $\bx$-independent, from \eqref{eq:Fkj} we deduce that
\[\begin{split}
-F_{10} 
&= (\partial_t^2 + \opL) \phi + \partial_t \phi \frac{\partial_t \det\Phi_0}{\det \Phi_0} \\
&= -\frac{|\mu|}{2} \frac{\theta \langle |J_{\bar\mu}|\bar\xi,\bar\xi\rangle^2}{1+i\theta\langle|J_{\bar\mu}|\bar\xi,\bar\xi\rangle}\\
&\quad- \frac{|\mu|^2}{4|\xi|^2} \Biggl[ \left(\tr|J_{\bar\mu}|-\frac{\langle|J_{\bar\mu}|\bar\xi,\bar\xi\rangle}{1+i\theta\langle|J_{\bar\mu}|\bar\xi,\bar\xi\rangle}\right) \langle \exp(-2i\theta|J_{\bar\mu}|) \left(|J_{\bar\mu}|+iJ_{\bar\mu}\right) \xi, x-x^t\rangle \\
&\qquad\qquad\qquad+2 \langle |J_{\bar\mu}| \exp(-2i\theta|J_{\bar\mu}|) \left(|J_{\bar\mu}|+iJ_{\bar\mu}\right) \xi, x-x^t \rangle \Biggr], \\
\end{split}\]
where we used \eqref{eq:dalphi_2step} and \eqref{eq:dtphi_dtDetphi_2step}. As we already know that $F_{02} = 1$, this completes the proof of the formulas \eqref{eq:Fkj_2step}.
\end{proof}

\begin{rem}
The formulas \eqref{eq:Fkj_2step} are consistent with the vanishing of second order of $F_{20}$ at $\bx = \bx^t$, as well as the value $2|\xi|$ for $F_{10}$ at $\bx=\bx^t$, already obtained in Lemma \ref{lem:crucial_coeff}. The same is true for the vanishing of $K = F_{10}+RF_{20}$ at $\bx=\bx^t$ in \eqref{eq:K_2step} below.
\end{rem}

\begin{lem}
With the notation of Lemma \ref{lem:Fkj_2step},
the coefficient $K$ from \eqref{eq:coeff_K} is given by
\begin{equation}\label{eq:K_2step}
\begin{split}
K &= 
\frac{|\mu|^2}{8|\xi|^2} \langle \exp(-2i\theta|J_{\bar\mu}|) \left(|J_{\bar\mu}|+iJ_{\bar\mu}\right)\xi,x-x^t\rangle\\
&\quad\times \Biggl[ \tr|J_{\bar\mu}| - \frac{\tr|J_{\bar\mu}|-2\langle |J_{\bar\mu}| \bar\xi,\bar\xi \rangle}{1+i\theta\langle |J_{\bar\mu}| \bar\xi,\bar\xi \rangle}   
 + i\theta  \frac{ 2\langle |J_{\bar\mu}|^2 \bar\xi,\bar\xi\rangle  - 3\langle |J_{\bar\mu}| \bar\xi, \bar\xi \rangle^2 }{(1+i\theta \langle |J_{\bar\mu}| \bar\xi, \bar\xi \rangle)^2} 
 \Biggr].
\end{split}
\end{equation}
\end{lem}
\begin{proof}
From \eqref{eq:Fkj_2step} we deduce
\[
\nabla_x F_{20} =
-\frac{|\mu|^2}{2|\xi|^2} \langle \exp(-2i\theta |J_{\bar\mu}|) \left(|J_{\bar\mu}|+i J_{\bar\mu}\right) \xi, x-x^t \rangle \exp(-2i\theta |J_{\bar\mu}|) \left(|J_{\bar\mu}|+i J_{\bar\mu}\right) \xi;
\]
thus, as the last expression is linear in $x-x^t$,
\begin{equation}\label{eq:dxF20_2step}
\begin{split}
\widetilde{\nabla_x F_{20}} &= \frac{1}{2} \nabla_x F_{20} \\
&= -\frac{|\mu|^2}{4|\xi|^2} \langle \exp(-2i\theta |J_{\bar\mu}|) \left(|J_{\bar\mu}|+i J_{\bar\mu}\right) \xi, x-x^t \rangle \exp(-2i\theta |J_{\bar\mu}|) \left(|J_{\bar\mu}|+i J_{\bar\mu}\right) \xi.
\end{split}
\end{equation}
Recall now that, by Lemma \ref{lem:phder}\ref{en:phder_hess_nondeg},
\[
\Phi_0 = \exp(\theta J_{\bar\mu}) \exp(-i\theta|J_{\bar\mu}|) \left( I + i \theta 	\left(|J_{\bar\mu}|+i J_{\bar\mu} \right)\bar\xi \, \bar\xi^T \right),
\]
thus
\[
\Phi_0^{-1} =  \left(I-i\frac{\theta}{1+i\theta \langle|J_{\bar\mu}| \bar\xi,\bar\xi\rangle} \left(|J_{\bar\mu}| +iJ_{\bar\mu}\right) \bar\xi \, \bar\xi^T \right) \exp(i\theta|J_{\bar\mu}|) \exp(-\theta J_{\bar\mu}),
\]
where we used that
\[
(I- vw^T)^{-1} = I + (1- v \cdot w)^{-1} vw^T \qquad\text{whenever } v \cdot w \neq 1.
\]
So, by \eqref{eq:fctnJ_2step},
\begin{equation}\label{eq:Phiinvxi_2step}
\begin{split}
&\Phi_0^{-1} \exp(-2i\theta |J_{\bar\mu}|) \left(|J_{\bar\mu}|+i J_{\bar\mu}\right) \xi \\
&= \left(I-i\frac{\theta}{1+i\theta \langle|J_{\bar\mu}| \bar\xi,\bar\xi\rangle} \left(|J_{\bar\mu}| +iJ_{\bar\mu}\right) \bar\xi \, \bar\xi^T \right) \left(|J_{\bar\mu}|+i J_{\bar\mu}\right) \xi  \\
&= \frac{1}{1+i\theta\langle|J_{\bar\mu}| \bar\xi,\bar\xi\rangle}  \left(|J_{\bar\mu}| +iJ_{\bar\mu}\right) \xi,
\end{split}
\end{equation}
where we also used that $\bar\xi^T \left(|J_{\bar\mu}| +iJ_{\bar\mu}\right) \xi = |\xi| \langle |J_{\bar\mu}| \bar\xi, \bar\xi\rangle$. Combining \eqref{eq:dxF20_2step} and \eqref{eq:Phiinvxi_2step}, we obtain
\begin{equation}\label{eq:PhiinvdxF20_2step}
\begin{split}
&\Phi_0^{-1} \widetilde{\nabla_x F_{20}} \\
&= - \frac{|\mu|^2}{4|\xi|^2} \frac{1}{1+i\theta\langle|J_{\bar\mu}| \bar\xi,\bar\xi\rangle} \langle \exp(-2i\theta |J_{\bar\mu}|) \left(|J_{\bar\mu}|+i J_{\bar\mu}\right) \xi, x-x^t \rangle \left(|J_{\bar\mu}|+iJ_{\bar\mu}\right) \xi.
\end{split}
\end{equation}

Now, from \eqref{eq:derxitheta} we deduce that, for all $N \in \NN$,
\begin{equation}\label{eq:nablaxi_langlerangle}
\begin{aligned}
\nabla_\xi \langle |J_{\bar\mu}|^N \bar\xi, \bar\xi \rangle &= \frac{2}{|\xi|} \left( |J_{\bar\mu}|^N \bar\xi  - \langle |J_{\bar\mu}|^N \bar\xi, \bar\xi \rangle \bar\xi \right), \\
\nabla_\xi (i\theta \langle |J_{\bar\mu}|^N \bar\xi, \bar\xi \rangle) &= \frac{i\theta}{|\xi|} \left( 2|J_{\bar\mu}|^N \bar\xi  - 3\langle |J_{\bar\mu}|^N \bar\xi, \bar\xi \rangle \bar\xi \right);
\end{aligned}
\end{equation}
this, together with \eqref{eq:det_2step} and \eqref{eq:derxitheta}, gives that
\[
\begin{split}
&\nabla_\xi \det\Phi_0 \\
&= \frac{i\theta}{|\xi|} \exp(-i\theta\tr|J_{\bar\mu}|) \left[ (\tr|J_{\bar\mu}|) (1+i\theta \langle |J_{\bar\mu}| \bar\xi, \bar\xi \rangle ) \bar\xi + 2 |J_{\bar\mu}| \bar\xi - 3 \langle |J_{\bar\mu}| \bar\xi,\bar\xi\rangle \bar\xi \, \right] 
\end{split}
\]
and
\[
\frac{\nabla_\xi \det\Phi_0}{\det\Phi_0} 
= \frac{i\theta}{|\xi|} \left((\tr|J_{\bar\mu}|) \bar\xi + \frac{2|J_{\bar\mu}|\bar\xi - 3\langle|J_{\bar\mu}|\bar\xi,\bar\xi\rangle \bar\xi}{1+i\theta \langle|J_{\bar\mu}|\bar\xi,\bar\xi\rangle}\right).
\]
Thus
\begin{equation}\label{eq:dxi_detPhi_detPhi_2step}
\frac{\partial_\xi \det\Phi_0}{\det\Phi_0} \left(|J_{\bar\mu}| + iJ_{\bar\mu}\right) \xi 
= i\theta \left((\tr|J_{\bar\mu}|) \langle |J_{\bar\mu}| \bar\xi, \bar\xi\rangle + \frac{2\langle|J_{\bar\mu}|^2 \bar\xi,\bar\xi\rangle - 3\langle|J_{\bar\mu}|\bar\xi,\bar\xi\rangle^2}{1+i\theta \langle|J_{\bar\mu}|\bar\xi,\bar\xi\rangle}\right),
\end{equation}
where we used that $J_{\bar\mu} |J_{\bar\mu}|$ is skew-symmetric; so, by \eqref{eq:PhiinvdxF20_2step},
\begin{equation}\label{eq:dxiPhiPhiinvdxF20_2step}
\begin{split}
&\frac{1}{2}\frac{\partial_\xi \det\Phi_0}{\det\Phi_0} \Phi_0^{-1} \widetilde{\nabla_x F_{20}} \\
&= - \frac{|\mu|^2}{8|\xi|^2} \frac{i\theta}{1+i\theta\langle|J_{\bar\mu}| \bar\xi,\bar\xi\rangle} \langle \exp(-2i\theta |J_{\bar\mu}|) \left(|J_{\bar\mu}|+i J_{\bar\mu}\right) \xi, x-x^t \rangle\\
&\quad\times \left((\tr|J_{\bar\mu}|) \langle |J_{\bar\mu}| \bar\xi, \bar\xi\rangle + \frac{2\langle|J_{\bar\mu}|^2 \bar\xi,\bar\xi\rangle - 3\langle|J_{\bar\mu}|\bar\xi,\bar\xi\rangle^2}{1+i\theta \langle|J_{\bar\mu}|\bar\xi,\bar\xi\rangle}\right).
\end{split}
\end{equation}

Now, arguing as in \eqref{eq:div_modJbarmu}, for all $N \in \NN$,
\begin{equation}\label{eq:divxixiNJmuJmuxi}
\Div_\xi \left[|\xi|^{-N} \left(|J_{\bar\mu}| + i J_\mu\right) \xi\right] = |\xi|^{-N} \left[ \tr(|J_{\bar\mu}|) -N \langle |J_{\bar\mu}| \bar\xi, \bar\xi \rangle \right]
\end{equation}
and, by \eqref{eq:nablaxi_langlerangle} (for $N=1$),
\[
\nabla_\xi [(1+i\theta \langle |J_{\bar\mu}| \bar\xi, \bar\xi \rangle)^{-N}]
= - N\frac{i\theta}{|\xi|} \frac{ 2|J_{\bar\mu}| \bar\xi  - 3\langle |J_{\bar\mu}| \bar\xi, \bar\xi \rangle \bar\xi }{(1+i\theta \langle |J_{\bar\mu}| \bar\xi, \bar\xi \rangle)^{1+N}},
\]
thus
\begin{equation}\label{eq:nablaxi1ithetaJmuxixiJmuJmuxi}
[\nabla_\xi[ (1+i\theta \langle |J_{\bar\mu}| \bar\xi, \bar\xi \rangle)^{-N}]] \cdot \left(|J_{\bar\mu}| + i J_{\bar\mu}\right) \xi
 = -Ni\theta \frac{ 2\langle |J_{\bar\mu}|^2 \bar\xi,\bar\xi\rangle  - 3\langle |J_{\bar\mu}| \bar\xi, \bar\xi \rangle^2 }{(1+i\theta \langle |J_{\bar\mu}| \bar\xi, \bar\xi \rangle)^{1+N}}.
\end{equation}
In addition, by \eqref{eq:flow_2step} and \eqref{eq:derxitheta},
\[
\partial_\xi x^t = \frac{1}{|\mu|} \frac{\exp(2\theta J_{\bar\mu})-I}{J_{\bar\mu}} - \frac{2\theta}{|\mu|} \exp(2\theta J_{\bar\mu}) \bar\xi \, \bar\xi^T,
\]
so, by \eqref{eq:fctnJ_2step},
\[\begin{split}
&(\partial_\xi x^t)^T \exp(-2i\theta |J_{\bar\mu}|) \left(|J_{\bar\mu}|+i J_{\bar\mu}\right) \xi \\
&= -\frac{i}{|\mu|} \frac{I-\exp(-2i\theta |J_{\bar\mu}|)}{|J_{\bar\mu}|} \left(|J_{\bar\mu}|+i J_{\bar\mu}\right) \xi 
- \frac{2\theta}{|\mu|} \langle |J_{\bar\mu}| \bar\xi,\bar\xi \rangle \xi
\end{split}\]
and, again by \eqref{eq:derxitheta},
\[\begin{split}
&\nabla_\xi \langle \exp(-2i\theta |J_{\bar\mu}|) \left(|J_{\bar\mu}|+i J_{\bar\mu}\right) \xi, x-x^t \rangle \\
&= \frac{2i\theta}{|\xi|} \langle |J_{\bar\mu}| \exp(-2i\theta |J_{\bar\mu}|) \left(|J_{\bar\mu}|+i J_{\bar\mu}\right) \xi, x-x^t \rangle \bar\xi \\
&\quad+ \exp(-2i\theta |J_{\bar\mu}|) \left(|J_{\bar\mu}|-i J_{\bar\mu}\right) (x-x^t) \\
&\quad +\frac{i}{|\mu|} \frac{I-\exp(-2i\theta |J_{\bar\mu}|)}{|J_{\bar\mu}|} \left(|J_{\bar\mu}|+i J_{\bar\mu}\right) \xi 
+ \frac{2\theta}{|\mu|} \langle |J_{\bar\mu}| \bar\xi,\bar\xi \rangle \xi.
\end{split}\]
Thus,
\begin{equation}\label{eq:nablaxiexp2ithetaJmuJmuJmuxixxtJmuJmuxi}
\begin{split}
&(\nabla_\xi \langle \exp(-2i\theta |J_{\bar\mu}|) \left(|J_{\bar\mu}|+i J_{\bar\mu}\right) \xi, x-x^t \rangle) \cdot \left( |J_{\bar\mu}| + iJ_{\bar\mu}\right) \xi \\ 
&= 2\langle |J_{\bar\mu}| \exp(-2i\theta |J_{\bar\mu}|) \left(|J_{\bar\mu}|+i J_{\bar\mu}\right) \xi, x-x^t \rangle (1+i\theta \langle |J_{\bar\mu}| \bar\xi,\bar\xi \rangle) 
 + 2\theta \frac{|\xi|^2}{|\mu|}  \langle |J_{\bar\mu}| \bar\xi,\bar\xi \rangle^2,
\end{split}
\end{equation}
where we used that
\begin{equation}\label{eq:square_JmuJmu}
(|J_{\bar\mu}|+iJ_{\bar\mu})^2=2|J_{\bar\mu}| (|J_{\bar\mu}|+iJ_{\bar\mu}), \qquad (|J_{\bar\mu}|+iJ_{\bar\mu})(|J_{\bar\mu}|-iJ_{\bar\mu}) = 0.
\end{equation}

So, from \eqref{eq:PhiinvdxF20_2step}, \eqref{eq:divxixiNJmuJmuxi}, \eqref{eq:nablaxi1ithetaJmuxixiJmuJmuxi} and \eqref{eq:nablaxiexp2ithetaJmuJmuJmuxixxtJmuJmuxi} we deduce that
\begin{equation}\label{eq:divxiPhiinvdxF20_2step}
\begin{split}
&\Div_{\xi} (\Phi_0^{-1} \widetilde{\nabla_x F_{20}}) \\
&= - \frac{|\mu|^2}{4} \frac{1}{1+i\theta\langle|J_{\bar\mu}| \bar\xi,\bar\xi\rangle} \\
&\qquad\qquad \times \langle \exp(-2i\theta |J_{\bar\mu}|) \left(|J_{\bar\mu}|+i J_{\bar\mu}\right) \xi, x-x^t \rangle \Div_\xi [|\xi|^{-2} \left(|J_{\bar\mu}|+iJ_{\bar\mu}\right) \xi] \\
&\quad - \frac{|\mu|^2}{4|\xi|^2}  \langle \exp(-2i\theta |J_{\bar\mu}|) \left(|J_{\bar\mu}|+i J_{\bar\mu}\right) \xi, x-x^t \rangle \\
&\qquad\qquad\times [\nabla_\xi [(1+i\theta\langle|J_{\bar\mu}| \bar\xi,\bar\xi\rangle)^{-1}]] \cdot \left(|J_{\bar\mu}|+iJ_{\bar\mu}\right) \xi \\
&\quad- \frac{|\mu|^2}{4|\xi|^2} \frac{1}{1+i\theta\langle|J_{\bar\mu}| \bar\xi,\bar\xi\rangle} [\nabla_\xi \langle \exp(-2i\theta |J_{\bar\mu}|) \left(|J_{\bar\mu}|+i J_{\bar\mu}\right) \xi, x-x^t \rangle] \cdot \left(|J_{\bar\mu}|+iJ_{\bar\mu}\right) \xi \\
&= - \frac{|\mu|^2}{4|\xi|^2} \frac{1}{1+i\theta\langle|J_{\bar\mu}| \bar\xi,\bar\xi\rangle} \\
&\quad\times \Biggl[ \langle \exp(-2i\theta |J_{\bar\mu}|) \left(|J_{\bar\mu}|+i J_{\bar\mu}\right) \xi, x-x^t \rangle \left( \tr|J_{\bar\mu}| -2 \langle |J_{\bar\mu}| \bar\xi, \bar\xi \rangle \right) \\
&\qquad -i\theta  \langle \exp(-2i\theta |J_{\bar\mu}|) \left(|J_{\bar\mu}|+i J_{\bar\mu}\right) \xi, x-x^t \rangle \frac{ 2\langle |J_{\bar\mu}|^2 \bar\xi,\bar\xi\rangle  - 3\langle |J_{\bar\mu}| \bar\xi, \bar\xi \rangle^2 }{1+i\theta \langle |J_{\bar\mu}| \bar\xi, \bar\xi \rangle} \\
&\qquad+ 2\langle |J_{\bar\mu}| \exp(-2i\theta |J_{\bar\mu}|) \left(|J_{\bar\mu}|+i J_{\bar\mu}\right) \xi, x-x^t \rangle (1+i\theta \langle |J_{\bar\mu}| \bar\xi,\bar\xi \rangle) \\
&\qquad+ 2\theta \frac{|\xi|^2}{|\mu|}  \langle |J_{\bar\mu}| \bar\xi,\bar\xi \rangle^2\Biggr].
\end{split}
\end{equation}

Hence, by Lemma \ref{lem:a2s}\ref{en:a2s_uind},
\[\begin{split}
&RF_{20} \\
&= \frac{1}{2}\frac{\partial_\xi \det\Phi_0}{\det\Phi_0} \Phi_0^{-1} \widetilde{\nabla_x F_{20}} + \Div_\xi (\Phi_0^{-1} {\widetilde{\nabla_{x} F_{20}}} ) \\
&=- \frac{|\mu|^2}{8|\xi|^2} \frac{1}{1+i\theta\langle|J_{\bar\mu}| \bar\xi,\bar\xi\rangle} 
 \langle \exp(-2i\theta |J_{\bar\mu}|) \left(|J_{\bar\mu}|+i J_{\bar\mu}\right) \xi, x-x^t \rangle \\
&\qquad\times \left( (\tr|J_{\bar\mu}|) (2+i\theta\langle |J_{\bar\mu}| \bar\xi, \bar\xi\rangle) -4 \langle |J_{\bar\mu}| \bar\xi, \bar\xi \rangle  
 - i\theta  \frac{ 2\langle |J_{\bar\mu}|^2 \bar\xi,\bar\xi\rangle  - 3\langle |J_{\bar\mu}| \bar\xi, \bar\xi \rangle^2 }{1+i\theta \langle |J_{\bar\mu}| \bar\xi, \bar\xi \rangle} \right) \\
&- \frac{|\mu|^2}{2|\xi|^2} \langle |J_{\bar\mu}| \exp(-2i\theta |J_{\bar\mu}|) \left(|J_{\bar\mu}|+i J_{\bar\mu}\right) \xi, x-x^t \rangle 
- \frac{|\mu|}{2} \frac{\theta \langle |J_{\bar\mu}| \bar\xi,\bar\xi \rangle^2}{1+i\theta\langle|J_{\bar\mu}| \bar\xi,\bar\xi\rangle}  
\end{split}\]
where we used \eqref{eq:dxiPhiPhiinvdxF20_2step} and \eqref{eq:divxiPhiinvdxF20_2step}. As $K=F_{10}+RF_{20}$, by combining this with the formula for $F_{10}$ in \eqref{eq:Fkj_2step} we obtain \eqref{eq:K_2step}.
\end{proof}

We now proceed to the proof of the formulas \eqref{eq:opLambda_coeff_2step}.

\begin{proof}[Proof of Proposition \ref{prp:Lambdaop_2step}]
We start by proving the formula for the coefficient $\Lambda_{00}$.
By \eqref{eq:K_2step},
\[\begin{split}
\widetilde{\nabla_x K} = \nabla_x K &=
\frac{|\mu|^2}{8|\xi|^2} \Biggl[ \tr|J_{\bar\mu}| - \frac{\tr|J_{\bar\mu}|-2\langle |J_{\bar\mu}| \bar\xi,\bar\xi \rangle}{1+i\theta\langle |J_{\bar\mu}| \bar\xi,\bar\xi \rangle}   
 + i\theta  \frac{ 2\langle |J_{\bar\mu}|^2 \bar\xi,\bar\xi\rangle  - 3\langle |J_{\bar\mu}| \bar\xi, \bar\xi \rangle^2 }{(1+i\theta \langle |J_{\bar\mu}| \bar\xi, \bar\xi \rangle)^2} 
 \Biggr] \\
&\quad\times \exp(-2i\theta|J_{\bar\mu}|) \left(|J_{\bar\mu}|+iJ_{\bar\mu}\right)\xi.
\end{split}\]
So, by \eqref{eq:Phiinvxi_2step},
\begin{equation}\label{eq:PhiinvdxK_2step}
\begin{split}
\Phi_0^{-1} \widetilde{\nabla_x K} 
&= \frac{|\mu|^2}{8|\xi|^2} \Biggl[ \tr|J_{\bar\mu}| - \frac{\tr|J_{\bar\mu}|-2\langle |J_{\bar\mu}| \bar\xi,\bar\xi \rangle}{1+i\theta\langle |J_{\bar\mu}| \bar\xi,\bar\xi \rangle}   
 + i\theta  \frac{ 2\langle |J_{\bar\mu}|^2 \bar\xi,\bar\xi\rangle  - 3\langle |J_{\bar\mu}| \bar\xi, \bar\xi \rangle^2 }{(1+i\theta \langle |J_{\bar\mu}| \bar\xi, \bar\xi \rangle)^2} 
 \Biggr] \\
&\quad\times \frac{1}{1+i\theta\langle|J_{\bar\mu}| \bar\xi,\bar\xi\rangle}  \left(|J_{\bar\mu}| +iJ_{\bar\mu}\right) \xi \\
&= \frac{|\mu|^2}{8|\xi|^2} A \left(|J_{\bar\mu}| +iJ_{\bar\mu}\right) \xi,
\end{split}
\end{equation}
where
\begin{equation}\label{eq:A}
A \defeq \frac{\tr|J_{\bar\mu}|}{1+i\theta\langle|J_{\bar\mu}| \bar\xi,\bar\xi\rangle} - \frac{\tr|J_{\bar\mu}|-2\langle |J_{\bar\mu}| \bar\xi,\bar\xi \rangle}{(1+i\theta\langle |J_{\bar\mu}| \bar\xi,\bar\xi \rangle)^2}   
 + i\theta  \frac{ 2\langle |J_{\bar\mu}|^2 \bar\xi,\bar\xi\rangle  - 3\langle |J_{\bar\mu}| \bar\xi, \bar\xi \rangle^2 }{(1+i\theta \langle |J_{\bar\mu}| \bar\xi, \bar\xi \rangle)^3} .
\end{equation}

Now, by \eqref{eq:dxi_detPhi_detPhi_2step} and \eqref{eq:PhiinvdxK_2step},
\[
\begin{split}
&\frac{1}{2} \frac{\partial_\xi \det\Phi_0}{\det\Phi_0} \Phi_0^{-1} \widetilde{\nabla_x K} \\
&= \frac{|\mu|^2}{16|\xi|^2} \Biggl[ \tr|J_{\bar\mu}| - \frac{\tr|J_{\bar\mu}|-2\langle |J_{\bar\mu}| \bar\xi,\bar\xi \rangle}{1+i\theta\langle |J_{\bar\mu}| \bar\xi,\bar\xi \rangle}   
 + i\theta  \frac{ 2\langle |J_{\bar\mu}|^2 \bar\xi,\bar\xi\rangle  - 3\langle |J_{\bar\mu}| \bar\xi, \bar\xi \rangle^2 }{(1+i\theta \langle |J_{\bar\mu}| \bar\xi, \bar\xi \rangle)^2} 
 \Biggr] \\
&\qquad\times \frac{i\theta}{1+i\theta\langle|J_{\bar\mu}|\bar\xi,\bar\xi\rangle} \left((\tr|J_{\bar\mu}|) \langle |J_{\bar\mu}| \bar\xi, \bar\xi\rangle + \frac{2\langle|J_{\bar\mu}|^2 \bar\xi,\bar\xi\rangle - 3\langle|J_{\bar\mu}|\bar\xi,\bar\xi\rangle^2}{1+i\theta \langle|J_{\bar\mu}|\bar\xi,\bar\xi\rangle}\right)\\
&= \frac{|\mu|^2}{16|\xi|^2} \Biggl[ \tr|J_{\bar\mu}| - \frac{\tr|J_{\bar\mu}|-2\langle |J_{\bar\mu}| \bar\xi,\bar\xi \rangle}{1+i\theta\langle |J_{\bar\mu}| \bar\xi,\bar\xi \rangle}   
 + i\theta  \frac{ 2\langle |J_{\bar\mu}|^2 \bar\xi,\bar\xi\rangle  - 3\langle |J_{\bar\mu}| \bar\xi, \bar\xi \rangle^2 }{(1+i\theta \langle |J_{\bar\mu}| \bar\xi, \bar\xi \rangle)^2} 
 \Biggr] \\
&\qquad\times \left(\tr|J_{\bar\mu}| - \frac{\tr|J_{\bar\mu}|}{1+i\theta \langle |J_{\bar\mu}| \bar\xi, \bar\xi\rangle} + i\theta \frac{2\langle|J_{\bar\mu}|^2 \bar\xi,\bar\xi\rangle - 3\langle|J_{\bar\mu}|\bar\xi,\bar\xi\rangle^2}{(1+i\theta \langle|J_{\bar\mu}|\bar\xi,\bar\xi\rangle)^2}\right) \\
&= \frac{|\mu|^2}{16|\xi|^2} \Biggl[ 
\tr^2 |J_{\bar\mu}| 
- 2 \tr|J_{\bar\mu}| \frac{\tr|J_{\bar\mu}|-\langle |J_{\bar\mu}| \bar\xi,\bar\xi \rangle}{1+i\theta\langle |J_{\bar\mu}| \bar\xi,\bar\xi \rangle} \\
&\qquad+ \tr|J_{\bar\mu}| \frac{\tr|J_{\bar\mu}|
-2\langle |J_{\bar\mu}| \bar\xi,\bar\xi \rangle}{(1+i\theta\langle |J_{\bar\mu}| \bar\xi,\bar\xi \rangle)^2} 
+2i\theta \tr|J_{\bar\mu}| \frac{2\langle|J_{\bar\mu}|^2 \bar\xi,\bar\xi\rangle - 3\langle|J_{\bar\mu}|\bar\xi,\bar\xi\rangle^2}{(1+i\theta \langle|J_{\bar\mu}|\bar\xi,\bar\xi\rangle)^2}  \\
&\qquad-2 i\theta (\tr|J_{\bar\mu}|-\langle |J_{\bar\mu}| \bar\xi, \bar\xi \rangle) \frac{ 2\langle |J_{\bar\mu}|^2 \bar\xi,\bar\xi\rangle  - 3\langle |J_{\bar\mu}| \bar\xi, \bar\xi \rangle^2 }{(1+i\theta \langle |J_{\bar\mu}| \bar\xi, \bar\xi \rangle)^3} \\
&\qquad+ (i\theta)^2 \frac{ (2\langle |J_{\bar\mu}|^2 \bar\xi,\bar\xi\rangle  - 3\langle |J_{\bar\mu}| \bar\xi, \bar\xi \rangle^2)^2 }{(1+i\theta \langle |J_{\bar\mu}| \bar\xi, \bar\xi \rangle)^4} \Biggr], 
\end{split}
\]
and a few manipulations give
\begin{equation}\label{eq:dxidetPhiPhiinvdxK_2step}
\begin{split}
&\frac{1}{2} \frac{\partial_\xi \det\Phi_0}{\det\Phi_0} \Phi_0^{-1} \widetilde{\nabla_x K} \\
&= \frac{|\mu|^2}{16|\xi|^2} \Biggl[ 
\tr^2 |J_{\bar\mu}| 
- 2 \tr|J_{\bar\mu}| \frac{\tr|J_{\bar\mu}|-\langle |J_{\bar\mu}| \bar\xi,\bar\xi \rangle}{1+i\theta\langle |J_{\bar\mu}| \bar\xi,\bar\xi \rangle} \\
&\qquad+2i\theta \tr|J_{\bar\mu}| \frac{2\langle|J_{\bar\mu}|^2 \bar\xi,\bar\xi\rangle - 3\langle|J_{\bar\mu}|\bar\xi,\bar\xi\rangle^2}{(1+i\theta \langle|J_{\bar\mu}|\bar\xi,\bar\xi\rangle)^2}  \\
&\qquad+ \tr|J_{\bar\mu}| \frac{\tr|J_{\bar\mu}|-2\langle |J_{\bar\mu}| \bar\xi,\bar\xi \rangle}{(1+i\theta\langle |J_{\bar\mu}| \bar\xi,\bar\xi \rangle)^2} 
+2\frac{ 2\langle |J_{\bar\mu}|^2 \bar\xi,\bar\xi\rangle - 3\langle |J_{\bar\mu}| \bar\xi, \bar\xi \rangle^2 }{(1+i\theta \langle |J_{\bar\mu}| \bar\xi, \bar\xi \rangle)^2}  \\
&\qquad-2 i\theta \tr|J_{\bar\mu}| \frac{ 2\langle |J_{\bar\mu}|^2 \bar\xi,\bar\xi\rangle - 3\langle |J_{\bar\mu}| \bar\xi, \bar\xi \rangle^2 }{(1+i\theta \langle |J_{\bar\mu}| \bar\xi, \bar\xi \rangle)^3} \\
&\qquad- 2 \frac{ 2\langle |J_{\bar\mu}|^2 \bar\xi,\bar\xi\rangle - 3\langle |J_{\bar\mu}| \bar\xi, \bar\xi \rangle^2 }{(1+i\theta \langle |J_{\bar\mu}| \bar\xi, \bar\xi \rangle)^3} 
+ (i\theta)^2 \frac{ (2\langle |J_{\bar\mu}|^2 \bar\xi,\bar\xi\rangle - 3\langle |J_{\bar\mu}| \bar\xi, \bar\xi \rangle^2)^2 }{(1+i\theta \langle |J_{\bar\mu}| \bar\xi, \bar\xi \rangle)^4} \Biggr] .
\end{split}
\end{equation}

Furthermore,
from \eqref{eq:A} it follows that
\begin{equation}\label{eq:dxiA}
\begin{split}
&\nabla_\xi A \cdot (|J_{\bar\mu}| + iJ_{\bar\mu}) \xi \\
&= \tr|J_{\bar\mu}| \nabla_\xi (1+i\theta\langle|J_{\bar\mu}| \bar\xi,\bar\xi\rangle)^{-1} \cdot (|J_{\bar\mu}| + iJ_{\bar\mu}) \xi\\
&\quad- (\tr|J_{\bar\mu}|-2\langle |J_{\bar\mu}| \bar\xi,\bar\xi \rangle) \nabla_\xi (1+i\theta\langle |J_{\bar\mu}| \bar\xi,\bar\xi \rangle)^{-2} \cdot (|J_{\bar\mu}| + iJ_{\bar\mu}) \xi\\
&\quad+ i\theta  (2\langle |J_{\bar\mu}|^2 \bar\xi,\bar\xi\rangle  - 3\langle |J_{\bar\mu}| \bar\xi, \bar\xi \rangle^2) \nabla_\xi(1+i\theta \langle |J_{\bar\mu}| \bar\xi, \bar\xi \rangle)^{-3} \cdot (|J_{\bar\mu}| + iJ_{\bar\mu}) \xi\\
&\quad+ 2(1+i\theta\langle |J_{\bar\mu}| \bar\xi,\bar\xi \rangle)^{-2}  \nabla_\xi \langle |J_{\bar\mu}| \bar\xi,\bar\xi \rangle \cdot (|J_{\bar\mu}| + iJ_{\bar\mu}) \xi \\
&\quad+ (1+i\theta \langle |J_{\bar\mu}| \bar\xi, \bar\xi \rangle)^{-3}  \nabla_\xi (2i\theta\langle |J_{\bar\mu}|^2 \bar\xi,\bar\xi\rangle  - 3i\theta\langle |J_{\bar\mu}| \bar\xi, \bar\xi \rangle^2) \cdot (|J_{\bar\mu}| + iJ_{\bar\mu}) \xi .
\end{split}
\end{equation}
Now, from \eqref{eq:nablaxi_langlerangle} we deduce that
\begin{equation}\label{eq:dxi2m3}
\begin{split}
&\nabla_\xi \big(2i\theta \langle |J_{\bar\mu}|^2 \bar\xi,\bar\xi\rangle  - 3i\theta \langle |J_{\bar\mu}| \bar\xi, \bar\xi \rangle^2\big) \cdot (|J_{\bar\mu}| + iJ_{\bar\mu}) \xi\\
&= \frac{i\theta}{|\xi|} \left( 4|J_{\bar\mu}|^2 \bar\xi  - 6\langle |J_{\bar\mu}|^2 \bar\xi, \bar\xi \rangle \bar\xi 
-12 \langle |J_{\bar\mu}| \bar\xi, \bar\xi \rangle |J_{\bar\mu}| \bar\xi  +15 \langle |J_{\bar\mu}| \bar\xi, \bar\xi \rangle^2 \bar\xi \right) \cdot (|J_{\bar\mu}| + iJ_{\bar\mu}) \xi\\
&= i\theta \left( 4\langle|J_{\bar\mu}|^3 \bar\xi,\bar\xi\rangle  - 18\langle |J_{\bar\mu}|^2 \bar\xi, \bar\xi \rangle \langle |J_{\bar\mu}| \bar\xi,\bar\xi \rangle 
 +15 \langle |J_{\bar\mu}| \bar\xi, \bar\xi \rangle^3  \right)
\end{split}
\end{equation}
and
\begin{equation}\label{eq:dxi0}
\nabla_\xi \langle |J_{\bar\mu}| \bar\xi, \bar\xi \rangle \cdot (|J_{\bar\mu}| + iJ_{\bar\mu}) \xi
= 2\left( \langle |J_{\bar\mu}|^2 \bar\xi,\bar\xi\rangle  - \langle |J_{\bar\mu}| \bar\xi, \bar\xi \rangle^2 \right).
\end{equation}
So, starting from \eqref{eq:dxiA}, by applying \eqref{eq:nablaxi1ithetaJmuxixiJmuJmuxi} with $N=1,2,3$ as well as \eqref{eq:dxi2m3} and \eqref{eq:dxi0}, we see that
\[
\begin{split}
&\nabla_\xi A \cdot (|J_{\bar\mu}| + iJ_{\bar\mu}) \xi \\
&= -i\theta \tr|J_{\bar\mu}| \frac{ 2\langle |J_{\bar\mu}|^2 \bar\xi,\bar\xi\rangle  - 3\langle |J_{\bar\mu}| \bar\xi, \bar\xi \rangle^2 }{(1+i\theta \langle |J_{\bar\mu}| \bar\xi, \bar\xi \rangle)^{2}}\\
&\quad+2i\theta \tr|J_{\bar\mu}| \frac{ 2\langle |J_{\bar\mu}|^2 \bar\xi,\bar\xi\rangle  - 3\langle |J_{\bar\mu}| \bar\xi, \bar\xi \rangle^2 }{(1+i\theta \langle |J_{\bar\mu}| \bar\xi, \bar\xi \rangle)^{3}} 
- 4i\theta \langle |J_{\bar\mu}|\bar\xi,\bar\xi\rangle \frac{ 2\langle |J_{\bar\mu}|^2 \bar\xi,\bar\xi\rangle  - 3\langle |J_{\bar\mu}| \bar\xi, \bar\xi \rangle^2 }{(1+i\theta \langle |J_{\bar\mu}| \bar\xi, \bar\xi \rangle)^{3}} \\
&\quad-3 (i\theta)^2 \frac{  (2\langle |J_{\bar\mu}|^2 \bar\xi,\bar\xi\rangle  - 3\langle |J_{\bar\mu}| \bar\xi, \bar\xi \rangle^2)^2  }{(1+i\theta \langle |J_{\bar\mu}| \bar\xi, \bar\xi \rangle)^{4}} \\
&\quad+ 4 \frac{ \langle |J_{\bar\mu}|^2 \bar\xi,\bar\xi\rangle  - \langle |J_{\bar\mu}| \bar\xi, \bar\xi \rangle^2 }{(1+i\theta\langle |J_{\bar\mu}| \bar\xi,\bar\xi \rangle)^{2}} \\
&\quad+ i\theta \frac{ 4\langle|J_{\bar\mu}|^3 \bar\xi,\bar\xi\rangle  - 18\langle |J_{\bar\mu}|^2 \bar\xi, \bar\xi \rangle \langle |J_{\bar\mu}| \bar\xi,\bar\xi \rangle  +15 \langle |J_{\bar\mu}| \bar\xi, \bar\xi \rangle^3 }{(1+i\theta \langle |J_{\bar\mu}| \bar\xi, \bar\xi \rangle)^{3}} ,
\end{split}
\]
and the terms in the above expression can be rearranged to give
\begin{equation}\label{eq:dxAfinal}
\begin{split}
&\nabla_\xi A \cdot (|J_{\bar\mu}| + iJ_{\bar\mu}) \xi \\
&= -i\theta \tr|J_{\bar\mu}| \frac{ 2\langle |J_{\bar\mu}|^2 \bar\xi,\bar\xi\rangle  - 3\langle |J_{\bar\mu}| \bar\xi, \bar\xi \rangle^2 }{(1+i\theta \langle |J_{\bar\mu}| \bar\xi, \bar\xi \rangle)^{2}}\\
&\quad+ 4 \frac{ \langle |J_{\bar\mu}|^2 \bar\xi,\bar\xi\rangle  - \langle |J_{\bar\mu}| \bar\xi, \bar\xi \rangle^2 }{(1+i\theta\langle |J_{\bar\mu}| \bar\xi,\bar\xi \rangle)^{2}} \\
&\quad+2i\theta \tr|J_{\bar\mu}| \frac{ 2\langle |J_{\bar\mu}|^2 \bar\xi,\bar\xi\rangle  - 3\langle |J_{\bar\mu}| \bar\xi, \bar\xi \rangle^2 }{(1+i\theta \langle |J_{\bar\mu}| \bar\xi, \bar\xi \rangle)^{3}} \\
&\quad+ i\theta \frac{ 4\langle|J_{\bar\mu}|^3 \bar\xi,\bar\xi\rangle  - 26\langle |J_{\bar\mu}|^2 \bar\xi, \bar\xi \rangle \langle |J_{\bar\mu}| \bar\xi,\bar\xi \rangle  +27 \langle |J_{\bar\mu}| \bar\xi, \bar\xi \rangle^3 }{(1+i\theta \langle |J_{\bar\mu}| \bar\xi, \bar\xi \rangle)^{3}} \\
&\quad-3 (i\theta)^2 \frac{  (2\langle |J_{\bar\mu}|^2 \bar\xi,\bar\xi\rangle  - 3\langle |J_{\bar\mu}| \bar\xi, \bar\xi \rangle^2)^2  }{(1+i\theta \langle |J_{\bar\mu}| \bar\xi, \bar\xi \rangle)^{4}}.
\end{split}
\end{equation}
On the other hand, by \eqref{eq:A} and \eqref{eq:divxixiNJmuJmuxi},
\begin{equation}\label{eq:Adxi}
\begin{split}
&A \Div_{\xi} (|\xi|^{-2} \left(|J_{\bar\mu}| +iJ_{\bar\mu}\right) \xi) = \frac{1}{|\xi|^2} \left( \tr|J_{\bar\mu}| - 2 \langle |J_{\bar\mu}|\bar\xi,\bar\xi\rangle \right) \\
&\times \left[ \frac{\tr|J_{\bar\mu}|}{1+i\theta\langle|J_{\bar\mu}| \bar\xi,\bar\xi\rangle} - \frac{\tr|J_{\bar\mu}|-2\langle |J_{\bar\mu}| \bar\xi,\bar\xi \rangle}{(1+i\theta\langle |J_{\bar\mu}| \bar\xi,\bar\xi \rangle)^2}   
 + i\theta  \frac{ 2\langle |J_{\bar\mu}|^2 \bar\xi,\bar\xi\rangle  - 3\langle |J_{\bar\mu}| \bar\xi, \bar\xi \rangle^2 }{(1+i\theta \langle |J_{\bar\mu}| \bar\xi, \bar\xi \rangle)^3}\right] .
\end{split}
\end{equation}

Now, by \eqref{eq:PhiinvdxK_2step} and the Leibniz rule,
\[
\Div_\xi(\Phi_0^{-1} \widetilde{\nabla_x K} ) = \frac{|\mu|^2}{8|\xi|^2} (\nabla_\xi A) \cdot \left(|J_{\bar\mu}| +iJ_{\bar\mu}\right) \xi
+ \frac{|\mu|^2}{8} A \Div_{\xi} (|\xi|^{-2} \left(|J_{\bar\mu}| +iJ_{\bar\mu}\right) \xi);
\]
therefore, by combining \eqref{eq:dxAfinal} and \eqref{eq:Adxi}, we deduce that
\begin{equation}\label{eq:divxiPhiinvdxK_2step}
\begin{split}
&\Div_\xi(\Phi_0^{-1} \widetilde{\nabla_x K} ) \\
&= \frac{|\mu|^2}{8|\xi|^2} \Biggl[
\tr|J_{\bar\mu}| \frac{\tr|J_{\bar\mu}| -2 \langle |J_{\bar\mu}| \bar\xi, \bar\xi \rangle}{1+i\theta\langle|J_{\bar\mu}| \bar\xi,\bar\xi\rangle} 
-i\theta \tr|J_{\bar\mu}| \frac{ 2\langle |J_{\bar\mu}|^2 \bar\xi,\bar\xi\rangle  - 3\langle |J_{\bar\mu}| \bar\xi, \bar\xi \rangle^2 }{(1+i\theta \langle |J_{\bar\mu}| \bar\xi, \bar\xi \rangle)^{2}}\\
&\quad- \frac{(\tr|J_{\bar\mu}|-2\langle |J_{\bar\mu}| \bar\xi,\bar\xi \rangle)^2}{(1+i\theta\langle |J_{\bar\mu}| \bar\xi,\bar\xi \rangle)^2}  
+ 4 \frac{ \langle |J_{\bar\mu}|^2 \bar\xi,\bar\xi\rangle  - \langle |J_{\bar\mu}| \bar\xi, \bar\xi \rangle^2 }{(1+i\theta\langle |J_{\bar\mu}| \bar\xi,\bar\xi \rangle)^{2}} \\
&\quad+3i\theta \tr|J_{\bar\mu}| \frac{ 2\langle |J_{\bar\mu}|^2 \bar\xi,\bar\xi\rangle  - 3\langle |J_{\bar\mu}| \bar\xi, \bar\xi \rangle^2 }{(1+i\theta \langle |J_{\bar\mu}| \bar\xi, \bar\xi \rangle)^{3}} \\
&\quad+ i\theta \frac{ 4\langle|J_{\bar\mu}|^3 \bar\xi,\bar\xi\rangle  - 30\langle |J_{\bar\mu}|^2 \bar\xi, \bar\xi \rangle \langle |J_{\bar\mu}| \bar\xi,\bar\xi \rangle  +33 \langle |J_{\bar\mu}| \bar\xi, \bar\xi \rangle^3 }{(1+i\theta \langle |J_{\bar\mu}| \bar\xi, \bar\xi \rangle)^{3}} \\
&\quad-3 (i\theta)^2 \frac{  (2\langle |J_{\bar\mu}|^2 \bar\xi,\bar\xi\rangle  - 3\langle |J_{\bar\mu}| \bar\xi, \bar\xi \rangle^2)^2  }{(1+i\theta \langle |J_{\bar\mu}| \bar\xi, \bar\xi \rangle)^{4}} \Biggr] .
\end{split}
\end{equation}

Finally, since
\[
RK = \frac{1}{2} \frac{\partial_\xi \det\Phi_0}{\det\Phi_0} \Phi_0^{-1} \widetilde{\nabla_x K} + \Div_\xi(\Phi_0^{-1} \widetilde{\nabla_x K})
\]
by Lemma \ref{lem:a2s}\ref{en:a2s_uind},
we can use \eqref{eq:dxidetPhiPhiinvdxK_2step} and \eqref{eq:divxiPhiinvdxK_2step} to deduce that
\begin{equation}\label{eq:RK_2step}
\begin{split}
RK 
&= \frac{|\mu|^2}{16|\xi|^2} \Biggl[ \tr^2|J_{\bar\mu}|-2 \tr|J_{\bar\mu}| \frac{\langle |J_{\bar\mu}| \bar\xi,\bar\xi \rangle}{1+i\theta\langle|J_{\bar\mu}|\bar\xi,\bar\xi\rangle} \\
&\quad+ \frac{12\langle|J_{\bar\mu}|^2 \bar\xi,\bar\xi\rangle - 22\langle|J_{\bar\mu}|\bar\xi,\bar\xi\rangle^2 - \tr^2|J_{\bar\mu}| + 6 \langle|J_{\bar\mu}|\bar\xi,\bar\xi\rangle \tr|J_{\bar\mu}|}{(1+i\theta\langle|J_{\bar\mu}|\bar\xi,\bar\xi\rangle)^2}\\
&\quad+ 4i\theta \tr|J_{\bar\mu}| \frac{2\langle|J_{\bar\mu}|^2\bar\xi,\bar\xi\rangle-3\langle|J_{\bar\mu}|\bar\xi,\bar\xi\rangle^2}{(1+i\theta\langle|J_{\bar\mu}|\bar\xi,\bar\xi\rangle)^3}\\
&\quad+ 2i\theta \frac{4\langle|J_{\bar\mu}|^3\bar\xi,\bar\xi\rangle-30\langle|J_{\bar\mu}|^2\bar\xi,\bar\xi\rangle \langle|J_{\bar\mu}|\bar\xi,\bar\xi\rangle+33\langle|J_{\bar\mu}|\bar\xi,\bar\xi\rangle^3}{(1+i\theta\langle|J_{\bar\mu}|\bar\xi,\bar\xi\rangle)^3}\\
&\quad- 5(i\theta)^2 \frac{(2\langle|J_{\bar\mu}|^2\bar\xi,\bar\xi\rangle-3\langle|J_{\bar\mu}|\bar\xi,\bar\xi\rangle^2)^2}{(1+i\theta\langle|J_{\bar\mu}|\bar\xi,\bar\xi\rangle)^4} 
- 2 \frac{2\langle|J_{\bar\mu}|^2\bar\xi,\bar\xi\rangle-3\langle|J_{\bar\mu}|\bar\xi,\bar\xi\rangle^2}{(1+i\theta\langle|J_{\bar\mu}|\bar\xi,\bar\xi\rangle)^3}
\Biggr].
\end{split}
\end{equation}

As $\Lambda_{00} = F_{00} + RK$ by \eqref{eq:opLambdacoeff_prelim}, by combining \eqref{eq:RK_2step} with the expression for $F_{00}$ in \eqref{eq:Fkj_2step}, we obtain 
the formula for $\Lambda_{00}$ in \eqref{eq:opLambda_coeff_2step}.

\smallskip

We proceed now with the formulas for the other coefficients.
From \eqref{eq:Fkj_2step} we deduce that
\[
\widetilde{\nabla_x F_{11}} = \nabla_x F_{11} 
= \frac{i |\mu|}{|\xi|} \exp(-2i\theta |J_{\bar\mu}|) \left(|J_{\bar\mu}|+iJ_{\bar\mu}\right) \xi,
\]
so, by \eqref{eq:opLambdacoeff_prelim} and \eqref{eq:Phiinvxi_2step},
\begin{equation}\label{eq:Lambda11_2step}
\Lambda_{11} = \Phi_0^{-1} \nabla_x F_{11} = \frac{i|\mu|}{|\xi|} \frac{1}{1+i\theta\langle|J_{\bar\mu}| \bar\xi,\bar\xi\rangle}  \left(|J_{\bar\mu}| +iJ_{\bar\mu}\right) \xi,
\end{equation}
which proves the formula for $\Lambda_{11}$ in \eqref{eq:opLambda_coeff_2step}.

\smallskip

Now, 
by \eqref{eq:dxi_detPhi_detPhi_2step} and \eqref{eq:Lambda11_2step},
\[\begin{split}
&\frac{1}{2} \frac{\partial_\xi \det\Phi_0}{\det\Phi_0} \Phi_0^{-1} \widetilde{\nabla_x F_{11}} \\
&=
i\frac{|\mu|}{2|\xi|} 
\frac{i\theta}{1+i\theta\langle|J_{\bar\mu}| \bar\xi,\bar\xi\rangle} 
 \left((\tr|J_{\bar\mu}|) \langle |J_{\bar\mu}| \bar\xi, \bar\xi\rangle + \frac{2\langle|J_{\bar\mu}|^2 \bar\xi,\bar\xi\rangle - 3\langle|J_{\bar\mu}|\bar\xi,\bar\xi\rangle^2}{1+i\theta \langle|J_{\bar\mu}|\bar\xi,\bar\xi\rangle}\right).
\end{split}\]
Moreover, 
from \eqref{eq:Lambda11_2step}, 
\eqref{eq:divxixiNJmuJmuxi}, \eqref{eq:nablaxi1ithetaJmuxixiJmuJmuxi}
 and \eqref{eq:derxitheta} we also deduce that
\[\begin{split}
\Div_{\xi} (\Phi_0^{-1} \widetilde{\nabla_x F_{11}})
&= \frac{i|\mu|}{|\xi|} [\nabla_{\xi} [(1+i\theta\langle|J_{\bar\mu}|\bar\xi,\bar\xi\rangle)^{-1}]] \cdot \left(|J_{\bar\mu}| +iJ_{\bar\mu}\right) \xi\\
&\qquad+ i|\mu| (1+i\theta\langle|J_{\bar\mu}|\bar\xi,\bar\xi\rangle)^{-1} \Div_\xi [|\xi|^{-1} \left(|J_{\bar\mu}| +iJ_{\bar\mu}\right) \xi] \\
&= \frac{i|\mu|}{|\xi|} \left[\frac{ \tr|J_{\bar\mu}| - \langle |J_{\bar\mu}| \bar\xi, \bar\xi \rangle }{1+i\theta \langle |J_{\bar\mu}| \bar\xi, \bar\xi \rangle} -i\theta \frac{ 2\langle |J_{\bar\mu}|^2 \bar\xi,\bar\xi\rangle  - 3\langle |J_{\bar\mu}| \bar\xi, \bar\xi \rangle^2 }{(1+i\theta \langle |J_{\bar\mu}| \bar\xi, \bar\xi \rangle)^{2}} \right].
\end{split}\]
Thus, by Lemma \ref{lem:a2s}\ref{en:a2s_uind},
\begin{equation}\label{eq:RF11_2step}
\begin{split}
RF_{11} &= \frac{1}{2} \frac{\partial_\xi \det\Phi_0}{\det\Phi_0} \Phi_0^{-1} \widetilde{\nabla_x F_{11}} + \Div_\xi(\Phi_0^{-1} \widetilde{\nabla_x F_{11}}) \\
&=  \frac{i|\mu|}{2|\xi|} \left[ \tr|J_{\bar\mu}| + \frac{ \tr|J_{\bar\mu}| - 2\langle |J_{\bar\mu}| \bar\xi, \bar\xi \rangle }{1+i\theta \langle |J_{\bar\mu}| \bar\xi, \bar\xi \rangle} -i\theta \frac{ 2\langle |J_{\bar\mu}|^2 \bar\xi,\bar\xi\rangle  - 3\langle |J_{\bar\mu}| \bar\xi, \bar\xi \rangle^2 }{(1+i\theta \langle |J_{\bar\mu}| \bar\xi, \bar\xi \rangle)^{2}} \right].
\end{split}
\end{equation}
As $\Lambda_{10} = F_{01}+RF_{11}$ by \eqref{eq:opLambdacoeff_prelim}, by combining \eqref{eq:RF11_2step} with the formula for $F_{01}$ in \eqref{eq:Fkj_2step}, we obtain the formula for $\Lambda_{10}$ in \eqref{eq:opLambda_coeff_2step}.

Recall the notation $\ldbrack v \rdbrack_j$ for the $j$th component of a vector $v$. Now, from \eqref{eq:PhiinvdxF20_2step} we deduce that, for all $j=1,\dots,d_1$,
\[\begin{split}
&\nabla_x \ldbrack\Phi_0^{-1} \widetilde{\nabla_x F_{20}}\rdbrack_j  \\
&= - \frac{|\mu|^2}{4|\xi|^2} \frac{1}{1+i\theta\langle|J_{\bar\mu}| \bar\xi,\bar\xi\rangle} \ldbrack(|J_{\bar\mu}|+iJ_{\bar\mu}) \xi\rdbrack_j \exp(-2i\theta |J_{\bar\mu}|) \left(|J_{\bar\mu}|+i J_{\bar\mu}\right) \xi;
\end{split}\]
thus, by \eqref{eq:Phiinvxi_2step},
\begin{equation}\label{eq:PhiinvdxPhiinvdxF20_2step}
\Phi_0^{-1} \nabla_x \ldbrack\Phi_0^{-1} \widetilde{\nabla_x F_{20}}\rdbrack_j 
 = - \frac{|\mu|^2}{4|\xi|^2} \frac{1}{(1+i\theta\langle|J_{\bar\mu}| \bar\xi,\bar\xi\rangle)^2} \ldbrack(|J_{\bar\mu}|+iJ_{\bar\mu}) \xi\rdbrack_j 
 \left(|J_{\bar\mu}| +iJ_{\bar\mu}\right) \xi.
\end{equation}
As $\Lambda_{02} = (\ldbrack\Phi_0^{-1} \nabla_x \ldbrack\Phi_0^{-1} \widetilde{\nabla_x F_{20}}\rdbrack_j\rdbrack_k)_{j,k=1}^{d_1}$ by \eqref{eq:opLambdacoeff_prelim}, from the expression \eqref{eq:PhiinvdxPhiinvdxF20_2step} we deduce the formula for $\Lambda_{02}$ in \eqref{eq:opLambda_coeff_2step}.

\smallskip

Notice now that, by \eqref{eq:dxi_detPhi_detPhi_2step} and \eqref{eq:PhiinvdxPhiinvdxF20_2step}, for all $j=1,\dots,d_1$,
\begin{equation}\label{eq:dxidetPhidetPhiPhiinvnablaxPhiinvnablaxF20j}
\begin{split}
&\frac{1}{2} \frac{\partial_\xi \det\Phi_0}{\det\Phi_0} \Phi_0^{-1} \nabla_x \ldbrack\Phi_0^{-1} \widetilde{\nabla_x F_{20}}\rdbrack_j  \\
&= - \frac{|\mu|^2}{8|\xi|^2} 
\frac{i\theta}{(1+i\theta\langle|J_{\bar\mu}| \bar\xi,\bar\xi\rangle)^2} \ldbrack(|J_{\bar\mu}|+iJ_{\bar\mu}) \xi\rdbrack_j \\
&\quad\times \left((\tr|J_{\bar\mu}|) \langle |J_{\bar\mu}| \bar\xi, \bar\xi\rangle + \frac{2\langle|J_{\bar\mu}|^2 \bar\xi,\bar\xi\rangle - 3\langle|J_{\bar\mu}|\bar\xi,\bar\xi\rangle^2}{1+i\theta \langle|J_{\bar\mu}|\bar\xi,\bar\xi\rangle}\right) .
\end{split}
\end{equation}
Furthermore, if $e_1,\dots,e_{d_1}$ is the standard basis of $\RR^{d_1}$, then
\[
\nabla_{\xi} \ldbrack(|J_{\bar\mu}|+iJ_{\bar\mu}) \xi\rdbrack_j = \nabla_{\xi} \langle (|J_{\bar\mu}|+iJ_{\bar\mu}) \xi, e_j \rangle = (|J_{\bar\mu}|-iJ_{\bar\mu}) e_j,
\]
thus, by \eqref{eq:square_JmuJmu},
\begin{equation}\label{eq:dxIJxiIJxi_2step}
\begin{split}
(\nabla_{\xi} \ldbrack\left(|J_{\bar\mu}|+iJ_{\bar\mu}\right) \xi\rdbrack_j) \cdot \left(|J_{\bar\mu}|+i J_{\bar\mu}\right) \xi 
&= \langle \left(|J_{\bar\mu}|+iJ_{\bar\mu}\right) \xi, \left(|J_{\bar\mu}|+iJ_{\bar\mu}\right) e_j \rangle \\
&= 2 \ldbrack|J_{\bar\mu}| \left(|J_{\bar\mu}|+iJ_{\bar\mu}\right) \xi\rdbrack_j.
\end{split}
\end{equation}
So, by \eqref{eq:PhiinvdxPhiinvdxF20_2step},
\eqref{eq:dxIJxiIJxi_2step},
 \eqref{eq:divxixiNJmuJmuxi}
and \eqref{eq:nablaxi1ithetaJmuxixiJmuJmuxi},
\begin{equation}\label{eq:divxiPhiinvnablaxPhiinvnablaxF20j}
\begin{split}
&\Div_\xi(\Phi_0^{-1} \nabla_x \ldbrack\Phi_0^{-1} \widetilde{\nabla_x F_{20}}\rdbrack_j) \\
&= - \frac{|\mu|^2}{4|\xi|^2} \ldbrack\left(|J_{\bar\mu}|+iJ_{\bar\mu}\right) \xi\rdbrack_j 
 \, (\nabla_\xi (1+i\theta\langle|J_{\bar\mu}| \bar\xi,\bar\xi\rangle)^{-2}) \cdot \left(|J_{\bar\mu}| +iJ_{\bar\mu}\right) \xi \\
&\quad - \frac{|\mu|^2}{4|\xi|^2} \frac{1}{(1+i\theta\langle|J_{\bar\mu}| \bar\xi,\bar\xi\rangle)^2} (\nabla_\xi \ldbrack(|J_{\bar\mu}|+iJ_{\bar\mu}) \xi\rdbrack_j )\cdot
 \left(|J_{\bar\mu}| +iJ_{\bar\mu}\right) \xi \\
&\quad - \frac{|\mu|^2}{4} \frac{1}{(1+i\theta\langle|J_{\bar\mu}| \bar\xi,\bar\xi\rangle)^2} \ldbrack(|J_{\bar\mu}|+iJ_{\bar\mu}) \xi\rdbrack_j 
 \Div_{\xi} [|\xi|^{-2} \left(|J_{\bar\mu}| +iJ_{\bar\mu}\right) \xi] \\
&= - \frac{|\mu|^2}{4|\xi|^2} \frac{1}{(1+i\theta\langle|J_{\bar\mu}| \bar\xi,\bar\xi\rangle)^2} \\
&\quad\times \Biggl[ \left[ \tr|J_{\bar\mu}| -2 \langle |J_{\bar\mu}| \bar\xi, \bar\xi \rangle  -2i\theta \frac{ 2\langle |J_{\bar\mu}|^2 \bar\xi,\bar\xi\rangle  - 3\langle |J_{\bar\mu}| \bar\xi, \bar\xi \rangle^2 }{1+i\theta \langle |J_{\bar\mu}| \bar\xi, \bar\xi \rangle} \right] \ldbrack(|J_{\bar\mu}|+iJ_{\bar\mu}) \xi\rdbrack_j \\
&\quad\quad+ 2 \ldbrack|J_{\bar\mu}| (|J_{\bar\mu}|+iJ_{\bar\mu}) \xi\rdbrack_j 
\Biggr].
\end{split}
\end{equation}
Thus, by \eqref{eq:dxidetPhidetPhiPhiinvnablaxPhiinvnablaxF20j} and \eqref{eq:divxiPhiinvnablaxPhiinvnablaxF20j},
\begin{equation}\label{eq:RPhiinvdxF20_2step}
\begin{split}
&R\ldbrack\Phi_0^{-1} \widetilde{\nabla_x F_{20}}\rdbrack_j \\
&= \frac{1}{2} \frac{\partial_\xi \det\Phi_0}{\det\Phi_0} \Phi_0^{-1} \nabla_x \ldbrack\Phi_0^{-1} \widetilde{\nabla_x F_{20}}\rdbrack_j + \Div_\xi(\Phi_0^{-1} \nabla_x \ldbrack\Phi_0^{-1} \widetilde{\nabla_x F_{20}}\rdbrack_j) \\
&= - \frac{|\mu|^2}{8|\xi|^2} 
\frac{1}{1+i\theta\langle|J_{\bar\mu}| \bar\xi,\bar\xi\rangle}\\
&\quad\times \Biggl[ \left[ \tr|J_{\bar\mu}| + \frac{\tr|J_{\bar\mu}| -4 \langle |J_{\bar\mu}| \bar\xi, \bar\xi \rangle}{1+i\theta \langle |J_{\bar\mu}| \bar\xi, \bar\xi \rangle}  -3i\theta \frac{ 2\langle |J_{\bar\mu}|^2 \bar\xi,\bar\xi\rangle  - 3\langle |J_{\bar\mu}| \bar\xi, \bar\xi \rangle^2 }{(1+i\theta \langle |J_{\bar\mu}| \bar\xi, \bar\xi \rangle)^2} \right] \\
&\qquad\qquad\times \ldbrack(|J_{\bar\mu}|+iJ_{\bar\mu}) \xi\rdbrack_j + \frac{4}{1+i\theta \langle |J_{\bar\mu}| \bar\xi, \bar\xi \rangle} \ldbrack|J_{\bar\mu}| (|J_{\bar\mu}|+iJ_{\bar\mu}) \xi\rdbrack_j 
\Biggr].
\end{split}
\end{equation}
As $\Lambda_{01} = \Phi_0^{-1} \nabla_x K + (R\ldbrack\Phi_0^{-1} \widetilde{\nabla_x F_{20}}\rdbrack_j)_{j=1}^{d_1}$ by \eqref{eq:opLambdacoeff_prelim}, from \eqref{eq:PhiinvdxK_2step} and \eqref{eq:RPhiinvdxF20_2step} we deduce the formula for $\Lambda_{01}$ in \eqref{eq:opLambda_coeff_2step}.

\smallskip

Since we already know from \eqref{eq:opLambdacoeff_prelim} that $\Lambda_{20} = 1$, this concludes the proof of the formulas \eqref{eq:opLambda_coeff_2step}.
\end{proof}

\end{document}